\documentclass[10pt,reqno]{amsart}

\usepackage{mathrsfs}
\usepackage[dvips]{graphicx}
\usepackage{amsmath,amsthm, amscd}
\usepackage[psamsfonts]{amssymb}
\usepackage[all]{xy}
\usepackage{wrapfig}

\newtheorem{Thm}{Theorem}[section]
\newtheorem{Prop}[Thm]{Proposition}
\newtheorem{Lem}[Thm]{Lemma}
\newtheorem{Cor}[Thm]{Corollary}

\newtheorem{Propt}[Thm]{Property}

\theoremstyle{remark}
\newtheorem{Rem}[Thm]{Remark}

\theoremstyle{definition}
\newtheorem{Def}[Thm]{Definition}
\newtheorem{Not}[Thm]{Notation}
\newtheorem{Assum}[Thm]{Assumption}
\newtheorem{Exa}[Thm]{Example}

\newcommand{\mysection}[2]{%
\vspace{2mm}\section{\bf #1}\label{#2}
}

\def\Z{{\mathbb Z}}
\def\R{{\mathbb R}}
\def\Q{{\mathbb Q}}

\def\calA{\mathscr{A}}

\def\calD{\mathscr{D}}
\def\calE{\mathscr{E}}

\def\calG{\mathscr{G}}

\def\calL{\mathscr{L}}
\def\calM{\mathscr{M}}

\def\calS{\mathscr{S}}

\def\deg{\mathrm{deg}}
\def\Hom{\mathrm{Hom}}
\def\tcoprod{\textstyle\coprod}

\def\Hom{\mathrm{Hom}}
\def\Aut{\mathrm{Aut}}
\def\Conf{C}
\def\bConf{\overline{C}}
\def\wConf{C^*}
\def\bwConf{\bConf^*}

\def\ve{\varepsilon}
\def\bvec#1{\mbox{\boldmath{$#1$}}}

\def\Diff{\mathrm{Diff}}
\def\wBDiff{\widetilde{B\Diff}}
\def\Emb{\mathrm{Emb}}
\def\fEmb{\Emb^{\mathrm{f}}}
\def\bcalA{\overline{\calA}}
\def\bcalD{\overline{\calD}}
\def\bcalM{\overline{\calM}}

\newcommand{\bacalM}{\bcalM\,\kern-.5mm^{\mathrm{Z}}}

\def\tcoprod{\textstyle\coprod}

\def\pbcalM{\bcalM\kern-.5mm\,'}
\def\ibcalD{\bcalD\kern-.5mm\,^\infty}
\def\ibcalA{\bcalA\kern-.5mm\,^\infty}
\def\ibcalM{\bcalM\kern-.5mm\,^\infty}

\def\Lk{\mathrm{Lk}}
\def\pr{p}
\def\Fr{\mathrm{Fr}}
\def\wgamma{\widetilde{\gamma}}
\def\even{\mathrm{even}}
\def\odd{\mathrm{odd}}
\def\dR{\mathrm{dR}}
\def\bEmb{\overline{\Emb}}
\def\st{\mathrm{st}}
\def\const{\mathrm{const}}
\def\bgamma{\overline{\gamma}}

\title[Addendum]{Addendum to: Some exotic nontrivial elements of the rational homotopy groups of $\Diff(S^4)$ (homological interpretation)}
\author{Tadayuki Watanabe}
\address{Department of Mathematics, Kyoto University,
Kyoto 606-8502, Japan}
\email{tadayuki.watanabe@math.kyoto-u.ac.jp}
\date{\today}

\begin{document}

{\noindent\footnotesize {\rm Preprint}}\par\vspace{15mm}
\maketitle
\vspace{-6mm}
\setcounter{tocdepth}{2}
\numberwithin{equation}{section}
\renewcommand{\thefootnote}{\fnsymbol{footnote}}

\begin{abstract}
In this addendum, we give a differential form interpretation of the proof of the main theorem of \cite{Wa4}, which gives lower bounds of the dimensions of $\pi_k(B\Diff(D^4,\partial))\otimes\Q$ in terms of the dimensions of Kontsevich's graph homology, and explain why it can be extended to arbitrary even dimensions $d\geq 4$. We attempted to make the proof accessible to more readers. Thus we do not assume familiarity with configuration space integrals nor knowledge of finite type invariants. Part of this addendum might be joined to the original article when it will be re-submitted to the journal. This is not aimed at giving a correction to the previous version.
\end{abstract}
\par\vspace{3mm}

\def\baselinestretch{1.06}\small\normalsize

%%%%%%%%%%%%%%%%%%%%%%%%%%%%%%%%%%%%%%%%%%%%
%%%%%%%%%%%%%%%%%%%%%%%%%%%%%%%%%%%%%%%%%%%%
\mysection{Introduction}{s:intro}

The extended result is the following. We refer the reader to \cite{Wa4} for backgrounds and consequences in 4-dimension.

\begin{Thm}[Theorem~\ref{thm:Z(G)}]\label{thm:main}
Let $d$ be an even integer such that $d\geq 4$.
For each $k\geq 1$, evaluation of Kontsevich's characteristic classes on $D^d$-bundles over $S^{(d-3)k}$ gives an epimorphism 
from $\pi_{(d-3)k}(B\Diff(D^d,\partial))\otimes\R$ to the space $\calA_k^\even\otimes\R$ of trivalent graphs (definition in \S\ref{ss:graphs}).
\end{Thm}

\begin{Rem}
Theorem~\ref{thm:main} gives no information about the mapping class group $\pi_0(\Diff(D^4,\partial))\cong \pi_1(B\Diff(D^4,\partial))$ because $\calA_1^\even=0$. The first nontrivial element is detected in $\calA_2^\even\cong \Q$ (Remark~\ref{rem:A_2}). It should be mentioned that after the first version of this paper was submitted to the arXiv, S.~Akbulut announced a proof that $\pi_0(\Diff(D^4,\partial))\neq 0$ based on his theory of corks (\cite{Ak}). Also, Budney and Gabai constructed some elements of $\pi_0(\Diff(D^4,\partial))$ explicitly in \cite[\S{5}]{BG}, and some structure of the group $\pi_0(\Diff(D^4,\partial))$ has been studied recently by D.~Gay (\cite{Ga}), Gay--Hartman (\cite{GH}), and an alternative proof of Gay's result is given by Krannich and Kupers in \cite{KK}.
\end{Rem}

\begin{Rem}
In our previous preprint \cite{Wa4}, we proved a result slightly different from Theorem~\ref{thm:main} in terms of Morse theory. 
The techniques used in this paper to prove Theorem~\ref{thm:main}, which uses differential forms, is quite different from \cite{Wa4}.
\end{Rem}

In fact, the construction needed is not different between $d=4$ and $d>4$ even. This is similar to the fact that the cocycles of $\Emb(S^1,\R^d)$ given by configuration space integrals are nontrivial for all $d\geq 4$ and $d=4$ is not exceptional there (\cite{Kon, CCL}). In earlier versions of this paper, we gave a proof of Theorem~\ref{thm:main} only for $d=4$ to simplify notations. However, we learned that some remarkable progresses on the topology of $\Diff(D^d,\partial)$ for higher even dimensions $d\geq 6$ have appeared recently (e.g., Weiss (\cite{We}), Boavida de Brito--Weiss (\cite{BdBW}), Fresse--Turchin--Willwacher, Fresse--Willwacher (\cite{FTW,FW}), Kupers--Randal-Williams (\cite{KRW})) and we thought it would be worth giving a proof of our result for arbitrary even integer $d\geq 4$. It would be very interesting to compare the results in this paper and those of \cite{We,BdBW,FTW,FW,KRW}.

Let $r\colon D^d\to D^d$ be the reflection $r(x_1,x_2,\ldots,x_d)=(-x_1,x_2,\ldots,x_d)$. The conjugation $r\circ g\circ r^{-1}$ for $g\in\Diff(D^d,\partial)$ gives an involution on $\Diff(D^d,\partial)$ which is a homomorphism, and hence an involution on $\pi_*(B\Diff(D^d,\partial))$. 
\begin{Cor}\label{cor:involution}
Let $d$ be an even integer such that $d\geq 4$.
The $(-1)^k$-eigenspace of the reflection involution in $\pi_{(d-3)k}(B\Diff(D^d,\partial))\otimes\R$ is nontrivial whenever $\calA_k^\even$ is nontrivial.
\end{Cor}
\begin{proof}
This follows from Theorem~\ref{thm:main} and Proposition~\ref{prop:involution} below. Namely, let $\pi_{(d-3)k}(B\Diff(D^d,\partial))\otimes\R=V_{(-1)^k}\oplus V_{(-1)^{k+1}}$ be the eigenspace decomposition with respect to the reflection involution. If $\xi\in V_{(-1)^{k+1}}$, then by Proposition~\ref{prop:involution}, we have $(-1)^kZ_k(\xi)=Z_k(\xi')=(-1)^{k+1}Z_k(\xi)$ and hence $Z_k(\xi)=0$. This shows that the image of $Z_k$ agrees with $Z_k(V_{(-1)^k})$.
\end{proof}
\begin{Rem}
For example, the $(+1)$-eigenspace of $\pi_{2d-6}(B\Diff(D^d,\partial))\otimes\R$ is at least one dimensional.
This is compatible with a result of Kupers and Randal-Williams (\cite[Corollary~7.15]{KRW}) that there is at least one dimensional nontrivial subspace in the $(+1)$-eigenspace of $\pi_i(B\Diff(D^d,\partial))\otimes\Q$ for some $i$ in $2d-9\leq i\leq 2d-5$ (the fourth band), $d\geq 6$ even, as pointed out in \cite{KRW}. As also pointed out in \cite[Example~6.9]{KRW}, Corollary~\ref{cor:involution} has a nontrivial consequence for the group $C(D^n)=\Diff(D^n\times I, \partial D^n\times I\cup D^n\times\{0\})$ of pseudo-isotopies. The following corollary holds since the $(+1)$-eigenspaces of $\pi_*(B\Diff(D^d,\partial))\otimes\R$ inject into $\pi_*(BC(D^{d-1}))\otimes\R$ (\cite[Example~6.9]{KRW}).
\end{Rem}
\begin{Cor}
Let $d$ be an even integer such that $d\geq 4$. If $k\geq 2$ is even and if $\calA_k^\even\neq 0$, then $\pi_{(d-3)k}(BC(D^{d-1}))\otimes\R\neq 0$.
\end{Cor}

\begin{Prop}[{\cite[Remark~7.16]{KRW}}]\label{prop:involution}
Let $d$ be an even integer such that $d\geq 4$.
For an element $\xi$ of $\pi_{(d-3)k}(B\Diff(D^d,\partial))\otimes\R$, let $\xi'$ be the element obtained from $\xi$ by the reflection involution $r$. Then we have
\[ Z_k(\xi')=(-1)^k Z_k(\xi). \]
\end{Prop}
A proof of Proposition~\ref{prop:involution} is given in Subsection~\ref{ss:csi}.

The method of this paper is essentially the same as \cite{Wa2}, where we studied the rational homotopy groups of $\Diff(D^{4k-1},\partial)$. Namely, we construct some explicit fiber bundles from trivalent graphs, by giving a higher-dimensional analogue of graph-clasper surgery, developed by Goussarov and Habiro for knots and 3-manifolds (\cite{Gou, Hab}). Then we compute the values of the characteristic numbers for the bundles, by giving a higher-dimensional analogue of Kuperberg--Thurston's computation of configuration space integrals for homology 3-spheres (\cite{KT,Les2}). Thus, what is new in this paper is to give higher-dimensional analogues of the ideas of Goussarov--Habiro and Kuperberg--Thurston so that they fit together well and to check that they indeed work. 

In earlier versions of this paper, we gave a proof of Theorem~\ref{thm:main} by means of parametrized Morse theory. We believe that the idea of the Morse theoretic proof is more straightforward and is suitable to understand why nontrivial values can be obtained, like the formula for the linking number counting crossings. However, in that proof it is unavoidable to describe thorough detailed arguments of transversality and orientation, which makes the paper surprisingly long, due to the inefficiency of the author. In this paper we attempted to make the proof accessible to more readers and gave a proof of Theorem~\ref{thm:main} by means of differential forms (or algebraic topology), as in \cite{Wa2}. It is easier for the author to write shorter proof with differential forms, though the main body of the proof is compressed into one lemma, whose proof is abstract and long. Nevertheless, the latter requires only elementary algebraic topology and we consider it convenient for most readers. 

%%%%%%%%%%%%%%%%%%%%%%%%%%%%5
\subsection{Contents of the paper}

The aim of this paper is to give a proof of Theorem~\ref{thm:main} by means of differential forms and to give a foundation of graph surgery which works for manifolds of arbitrary dimensions $\geq 3$. 
There are roughly three ingredients in this paper.
\begin{enumerate}
\item[(i)] Kontsevich's characteristic classes for framed disk bundles defined by a graph complex and configuration space integrals. This will be explained in \S\ref{s:kon}.
\item[(ii)] Surgery on ``graph claspers'', a higher dimensional analogue of Goussarov--Habiro's theory. This will be explained mainly in \S\ref{s:cycles}, and technical details are described in \S\ref{s:proof-borr}. 
\item[(iii)] That Kontsevich's configuration space integral invariants can be computed explicitly for the disk bundles constructed by graph clasper surgeries. The method for the computation is a higher dimensional analogue of Kuperberg--Thurston's computation of configuration space integrals for homology 3-spheres (\cite[Theorem~2]{KT}), for which a detailed exposition has been given by Lescop (\cite{Les2}). This will be explained in \S\ref{s:computation}, \S\ref{s:proof-localization}, \S\ref{s:localization-ii}.
\end{enumerate}

In the appendices, we will explain about the following.
\begin{enumerate}
\item[(A)] Smooth manifolds with corners.
\item[(B)] Blow-up in differentiable manifolds.
\item[(C)] Fulton--MacPherson compactification.
\item[(D)] Orientations on manifolds and on their intersections.
\item[(E)] Well-definedness of Kontsevich's characteristic class.
\item[(F)] Homology class of the diagonal.
\end{enumerate}

The readers who do not need to check the technical details for the moment can read only \S\ref{s:kon}--\ref{s:computation}.

\S\ref{s:cycles}, \ref{s:proof-borr} of this paper corresponds to \S{4} of \cite{Wa4}.
\S\ref{s:kon}, \ref{s:computation} of this paper can be considered as simplifications of \S{2}, {3}, {5} of \cite{Wa4}. Proofs of \S{4.8} of \cite{Wa4} was separated and joined to \cite{Wa3}. The correspondence is roughly as follows.
\begin{center}
\begin{tabular}{c||c|c|c|c}\hline
 This paper & \S\ref{s:kon} & \S\ref{s:cycles}, \ref{s:proof-borr} & \S\ref{s:computation}--\ref{s:localization-ii} & \\\hline
 \cite{Wa4} & \S{2} & \S{4} & \S{3}, \S{5} & \S{4.8}\\\hline
 \cite{Wa3} & & & & $\bigcirc$\\\hline
\end{tabular}
\end{center}

%%%%%%%%%%%%%%%%%%%%%%%%%%%%5
\subsection{What is different from higher odd dimensional case in \cite{Wa2}.}

As we mentioned above, the idea of the proof of Theorem~\ref{thm:main} is essentially the same as that of \cite{Wa2}, although there are some technical differences. 

In a $(2k+1)$-dimensional manifold, there is a Hopf link consisting of two unknotted round $k$-spheres, which are linked together with linking number 1 (see \S\ref{ss:borromean}). The special case $k=1$ corresponds to the Hopf link of circles in a 3-manifold, which is used in \cite{Gou,Hab,KT}. Thus many constructions in dimension 3 can be generalized to higher odd dimensions in a similar way just by replacing 1-spheres with $k$-spheres. For example, a closed surface of genus 3 is $(S^k\times S^k)\#(S^k\times S^k)\#(S^k\times S^k)$, a solid handlebody of genus 3 is the boundary connected sum $(D^{k+1}\times S^k)\,\natural\,(D^{k+1}\times S^k)\,\natural\,(D^{k+1}\times S^k)$. 

For higher even-dimensional manifolds of dimension $d$, we need to consider Hopf links with components of different dimensions, namely, a pair $p,q$ of integers such that $1\leq p<q\leq d-2$ and $p+q=d-1$. We found that we need only to consider Hopf links for a fixed pair $p,q$, say $(p,q)=(1,d-2)$, to define surgeries for all the trivalent graphs, which are given by links of handlebodies whose handles are linked along Hopf links and which are arranged along an embedded trivalent graph. Moreover, we need only to consider combinations of two types of handlebodies (type I and II) to generate trivalent graph claspers which can be detected by Kontsevich's characteristic classes. We checked that by explicitly describing surgeries on the handlebodies.

In \cite{Wa2}, we followed the line of the computation of \cite{KT,Les2} of Kontsevich's invariants for homology 3-spheres. When the dimension of the manifold is $2k+1\geq 5$, it turned out that many of the steps in the computation of \cite{KT,Les2} can be skipped by dimensional reasons. On the other hand, for even dimensions, such a shortcut fails and we needed to give higher dimensional analogues of all the steps needed. At the time we wrote \cite{Wa2}, we were not able to do so, however, we did that later with the help of \cite{Les3}. Also, the proof of Lemma~\ref{lem:F(a)-2} for bundles is not a straightforward analogue of the corresponding lemma \cite[Lemma~11.13]{Les3} for 3-manifolds. 

%%%%%%%%%%%%%%%%%%%%%%%%%%%%5
\subsection{Notations and conventions}\label{ss:notations}

\begin{enumerate}

\item[(a)] The diagonal $\{(x,x)\in X\times X\mid x\in X\}$ is denoted by $\Delta_X$. We identify its normal bundle $N\Delta_X$ and tangent bundle $T\Delta_X$ with $TX$ in a canonical manner, namely, identifying $(-v,v)\in N_{(x,x)}\Delta_X$, $(v,v)\in T_{(x,x)}\Delta_X$ with $v\in T_xX$, as in (\ref{eq:tangent-normal}).

\item[(b)] Let $I$ denote the interval $[0,1]$.

\item[(c)] We abbreviate the vector field $\displaystyle\frac{\partial}{\partial x_i}$ as $\partial x_i$.

\item[(d)] Throughout this paper, we assume that manifolds and maps between manifolds are smooth, unless otherwise stated.

\item[(e)] For manifolds with corners, smooth maps between them and their (strata) transversality, we follow \cite[Appendix]{BTa}. See also Appendix \ref{s:mfd-corners} in this paper.

\item[(f)] For a sequence of submanifolds $A_1,A_2,\ldots,A_r\subset W$ of a smooth Riemannian manifold $W$, we say that the intersection $A_1\cap A_2\cap \cdots\cap A_r$ is {\it transversal} if for each point $x$ in the intersection, the subspace $N_xA_1+N_xA_2+\cdots+N_xA_r\subset T_xW$ is the direct sum $N_xA_1\oplus N_xA_2\oplus\cdots\oplus N_xA_r$, where $N_xA_i$ is the orthogonal complement of $T_xA_i$ in $T_xW$ with respect to the Riemannian metric. Note that the transversality property does not depend on the choice of Riemannian metric.

\item[(g)] Homology and cohomology are considered over $\R$ if the coefficient ring is not specified. 

\item[(h)] For a fiber bundle $\pi\colon E\to B$, we denote by $T^vE$ the (vertical) tangent bundle along the fiber $\mathrm{Ker}\,d\pi\subset TE$. Let $ST^vE$ denote the subbundle of $T^vE$ of unit spheres. Let $\partial^v E$ denote the fiberwise boundaries: $\bigcup_{b\in B}\partial(\pi^{-1}\{b\})$. 

\item[(i)] We represent an orientation of a manifold $M$ by a nowhere-zero section of $\bigwedge^{\dim{M}} TM$ and use the symbol $o(M)$ for orientation of $M$. When $\dim{M}=0$, we give an orientation of $M$ by a choice of sign $\pm 1$ at each point, as usual.
We orient the boundary of a manifold by the outward-normal-first convention. We orient the total space of a fiber bundle over an oriented manifold by the rule $o(\mbox{base})\wedge o(\mbox{fiber})$.  
In Appendix \ref{s:ori}, we describe more orientation conventions adopted in this paper.

\item[(j)] We interpret a normal framing of a submanifold $A$ of a manifold $X$ of codimension $r$ by a sequence of sections $(s_1,\ldots,s_r)$ of the normal bundle $NA$ of $A$ that restricts to an ordered basis of each fiber of $NA$. 

\item[(k)] In Appendix \ref{s:blow-up}, we recall the definition of the blow-up in differentiable manifolds.

\end{enumerate}

%%%%%%%%%%%%%%%%%%%%%%%%%%%%5
\subsection{Acknowledgements}\label{ss:acknowledge}

I would like to thank B.~Botvinnik, R.~Budney, K.~Fujiwara, D.~Gabai, D.~Kosanovi\'{c}, M.~Krannich, A.~Kupers, F.~Laudenbach, A.~Lobb, S.~Moriya, K.~Ono, M.~Powell, O.~Randal-Williams, J.~Reinhold, K.~Sakai, T.~Sakasai, T.~Shimizu, C.~Taubes, P.~Teichner, M.~Weiss for helpful comments or questions. I would like to thank the organizers of ``HCM Workshop: Automorphisms of Manifolds (Hausdorff Center, 2019)'' for giving me an important opportunity to present my result. This work was partially supported by JSPS Grant-in-Aid for Scientific Research 21K03225, 20K03594, 17K05252, 15K04880, and by Research Institute for Mathematical Sciences, a Joint Usage/Research Center located in Kyoto University. 

I am deeply grateful to the referees for spending much time to read my paper and for giving me numerous important comments.

%%%%%%%%%%%%%%%%%%%%%%%%%%%%%%%
%%%%%%%%%%%%%%%%%%%%%%%%%%%%%%%
\mysection{Kontsevich's characteristic class}{s:kon}

The aim of this section is to give a self-contained exposition of Kontsevich's characteristic classes for even dimensional disk bundles, which were developed in \cite{Kon} and play a crucial role in the main result of this paper. {\it There are no new results in this section}.
We try to make the exposition as complete as possible since there seems to be no literature about the detail of that for higher even dimensions, though necessary ideas are given in \cite{Kon}\footnote{For 3-dimensional rational homology spheres, there are several expositions about Axelrod--Singer's or Kontsevich's configuration space integral invariants (\cite{Fu,BC,KT,Les1,Wa3}) other than the original papers (\cite{AS,Kon}). Among these, Lescop's \cite{Les1} (also \cite{Les4}) gives a thorough exposition of a complete detail of the definition and well-definedness of the invariant and that was helpful to write this section.}. 
What will be needed in the proof of our main result from this section are the definition of Kontsevich's invariant and the statement of Theorem~\ref{thm:kon} and of its corollary. 

%%%%%%%%%%%%%%%%%%%%%%%%%%%%%%%
\subsection{Framed smooth fiber bundles and classifying spaces}\label{ss:sphere-bundles}
%%%%%%%%%%%%%%%%%%%
\subsubsection{$(X,A)$-bundle}

In this paper, we consider {\it pointed} smooth fiber bundles, where we say that a smooth fiber bundle is pointed if the base space is a pointed space and if the bundle is equipped with a smooth identification of the fiber over the basepoint with a standard model of the fiber. Let $X$ be a compact manifold and $A$ be a submanifold of $X$. An {\it $(X,A)$-bundle} is a pointed $X$-bundle $E\to B$ over a pointed space $B$, equipped with maps of smooth fiber bundles
\begin{equation}\label{eq:pointed}
 \xymatrix{
  A \ar[r]^-{\widetilde{i}} \ar[d] & B\times A \ar[r]^-{\varphi} \ar[d]_-{\pr_1} & E \ar[d] \\
  \ast \ar[r]^-{i} & B \ar[r]^-{=} & B
} 
\end{equation}
where $i$ is the inclusion map of the basepoint $*$, $\widetilde{i}$ is given by the identification $A=\{*\}\times A$, $\pr_1$ is the projection onto the first factor, and $\varphi$ is a fiberwise embedding such that $\varphi\circ \widetilde{i}$ agrees with the inclusion $A\subset X$ into the fiber over $*$. In other words, an $X$-bundle equipped with trivializations on a subbundle with fiber $A$ (given by $\varphi$) and on the fiber over $*$, which are compatible on their intersection $A\subset \pi^{-1}(*)$. This can instead be defined as pointed $X$-bundles with structure group $\Diff(X,A)$, the group of diffeomorphisms $X\to X$ each of which fixes a neighborhood of $A$ pointwise, or equivalently, as $X$-bundles corresponding to a pointed classifying map from a pointed space to $B\Diff(X,A)$. The main objects in this paper are $(D^d,\partial D^d)$-bundles, or $(D^d,\partial)$-bundles for short. 

Studying a $(D^d,\partial)$-bundle is equivalent to studying a $(S^d,U_\infty)$-bundle, where $S^d=\R^d\cup\{\infty\}$ and $U_\infty$ is a small $d$-ball about $\infty$, and we will often consider the latter instead. More explicitly, a $(D^d,\partial)$-bundle over $B$ can be canonically extended to an $S^d$-bundle by attaching a trivial bundle over $B$ with fiber the disk $\{x\in S^d=\R^d\cup\{\infty\}\mid |x|\geq 1\}$, along the boundaries where the bundles are trivialized.

%%%%%%%%%%%%%%%%
\subsubsection{Framed $(X,A)$-bundle}

Now suppose that $TX$ is trivial and we fix a trivialization $\tau\colon TX\stackrel{\cong}{\to} \R^{\dim{X}}\times X$, which we think as a standard one. For an $X$-bundle $\pi\colon E\to B$, let $T^vE:=\mathrm{Ker}\,d\pi$, that is, the linear subbundle of $TE$ whose fiber over $z\in E$ is the subspace $\mathrm{Ker}(d\pi_z\colon T_zE\to T_{\pi(z)}B)\subset T_zE$. Suppose that a Riemannian metric on $T^vE$ is given. A {\it vertical framing} on $T^vE$ is a trivialization $T^vE\stackrel{\cong}{\to} \R^{\dim{X}}\times E$. For an $(X,A)$-bundle, we usually consider a vertical framing that agrees with the standard one $\tau$ on $\varphi(B\times A)\cup \pi^{-1}(*)=(B\times A)\cup \pi^{-1}(*)$, where $\varphi$ is the map in (\ref{eq:pointed}). We call such a framed bundle a pointed framed bundle.

%%%%%%%%%%%%%%%%%%
\subsubsection{Classifying space for framed $(X,A)$-bundles}

Let $\Fr(X,A;\tau)$ be the space of framings on $X$ that agree with $\tau$ on $A$, equipped with the topology as the subspace of the section space of the principal $SO_d$-bundle over $X$ associated to $TX$, which is also known as the oriented orthonormal frame bundle. Then $\Fr(X,A;\tau)$ is naturally a left $\Diff(X,A)$-space by $g\cdot \sigma=\sigma\circ (dg)^{-1}$ for $g\in \Diff(X,A), \sigma\in\Fr(X,A;\tau)$. We set
\[ \wBDiff(X,A;\tau):=E\Diff(X,A)\times_{\Diff(X,A)} \Fr(X,A;\tau). \]
This is a fiber bundle over $B\Diff(X,A)$ with fiber
\[ \Fr(X,A;\tau)\simeq \mathrm{Map}((X,A),(SO_d,\mathrm{id})). \]
This homotopy equivalence depends on the choice of $\tau$. 
Then $\wBDiff(X,A;\tau)$ is the classifying space for pointed framed $(X,A)$-bundles in the sense that there is a natural bijection between $[(B,*),(\wBDiff(X,A;\tau),*)]$ with the set of isomorphism classes of framed $(X,A)$-bundle over $B$. Since there is a (pointed) homotopy equivalence $\Fr(D^d,\partial D^d;\tau)\simeq \Omega^d SO_d$, we have a fiber sequence
\begin{equation}\label{eq:fib-seq-1}
 \Omega^d SO_d \to \wBDiff(D^d,\partial;\tau) \to B\Diff(D^d,\partial). 
\end{equation}

%%%%%%%%%%%%%%%%%%%%%%%%%%%%%%%
\subsection{Graph complex}\label{ss:graphs}

We recall the notion of Kontsevich's graph complex given in \cite{Kon} relevant to even dimensional manifolds. 

%%%%%%%%%%%%%%%%
\subsubsection{Space of graphs}

By a {\it graph} we mean a finite connected graph with valence at least 3. For a graph $\Gamma$ with $v$ vertices and $e$ edges, a {\it label} is a choice of bijections $\rho\colon \{\mbox{vertices of $\Gamma$}\}\to \{1,2,\ldots,v\}$ and $\mu\colon \{\mbox{edges of $\Gamma$}\}\to \{1,2,\ldots,e\}$. 
We identify two labelled graphs related by a label preserving graph isomorphism. An {\it orientation} of $\Gamma$ is a choice of an orientation of the real vector space
\[ \R^{\{\mathrm{edges\,of\,\Gamma}\}}. \] 
A label $(\rho,\mu)$ on a graph $\Gamma$ canonically determines an orientation of $\Gamma$, which we denote by $o(\Gamma,\rho,\mu)$. In this way, we consider a labelled graph also as an oriented graph. Let $V_{v,e}^\even$ be the vector space over $\Q$ generated by labelled graphs $(\Gamma,\rho,\mu)$ with $v$ vertices and $e$ edges, modulo the relations
\begin{equation}\label{eq:label-change}
\begin{array}{lll}
  \mbox{(i)} & (\Gamma,\rho',\mu')=-(\Gamma,\rho,\mu) \quad&\mbox{if $\mu'$ and $\mu$ differ by an odd permutation,}\\
  \mbox{(ii)} & (\Gamma,\rho,\mu)=0\quad&\mbox{if $\Gamma$ has a self-loop.}
\end{array}
\end{equation}
It follows from the relation (i) that $(\Gamma,\rho,\mu)$ is zero in $V_{v,e}^\even$ if it has a pair of vertices with multiple edges between them. The equivalence class of $(\Gamma,\rho,\mu)$ in $V_{v,e}^\even$ without self-loop bijectively corresponds to the oriented graph $(\Gamma,o(\Gamma,\rho,\mu))$ considered modulo the relation $(\Gamma,-o)=-(\Gamma,o)$. We will omit $\rho,\mu$ from the notation of labelled graph, and use the same notation $\Gamma$ for the equivalence class of a labelled graph $\Gamma$ in $V_{v,e}^\even$ to avoid complicated notations.

%%%%%%%%%%%%%%%
\subsubsection{Graph complex}\label{ss:graph-complex}

We set
\[ \calG^\even=\bigoplus_{v,e} V_{v,e}^\even. \]
As in \cite[Definition~3.6]{BNM}, we impose a bigrading on $\calG^\even$ by the ``degree'' $k=e-v=-\chi(\Gamma)=b_1(\Gamma)-1$, and the ``excess'' $\ell=2e-3v$\footnote{In \cite{BNM}, $\calG^\even$ is denoted by ${}^{bc}C$, and $P_k\calG_\ell^\even$ is denoted by ${}^{bc}C_k^\ell$.}. We denote by $P_k\calG^\even$ the subspace of $\calG^\even$ of degree $k$, and by $\calG^\even_\ell$ the subspace of $\calG^\even$ of excess $\ell$, where we observe $\calG^\even_{-1}=0$. The graded vector space $\calG^\even$ is made into a chain complex by the differential $\delta\colon \calG^\even_\ell \to\calG^\even_{\ell+1}$ defined on an element represented by a labelled graph $\Gamma$ without self-loop as
\[ \delta(\Gamma,o):=\sum_{{{i:\,\mathrm{edge}}\atop{\mathrm{of\,\Gamma}}}}(\Gamma/i,o[i]), \]
where $o=o(\Gamma)$ or $-o(\Gamma)$, $\Gamma/i$ is the labelled graph obtained from $\Gamma$ by contracting the edge $i$, equipped with the induced label: if the endpoints of the edge $i$ are $j_0,j_1$ with $j_0<j_1$, then the set of vertices of $\Gamma/i$ is labelled by shifting the labels $\{j_1+1,j_1+2,\ldots,v\}$ in $\{1,\ldots,v\}-\{j_1\}$ by $-1$, the set of edges of $\Gamma/i$ is labelled by shifting the labels $\{i+1,i+2,\ldots,e\}$ in $\{1,\ldots,e\}-\{i\}$ by $-1$. The orientation on $\Gamma/i$, denoted by $o[i]$, induced from an orientation $o$ of $\Gamma$ is determined by the rule
\begin{equation}\label{eq:induced-ori}
 i\wedge o[i]=o
\end{equation}
as an element of the vector space $\bigwedge^e\R^{\{\mathrm{edges\,of\,\Gamma}\}}$. Even if $o=o(\Gamma)$, the induced orientation $o[i]$ may be either $o(\Gamma/i)$ or $-o(\Gamma/i)$. 
It follows from $(o[i])[j]=-(o[j])[i]$ that $\delta\circ \delta=0$. The chain complex $(\calG^\even,\delta)$ is a version of Kontsevich's graph complex in \cite{Kon}. The ``graph cohomology'' is defined by
\[ H^\ell(\calG^\even)=\frac{\mathrm{Ker}\,(\delta\colon \calG^\even_\ell\to \calG^\even_{\ell+1})}{\mathrm{Im}\,(\delta\colon \calG^\even_{\ell-1}\to \calG^\even_\ell)}. \]
Note that $\delta$ preserves the degree and thus $H^\ell(\calG^\even)=\bigoplus_k H^\ell(P_k\calG^\even)$, and it makes sense to set $P_kH^\ell(\calG^\even)=H^\ell(P_k\calG^\even)$. 

We will also consider the dual chain complex $(\calG^\even,\delta^*)$, which is defined by identifying $\calG^\even_\ell$ with $\Hom(\calG^\even_\ell,\Q)$ by the canonical basis given by graphs, and by letting $\delta^*$ be the dual of $\delta$. The ``graph homology''\footnote{In \cite{Wi}, the complex $(\calG^\even,\delta^*)$ is denoted by $\mathsf{GC}_d$.} is defined by
\[ H_\ell(\calG^\even)=\frac{\mathrm{Ker}\,(\delta^*\colon \calG^\even_\ell\to \calG^\even_{\ell-1})}{\mathrm{Im}\,(\delta^*\colon \calG^\even_{\ell+1}\to \calG^\even_\ell)}. \]

%%%%%%%%%%%%%%%%
\subsubsection{The $0$-th graph (co)homology}

Since $\calG^\even_{-1}=0$, we have
\[ H^0(\calG^\even)=\mathrm{Ker}\,(\delta\colon \calG^\even_0\to \calG^\even_1),\quad H_0(\calG^\even)=\calG^\even_0/\delta^*(\calG^\even_1), \]
where $\calG^\even_0$ is the subspace of trivalent graphs. It follows from the definition of $\delta^*$ that $\delta^*(\calG^\even_1)$ is spanned by the IHX relation shown in Figure~\ref{fig:IHX}.
\begin{figure}
\begin{center}
\includegraphics[height=15mm]{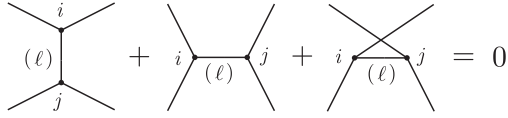}
\end{center}
\caption{IHX relation. Each term is the equivalence class in $\calG^\even$ of a labelled graph.}\label{fig:IHX}
\end{figure}
We set
\[ \calA_k^\even:=P_kH_0(\calG^\even)=P_k\calG^\even_0/{\mbox{IHX}}. \] 
Any class in $H^0(\calG^\even)$ can be obtained in the following way. Let
$\calL_k^\even$ be the set of all labelled trivalent graphs with $2k$ vertices and with no multiple edges and no self-loops, and let
\[ \zeta_k:=\sum_{\Gamma\in\calL_k^\even}\Gamma\otimes\Gamma \in P_k\calG^\even_0\otimes P_k\calG^\even_0. \]
It is obvious that any element $\gamma \in P_k\calG^\even_0$ can be represented as
\[ \gamma=(W\otimes \mathrm{id})\zeta_k=\sum_{\Gamma\in\calL_k^\even}W(\Gamma)\Gamma \]
for some linear map $W\colon P_k\calG^\even_0\to \Q$. 
Since we have
\[ (\mathrm{id}\otimes \delta)\zeta_k=\sum_{\Gamma\in\calL_k^\even}\Gamma\otimes \delta\Gamma=\sum_{\Gamma'}\delta^* \Gamma'\otimes \Gamma'\in P_k\calG^\even_0\otimes P_k\calG^\even_1, \]
where the sum of $\Gamma'$ is over all generating labelled graphs in $P_k\calG^\even_1$, it follows that $\delta\gamma=0$ if and only if $W(\delta^*(P_k\calG^\even_1))=0$, or equivalently, $w$ factors through a linear map $\overline{W}\colon \calA_k^\even\to \Q$. Hence any class $[\gamma]\in P_kH^0(\calG^\even)$ can be written uniquely as
\[ [\gamma]=(\overline{W}\otimes\mathrm{id})([\cdot]\otimes\mathrm{id})\zeta_k \]
for some linear map $\overline{W}\colon \calA_k^\even\to \Q$. We define
\begin{equation}\label{eq:universal}
 \widetilde{\zeta}_k:=\frac{1}{(2k)!(3k)!}([\cdot]\otimes\mathrm{id})\zeta_k=\frac{1}{(2k)!(3k)!}\sum_{\Gamma\in\calL_k^\even}[\Gamma]\otimes\Gamma\in\calA_k^\even\otimes P_k\calG^\even_0, 
\end{equation}
which can be considered as the universal class in $P_kH^0(\calG^\even;\calA_k^\even)$. The reason for the coefficients $\frac{1}{(2k)!(3k)!}$ in the formula of $\widetilde{\zeta}_k$ is just to avoid a coefficient in the right hand side of Theorem~\ref{thm:Z(G)}(3).

%%%%%%%%%%%%%%%%%%%%%%%%%%
%\subsubsection{Examples}
%
\if0
\begin{Prop}\label{prop:A_2}
 
\end{Prop}
\begin{proof}
By the label change relation (\ref{eq:label-change}) (i) and the IHX relation, one may see that a graph with a self-loop or multiple edges vanishes in $\calA_k^\even$. It follows that $\calA^\even_1=0$ and $\calA^\even_2$ is spanned by the class of $G_4$. Since the IHX relation for $\calA^\even_2$ imposes no restriction other than the vanishings of the graphs with a self-loop or multiple edges, we need only to check that the label change relation does not make $\calA^\even_2$ trivial. The automorphism group of $G_4$ is isomorphic to the symmetric group $\mathfrak{S}_4$ and each automorphism changes the labels of edges by an even permutation, namely, an automorphism of $G_4$ never produces $-G_4$, which implies that $G_4\neq 0$ in $\calA^\even_2$.
\end{proof}
\fi
%%%%%%%%%%%%%%%%%%%%%%%%%%%

\begin{Rem}\label{rem:A_2}
It is an easy exercise to see that $\calA_1^\even=0$, and $\calA_2^\even$ is 1-dimensional and generated by the class of the complete graph $W_4$ on four vertices with some labels. That $W_4$ represents a nontrivial class in $\calA^\even_2$ is a special case of \cite[Example~2.5]{CGP}. 
One may also easily check that $\calA^\even_3=0$. The dimensions of $\calA^\even_k$ for $4\leq k\leq 9$ are computed in \cite{BNM} as in the following table (${}^{bc}H_k^0$ in the notation of \cite{BNM} is $P_kH^0(\calG^\even)$, so that $\dim\,\calA_k^\even=\dim {}^{bc}H_k^0$).
\par\medskip
\begin{center}
\begin{tabular}{c|ccccccccc}\hline
$k$ & 1 & 2 & 3 & 4 & 5 & 6 & 7 & 8 & 9\\\hline
$\dim\,\calA^\even_k$ & 0 & 1 & 0 & 0 & 1 & 0 & 0 & 0 & 1\\\hline
\end{tabular}
\end{center}
\par\medskip
A lot more is known about $H_*(\calG^\even)$, e.g. lower bounds through \cite{Br,Wi} and the Euler characteristics (\cite{WZ}).
\end{Rem}
\par\medskip

%%%%%%%%%%%%%%%%%%%%%%%%%%%%%%%
\subsection{Compactification of configuration spaces}\label{ss:conf}

%%%%%%%%%%%%%%%%%
\subsubsection{Differential geometric analogue of the Fulton--MacPherson compactification due to Axelrod--Singer and Kontsevich}

Let $X$ be a manifold without boundary. 
The configuration space of labelled tuples of $n$ points on $X$ is
\[ \Conf_n(X)=\{(x_1,\ldots,x_n)\in X^n\mid x_i\neq x_j\mbox{ if }i\neq j\}. \]
For a subset $\Lambda$ of $N=\{1,2,\ldots,n\}$, we consider the blow-up $B\ell_{\Delta(\Lambda)}(X^\Lambda)$, where $\Delta(\Lambda)\subset X^\Lambda$ denotes the small diagonal $\{(x,\ldots,x)\in X^\Lambda\mid x\in X\}$. Roughly, the blowing up of $X^\Lambda$ along $\Delta(\Lambda)$ replaces $\Delta(\Lambda)$ with its normal sphere bundle $SN\Delta(\Lambda)$. See Appendix~\ref{s:blow-up} for more information about blow-ups. Let $\Conf_\Lambda(X)\subset X^\Lambda$ denote the configuration space of points labelled by $\Lambda$, analogously defined by replacing $N$ with $\Lambda$ in the above definition of $\Conf_n(X)$. There is a natural map $\Conf_\Lambda(X)\to B\ell_{\Delta(\Lambda)}(X^\Lambda)$ into the interior of $B\ell_{\Delta(\Lambda)}(X^\Lambda)$. By precomposing the forgetful map $\Conf_n(X)\to \Conf_\Lambda(X)$, a map $i_\Lambda\colon \Conf_n(X)\to B\ell_{\Delta(\Lambda)}(X^\Lambda)$ is defined. The inclusion $\Conf_n(X)\to X^n$ and the maps $i_\Lambda$ give an embedding
\begin{equation}\label{eq:AS}
 \Conf_n(X)\to X^n\times\prod_{|\Lambda|\geq 2}B\ell_{\Delta(\Lambda)}(X^\Lambda). 
\end{equation}
Then the space $\bConf_n(X)$ is defined to be the closure of the image of this map.
The following properties are proved in \cite{FM,AS} (see also Theorem~4.4, Propositions~1.4, 6.1 of \cite{Si})\footnote{More precisely, Proposition~\ref{prop:FM} (3) was proved in \cite[\S{3}]{FM} for nonsingular algebraic varieties over algebraically closed fields by constructing $\bConf_{n+1}(X)\to \bConf_n(X)$ by a sequence of blow-ups. In \cite{AS,Kon}, an analogue of the construction of \cite{FM} was given for differentiable manifolds. That the construction of \cite{Si} for $X=\R^m$ is canonically diffeomorphic to that of \cite{AS} (given via (\ref{eq:AS})) follows by an analogue of \cite[Corollary~4.1a]{FM} and since an image in $X^n\times S^k$ for the fiber $S^k$ of the sphere bundle $\partial B\ell_{\Delta(\Lambda)}(X^\Lambda)$ over $\Delta(\Lambda)$ with canonical trivialization recovers a unique lift in $X^n\times B\ell_{\Delta(\Lambda)}(X^\Lambda)$.}.

\begin{Prop}[Fulton--MacPherson, Axelrod--Singer]\label{prop:FM}
\begin{enumerate}
\item $\bConf_n(X)$ is a manifold with corners. 
\item If $X$ is compact, so is $\bConf_n(X)$.
\item The forgetful map $\Conf_m(X)\to \Conf_n(X)$ for $m>n$ which forgets the last $m-n$ factors extends to a smooth map $\bConf_m(X)\to \bConf_n(X)$. The same is true for other choices of the $m-n$ factors.
\end{enumerate}
\end{Prop}

The structure of manifold with corners on $\bConf_n(X)$ can be obtained from $X^n$ by a sequence of blow-ups. Let $X^n(r):=X^n\times\prod_{|\Lambda|\geq r}B\ell_{\Delta(\Lambda)}(X^\Lambda)$. 
Then there is a sequence of embeddings $\iota_r$ and projections $q_r$:
\begin{equation}\label{eq:seq-emb}
 \xymatrix{
  & \Conf_n(X) \ar[ld]_-{\iota_{n+1}} \ar[d]_-{\iota_{n}} \ar[rd]_-{\iota_{n-1}} \ar[rrrd]^-{\iota_2}& & & \\
  X^n=X^n(n+1) & X^n(n) \ar[l]^-{q_{n}} & X^n(n-1) \ar[l]^-{q_{n-1}} & \cdots \ar[l]^-{q_{n-2}} & X^n(2) \ar[l]^-{q_2}
} \end{equation}
where $q_r$ is the forgetful map which forgets the factors $B\ell_{\Delta(\Lambda)}(X^\Lambda)$ for $|\Lambda|=r$. Let $\bConf_n^{(r)}(X)$ be the closure of the image of $\iota_r$ in $X^n(r)$ of (\ref{eq:seq-emb}). Then one can show that for $r>2$, $\bConf_n^{(r-1)}(X)$ can be obtained from $\bConf_n^{(r)}(X)$ by a sequence of blow-ups along submanifolds of codimension $d(r-2)$ and thus $\bConf_n(X)=\bConf_n^{(2)}(X)$ can be obtained from $X^n$ by a sequence of blow-ups (Lemma~\ref{lem:seq-blowup})\footnote{This sequence of blow-ups is different from the analogue of the successive blow-ups given in \cite{FM,AS} in their derivation of the definition from (\ref{eq:AS}). The sequence of blow-ups along (\ref{eq:seq-emb}) will be convenient for our purpose. } 

We will also use the following important property of $\bConf_n(X)$ given in \cite[Corollaries~4.5, 4.9]{Si}.
\begin{Prop}[Sinha]\label{prop:sinha}
\begin{enumerate}
\item The inclusion $\Conf_n(X)\to \bConf_n(X)$ to the interior is a homotopy equivalence. 
\item The diagonal action of $\Diff(X)$ on $\Conf_n(X)$ extends to an action on $\bConf_n(X)$.
\end{enumerate}
\end{Prop}

In \cite{Si}, there are also explicit charts near the boundary (and corners) of $\bConf_n(X)$. The following is a compactification of $\Conf_n(\R^d)$, given in \cite{BTa}. 
\begin{Def}\label{def:compactif-Sd}
For $S^d=\R^d\cup\{\infty\}$, we define the space $\bConf_n(S^d;\infty)$ to be the preimage of $\{\infty\}$ under the map $\rho^{n+1}\colon \bConf_{n+1}(S^d)\to S^d$ induced by the projection\\ $(x_1,\ldots,x_{n+1})\mapsto x_{n+1}$. 
\end{Def}

\begin{Lem}[Proof in \S\ref{ss:proof-compactif-Sd-mfd}]\label{lem:compactif-Sd-mfd}
The map $\rho^{n+1}\colon \bConf_{n+1}(S^d)\to S^d$ is a fiber bundle such that the fiber $\bConf_n(S^d;\infty)$ is a manifold with corners.
\end{Lem}
An example of the construction of the compactification $\bConf_2(S^d;\infty)$ of $\Conf_2(\R^d)$ is given in \S\ref{ss:two-pt}.

%%%%%%%%%%%%%%%%%
\subsubsection{Codimension 1 strata}

We give a description of the codimension 1 strata of $\bConf_n(S^d;\infty)$, following \cite{AS,Kon,BTa,Si,Les4}. We refer the reader to these references for detail. By the definition of $\bConf_n(X)$ given above, the codimension 1 strata of $\bConf_n(S^d;\infty)$ are caused by the boundaries of the factors $B\ell_{\Delta(\Lambda)}(X^\Lambda)$ in (\ref{eq:AS}) (see the proof of Lemma~\ref{lem:seq-blowup} for the meaning of ``caused by''). Thus the set of codimension 1 strata of $\bConf_n(S^d;\infty)$ can be parametrized by subsets $\Lambda\subset N\cup\{\infty\}$ with $|\Lambda|\geq 2$. Now we set $X=S^d$, $X^\circ=S^d-\{\infty\}=\R^d$, though the following description is also valid for almost parallelizable $d$-manifolds.

\begin{Def}\label{def:codim1-stratum}
\begin{enumerate} 
\item Let $S_\Lambda$ be the codimension 1 stratum of $\bConf_n(S^d;\infty)$ corresponding to $\Lambda$. 
\item For a finite dimensional real vector space $W$ and an integer $r\geq 2$, let $\wConf_r(W)$ be the quotient of $\Conf_r(W)$ by the subgroup of affine transformations in $W$ generated by the diagonal actions of translations and multiplications of positive real number\footnote{In \cite{Si}, $\bConf_n(X)$, $\wConf_r(W)$, $S_\Lambda$ are denoted by $\Conf_n[X]$, $\widetilde{\Conf}_r(W)$, $\Conf_T(X)$, respectively.}. 
The space $\wConf_r(\R^d)$ can be identified with the subspace of $\Conf_r(\R^d)$ of $(y_1,\ldots,y_r)$ characterized by 
\begin{align}
 &|y_1|^2+\cdots+|y_r|^2=1,\quad y_1+\cdots+y_r=0, \quad \mbox{or } \label{eq:barycenter} \\
 &|y_1|^2+\cdots+|y_{r-1}|^2=1,\quad y_r=0. \label{eq:last-on-origin}
\end{align}
\item The compactification $\bwConf_r(\R^d)$ is defined as the closure of $\wConf_r(\R^d)$ in $\bConf_r(\R^d)$. This has the structure of a manifold with corners induced from $\bConf_r(\R^d)$. The compactification $\wConf_r(W)$ is defined analogously.
\item Let $\wConf_r(TX)$ denote the $\wConf_r(\R^d)$-bundle over $X$ associated to the oriented orthonormal frame bundle over $X$.
The $\bwConf_r(\R^d)$-bundle $\bwConf_r(TX)$ is defined by replacing the $SO_d$-space $\wConf_r(\R^d)$ with $\bwConf_r(\R^d)$ in the definition of $\wConf_r(TX)$.
\end{enumerate}
\end{Def}

The strata $S_\Lambda$ and its closures can be described explicitly as follows.

When $\infty\notin \Lambda$, let 
\[ \Conf_{n,\Lambda}(X^\circ):=\{(x_1,\ldots,x_n)\in (X^\circ)^n\mid x_i=x_j\,(i\neq j)\mbox{ if and only if }i,j\in \Lambda\}. \]
There is a diffeomorphism $\Conf_{n,\Lambda}(X^\circ)\cong \Conf_{n-r+1}(X^\circ)$, where $r=|\Lambda|$. 
Then the stratum $S_\Lambda$ of $\bConf_n(S^d;\infty)$ can be identified with the pullback of the bundle $\wConf_r(TX)\to X$ by the projection $\Conf_{n,\Lambda}(X^\circ)\to X$, which forgets the $(n-r)$-factors labelled by $N-\Lambda$ and maps the multiple factors for $\Lambda$ to $X^\circ\subset X$ by the natural map.
\begin{equation}\label{eq:S_Lambda}
 S_\Lambda=\lim\left(\vcenter{\xymatrix{
     & \wConf_r(TX) \ar[d] \\
  \Conf_{n,\Lambda}(X^\circ) \ar[r] & X
}}\right) 
\end{equation}
A framing on $X^\circ$ induces a trivialization $\wConf_r(TX^\circ)\stackrel{\cong}{\longrightarrow} X^\circ\times \wConf_r(\R^d)$ and a diffeomorphism
\[ S_\Lambda\cong C_{n,\Lambda}(X^\circ)\times \wConf_r(\R^d). \]
The projection $S_\Lambda\to C_{n,\Lambda}(X^\circ)$ is compatible near $S_\Lambda$ with the bundle projection $\Conf_n(X^\circ)\to \Conf_{n-r+1}(X^\circ)$, which forgets points with labels in a subset of $\Lambda$ with $r-1$ elements. 
Then the closure $\overline{S}_\Lambda$ of $S_\Lambda$ in $\bConf_n(S^d;\infty)$ is diffeomorphic to
\begin{equation}\label{eq:SA}
 \bConf_{n-r+1}(S^d;\infty)\times \bwConf_r(\R^d). 
\end{equation}

The case $\infty\in \Lambda$ is similar. In this case, we consider the pullback by the map $\Conf_{N-\Lambda}(X^\circ)\times\{\infty\}\to \{\infty\}$ instead of the bottom horizontal map in the diagram in (\ref{eq:S_Lambda}), where we set $r=|\Lambda|$, so that $|N-\Lambda|=n-|\Lambda-\{\infty\}|=n-r+1$. Hence we have
\begin{equation}\label{eq:SA-infty}
\begin{split}
S_\Lambda&=C_{N-\Lambda}(X^\circ)\times \wConf_r(T_\infty X),\\
\overline{S}_\Lambda&= \bConf_{N-\Lambda}(S^d;\infty)\times \bwConf_r(T_\infty X).
\end{split}
\end{equation}

%%%%%%%%%%%%%%%%%%%%%%%%%%%%%
\subsubsection{Unusual coordinates on $\wConf_r(T_\infty X)$}\label{ss:unusual-coord}
We will use seemingly unusual coordinates on $\wConf_r(T_\infty X)$ ((\ref{eq:Cr-infty}) below) in which the origin does not correspond to $\infty$, so that it is consistent with the coordinate system of $\Conf_r(X^\circ)=\Conf_r(\R^d)$ with respect to the limit. 
To fix such a coordinate system, we identify $T_\infty X-\{0\}$ with $T_0X-\{0\}$ through the diffeomorphism $\sigma\colon T_\infty X - \{0\} \stackrel{\cong}{\leftarrow} S^d-\{0,\infty\} \stackrel{\cong}{\to} T_0 X -\{0\}$ given by the stereographic projections\footnote{See e.g., \cite[Ch.I-(1.2)]{Kos}. In the notation of \cite{Kos}, $\sigma$ is $h_+\circ h_-^{-1}$ and by the formula for $h_\pm$, it follows that $\sigma(y)=\frac{y}{|y|^2}$. 
This identification can be visualized by considering $S^d-\{0,\infty\}$ as an $S^{d-1}$-family of geodesic arcs between $0$ and $\infty$, so that a linear half-ray from the origin in $T_0X$ corresponds to another linear half-ray to the origin in $T_\infty X$.}. This is equivariant with respect to the positive scalar multiplications in the sense that $\sigma(ay)=\frac{1}{a}\sigma(y)$ for $a>0$. The following lemma is evident.
\begin{Lem}
The diffeomorphism $\sigma\colon T_\infty X - \{0\}\to T_0 X -\{0\}$ induces a diffeomorphism $\Conf_{r-1}(\sigma)\colon \Conf_{r-1}(T_\infty X-\{0\})\to \Conf_{r-1}(T_0 X -\{0\})$, equivariant with respect to the positive scalar multiplications $(y_1,\ldots,y_{r-1})\mapsto (ay_1,\ldots,ay_{r-1})$ and $(y_1,\ldots,y_{r-1})\mapsto (a^{-1}y_1,\ldots,a^{-1}y_{r-1})$. Hence it induces a diffeomorphism 
\[ \Conf_r^*(\sigma)\colon \Conf_r^*(T_\infty X)\to \Conf_r^*(T_0X)=\Conf_r^*(\R^d). \]
\end{Lem}
We identify $\Conf_r^*(T_\infty X)$ with $\Conf_r^*(\R^d)$ via the diffeomorphism $\Conf_r^*(\sigma)$. Since $\Conf_r^*(\R^d)$ can be naturally identified with a subspace of $\Conf_r(\R^d)$ as in Definition~\ref{def:codim1-stratum} (2), we obtain the following explicit coordinates:
\begin{equation}\label{eq:Cr-infty}
\begin{split}
&\wConf_r(T_\infty X)\\
&=\{(y_1,\ldots,y_{r-1})\in(\R^d-\{0\})^{r-1}\mid |y_1|^2+\cdots+|y_{r-1}|^2=1,\,y_i\neq y_j\mbox{ if }i\neq j\}.
\end{split} 
\end{equation}
The right hand side is identified with $\wConf_r(\R^d)$ by considering one of the $r$ points is restrained at the origin (as in (\ref{eq:last-on-origin})), which in (\ref{eq:SA-infty}) plays the role of the limit point where the non-infinite $n-r+1$ points gather together, and which is the alternative of putting the infinity at the origin. These coordinates will be used in Lemma~\ref{lem:phi-extend} and in the derivation of (\ref{eq:omega-S_A-2}).

\begin{Rem}
The coordinates (\ref{eq:Cr-infty}) obtained via the identification by $\Conf_r^*(\sigma)$ look unusual but natural when taking relative directions. For example, we fix points $x,x'\in\R^d-\{0\}$, $x\neq x'$, and consider a smooth path $a\colon [1,\infty)\to (S^d)^{\times 3}$ given by $t\mapsto (tx,tx',\infty)$, which converges to $(\infty,\infty,\infty)$ as $t\to \infty$. Taking the unit direction $(x_1,x_2,\infty)\mapsto \frac{x_2-x_1}{|x_2-x_1|}\in S^{d-1}$ on the path $a$ gives a map $\phi_a\colon [1,\infty)\to S^{d-1}$, which is a constant map in this case. If we consider $\Conf_2(\R^d)\times\{\infty\}$ as a subset of $\bConf_2(S^d;\infty)$, the path $a$ can be extended to a path $\bar{a}\colon [1,\infty]\to \bConf_2(S^d;\infty)$ such that $\bar{a}(\infty)\in S_{\{1,2,\infty\}}=\wConf_3(T_\infty X)$. With the coordinates (\ref{eq:Cr-infty}), the limit point $\bar{a}(\infty)$ agrees with $(x,x')$ up to a scalar multiplication and $\phi_a$ can be extended to $[1,\infty]\to S^{d-1}$ by the same formula $\frac{x_2-x_1}{|x_2-x_1|}$. 
\end{Rem}

The coordinate description (\ref{eq:Cr-infty}) of $\wConf_r(T_\infty X)$ also allows us to consider it as a subspace of $C_r(\R^d)$ by mapping $(y_1,\ldots,y_{r-1})$ to $(y_1,\ldots,y_{r-1},0)$ and hence as a subspace of $\bConf_r(\R^d)$. Then the compactification $\bwConf_r(T_\infty X)$ can be obtained by the closure of $\wConf_r(T_\infty X)$ in $\bConf_r(\R^d)$. This is compatible with the compactification of $\wConf_r(T_\infty X)$ obtained by identifying $T_\infty X$ with $\R^d$ and $\wConf_r(T_\infty X)$ with $\wConf_r(\R^d)\subset \bConf_r(\R^d)$ in a usual way.

%%%%%%%%%%%%%%%%%
\subsubsection{Example: the case of two points}\label{ss:two-pt}

We describe the structure of a manifold with corners on $\bConf_2(S^d;\infty)$, following \cite[Section III]{BTa} and \cite[\S{3}]{Les1}. 
According to (\ref{eq:fiberwise-closure}) in the proof of Lemma~\ref{lem:compactif-Sd-mfd}, the compactification $\bConf_2(S^d;\infty)$ can be obtained by the closure of the embedding
\begin{equation}\label{eq:emb-2}
 \begin{split}
  \iota'\colon \Conf_2(X^\circ)\to &\,\, X^2\times B\ell_{\Delta(\{1,2,\infty\})}(X^2\times\{\infty\})\\
  &\times B\ell_{\Delta(\{1,\infty\})}(X\times\{\infty\})\times B\ell_{\Delta(\{2,\infty\})}(X\times\{\infty\})\times B\ell_{\Delta(\{1,2\})}(X^2),
\end{split} 
\end{equation}
where $B\ell_{\Delta(\{1,2,\infty\})}(X^2\times\{\infty\})\cong B\ell_{\{(\infty,\infty)\}}(X^2)$, $B\ell_{\Delta(\{i,\infty\})}(X\times\{\infty\})\cong B\ell_{\{\infty\}}(X)$. We claim that $\bConf_2(S^d;\infty)$ is obtained from $X^2\times\{\infty\}$ by the sequence of blow-ups along the strata $\Delta_{\{1,2,\infty\}}\subset \Delta_{\{1,\infty\}}\cup\Delta_{\{2,\infty\}}\cup\Delta_{\{1,2\}}$. Indeed, there is a sequence of embeddings analogous to (\ref{eq:seq-emb}):
\[
 \xymatrix{
  \Conf_2(X^\circ) \ar[d]_-{\iota_{4}} \ar[rd]_-{\iota_{3}} \ar[rrd]^-{\iota_{2}} & &  \\
  X^2=X^2(4) & X^2(3) \ar[l]^-{q_{3}} & X^2(2) \ar[l]^-{q_{2}} 
} \]
where $X^2(3)=X^2\times B\ell_{\Delta(\{1,2,\infty\})}(X^2\times\{\infty\})$ and $X^2(2)$ is the right hand side of (\ref{eq:emb-2}).
Let $\bConf_2^{(r)}(S^d;\infty)$ be the closure of the image of $\iota_r$. It is straightforward that $\bConf_2^{(4)}(S^d;\infty)=X^2$ and $\bConf_2^{(3)}(S^d;\infty)\cong B\ell_{\{(\infty,\infty)\}}(X^2)$. The next term $\bConf_2^{(2)}(S^d;\infty)=\bConf_2(S^d;\infty)$ is obtained by blowing up $\bConf_2^{(3)}(S^d;\infty)$ along the closures of the preimages of the strata $X^\circ\times\{\infty\}$, $\{\infty\}\times X^\circ$, $\Delta_{X^\circ}$ under $q_3$ (see Figure~\ref{fig:Bl-diag-infty}). 

Let $\overline{S}_{\{1,2,\infty\}}$ be $q_2^{-1}(\partial \bConf_2^{(3)}(S^d;\infty))$, and let $\overline{S}_{\{1,\infty\}}$, $\overline{S}_{\{2,\infty\}}$, $\overline{S}_{\{1,2\}}$ be the (closed) codimension 1 strata obtained by the blow-ups along the closures of the preimages of $X^\circ\times\{\infty\}$, $\{\infty\}\times X^\circ$, $\Delta_{X^\circ}$, respectively. Then the boundary of $\bConf_2(S^d;\infty)$ is
\[ \overline{S}_{\{1,2,\infty\}}\cup \overline{S}_{\{1,\infty\}}\cup \overline{S}_{\{2,\infty\}}\cup \overline{S}_{\{1,2\}}, \]
where the pieces are glued together along the strata of $\bConf_2(S^d;\infty)$ of codimension $\geq 2$. The product structures (\ref{eq:SA}) and (\ref{eq:SA-infty}) for this case can be given directly as follows.
\begin{figure}
\[ \includegraphics[height=25mm]{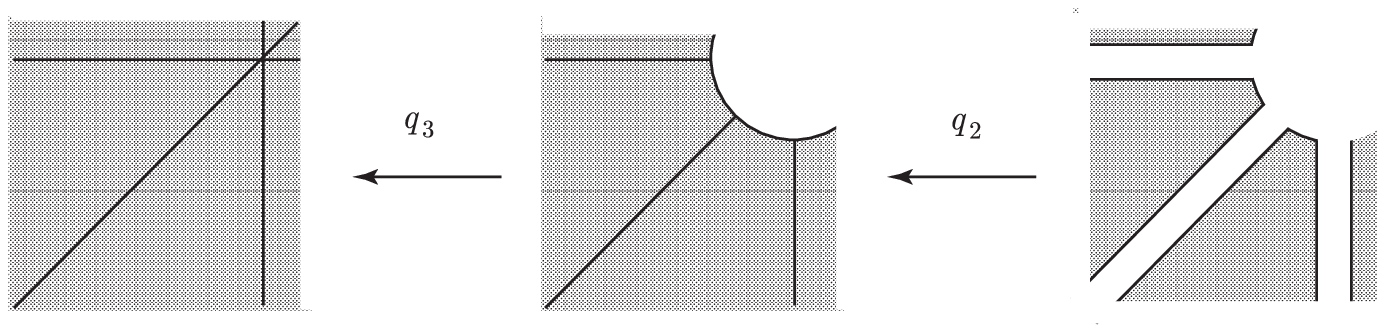} \]
\caption{$\bConf_2^{(4)}(S^d;\infty)\stackrel{q_3}{\longleftarrow}\bConf_2^{(3)}(S^d;\infty)\stackrel{q_2}{\longleftarrow}\bConf_2^{(2)}(S^d;\infty)$}\label{fig:Bl-diag-infty}
\end{figure}
\begin{figure}
\[ \includegraphics[width=90mm]{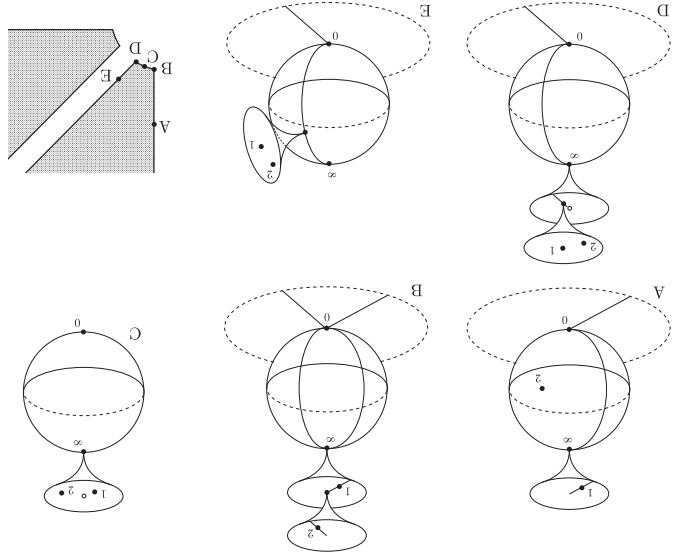} \]
\caption{Points in $\partial\bConf_2(S^d;\infty)$. $A\in S_{\{1,\infty\}}$, $B\in \overline{S}_{\{1,\infty\}}\cap\overline{S}_{\{1,2,\infty\}}$, $C\in S_{\{1,2,\infty\}}$, $D\in \overline{S}_{\{1,2,\infty\}}\cap\overline{S}_{\{1,2\}}$, $E\in S_{\{1,2\}}$}
\end{figure}

\begin{enumerate}
\item The stratum $\overline{S}_{\{1,2,\infty\}}=\bwConf_3(T_\infty X)$ is the blow-up of $\partial\bConf_2^{(3)}(S^d;\infty)=S^{2d-1}=\{(y_1,y_2)\in(\R^d)^2\mid |y_1|^2+|y_2|^2=1\}$ along the codimension $d$ submanifold $D=(\{y_1=0\}\cup \{y_2=0\}\cup \{y_1=y_2\})\cap S^{2d-1}$. 
\item The stratum $\overline{S}_{\{1,\infty\}}$ is $\partial B\ell_{\{0\}}(T_\infty X)\times \bConf_1(S^d;\infty)=\bwConf_2(T_\infty X)\times \bConf_1(S^d;\infty)$. 
\item The stratum $\overline{S}_{\{2,\infty\}}$ is $\bConf_1(S^d;\infty)\times \partial B\ell_{\{0\}}(T_\infty X)=\bConf_1(S^d;\infty)\times \bwConf_2(T_\infty X)$. 
\item The stratum $\overline{S}_{\{1,2\}}$ is $\Delta_{\bConf_1(S^d;\infty)}\times \partial B\ell_{\{(0,0)\}}((T_{(0,0)}\Delta_{\bConf_1(S^d;\infty)})^\perp)=\Delta_{\bConf_1(S^d;\infty)}\times \bwConf_2(\R^d)$ by the canonical identification $(T_{(0,0)}\Delta_{\bConf_1(S^d;\infty)})^\perp=T_0\bConf_1(S^d;\infty)=\R^d$. 
\end{enumerate}

The proof of the following lemma was given in \cite[p.5266--5267]{BTa} and \cite[\S{3.2}]{Les1}. A more detailed proof is given in Appendix~\ref{s:phi-extend}.
\begin{Lem}\label{lem:phi-extend}
The smooth map $\phi\colon \Conf_2(\R^d)\to S^{d-1}$ defined by 
\[ \phi(x_1,x_2)=\frac{x_2-x_1}{|x_2-x_1|} \]
extends to a smooth map $\overline{\phi}\colon \bConf_2(S^d;\infty)\to S^{d-1}$. The extension $\overline{\phi}$ on the boundary of $\bConf_2(S^d;\infty)$ is explicitly given as follows\footnote{One can observe that the signs of $\pm\phi'$ are correct by drawing a picture for $d=1$. Note that we have chosen unusual coordinates on $C_r^*(T_\infty X)$ as in \S\ref{ss:unusual-coord}.}:
\begin{enumerate}
\item On the stratum $\overline{S}_{\{1,2,\infty\}}=B\ell_D(\{(y_1,y_2)\in(\R^d)^2\mid |y_1|^2+|y_2|^2=1\})$, $\overline{\phi}=\phi'\circ \overline{i}$, where $\overline{i}\colon\overline{S}_{\{1,2,\infty\}}\to \bConf_2(\R^d)$ is the map induced by the embedding $i:S_{\{1,2,\infty\}}=\Conf_3^*(T_\infty X)\to \Conf_2(\R^d-\{0\})$ given by (\ref{eq:Cr-infty}), and $\phi'\colon \bConf_2(\R^d)\to S^{d-1}$ is the smooth extension of $\phi$ defined by the coordinates of the blow-up (Lemma~\ref{lem:bl_extension}(3)). 

\item On the stratum $\overline{S}_{\{1,\infty\}}$, $\overline{\phi}$ is the composition
\begin{equation}\label{eq:phi-extend-2}
 \overline{S}_{\{1,\infty\}}=\bwConf_2(T_\infty X)\times \bConf_1(S^d;\infty)\stackrel{\pr_1}{\longrightarrow}\bwConf_2(T_\infty X)\xrightarrow[\cong]{-\phi'} S^{d-1}. 
\end{equation}
\item On the stratum $\overline{S}_{\{2,\infty\}}$, $\overline{\phi}$ is the composition
\begin{equation}\label{eq:phi-extend-3}
 \overline{S}_{\{2,\infty\}}=\bConf_1(S^d;\infty)\times \bwConf_2(T_\infty X)\stackrel{\pr_2}{\longrightarrow}\bwConf_2(T_\infty X)\xrightarrow[\cong]{\phi'} S^{d-1}.  
\end{equation}
\item On the stratum $\overline{S}_{\{1,2\}}$, $\overline{\phi}$ is the composition
\begin{equation}\label{eq:phi-extend-4}
 \overline{S}_{\{1,2\}}=\Delta_{\bConf_1(S^d;\infty)}\times \bwConf_2(\R^d)\stackrel{\pr_2}{\longrightarrow}\bwConf_2(\R^d)\xrightarrow[\cong]{\phi'} S^{d-1}. 
\end{equation}
\end{enumerate}
\end{Lem}

In each case of (1)--(4), we take a projection to the space of `limit configurations': $\bConf_2(\R^d)$, $\bwConf_2(\R^d)$ etc., that is a subspace of $\bConf_2(\R^d)$, then take the relative direction from the first point $y_1$ to the second point $y_2$. In (2), $y_2$ (in the limit configuration of $\bwConf_2(T_\infty X)$) is assumed to be at the origin, so the relative direction from $y_1$ to $y_2=0$ agrees with $-\phi'$. In (3), $y_1$ is assumed to be at the origin, so the relative direction from $y_1=0$ to $y_2$ agrees with $\phi'$. In (4), the orthogonal projection $T_xX\times T_xX\to N_{(x,x)}\Delta_X\to \R^d$ is the limit of $(x_1,x_2)\mapsto (\frac{x_1-x_2}{2},\frac{x_2-x_1}{2})\mapsto \frac{x_2-x_1}{2}$ as in (\ref{eq:tangent-normal}), the relative direction for the limit configuration agrees with $\phi'$.

%%%%%%%%%%%%%%%%%%%%%%%%%%%%%%%
\subsection{Propagator}\label{ss:propagator}

We need to fix a certain closed form on the configuration space called a propagator to define the configuration space integrals. 

%%%%%%%%%%%%%%%%%%%
\subsubsection{de Rham Cohomology of $\bConf_2(S^d;\infty)$}

Throughout this subsection, we assume $d>1$.  Since $\overline{\phi}\colon \bConf_2(S^d;\infty)\to S^{d-1}$ is a homotopy equivalence, it follows that
\[ H^*(\bConf_2(S^d;\infty))=H^*(S^{d-1})\cong\left\{\begin{array}{ll}
\R & (*=0,d-1),\\
0 & (\mbox{otherwise}).
\end{array}\right. \]
In particular, $H^{d-1}(\bConf_2(S^d;\infty))$ is generated by $[\bar{\phi}^*\mathrm{Vol}_{S^{d-1}}]$, where
\begin{equation}\label{eq:vol}
 \mathrm{Vol}_{S^{d-1}}=\frac{1}{\mathrm{vol}(S^{d-1})}\sum_{i=1}^d(-1)^{i-1}x_idx_1\wedge\cdots\wedge dx_{i-1}\wedge dx_{i+1}\wedge\cdots\wedge dx_d, 
\end{equation}
and $\mathrm{vol}(S^{d-1})$ is the volume of the unit sphere $S^{d-1}$ in $\R^d$, so that $\displaystyle\int_{S^{d-1}}\mathrm{Vol}_{S^{d-1}}=1$. 
By Poincar\'{e}--Lefschetz duality, 
\[ H^*(\bConf_2(S^d;\infty),\partial \bConf_2(S^d;\infty))\cong H_{2d-*}(S^{d-1})\cong\left\{\begin{array}{ll}
\R & (*=d+1,2d),\\
0 & (\mbox{otherwise}).
\end{array}\right. \]
The following lemma is evident from the explicit formula (\ref{eq:vol}).
\begin{Lem}\label{lem:vol-symmetry}
Let $\iota\colon S^{d-1}\to S^{d-1}$ be the involution $\iota(x)=-x$. Then we have
\[ \iota^*\mathrm{Vol}_{S^{d-1}}=(-1)^d\,\mathrm{Vol}_{S^{d-1}}. \]
\end{Lem}

%%%%%%%%%%%%%%%%%
\subsubsection{Propagator in a fiber}

Suppose we are given a framing $\tau\colon T(S^d-\{\infty\})\stackrel{\cong}{\to} (S^d-\{\infty\}) \times \R^d$ on $S^d-\{\infty\}=\R^d$ that agrees with the standard framing $\tau_0$ of $\R^d$ outside a $d$-ball of finite radius about the origin. Then $\tau$ induces a smooth map
\[ p(\tau)\colon \partial \bConf_2(S^d;\infty) \to S^{d-1}, \]
which extends the map obtained by restricting $\bar{\phi}$ of Lemma~\ref{lem:phi-extend} to $\overline{S}_{\{1,2,\infty\}}\cup \overline{S}_{\{1,\infty\}}\cup \overline{S}_{\{2,\infty\}}$ and agrees on $\overline{S}_{\{1,2\}}$ with the composition 
\[ \overline{S}_{\{1,2\}}\stackrel{\cong}{\longrightarrow} \Delta_{\bConf_1(S^d;\infty)}\times S^{d-1}\stackrel{\pr_2}{\longrightarrow} S^{d-1}, \] 
where the first map is induced by $\tau$. 
By Lemma~\ref{lem:phi-extend}, $p(\tau)$ can indeed be extended to a smooth map.

\begin{Lem}[Propagator in fiber]\label{lem:propagator1}
Let $\tau$ be a framing of $T(S^d-\{\infty\})$ that is standard near $\infty$.  

\begin{enumerate}
\item The closed $(d-1)$-form $p(\tau)^*\mathrm{Vol}_{S^{d-1}}$ on $\partial \bConf_2(S^d;\infty)$ can be extended to a closed form $\omega$ on $\bConf_2(S^d;\infty)$ so that its cohomology class $[\omega]$ agrees with $[{\bar{\phi}}^*\mathrm{Vol}_{S^{d-1}}]$. 
\item For a fixed framing $\tau$, the extension $\omega$ is unique in the sense that for two such extensions $\omega$ and $\omega'$, there is a $(d-2)$-form $\mu$ on $\bConf_2(S^d;\infty)$ that vanishes on $\partial \bConf_2(S^d;\infty)$ such that
\[ \omega'-\omega=d\mu. \]
\end{enumerate}
We call such an extended form a \emph{propagator} for $\tau$.
\end{Lem}
\begin{proof}
The proof is an analogue of \cite[Lemma~2.1]{Tau}, \cite[p.2]{BC2}, or \cite[Lemmas~2.3, 2.4]{Les1}. 
The assertion (1) follows immediately from the long exact sequence of the pair
\[ 0=H^{d-1}(\bConf,\partial\bConf)\to H^{d-1}(\bConf)\to H^{d-1}(\partial\bConf)\to H^d(\bConf,\partial \bConf)=0, \]
where we abbreviate as $\bConf=\bConf_2(S^d;\infty)$. Here both $[\omega]$ and $[\bar{\phi}^*\mathrm{Vol}_{S^{d-1}}]$ restrict to the same generator of the de Rham cohomology of $*\times S^d\subset SN\Delta_{\R^d}$, their cohomology classes agree. The assertion (2) follows since the difference $\omega'-\omega$ vanishes on $\partial\bConf$ and represents 0 of $H^{d-1}(\bConf,\partial\bConf)$, which is the cohomology of the subcomplex of the de Rham complex $\Omega_\dR^*(\bConf)$ of forms that vanish on $\partial \bConf$.
\end{proof}

%%%%%%%%%%%%%%%%
\subsubsection{Propagator in family}

The group $\Diff(S^d,U_\infty)$ acts on $\bConf_n(S^d;\infty)$ by extending the diagonal action $g\cdot(x_1,\ldots,x_n)=(g\cdot x_1,\ldots,g\cdot x_n)$ on $\Conf_n(\R^d)$. Namely, $\Diff(S^d,U_\infty)$ acts diagonally on the target space of the embedding $\iota'$ of (\ref{eq:proj_infty}) which induces an automorphism of the subspace $\bConf_n(S^d;\infty)=\mathrm{Closure}\,(\mathrm{Im}\,\iota')$. 
For a $(D^d,\partial)$-bundle $\pi\colon E\to B$, we consider the associated $\bConf_n(S^d;\infty)$-bundle
\[ \bConf_n(\pi)\colon E\bConf_n(\pi)\to B. \]
Its fiberwise restriction to the boundary of the fiber gives the subbundle
\[ \bConf_n^\partial(\pi)\colon \partial^v E\bConf_n(\pi)\to B. \]
A vertical framing $\tau_E\colon T^vE\stackrel{\cong}{\to} E\times \R^d$ induces a smooth map
\[ p(\tau_E)\colon \partial^v E\bConf_2(\pi)\to S^{d-1} \]
by applying a similar construction as above in each fiber. 

\begin{Lem}[Propagator in family]\label{lem:propagator2}
Suppose that $B$ is a manifold. 
\begin{enumerate}
\item The closed $(d-1)$-form $p(\tau_E)^*\mathrm{Vol}_{S^{d-1}}$ on $\partial^v E\bConf_2(\pi)$ can be extended to a closed form $\omega$ on $E\bConf_2(\pi)$. 
\item For a fixed framing $\tau_E$, the extension $\omega$ is unique in the sense that for two such extensions $\omega$ and $\omega'$, there is a $(d-2)$-form $\mu$ on $E\bConf_2(\pi)$ that vanishes on $\partial^v E\bConf_2(\pi)$ such that
\[ \omega'-\omega=d\mu. \]
\end{enumerate}
We call such an extended form a \emph{propagator (in family)} for $\tau_E$.
\end{Lem}
\begin{proof}
The Leray--Serre spectral sequence of the relative fibration
\[ (\bConf,\partial\bConf)\to (E\bConf_2(\pi),\partial^vE\bConf_2(\pi))\to B, \]
has $E_2$-term $E_2^{p,q}\cong H^p(B;\{H^q(\bConf_b,\partial\bConf_b)\}_{b\in B})$, where $\{H^q(\bConf_b,\partial\bConf_b)\}_{b \in B}$ is the local coefficient system on $B$ given by the cohomology of the fiber. Also, we know that $H^q(\bConf,\partial \bConf)=0$ for $q<d+1$. Hence we have
\[ H^n(E\bConf_2(\pi),\partial^vE\bConf_2(\pi))=0\mbox{ for $n<d+1$}, \]
and the natural map $H^{d-1}(E\bConf_2(\pi))\to H^{d-1}(\partial^v E\bConf_2(\pi))$ is an isomorphism. This implies the assertion (1). The proof of the assertion (2) is the same as Lemma~\ref{lem:propagator1}(2).
\end{proof}

\begin{Cor}\label{cor:omega-cobordism}
Suppose that $(\pi\colon E\to B,\tau_E)$ is a framed $(D^d,\partial)$-bundle over a cobordism $B$ between closed manifolds $A_0$ and $A_1$. Suppose given propagators $\omega_0$ and $\omega_1$ for $\tau_E$ on $\bConf_2(\pi)^{-1}(A_0)$ and $\bConf_2(\pi)^{-1}(A_1)$, respectively. Then there exists a propagator $\omega$ for $\tau_E$ on $E\bConf_2(\pi)$ that restricts to $\omega_0$ and $\omega_1$ on $\bConf_2(\pi)^{-1}(A_0)$ and $\bConf_2(\pi)^{-1}(A_1)$, respectively.
\end{Cor}
\begin{proof}
We identify a collar neighborhood of $B$ with $A_0\times [0,\ve] \cup A_1\times [1-\ve,1]$ and accordingly identify as $\bConf_2(\pi)^{-1}(A_0\times[0,\ve])=\bConf_2(\pi)^{-1}(A_0)\times[0,\ve]$ and $\bConf_2(\pi)^{-1}(A_1\times[1-\ve,1])=\bConf_2(\pi)^{-1}(A_0)\times[1-\ve,1]$. Then we may pull back $\omega_0$ and $\omega_1$ to $\bConf_2(\pi)^{-1}(A_0)\times[0,\ve]$ and $\bConf_2(\pi)^{-1}(A_1)\times[1-\ve,1]$, respectively. Moreover, we assume without loss of generality that $\tau_E$ is compatible with these product structures. Let
\[ B'=B-\bigl((A_0\times [0,\ve))\cup (A_1\times (1-\ve,1])\bigr). \]
By Lemma~\ref{lem:propagator2}(1), there exists a propagator $\omega_a$ on $E\bConf_2(\pi)$ for $\tau_E$. By Lemma~\ref{lem:propagator2}(2), there are $(d-2)$-forms $\mu_0$ and $\mu_1$ on the collar neighborhoods such that they vanish on $\partial^vE\bConf_2(\pi)$, and 
\[ \omega_0-\omega_a=d\mu_0,\quad \omega_1-\omega_a=d\mu_1 \]
where they make sense. We take a smooth function $\chi\colon E\bConf_2(\pi)\to [0,1]$ that takes the value 1 on $\bConf_2(\pi)^{-1}(\partial B)$ and takes the value 0 on $\bConf_2(\pi)^{-1}(B')$. Let $\mu'$ be a $(d-2)$-form on $E\bConf_2(\pi)$ extending $\mu_0$ and $\mu_1$, which vanish on $\partial^vE\bConf_2(\pi)$. We set
\[ \omega=\omega_a+d(\chi \mu'), \]
which is well-defined as a smooth closed $(d-1)$-form on $E\bConf_2(\pi)$. As $\chi \mu'$ vanishes on $\partial^vE\bConf_2(\pi)$, we have $\omega|_{\partial^vE\bConf_2(\pi)}=\omega_a|_{\partial^vE\bConf_2(\pi)}$ and 
\[ \omega|_{\bConf_2(\pi)^{-1}(A_i)}=\omega_i\quad\mbox{ for $i=0,1$}. \]
This completes the proof.
\end{proof}

%%%%%%%%%%%%%%%%%%%%%%%%%%%%%%%
\subsection{Configuration space integrals}\label{ss:csi}

%%%%%%%%%%%%%%%%%
\subsubsection{Kontsevich's integral}

Now we assume that $d$ is even and $d\geq 4$. Let $\pi\colon E\to B$ be an $(D^d,\partial)$-bundle over a closed oriented manifold $B$, equipped with a vertical framing $\tau_E$. Let $\bConf_n(\pi)\colon E\bConf_n(\pi)\to B$ be the $\bConf_n(S^d;\infty)$-bundle associated to $\pi$. We take a propagator $\omega$ in family $E\bConf_2(\pi)$ for $\tau_E$ as in Lemma~\ref{lem:propagator2}. For a labelled graph $\Gamma=(\Gamma,\rho,\mu)\in\calG^\even$ without self-loop, we choose orientations on edges of $\Gamma$, namely, make a choice of the order of the two boundary vertices of each edge. This choice determines the projection map
\[ \phi_i\colon E\bConf_v(\pi)\to E\bConf_2(\pi) \]
defined by forgetting the points other than the two points for the labels of the boundary vertices of the edge $i$, which is smooth by Proposition~\ref{prop:FM}.
\begin{Def}\label{def:I} We set
\begin{equation}\label{eq:omega(Gamma)}
 \begin{split}
  &\omega(\Gamma):=\bigwedge_{{{i:\,\mathrm{edge}}\atop{\mathrm{of\,\Gamma}}}}\phi_i^*\omega\in\Omega_\dR^{(d-1)e}(E\bConf_v(\pi)),\\
  &I(\Gamma):=\bConf_v(\pi)_*\, \omega(\Gamma)\in\Omega_\dR^{(d-1)e-dv}(B),
\end{split} 
\end{equation}
where $\bConf_v(\pi)_*\colon \Omega_\dR^{(d-1)e}(E\bConf_v(\pi))\to \Omega_\dR^{(d-1)e-dv}(B)$ denotes the pushforward or integration along the fibers (\cite[p.61]{BTu}, \cite[Ch.VII]{GHV}, see also \S\ref{ss:pushforward}).
This extends linearly to the linear map
\[ I\colon P_k\calG^\even_\ell \to \Omega_\dR^{(d-3)k+\ell}(B), \]
where $k=e-v$, $\ell=2e-3v$ as in \S\ref{ss:graph-complex}.
\end{Def}
Note that the integral along the fibers (\ref{eq:omega(Gamma)}) is convergent since the fiber $\bConf_v(S^d;\infty)$ is compact.

\begin{Thm}[Kontsevich \cite{Kon}. Proof in \S\ref{s:wd-kon}]\label{thm:kon}
Let $d$ be an even integer such that $d\geq 4$.
\begin{enumerate}
\item $I$ is a chain map up to sign, namely, 
\[ dI(\Gamma)=(-1)^{(d-3)k+\ell+1}I(\delta\Gamma) \]
for $\Gamma\in P_k\calG^\even_\ell$. In particular, if $\gamma\in P_k\calG^\even_\ell$ is such that $\delta\gamma=0$, then
$dI(\gamma)=0$. If $\gamma$ is such that $\gamma=\delta\eta$, then $I(\gamma)=(-1)^{(d-3)k+\ell+1}dI(\eta)$. 
Hence $I$ induces a linear map $I_*\colon  P_kH^\ell(\calG^\even;\Q)\to H^{(d-3)k+\ell}(B;\R)$.
\item $I_*$ does not depend on the choice of propagator $\omega$ in family for $\tau_E$.
\item $I_*$ does not depend on the choice of edge orientations (used to define $\phi_i$).
\item $I_*$ is invariant under a homotopy of $\tau_E$.
\item $I_*$ gives characteristic classes of framed $(D^d,\partial)$-bundles, that is, $I_*$ is natural with respect to bundle morphisms of framed $(D^d,\partial)$-bundles, in the sense that the following diagram for a framed bundle map over $f\colon B\to B'$ commutes.
\[ \xymatrix{
  P_k H^\ell(\calG^\even;\Q) \ar[r]^-{I_*} \ar[rd]^-{I_*} & H^{(d-3)k+\ell}(B;\R)\\
  & H^{(d-3)k+\ell}(B';\R)  \ar[u]_-{f^*}
} \]
\end{enumerate}
\end{Thm}
\begin{Rem}
When $d$ is odd and at least 3, the construction in Definition~\ref{def:I} is also valid if $\calG_\ell^\even$ is replaced by another version $\calG_\ell^\odd$, which is defined similarly as $\calG_\ell^\even$, except that $\R^{\{\mathrm{edges\,of\,\Gamma}\}}$ is replaced by $\R^{\{\mathrm{edges\,of\,\Gamma}\}}\oplus H^1(\Gamma;\R)$ and that the ``induced ori'' in the definition of $\delta$ (\S\ref{ss:graph-complex}) is defined suitably, as in \cite[p.109]{Kon}. The statement (3) of Theorem~\ref{thm:kon} is not true for $d$ odd, although other statements are true also for $d$ odd. The odd case was studied in \cite{Wa1,Wa2}.
\end{Rem}
Since the universal class $\widetilde{\zeta}_k\in P_kH_0(\calG^\even)\otimes P_k\calG^\even_0$ in (\ref{eq:universal}) satisfies $(\mathrm{id}\otimes\delta)\widetilde{\zeta}_k=0$, it follows from Theorem~\ref{thm:kon} (1) that it gives a class
\begin{equation}\label{eq:I-universal}
 I_*([\widetilde{\zeta}_k])=\frac{1}{(2k)!(3k)!}\sum_{\Gamma\in\calL_k^\even} I(\Gamma)[\Gamma] \in H^{(d-3)k}(B;\calA^\even_k\otimes\R). 
\end{equation}
Recall that $\calL_k^\even$ the set of all labelled trivalent graphs with $2k$ vertices with no multiple edges and no self-loops.
When $\dim{B}=(d-3)k$, the evaluation of this class at the fundamental class of $B$ produces an element of $\calA^\even_k\otimes\R$. 
\begin{Cor}\label{cor:bordism-inv}
Let $d$ be an even integer such that $d\geq 4$. The evaluation of $I_*([\widetilde{\zeta}_k])$ for bundles over closed oriented manifold $B$ of dimension $(d-3)k$ gives well-defined linear maps
\[ \begin{split}
  &Z_{k}\colon \pi_{(d-3)k}(\wBDiff(D^d,\partial;\tau_0))\otimes\R\to \calA^\even_k\otimes\R,\\
  &Z_{k}^\Omega\colon \Omega_{(d-3)k}^{SO}(\wBDiff(D^d,\partial;\tau_0))\otimes\R\to \calA^\even_k\otimes\R. 
\end{split}\]
Furthermore, the real homotopy group $\pi_{(d-3)k}(\wBDiff(D^d,\partial;\tau_0))\otimes\R$ can be replaced with $\pi_{(d-3)k}(B\Diff(D^d,\partial))\otimes\R$ in the sense that the natural map $\wBDiff(D^d,\partial;\tau_0)\to B\Diff(D^d,\partial)$ induces an isomorphism in $\pi_{(d-3)k}(-)\otimes\R$.
\end{Cor}
\begin{proof}
We consider a framed $(D^d,\partial)$-bundle over an oriented cobordism $B$ between $(d-3)k$-dimensional manifolds $A_0$ and $A_1$. Let $i_q\colon A_q\to B$, $q=0,1$, be the inclusion. Since $\zeta=I([\widetilde{\zeta}_k])$ gives a closed $(d-3)k$-form on $B$ with coefficients in $\calA^\even_k$, we have
\[ \int_{A_1}i_1^*\zeta - \int_{A_0}i_0^*\zeta = \int_{\partial B} \zeta = \int_B d\zeta =0 \]
by Theorem~\ref{thm:kon} and the Stokes Theorem. This shows the well-definedness of the map. The linearity follows from the linearity of the integrals. 

That $\pi_{(d-3)k}(\wBDiff(D^d,\partial;\tau_0))\otimes\R$ can be replaced with $\pi_{(d-3)k}(B\Diff(D^d,\partial))\otimes\R$ follows since in the long exact sequence for the fibration (\ref{eq:fib-seq-1}) the term $\pi_i(\Omega^dSO_d)\otimes \R$ is zero for $i=(d-3)k, (d-3)k-1$ when $d$ is even, $d\geq 4$, and $k\geq 1$. Indeed, the rational homology of $SO_d$ for $d$ even is well-known (e.g. \cite[3.D]{HatAT}):
\[ H_*(SO_{2n};\Q)\cong \bigwedge(x_3,x_7,\ldots,x_{4n-5},x_{2n-1})\quad(x_i\in H_i(SO_{2n};\Q)). \]
It follows from the isomorphism $\pi_*(G)\otimes\Q\cong PH_*(G;\Q)$ ($P$ denotes the primitive part of a Hopf algebra, \cite[Appendix]{MM}) for $G=SO_{2n}$ that $\pi_*(SO_{2n})\otimes\Q\cong\pi_*(S^3\times S^7\times\cdots\times S^{4n-5}\times S^{2n-1})\otimes\Q$. In particular, the highest $i$ such that $\pi_i(SO_d)\otimes\Q\neq 0$ for $d$ even is $2d-5$ and we have $\{(d-3)k-1+d\}-(2d-5)=(d-3)(k-1)+1>0$. 
\end{proof}
\begin{Rem}
The connecting homomorphism
\[ \pi_{(d-3)k}(B\Diff(D^d,\partial))\otimes\R\to \pi_{(d-3)k-1}(\Omega^dSO_d)\otimes\R \]
is zero when $d$ is even, $d\geq 4$, and $k\geq 1$. On the other hand, without tensoring with $\R$, the group $\pi_i(\Omega^d SO_d)$ may be nontrivial for many $i$. Thus, it would be natural to ask what the homomorphism
\[ \pi_i(B\Diff(D^d,\partial))\to \pi_{i-1+d}(SO_d) \]
is. Since the elements constructed by graph clasper surgery in \S\ref{s:cycles} admit vertical framings, they are in the kernel of this map. As in earlier versions of this paper, one could define configuration space integrals over $\Z$ or $\Z[\frac{1}{M_k}]$ for some explicit integer $M_k$ in terms of piecewise smooth chains in the infinitesimal configuration spaces $\bwConf_{2k}(V)$ or its quotient by $\mathfrak{S}_{2k}$ associated to a vector bundle $V$. They might be related to the above question. Nontriviality of the corresponding homomorphism for $\pi_6(B\Diff(D^{11},\partial))$ is proved in \cite{CSS}.
\end{Rem}

\begin{proof}[Proof of Proposition~\ref{prop:involution}]
Let $\pi\colon E\to S^{(d-3)k}$ and $\pi'\colon E'\to S^{(d-3)k}$ be the $(D^d,\partial)$-bundles corresponding to $\xi$ and $\xi'$, respectively. 
The involution $r$ induces an isomorphism $r\colon E\bConf_n(\pi')\to E\bConf_n(\pi)$. For a propagator $\omega$ on $E\bConf_2(\pi)$, the pullback $r^*\omega$ is $-1$ times a propagator on $E\bConf_2(\pi')$ since the restriction of $r$ to a single normal $(d-1)$-sphere over a point of the diagonal $\Delta_E$ is orientation reversing. Also, $r_*o(E\bConf_{2k}(\pi'))=(-1)^{2k}o(E\bConf_{2k}(\pi))$.
Hence we have
\[ \begin{split}
\int_{E\bConf_{2k}(\pi')}\omega(\Gamma')_{\pi'}&=(-1)^{3k}\int_{E\bConf_{2k}(\pi')}r^*\omega(\Gamma')_{\pi}
=(-1)^{2k}(-1)^{3k}\int_{E\bConf_{2k}(\pi)}\omega(\Gamma')_{\pi}.
\end{split}\]
\end{proof}

%%%%%%%%%%%%%%%%%%%%%%%%%%%%%%%

%%%%%%%%%%%%%%%%%%%%%%%%%%%%%%%
%%%%%%%%%%%%%%%%%%%%%%%%%%%%%%%
\mysection{Surgery on graph claspers}{s:cycles}

In this section, we construct $(D^d,\partial)$-bundles by an analogue of Goussarov--Habiro's graph-clasper surgery that will be detected by $Z_k$ of Corollary~\ref{cor:omega-cobordism}, and review some fundamental properties of the surgery.  

%%%%%%%%%%%%%%%%%%%%%%%%%%%%%%%
\subsection{Hopf link and Borromean link (e.g., \cite[\S{3}]{Ma})}\label{ss:borromean}

Graph-clasper surgery is constructed by combining Hopf links and Borromean links. If $d$ is a positive integer and if $p,q$ are integers such that $0< p,q< d-1$ and $p+q=d-1$, then the Hopf link is defined as the two-component link $H(p,q)_d\colon S^p\cup S^q\to \R^d$, whose components are given by the inclusions of the following submanifolds
\[ \begin{split}
&\textstyle \{(t,u,v)\in\R^d\mid |t|^2+|u|^2=1,\,\, v=0\}, \\
&\textstyle \{(t,u,v)\in\R^d\mid |t-1|^2+|v|^2=1,\,\, u=0\},
\end{split} \]
where we consider $\R^d=\R\times \R^p\times \R^q$. A standard (normal) framing for the Hopf link is given as follows. Let $n_1,n_2$ be the outward unit normal vector field on the two components $H(p,q)_d(S^p)\subset \R\times\R^p\times\{0\}$ and $H(p,q)_d(S^q)\subset \R\times\{0\}\times\R^q$, respectively, both codimension 1. Then the normal framings on the two components in $\R^d$ are given by $(n_1,\partial v_1,\ldots, \partial v_q)$, $(n_2,\partial u_1,\ldots, \partial u_p)$, respectively. See \S\ref{ss:notations}(j) for the convention of normal framing.

If $d$ is a positive integer and if $p,q,r$ are integers such that $0<p,q,r<d-1, p+q+r=2d-3$, then the Borromean link is defined as the three-component link
$S^p\cup S^q\cup S^r\to \R^d$, 
whose components are given by the inclusions of the following submanifolds
\begin{equation}\label{eq:bor}
\begin{split}
&\textstyle L_1=\{(x,y,z)\in\R^d\mid \frac{|y|^2}{4}+|z|^2=1,\,\,x=0\}, \\
&\textstyle L_2=\{(x,y,z)\in\R^d\mid \frac{|z|^2}{4}+|x|^2=1,\,\,y=0\}, \\
&\textstyle L_3=\{(x,y,z)\in\R^d\mid \frac{|x|^2}{4}+|y|^2=1,\,\, z=0\},
\end{split} 
\end{equation}
where we consider $\R^d=\R^{d-p-1}\times \R^{d-q-1}\times \R^{d-r-1}$. We denote by $B(p,q,r)_d$ this link.
A standard (normal) framing for the Borromean link is given as follows. Let $n_1,n_2,n_3$ be the outward unit normal vector field on the three components $L_1\subset \{0\}\times\R^{p+1}$, $L_2\subset \R^{d-p-1}\times\{0\}\times\R^{d-r-1}$, $L_3\subset \R^{r+1}\times\{0\}$, respectively. Then the normal framings on the three components in $\R^d$ are given by $(n_1,\partial x_1,\ldots, \partial x_{d-p-1})$, $(n_2,\partial y_1,\ldots, \partial y_{d-q-1})$, $(n_3,\partial z_1,\ldots, \partial z_{d-r-1})$, respectively. The Borromean links have the following significant feature, which is well-known, or can be checked easily from the coordinate description (\ref{eq:bor}).
\begin{Lem}\label{lem:borromean}
If one of the three components in a Borromean link is removed, then the link consisting of the remaining components can be isotoped into an unlink. Here, the trivializing isotopy can be taken so that it fixes neighborhoods of the points
\[ (0,\ldots,0,-2)\times 0\times 0,\quad 0\times (0,\ldots,0,-2)\times 0,\quad 0\times 0\times (0,\ldots,0,-2)\]
in $\R^{d-p-1}\times\R^{d-q-1}\times\R^{d-r-1}$ on the components.
\end{Lem}

\begin{Rem}
\begin{enumerate}
\item We will also call a link that is isotopic to $H(p,q)_d$ (resp. $B(p,q,r)_d$) a Hopf link (resp. a Borromean link). We will use the same symbol $H(p,q)_d$ (resp. $B(p,q,r)_d$) for its isotopic alternative, abusing of notation (like $T(p,q)$, $\Sigma(p,q,r)$ in low-dimensional topology). Similar convention applies to $B(\underline{p},\underline{q},\underline{r})_d$ etc. in Definition~\ref{def:long-borromean-link} below.

\item For each component $L_i$ in the Borromean link, let $D_i$ be the standard spanning disk defined by replacing the `$=1$' by `$\leq 1$' in (\ref{eq:bor}). 
The spanning disks $D_i$ have natural coorientations $\partial x_1\wedge\cdots\wedge \partial x_{d-p-1}$, $\partial y_1\wedge\cdots\wedge \partial y_{d-q-1}$, $\partial z_1\wedge\cdots\wedge \partial z_{d-r-1}$, respectively. They determine the orientations of the components of $B(p,q,r)_d$ by the rule (\ref{eq:coori}). 
\end{enumerate}
\end{Rem}

The spanning disks $D_i$ have triple intersection at the origin and its intersection number is $+1$. The intersection of the spanning disk $D_i$ of $L_i$ with an other component $L_j$, which is a sphere or empty, can be resolved by a surgery, which is given by attaching to $D_i$ the normal sphere bundle of $L_j$ restricted to a submanifold of $L_j$ and by removing the interior of the normal disk bundle of $D_i\cap L_j$ whose boundary agrees with the boundary of the normal sphere bundle attached. The detail of this surgery is described in \cite[\S{3.3}]{Tak}. Let $D_i'$ be the result of the surgery for $D_i$ (Figure~\ref{fig:borromean-link}). The following lemma is evident from the definition of the Borromean link by (\ref{eq:bor}).
\begin{figure}
\[ \includegraphics[height=40mm]{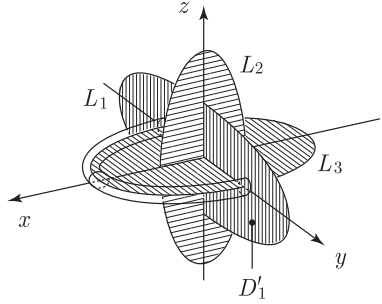} \]
\caption{The spanning surface $D_1'$ of $L_1$.}\label{fig:borromean-link}
\end{figure}

\begin{Lem}\label{lem:spanning-disk}
\begin{enumerate}
\item $D_i'$ is a compact submanifold of $\R^d$ bounded by $L_i$, which is disjoint from other two link components and is diffeomorphic to $D_i\#(S^u\times S^v)$ for some $u,v$ such that $u+v=\dim{L_i}+1$. More explicitly, 
\[ \begin{split}
  &D_1'\cong D_1\# (S^{d-q-1}\times S^{d-r-1}),\quad D_2'\cong D_2\#(S^{d-p-1}\times S^{d-r-1}),\\
  &D_3'\cong D_3\#(S^{d-p-1}\times S^{d-q-1}).
\end{split}\]
\item The normal bundle of $D_i'$ is trivial.
\item $D_1'\cap D_2'\cap D_3'=D_1\cap D_2\cap D_3$ and the triple intersection number of $D_1',D_2',D_3'$ counted with sign is $+1$.
\end{enumerate}
\end{Lem}

\begin{Def}[Suspension of the Borromean link]\label{def:suspension}
The {\it suspension} of the Borromean link $B(p,q,r)_d$ is the link in $\R^{d+1}$ defined by replacing $z\in \R^{d-r-1}$ in the equations (\ref{eq:bor}) for the three components with $z'=(z,t)\in\R^{d-r-1}\times\R$, which is $B(p+1,q+1,r)_{d+1}$ and its intersection with $\R^d\times\{0\}$ is $B(p,q,r)_d$. 
The normal framing of $B(p,q,r)_d$ extends naturally to $B(p+1,q+1,r)_{d+1}$ by extending the outward unit normal vector fields. By symmetry of the equations (\ref{eq:bor}), suspensions for other variables $x$, $y$ are defined similarly. 
\end{Def}
Also, the explicit conditions in (\ref{eq:bor}) suggest that the ``desuspension'' is possible whenever two of the $p,q,r$ are at least $2$. For example, if $p,q\geq 2$, then that $B(p,q,r)_d$ is the suspension of $B(p-1,q-1,r)_{d-1}$ can be seen by restricting $z=(z',t)\in\R^{d-r-1}=\R^{(d-1)-r-1}\times \R$ to $(z',0)$. 

%%%%%%%%%%%%%%%%%%%%%%%%%%%%%%%
\subsection{Long Borromean link}\label{ss:long-borromean}

\begin{Def}
For $0<p,q,r<d$, let $\fEmb(I^p\cup I^q\cup I^r,I^d)$ denote the space of strata preserving (Appendix~\ref{s:mfd-corners}), normally framed embeddings of $I^p\cup I^q\cup I^r$ into $I^d$ such that 
\begin{enumerate}
\item the preimage of $\partial I^d$ agrees with the boundary of the domain, and
\item embeddings are transversal to the boundary.
\end{enumerate}
We allow components and normal framings on them to be non standard near the boundary, though what we will need later is the subspace of $\fEmb(I^p\cup I^q\cup I^r,I^d)$ defined by imposing some boundary conditions. 
We call an affine embedding $f\colon \R^p\to \R^d$ or its restriction to $f^{-1}(I^d)$, suitably reparametrized so that the restriction is an embedding from $I^p=f^{-1}(I^d)$, a {\it standard inclusion}.
We call an element of $\fEmb(I^p\cup I^q\cup I^r,I^d)$ a {\it (framed) string link}, and call a path in $\fEmb(I^p\cup I^q\cup I^r,I^d)$ a {\it (framed) isotopy} of framed long embeddings.
\end{Def}

The subspace of $\fEmb(I^p\cup I^q\cup I^r,I^d)$ of framed embeddings such that some framed components are standard near the boundaries, i.e., agree with standard inclusions near the boundaries, is denoted like $\fEmb(\underline{I}^p\cup I^q\cup I^r,I^d)$, where the underlined component(s) is standard near the boundary. Here, we fix a standard inclusion $L_{\mathrm{st}}\colon I^p\cup I^q\cup I^r\to I^d$ given by
\[ I^p\stackrel{\subset}{\to} I^{d-1}\stackrel{=}{\to} \{p_1\}\times I^{d-1},I^q\stackrel{\subset}{\to} I^{d-1}\stackrel{=}{\to} \{p_2\}\times I^{d-1},I^r\stackrel{\subset}{\to} I^{d-1}\stackrel{=}{\to} \{p_3\}\times I^{d-1}\]
for fixed distinct points $p_1,p_2,p_3\in (0,1)$, where the inclusion $I^p\stackrel{\subset}{\to} I^{d-1}$ etc. is given by $(x_1,\ldots,x_p)\mapsto (x_1,\ldots,x_p,\frac{1}{2},\ldots,\frac{1}{2})$ etc. We equip the standard inclusion with the standard normal framing given by the euclidean coordinates. 
The subspace of $\fEmb(\underline{I}^p\cup \underline{I}^q\cup \underline{I}^r,I^d)$ consisting of framed embeddings that are relatively isotopic to the standard inclusion is denoted by $\fEmb_0(\underline{I}^p\cup \underline{I}^q\cup \underline{I}^r,I^d)$.

\begin{Def}[Long Borromean link]\label{def:long-borromean-link}
Given a link $L\colon \R^p\cup \R^q\cup \R^r \to \R^d$ consisting of disjoint standard inclusions, and a Borromean link $B(p,q,r)_d$ that is disjoint from $L$, we join the images of $\R^p$ and $S^p$, $\R^q$ and $S^q$, $\R^r$ and $S^r$, by three mutually disjoint arcs that are also disjoint from components of the links $L$ and of the spanning disks $D_i$ of $B(p,q,r)_d$ except their endpoints. Then replace the arcs with thin tubes $S^{p-1}\times I$, $S^{q-1}\times I$, $S^{r-1}\times I$ to construct connected sums. The result is a long link 
$B(\underline{p},\underline{q},\underline{r})_d\colon \R^p\cup \R^q\cup \R^r\to \R^d$ with a natural framing $F_D$ in the sense of connected sum of framed submanifolds (e.g., \cite[Ch.IX,2]{Kos}). 
\end{Def}

One may also consider partial connected sum, which joins $B(p,q,r)_d$ to the link $L$ of standard inclusions with less components and denote the resulting embedding by $B(\underline{p},\underline{q},r)_d$ etc. Long Borromean embeddings $I^p\cup I^q\cup I^r\to I^d$ such that the preimage of $\partial I^d$ is $\partial I^p\cup \partial I^q\cup \partial I^r$ can also be defined similarly and we denote them by the same symbols as above. A natural analogue of Lemma~\ref{lem:borromean} for the long Borromean link holds. Also a natural analogue of Lemma~\ref{lem:spanning-disk} for the long Borromean link holds: For each component $\underline{L_i}$ in the long Borromean link, let $\underline{D_i}$ be the standard spanning disk obtained from $D_i$ by boundary connect-summing the half-planes
\begin{equation}\label{eq:half-planes}
 \begin{split}
  &\{p_1\}\times [0,\textstyle\frac{1}{2}]\times I^p\times \bigl\{(\textstyle\underbrace{\textstyle\frac{1}{2},\ldots,\frac{1}{2}}_{d-2-p})\bigr\},\quad
  \{p_2\}\times [0,\textstyle\frac{1}{2}]\times I^q\times \bigl\{(\textstyle\underbrace{\textstyle\frac{1}{2},\ldots,\frac{1}{2}}_{d-2-q})\bigr\},\\
  &\{p_3\}\times [0,\textstyle\frac{1}{2}]\times I^r\times \bigl\{(\textstyle\underbrace{\textstyle\frac{1}{2},\ldots,\frac{1}{2}}_{d-2-r})\bigr\}.
\end{split} 
\end{equation}
The intersection of the spanning disk $\underline{D_i}$ of $\underline{L_i}$ with an other component $\underline{L_j}$, which is a sphere or empty, can be resolved by a surgery as before. Let $\underline{D_i'}$ be the result of the surgery for $\underline{D_i}$ (Figure~\ref{fig:borromean-link-long}). 
\begin{Lem}\label{lem:spanning-disk-long}
\begin{enumerate}
\item $\underline{D_i'}$ is a compact submanifold of $I^d$ whose boundary agrees with that of the $i$-th half plane in (\ref{eq:half-planes}), which is disjoint from other two string link components and is diffeomorphic to $\underline{D_i}\#(S^u\times S^v)$ for some $u,v$ such that $u+v=\dim{L_i}+1$. 
\item The normal bundle of $\underline{D_i'}$ is trivial.
\item $\underline{D_1'}\cap \underline{D_2'}\cap \underline{D_3'}=D_1\cap D_2\cap D_3$ and the triple intersection number of $\underline{D_1'},\underline{D_2'},\underline{D_3'}$ counted with sign is $+1$.
\end{enumerate}
\end{Lem}
\begin{figure}
\[ \includegraphics[height=40mm]{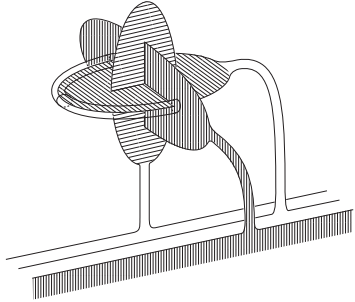} \]
\caption{Long Borromean link and the spanning surface $\underline{D_1'}$.}\label{fig:borromean-link-long}
\end{figure}

A suspension of the long string link $B(\underline{p},\underline{q},\underline{r})_d$ can be defined analogously to that of $B(p,q,r)_d$. In fact, a suspension can be defined for more general string links. A precise definition of a suspension of a string link is given in Definition~\ref{def:suspension-string} later, which is slightly complicated. What will be important below is the following lemma, which can be seen from Definition~\ref{def:suspension-string}. 

\begin{Lem}
The following procedures yield the same result up to relative isotopy:
\begin{enumerate}
\item $B(p,q,r)_d
  \stackrel{{\mathrm{connected}}\atop{\mathrm{sum}}}{\xrightarrow{\hspace*{8mm}}} 
  B(\underline{p},\underline{q},\underline{r})_d
  \stackrel{\mathrm{suspension}}{\xrightarrow{\hspace*{8mm}}} 
  \bigl\{B(\underline{p},\underline{q},\underline{r})_d\bigr\}'$.
\item $B(p,q,r)_d
  \stackrel{\mathrm{suspension}}{\xrightarrow{\hspace*{8mm}}} 
  B(p+1,q+1,r)_{d+1}
  \stackrel{{\mathrm{connected}}\atop{\mathrm{sum}}}{\xrightarrow{\hspace*{8mm}}} 
  B(\underline{p+1},\underline{q+1},\underline{r})_{d+1}$.
\end{enumerate}
\end{Lem}

%%%%%%%%%%%%%%%%%%%%%%%%%%%%%%%
\subsection{Vertex oriented arrow graph}\label{ss:arrow-graph}

We impose extra combinatorial structures on a labelled trivalent graph: an arrow orientation and a vertex orientation. They are used to decompose the graph into two types of vertices, each equipped with an orientation. 

\subsubsection{Arrow graph}
We orient each edge of a trivalent graph such that each vertex has both input and output incident edges. That any trivalent graph without self-loop admits such an orientation follows by induction on the number of edges: there is an edge $e$ in a trivalent graph without self-loop such that removing $e$ yields a graph with two bivalent vertices. Then merging the two edges indicent to each bivalent vertex gives a trivalent graph with less edges. We call a trivalent graph without self-loop equipped with such an orientation an {\it arrow graph}. Possible status of input/output of the three incident edges at a vertex of an arrow graph are as shown in the following picture:
\[ \includegraphics[height=23mm]{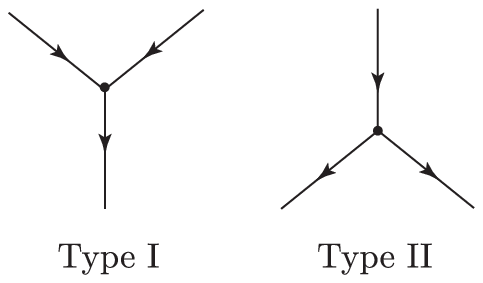} \]
Note that it is possible to include graphs with self-loops in the following constructions though we exclude these for simplicity.

\subsubsection{Vertex orientation}
To define vertex orientation, we decompose each edge $e$ of an arrow graph $\Gamma$ into half-edges $H(e)=\{e_-,e_+\}$ ordered according to the arrow orientation of $e$, namely, so that $e_-$ includes the input vertex and $e_+$ includes the output vertex. We denote by $\frac{1}{2}\mathrm{Edges}(\Gamma)$ the set of all half-edges in $\Gamma$. Then a {\it vertex orientation} of a vertex $v$ of $\Gamma$ is a choice of linear ordering of the three half-edges meeting at $v$.

\subsubsection{Half-edge orientation}\label{ss:he-ori}
Given a vertex labelled arrow graph, the following notions of orientations are canonically equivalent:
\begin{itemize}
\item[(a)] An orientation of $\R^{\mathrm{Edges}(\Gamma)}$ (as in \S\ref{ss:graph-complex}). 
\item[(b)] An orientation of $\R^{\frac{1}{2}\mathrm{Edges}(\Gamma)}$.
\end{itemize}
Here we consider $\R^{\frac{1}{2}\mathrm{Edges}(\Gamma)}$ as a graded vector space by setting the degrees of the half-edges in $H(e)=\{e_-,e_+\}$ as $\deg\,{e_+}=1$, $\deg\,{e_-}=d-2$ for each edge $e$. The correspondence between them is canonically given using the arrow orientation by
\[ \begin{split}
&e_1\wedge\cdots\wedge e_{3k}\leftrightarrow(e_{1+}\wedge e_{1-})\wedge \cdots\wedge (e_{3k+}\wedge e_{3k-})\quad (H(e_i)=\{e_{i-},e_{i+}\}).
\end{split} \]

\subsubsection{Vertex labelled, vertex oriented arrow graph compatible with (a)-orientation}\label{ss:vertex-ori-compatible}

If a vertex labelled, vertex oriented arrow graph is given, then an orientation in the sense of (b) above is given by 
\[ \upsilon_1\wedge\upsilon_2\wedge\cdots\wedge \upsilon_{2k},\quad  \upsilon_i=e_{p\pm}\wedge e_{q\pm}\wedge e_{r\pm}, \]
where $e_{p\pm}, e_{q\pm}, e_{r\pm}$ are the half-edges meeting at the $i$-th vertex ($\pm$ are determined by the arrow orientation). When $d$ is even, the term $\upsilon_i$ determine the relative orders of the degree 1 half-edges at each type I vertex, up to an even number of transpositions. 

In this section, we fix one choice of vertex orientation and arrow orientations for a given labelled trivalent graph so that they give a compatible orientation in the sense of (b) determined by the edge labels.

%%%%%%%%%%%%%%%%%%%%%%%%%%%%%%%
\subsection{Y-link associated to trivalent graph}\label{ss:Y-link}

Let $X$ be a compact $d$-manifold. Given a framed embedding $f\colon \Gamma\to \mathrm{Int}\,X$ of a vertex labelled, vertex oriented arrow graph $\Gamma$ whose restriction to each edge is smooth, we associate a Y-link $G=G_1\cup\cdots\cup G_{2k}$ in $X$ as follows (Figure~\ref{fig:G-to-Y-graph}).
\setcounter{footnote}{0}
\begin{enumerate}
\item For each edge $e$ of $\Gamma$, let $P(e)\subset \mathrm{Int}\,X$ be a small closed $d$-ball centered at the middle point of $f(e)$ such that $P(e)$ is disjoint from vertices and other edges of $f(\Gamma)$. Further, we assume that $P(e)\cap P(e')=\emptyset$ if $e\neq e'$, and that $P(e)\cap f(e)$ is diffeomorphic to a closed interval.
\item We decompose the closed interval $P(e)\cap f(e)$ into three subintervals: $P(e)\cap f(e)=[a,b]\cup [b,c]\cup [c,d]$, in a way that the image of the input (resp. output) vertex under $f$ is $a$ (resp. $d$). Then we remove the middle one $[b,c]$ and attach a suitably rescaled standard Hopf link $S^1\cup S^{d-2}\to \mathrm{Int}\,P(e)$ instead, so that the image of $S^{d-2}$ is attached to $b\in [a,b]$ and the image of $S^1$ is attached to $c\in [c,d]$. 
\[ \includegraphics[height=20mm]{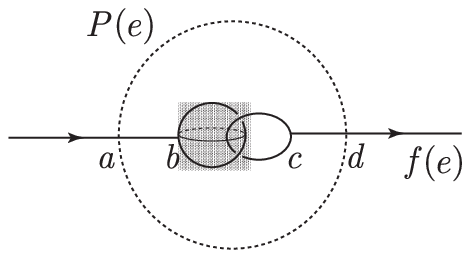} \]
\item We give orientations of the components of the Hopf link by $\partial u_1$ at $(1,0,\ldots,0)\in H(1,d-2)_d(S^1)$ and by $\partial v_1\wedge\cdots\wedge \partial v_{d-2}$ at $(0,0,\ldots,0)\in H(1,d-2)_d(S^{d-2})$ in the coordinates of \S\ref{ss:borromean}\footnote{Note that the latter is opposite to the usual one induced from the standard orientation of the $tv$-plane $\R\times\{0\}\times\R^q$.}. These are chosen so that their linking number is $+1$.
\end{enumerate}
\begin{figure}
\begin{center}
\includegraphics[height=40mm]{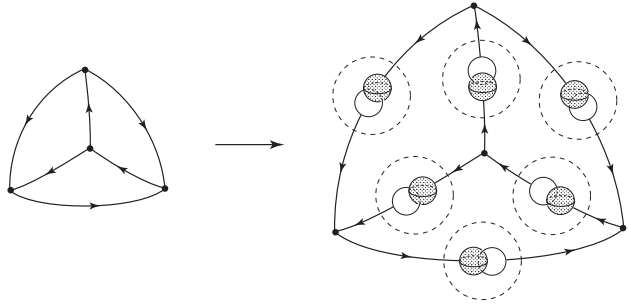}
\end{center}
\caption{An embedded arrow graph to a Y-link}\label{fig:G-to-Y-graph}
\end{figure}

Here, the linking number of a two component link $a\cup b\colon S^p\cup S^q\to \mathrm{Int}\,P(e)$ with $p+q=d-1$ is defined by the usual formula:
\begin{equation}\label{eq:Lk}
\begin{split} 
&\mathrm{Lk}(a,b)=\int_{S^p\times S^q}\phi^*\mathrm{Vol}_{S^{d-1}},\\
&\phi\colon S^p\times S^q\to S^{d-1}; \quad\phi(x,y)=\frac{b(y)-a(x)}{|b(y)-a(x)|},
\end{split}
\end{equation}
where we identify $\mathrm{Int}\,P(e)$ with an open set of $\R^d$, $\mathrm{Vol}_{S^{d-1}}$ is the unit volume form in (\ref{eq:vol}), and we give orientation of $S^p\times S^q$ by $o(S^p)\wedge o(S^q)$ (as in \S\ref{ss:ori-prod}).

The above procedure gives a disjoint union $G_1\cup G_2\cup \cdots \cup G_{2k}$ of path-connected objects with $2k=|V(\Gamma)|$ components. We call each component $G_i$ a {\it Y-graph}, and $G=G_1\cup G_2\cup \cdots \cup G_{2k}$ a {\it Y-link} (or a {\it graph clasper}). There are two types for a Y-graph, according to whether the corresponding vertex is of type I or II in the following figure:
\[  \includegraphics[height=27mm]{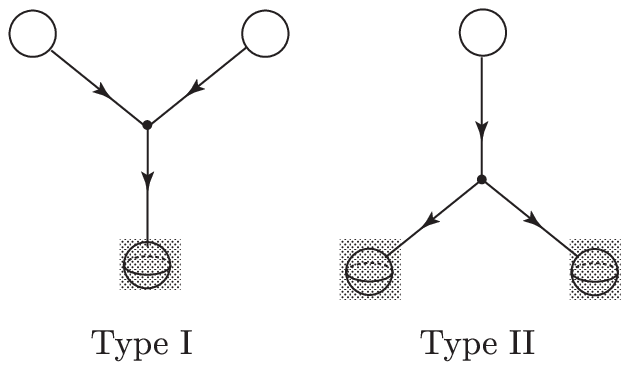} \]

By taking a small smooth closed tubular neighborhood $V_i\subset\mathrm{Int}\,X$ for each component $G_i$, we obtain a tuple $\vec{V}_G=(V_1,\ldots,V_{2k})$ of mutually disjoint handlebodies in $\mathrm{Int}\,X$. Here, by a small closed tubular neighborhood of $G_i$, we mean the union of piecewise small tubular neighborhoods, where we consider $G_i$ consists of three oriented spheres (consisting of $S^1$ and $S^{d-2}$), a trivalent vertex, and three edges connecting them. We take the radii of the tubular neighborhoods of edges to be less than half the radii of the tubular neighborhoods of the vertex components and we smooth the corners.

%%%%%%%%%%%%%%%%%%%%%%%%%%%%%%%
\subsection{Surgery along Y-links}\label{ss:surgery-Y}

The surgery on a Y-graph will be defined by a parametrized Borromean surgery, which roughly replaces the exterior of a trivial string link with the exterior of Borromean string link. 
We shall construct a $(X,\partial)$-bundle by a family of surgeries along $\vec{V}_G=(V_1,\ldots,V_{2k})$.
We take a smooth family $\alpha_i\colon K\to \Diff(\partial V_i)$ of diffeomorphisms parametrized by a compact manifold $K$ with $\partial K=\emptyset$. This defines a bundle automorphism $\bar{\alpha}_i\colon K\times \partial V_i\to K\times \partial V_i$ of the trivial $\partial V_i$-bundle over $K$ by $\bar{\alpha}_i(t,x)=(t,\alpha_i(t)x)$. We put
\begin{equation}\label{eq:alpha-gluing}
 (K\times X)^{V_i,\alpha_i}:=(K\times (X-\mathrm{Int}\,V_i))\cup_{\bar{\alpha}_i}(K\times V_i), 
\end{equation}
where the fiberwise boundaries are glued together by $\bar{\alpha}_i$ in a way that $(t,x)\in K\times\partial V_i\subset K\times V_i$ is identified to $\bar{\alpha}_i(t,x)\in K\times\partial V_i\subset K\times (X-\mathrm{Int}\,V_i)$. This defines a surgery along $V_i$ with respect to $\alpha_i$, which yields a smooth fiber bundle over $K$. The product structures on the two parts induce a bundle projection $\pi(\alpha_i)\colon (K\times X)^{V_i,\alpha_i}\to K$. 

Since the handlebodies $V_i$ are mutually disjoint, the surgery can be done for every $V_i$ simultaneously. Namely, taking $\vec{\alpha}=(\alpha_1,\ldots,\alpha_{2k})$, $\alpha_i\colon K_i\to \Diff(\partial V_i)$, we do surgery at each $V_i$ by using $\alpha_i$, and then we obtain a family of surgeries parametrized by $K_1\times \cdots \times K_{2k}$ and a bundle projection
\[ \pi(\vec{\alpha})\colon \bigl(K_1\times\cdots\times K_{2k}\times X\bigr)^{\vec{V}_G,\vec{\alpha}}\to K_1\times\cdots\times K_{2k}.\]
More precisely, let 
\[ V_\infty=X-\mathrm{Int}\,(V_1\cup\cdots\cup V_{2k}) \]
and we define $((\prod_{i=1}^{2k}K_i)\times X)^{\vec{V}_G,\vec{\alpha}}$ by the parametrized gluing of the two trivial bundles
\[ \begin{split}
&\Bigl(\prod_{i=1}^{2k}K_i\Bigr)\times V_\infty\quad\mbox{and}\quad
\Bigl(\prod_{i=1}^{2k}K_i\Bigr)\times (V_1\cup\cdots\cup V_{2k})
\end{split} \]
along the fiberwise boundary $(\prod_{i=1}^{2k}K_i)\times(\partial V_1\cup\cdots \cup\partial V_{2k})$ by the map
\[ \begin{split}
 \vec{\alpha}{}'\colon&\, \Bigl(\prod_{i=1}^{2k}K_i\Bigr)\times(\partial V_1\cup\cdots \cup\partial V_{2k})
\to \Bigl(\prod_{i=1}^{2k}K_i\Bigr)\times(\partial V_1\cup\cdots\cup \partial V_{2k}); \\
&(t_1,\ldots,t_{2k},x)\mapsto (t_1,\ldots,t_{2k},(\alpha_1(t_1)\cup\cdots\cup \alpha_{2k}(t_{2k}))x).
\end{split}\]
This defines a surgery along a Y-link with respect to $\vec{\alpha}$, which yields a smooth fiber bundle over $\prod_i K_i$. 

In the following, we take $\alpha_i=\alpha_{\mathrm{I}}$ or $\alpha_{\mathrm{II}}$ defined below for each $i$. We write $V=V_i$ for simplicity. 
\begin{enumerate}
\item If $V$ is of type I, we take $K=S^0=\{-1,1\}$, and we let $\alpha_\mathrm{I}\colon S^0\to \Diff(\partial V)$ map $(-1)$ to the identity map of $\partial V$, and $\alpha_\mathrm{I}(1)$ be a ``Borromean twist associated to $B(\underline{d-2},\,\underline{d-2},\,\underline{1})_d$'' constructed in \S\ref{ss:p-borr-I}.

\item If $V$ is of type II, we take $K=S^{d-3}$ and we let $\alpha_{\mathrm{II}}\colon S^{d-3}\to \Diff(\partial V)$ be a ``parametrized Borromean twist associated to $B(\underline{d-2},\,\underline{d-2},\,\underline{1})_d$'' constructed in \S\ref{ss:p-borr-II}.
\end{enumerate}

We now consider the special case $X=D^d$ and define the main construction.

\begin{Def}\label{def:pi-Gamma} Let $\Gamma$ be a vertex oriented, vertex labelled arrow graph with $2k$ vertices without self-loop. Fix a framed embedding $f\colon \Gamma\to \mathrm{Int}\,D^d$. We use the framing from $f$ and the vertex orientation of \S\ref{ss:arrow-graph} to associate the components in the Borromean string link $B(\underline{d-2},\,\underline{d-2},\,\underline{1})_d$ to the three handles of a handlebody $V_i$ at each vertex. According to the type of the $i$-th vertex of $\Gamma$, we put $\alpha_i=\alpha_{\mathrm{I}}$ or $\alpha_{\mathrm{II}}$, and let $\vec{\alpha}=(\alpha_1,\ldots,\alpha_{2k})$. Then we define a smooth fiber bundle $\pi^\Gamma\colon E^\Gamma\to B_\Gamma$ by
\[ \pi^\Gamma=\pi(\vec{\alpha}),\quad 
B_\Gamma=\prod_{i=1}^{2k}K_i,\quad 
E^\Gamma=(B_\Gamma\times D^d)^{\vec{V}_G,\vec{\alpha}}. \]
We also consider the straightforward analogue of this surgery for $(S^d,U_\infty)$-bundles which is given by replacing $D^d$ with $S^d$ in the definition above, to compute invariants in \S\ref{s:computation}. 
\end{Def}

In a joint work with Botvinnik (\cite[\S{3}]{BW}), we give another interpretation of $\pi^\Gamma$ in terms of surgeries on families of framed links in $D^d$, which would be more simple, though Definition~\ref{def:pi-Gamma} is suitable for proving the main theorem of this paper.

\begin{Thm}[Proof in \S\ref{ss:framed-surgery} for (1), (2) and in \S\ref{s:computation} for (3)]\label{thm:Z(G)}
Let $d$ be an even integer such that $d\geq 4$. Let $\Gamma$ be as in Definition~\ref{def:pi-Gamma}.
\begin{enumerate}
\item $\pi^\Gamma\colon E^\Gamma\to B_\Gamma$ is a $(D^d,\partial)$-bundle and admits a canonical vertical framing $\tau^\Gamma$.
\item The framed $(D^d,\partial)$-bundle bordism class of $(\pi^\Gamma\colon E^\Gamma\to B_\Gamma,\tau^\Gamma)$ is contained in the image of the natural map
\[ H\colon \pi_{(d-3)k}(\wBDiff(D^d,\partial))\to \Omega_{(d-3)k}^{SO}(\wBDiff(D^d,\partial)). \]
\item If $\Gamma$ has no multiple edges, we have 
\[ Z_k^\Omega(\pi^\Gamma;\tau^\Gamma)=\pm[\Gamma],\]
where the sign depends only on $k$ (not on $\Gamma$ in $P_k\calG_0^\even$).
\end{enumerate}
\end{Thm}

Theorem~\ref{thm:main} follows immediately from Theorem~\ref{thm:Z(G)}. Namely, let
\[ \Psi_k\colon P_k\calG^\even_0\to \mathrm{Im}\,H\otimes \Q\]
be a $\Q$-linear function defined by $\Psi_k(\Gamma)=[\pi^\Gamma\colon E^\Gamma\to B_\Gamma]$ by fixing labels and arrows on $\Gamma$ arbitrarily for each class. Recall that $P_k\calG^\even_0$ is the subspace of $\calG^\even$ spanned by trivalent graphs of degree $k$. Then by Theorem~\ref{thm:Z(G)}(3), the composition
\[ P_k\calG^\even_0\otimes\R\stackrel{\Psi_k}{\longrightarrow} \mathrm{Im}\,H\otimes \R
 \stackrel{\pm Z_k^\Omega}{\longrightarrow} P_k H_0(\calG^\even;\R)=\calA_k^\even\otimes\R \]
agrees with the quotient map $P_k\calG^\even_0\otimes\R\to P_k H_0(\calG^\even;\R)$. Hence $Z_k=Z_k^\Omega\circ H$ is surjective over $\R$ and Theorem~\ref{thm:main} follows. 
\begin{Rem}\label{rem:depend-choice}
We have chosen the framed embedding $f$, the labels, vertex orientation, and arrow orientations on graphs to define $\Psi_k$ as an auxiliary data. In particular, we have not proved $\Psi_k(-\Gamma)=-\Psi_k(\Gamma)$, which seems likely to be true. 
We do not know whether the bordism class of $\Psi_k(\Gamma)$ changes under a change of the choice of the vertex orientation and the arrow orientations which preserves graph orientation. Although it would not be hard to determine the effect of different choices in the bordism group, it is not necessary for our purpose.
\end{Rem}

Let $X$ be a compact $d$-manifold. For a framed embedding $f\colon \Gamma\to X$ of a vertex oriented labelled arrow graph $\Gamma$ with $2k$ vertices, one may also consider the $(X, \partial)$-bundle $\pi^f\colon E^f\to B_\Gamma$ by surgery on $f$ given by replacing $D^d$ in Definition~\ref{def:pi-Gamma} with $X$. The following theorem can be proved just by replacing $D^d$ with $X$ in the proof of Theorem~\ref{thm:Z(G)} (1), (2).
\begin{Thm}
The relative bundle bordism class of $\pi^f$ represents an element of $\Omega_{(d-3)k}^{SO}(B\Diff(X,\partial))$, which is contained in the image of the natural map\\
$H\colon \pi_{(d-3)k}(B\Diff(X,\partial))\to \Omega_{(d-3)k}^{SO}(B\Diff(X,\partial))$.
\end{Thm}
The class of $\pi^f$ does not change if $f$ is replaced within the same homotopy class, which can be described by $\Gamma$ as above with edges decorated by elements of $\pi_1(X)$, considered modulo certain relations as in \cite[p.566]{GL}. Note that the same remark as Remark~\ref{rem:depend-choice} applies to this case.

%%%%%%%%%%%%%%%%%
\begin{Exa}[$k=2$, $\Gamma=W_4$]
Now we consider the complete graph $W_4$, edge-oriented as in the following picture:
\[ \includegraphics[height=25mm]{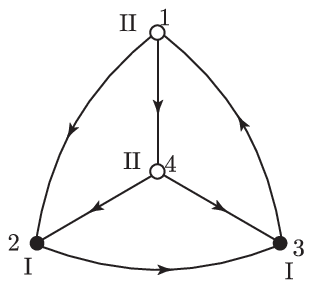} \]
 In this case, $B_{W_4}=K_1\times K_2\times K_3\times K_4$, where $K_1=K_4=S^{d-3}$ and $K_2=K_3=S^0$. Hence $B_{W_4}$ is the disjoint union of four components $B_{t_2,t_3}=K_1\times \{(t_2,t_3)\}\times K_4$, $t_2,t_3=\pm 1$, each canonically diffeomorphic to $S^{d-3}\times S^{d-3}$. It will turn out from Lemma~\ref{lem:A} that the restriction of the $(D^d,\partial)$-bundle $\pi^{W_4}\colon E^{W_4}\to B_{W_4}$ over $B_{t_2,t_3}$, $(t_2,t_3)\neq (1,1)$, is a trivial $(D^d,\partial)$-bundle. Let us focus on the restriction of $\pi^{W_4}$ to the only component $E^{W_4}_{1,1}:=(\pi^{W_4})^{-1}(B_{1,1})$ that may be nontrivial. This is constructed by gluing the pieces
\[ \begin{split}
  &B_{1,1}\times V_\infty,\quad \widetilde{V}_1'=\widetilde{V}_1\times K_4,\quad \widetilde{V}_4'=K_1\times \widetilde{V}_4,\quad B_{1,1}\times V_2(1),\quad B_{1,1}\times V_3(1)
\end{split}\]
along their boundaries
\[ \begin{split}
  &B_{1,1}\times (\partial V_1\cup\partial V_2\cup \partial V_3\cup\partial V_4\cup \partial D^d),\quad B_{1,1}\times \partial V_1,\quad B_{1,1}\times \partial V_4,\\
  &B_{1,1}\times \partial V_2,\quad B_{1,1}\times \partial V_3.
\end{split}\]
The identifications are given by using the trivializations $\partial\widetilde{V}_\lambda=K_\lambda\times \partial V_\lambda$. 

Let us look at the restrictions of $\pi^{W_4}|_{E_{1,1}^{W_4}}$ to the preimages of the two submanifold cycles $\gamma_1=S^{d-3}\times \{t_4^0\}$ and $\gamma_2=\{t_1^0\}\times S^{d-3}$ in $B_{1,1}$, where $t_\lambda^0$ is a basepoint of $K_\lambda$. The restricted bundle over $\gamma_1$ does not depend on the parameter $t_1\in\gamma_1$ outside $\widetilde{V}_1\times \{t_4^0\}$. The restricted bundle over $\gamma_2$ does not depend on the parameter $t_2$ outside $\{t_1^0\}\times \widetilde{V}_4$. Again it will turn out that these restricted bundles are both trivial by Lemma~\ref{lem:A} and there is a trivialization of the bundle over the $(d-3)$-skeleton $\gamma_1\cup \gamma_2$ of $B_{1,1}$. Moreover, it will turn out that this trivialization cannot be extended to the bundle over $B_{1,1}$. The obstruction can be detected by $Z_2$ (Theorem~\ref{thm:Z(G)}). \qed
\end{Exa}

%%%%%%%%%%%%%%%%%%%%%%%%%%%%%%%%%%%%%%%%%%%%%%%%%%%%%%%%
\subsection{Standard coordinates on $V_i$}\label{ss:coord-V}

As a preliminary to define the Borromean surgeries, we fix coordinates on $V_i$ using the vertex orientation fixed as in \S\ref{ss:arrow-graph}. Let $T$ be a handlebody obtained from a $(d-1)$-disk by removing several $(d-3)$-handles and 0-handles, and we put 
\[ V=T\times I.\]
We fix an explicit coordinates on $T$ as follows. We fix three distinct points $p_1,p_2,p_3\in(-1,1)$ and let $T_0=[-1,1]^{d-1}$, and for $n=1,2,3$ and small $\ve>0$, we define $T$ as follows (Figure~\ref{fig:T-type-I-II}). 
\[ \begin{split} 
  &h_n^1=\{(x_1,x_2)\in\R^2\mid (x_1-p_n)^2+x_2^2< \ve^2\}\times[-1,1]^{d-3},\\
  &h_n^0=\{(x_1,\ldots,x_{d-1})\in \R^{d-1}\,|\, (x_1-p_n)^2+x_2^2+\cdots+x_{d-1}^2< \ve^2\},\\
  &T=T_0-(h_1^{e_1}\cup h_2^{e_2}\cup h_3^{e_3}),\qquad (e_1,e_2,e_3)=\left\{\begin{array}{ll}
  (1,1,0) & \mbox{($V$: type I)}\\
  (1,0,0) & \mbox{($V$: type II)}
  \end{array}\right.
\end{split} \]
Let $H_n^e=h_n^e\times I$. 

Now we use the vertex orientation to fix the correspondence between handles of $V$ and components of the link. Namely, we rearrange the order of the three half-edges within its class of vertex orientation at the $i$-th vertex so that the first one or two are of degree 1 (or incoming) and the rest are of degree $d-2$ (or outgoing). Then this order of half-edges determines a correspondence between the spheres $S^0$ or $S^{d-3}$ associated with the half-edges of that trivalent vertex and the three components $h_1^{e_1}, h_2^{e_2}, h_3^{e_3}$.

\begin{figure}
\begin{center}
\begin{tabular}{ccc}
\includegraphics[height=30mm]{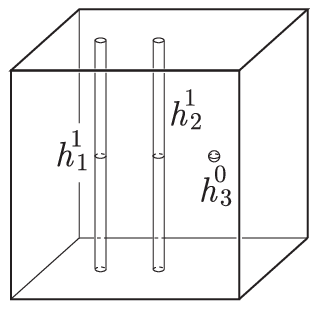} & & \includegraphics[height=30mm]{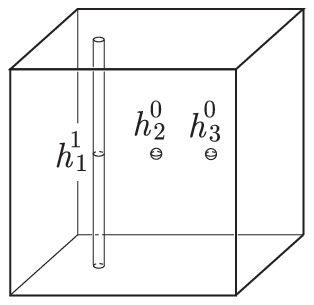}\\
(1) & & (2)
\end{tabular}
\end{center}
\caption{(1) $T$ in $V$ of type I. (2) $T$ in $V$ of type II.}\label{fig:T-type-I-II}
\end{figure}

We take standard cycles $b_1,b_2,b_3$ of $V$ that generates the reduced integral homology of $V$. We choose an identification $V_i=V$ so that the the homology classes of the cycles $b_1,b_2,b_3$ correspond to those of the oriented sphere components from the Hopf links introduced in \S\ref{ss:Y-link}. When $V$ is of type I, we let $b_1,b_2,b_3\subset T\times\{1\}\subset \partial V$ be defined by
\[ \begin{split}
  &b_1=S^1_{2\ve}(p_1,0)\times \{\textstyle(\underbrace{0,\ldots,0}_{d-3})\},\quad 
  b_2=S^1_{2\ve}(p_2,0)\times \{\textstyle(\underbrace{0,\ldots,0}_{d-3})\},\\
  &b_3=S^{d-2}_{2\ve}(p_3,\underbrace{0,\ldots,0}_{d-2}).
\end{split}\]
Here, we denote by $S^1_\delta(a,b)\subset\R^2$, $S^{d-2}_\delta(a,\boldsymbol{b},c)\subset \R^{d-1}$, the codimension 1 round spheres of radius $\delta$, centered at $(a,b)\in\R^2$, $(a,\boldsymbol{b},c)\in\R^{d-1}=\R\times \R^{d-3}\times \R$ respectively. We consider $b_1,b_2$ as 1-cycles by counter-clockwise orientations in circles of $\R^2$. We consider $b_3$ as a $(d-2)$-cycle by inducing an orientation from a $(d-1)$-disk of radius $2\ve$ in $\R^{d-1}$ by outward-normal-first convention. When $V$ is of type II, we replace $b_2$ for type I with
\[ b_2=S^{d-2}_{2\ve}(p_2,\underbrace{0,\ldots,0}_{d-2}) \]
with an orientation given similarly as $b_3$.

%%%%%%%%%%%%%%%%%%%%%%%%%%%%%%%%%%%%%%%%%%%%%%%%%%%%%%%%
\subsection{Borromean surgery of type I}\label{ss:p-borr-I}

\subsubsection{Twisted handlebody $V'$ of type I}
We shall define ``Borromean twist'' $\alpha_\mathrm{I}$ as announced before Definition~\ref{def:pi-Gamma}. The handlebody $V$ of type I is diffeomorphic to a handlebody obtained from $I^d$, where we identify $T_0\times I=[-1,1]^{d-1}\times I$ with $I^d$, by removing two open $(d-2)$-handles $H_1^1$ and $H_2^1$ and one 1-handle $H_3^0$, which are thin. We now define another handlebody $V'$, which is obtained from $V$ by changing the thin handles as follows. We represent the relative isotopy class of the thin handles in $I^d$ by a framed string link relative to the attaching region, in the sense that the map 
\[ \mathrm{res}\colon \Emb(\underline{H}_1^1\cup\underline{H}_2^1\cup\underline{H}_3^0,I^d) \to \fEmb(\underline{I}^{d-2}\cup\underline{I}^{d-2}\cup\underline{I}^1,I^d) \]
induced by restriction is a homotopy equivalence. Since framed string links here are assumed to be standard near the boundary, a framed string link induces a trivialization of the sides of the closed handles $\overline{H}_n^e$ as sphere bundles over the cores, which is canonically extended to a parametrization of the boundary of the complement of the images of the embeddings of the open handles $H_n^e$ in $I^d$. Then we have a natural map
\begin{equation}\label{eq:c}
 c_*\colon \pi_0(\Emb(\underline{H}_1^1\cup\underline{H}_2^1\cup\underline{H}_3^0,I^d))\to \calS^H(V,\partial V),
\end{equation}
given by taking the complement, where the right hand side is the set of relative diffeomorphism classes of the pairs $(W,\partial W)$ of compact $d$-manifolds with $\partial W=\partial (T\times I)$ such that $H_*(W;\Z)\cong H_*(T\times I;\Z)$. The image of the class of the standard embedding under the map $c_*$ gives $(V,\partial V)$. The image of the framed Borromean string link $B(\underline{d-2},\,\underline{d-2},\,\underline{1})_d$ under $c_*$ gives another relative diffeomorphism class, which we denote by $(V',\partial V')$. We identify the boundary $\partial V'$, which is the union of $T\times\{0,1\}$ and the sides of the handles, with $\partial V$ by using the parametrization of embeddings of the handles.
\begin{Rem}
Although the relative diffeomorphism class of $(V',\partial V'=\partial V)$ suffices to define the surgery of type I in Definition~\ref{def:pi-Gamma}, we describe below a further property of the surgery. Namely, that the surgery can be obtained by attaching the standard handlebody along its boundary by a twisting map.
\end{Rem}

\subsubsection{Mapping cylinder structure on $V'$}\label{ss:mc_I}
For the type I handlebody $V$, we will see that the handlebody $V'$ thus obtained can be realized as the mapping cylinder of a relative diffeomorphism $\varphi_0\colon (T,\partial T)\to (T,\partial T)$, which is defined by 
\[ C(\varphi_0)=(T\times I)\cup_{\varphi_0}(T\times \{0\}), \]
where we consider the $T\times\{0\}$ on the right as a copy of the original one $T$, and identify each $(x,0)\in T\times \{0\}\subset T\times I$ with $(\varphi_0(x),0)\in T\times\{0\}$.
Note that the boundary of $C(\varphi_0)$ is $(T\times \{1\})\cup (\partial T\times I)\cup_{\partial\varphi_0} (T\times \{0\})=(T\times \{0,1\})\cup_{\mathrm{id_{\partial T\times\{0,1\}}}} (\partial T\times I)=\partial V$ and we consider that the canonical identification $\partial C(\varphi_0)=\partial V$ is a part of the structure of the mapping cylinder.

\begin{Prop}[Proof in \S\ref{ss:proof-T-bundle-I}]\label{prop:T-bundle-I}
For a handlebody $V$ of type I, there exists a relative diffeomorphism $\varphi_0\colon (T,\partial T)\to (T,\partial T)$ and a relative diffeomorphism $(V',\partial V)\to (C(\varphi_0),\partial V)$ that restricts to $\mathrm{id}$ on $\partial V$.
\end{Prop}
The relative diffeomorphism $\varphi_0\colon (T,\partial T)\to (T,\partial T)$ of Proposition~\ref{prop:T-bundle-I} extends to a self-diffeomorphism $\varphi_\mathrm{I}$ of $\partial V=(T\times\{0,1\})\cup (\partial T\times I)$ by setting $\varphi_0$ on $T\times\{0\}$ and $\mathrm{id}$ otherwise. 
\begin{Def}[Type I Borromean twist]\label{def:alpha-I}
We define the map $\alpha_\mathrm{I}\colon S^0\to \Diff(\partial V)$ by $\alpha_\mathrm{I}(-1)=\mathrm{id}$, $\alpha_\mathrm{I}(1)=\varphi_\mathrm{I}$. Let $\widetilde{V}$ be the total space of the bundle $V'\cup (-V)\to S^0$ that is the disjoint union of $V'\to \{1\}$ and $-V\to \{-1\}$. 

\end{Def}

\begin{Rem}
\begin{enumerate}
\item We assume that the corners arose in the construction above are all smoothed (in the sense of \cite[Ch.2,2.6]{Wal} or \cite[Ch.3,3.3]{Tam}).
\item When $d=3$, the surgery on Y-graph in \cite{Gou,Hab} is given by surgery for $\alpha_{\mathrm{I}}$ of Definition~\ref{def:alpha-I}.
\end{enumerate}
\end{Rem}

%%%%%%%%%%%%%%%%%%%%%%%%%%%%%%%%%%%%%%%%%%%%%%%%%%%%%%%%
\subsection{Parametrized Borromean surgery of type II}\label{ss:p-borr-II}

\subsubsection{Family $\widetilde{V}$ of twisted handlebodies of type II}
We define the ``parametrized Borromean twist'' $\alpha_\mathrm{II}\in \Omega^{d-3}\Diff(\partial V)$, announced before Definition~\ref{def:pi-Gamma}.
The handlebody $V$ of type II is diffeomorphic to a handlebody obtained from $I^d$ by removing one $(d-2)$-handle and two 1-handles, which are thin. We now define a $(V,\partial)$-bundle $\widetilde{V}\to S^{d-3}$, which is obtained from a trivial $V$-bundle over $S^{d-3}$ by changing the trivial family of thin handles as follows. We construct $\widetilde{V}$ by taking the image under the map 
\[ c_*\colon \pi_{d-3}(\fEmb_0(\underline{I}^{d-2}\cup \underline{I}^1\cup \underline{I}^1,I^d))\to \pi_{d-3}(B\Diff(V,\partial)), \]
which is given by taking the complement, of the class of a certain loop
\[ \beta\in\Omega^{d-3}\fEmb_0(\underline{I}^{d-2}\cup \underline{I}^1\cup \underline{I}^1,I^d) \]
corresponding to a framed Borromean link $B(\underline{d-2},\,\underline{d-2},\,\underline{1})_d$, based at the standard inclusion. We will define $\beta$ later in \S\ref{ss:proof-T-bundle-II}. Roughly, the loop $\beta$ is constructed by replacing the second component in $B(\underline{d-2},\,\underline{d-2},\,\underline{1})_d$ with a $(d-3)$-parameter family of 1-disks with framing, so that the locus of the family of 1-disks recovers the original $(d-2)$-disk component after a small change on the boundary. Then the image of the homotopy class of $\beta$ under $c_*$ gives a $(V,\partial)$-bundle $\widetilde{V}\to S^{d-3}$.

\subsubsection{Mapping cylinder structure on the bundle $\widetilde{V}$}\label{ss:mc_II}
We will show that thus obtained $(V,\partial)$-bundle $\widetilde{V}$ is a $(d-3)$-parameter family of mapping cylinders for an element of $\pi_{d-3}(\Diff(T,\partial T))$. For a given smooth family of relative diffeomorphisms $\varphi_{0,t}\colon (T,\partial T)\to (T,\partial T)$ ($t\in S^{d-3}$), let $\bar{\varphi}\colon S^{d-3}\times T\to S^{d-3}\times T$ be the map defined by $\bar{\varphi}(t,x)=(t,\varphi_{0,t}(x))$. Here we say that an $S^{d-3}$-family of diffeomorphisms $\varphi_{0,t}$ in $\Diff(T,\partial)$ is smooth if the associated map $\bar{\varphi}$ is smooth, as usual. Now we set 
\[ \widetilde{C}(\{\varphi_{0,t}\})=
   (S^{d-3}\times T\times I)\cup_{\bar{\varphi}}(S^{d-3}\times T\times\{0\}),\]
where we consider $S^{d-3}\times T\times \{0\}$ on the right as a copy of $S^{d-3}\times T$, and identify each $(t,x,0)\in S^{d-3}\times T\times\{0\}\subset S^{d-3}\times T\times I$ with $(\bar{\varphi}(t,x),0)\in S^{d-3}\times T\times\{0\}$. This has a natural structure of a $(V,\partial)$-bundle over $S^{d-3}$ whose boundary is $S^{d-3}\times \partial V$. 
\begin{Prop}[Proof in \S\ref{ss:proof-T-bundle-II}]\label{prop:T-bundle-II}
For a handlebody $V$ of type II, there exist a smooth family of relative diffeomorphisms $\varphi_{0,t}\colon (T,\partial T)\to (T,\partial T)$ ($t\in S^{d-3}$) with $\varphi_{0,*}=\mathrm{id}$ for the basepoint $*\in S^{d-3}$, and a relative bundle isomorphism
\[ (\widetilde{V},S^{d-3}\times \partial V)\to (\widetilde{C}(\{\varphi_{0,t}\}),S^{d-3}\times \partial V)\]
that restricts to $\mathrm{id}$ on the boundary $S^{d-3}\times \partial V$. 
\end{Prop}

\begin{Def}[Type II Borromean twist]\label{def:alpha-II}
We define the map $\alpha_\mathrm{II}\colon S^{d-3}\to \Diff(\partial V)$ by extending $\{\varphi_{0,t}\}$ to a $(d-3)$-parameter family of diffeomorphisms of $\partial V$ by $\mathrm{id}$ on the complement of $T\times\{0\}$ in $\partial V$. 
\end{Def}

There is a natural ``graphing'' map 
\[ \Psi\colon \pi_{d-3}(\fEmb_0(\underline{I}^{d-2}\cup \underline{I}^1\cup \underline{I}^1,I^d))\to \pi_0(\fEmb(\underline{I}^{2d-5}\cup \underline{I}^{d-2}\cup \underline{I}^{d-2},I^{2d-3})), \]
which is obtained by representing a $(d-3)$-parameter family of framed long embeddings in $\fEmb_0(\underline{I}^{d-2}\cup \underline{I}^1\cup \underline{I}^1,I^d)$ by a single map 
$(\underline{I}^{d-2}\cup \underline{I}^1\cup \underline{I}^1)\times I^{d-3}\to I^d\times I^{d-3}$ with the corresponding framing. The following lemma will be used in Lemma~\ref{lem:S(a)}.
\begin{Lem}[Proof in \S\ref{ss:normal-framing-beta}]\label{lem:B(3,2,2)}
The image of $[\beta]\in \pi_{d-3}(\fEmb_0(\underline{I}^{d-2}\cup \underline{I}^1\cup \underline{I}^1,I^d))$ under $\Psi$ is the class of $B(\underline{2d-5},\,\underline{d-2},\,\underline{d-2})_{2d-3}$ with the normal framing $F_D$ given in \S\ref{ss:borromean} and Definition~\ref{def:long-borromean-link}.
\end{Lem}

%%%%%%%%%%%%%%%%%%%%%%%%%%%%%%%
\subsection{Framed handlebody replacement}\label{ss:framed-surgery}

We shall see that the surgery of type I or II is compatible with framing and that surgery along a graph clasper gives an element of the homotopy group of $\wBDiff(D^d,\partial)$. Let $V$ be the standard model in \S\ref{ss:coord-V} of the handlebody of type I or II. 

\begin{Prop}\label{prop:framed-handle-II}
\begin{enumerate}
\item There is a bundle isomorphism 
\[ \widetilde{\varphi}\colon \widetilde{V}\to K\times V \]
that induces $\bar{\alpha}\colon K\times \partial V=\partial \widetilde{V}\to K\times \partial V$. Here, $K=S^0$ or $S^{d-3}$, $\bar{\alpha}=\bar{\alpha}_{\mathrm{I}}$ or $\bar{\alpha}_{\mathrm{II}}$, and the identification $\partial \widetilde{V}=K\times\partial V$ is the trivialization given by the mapping cylinder construction of Proposition~\ref{prop:T-bundle-I} or \ref{prop:T-bundle-II}.
\item The vertical framing on $\widetilde{V}$ induced from the standard framing $\mathrm{st}$ on $T_0\times I\subset \R^d$ has the property that it can be modified by a homotopy supported in a small neighborhood of $\partial V$ into one whose restriction to $\partial\widetilde{V}$ agrees with $(d\widetilde{\varphi})^{-1}(\mathrm{st}|_{\partial V})$.
\end{enumerate}
\end{Prop}
\begin{proof} The assertion (1) follows from Proposition~\ref{prop:T-bundle-I} or \ref{prop:T-bundle-II}. The assertion (2) follows from \cite[Lemma~A]{Wa3}.
\end{proof}

That the homotopy of (2) is supported in a small neighborhood of $\partial V$ will be used in the proof of Lemma~\ref{lem:Sigma-Sigma-deg}. 
Proposition~\ref{prop:framed-handle-II} gives a trivialization of the bundle $\widetilde{V}$ as a $V$-bundle, but not as a $(V,\partial)$-bundle. 
Propositions~\ref{prop:framed-handle-II} shows that the surgeries of type I and II are framed ones, in the sense of the following corollary.
\begin{Cor}\label{cor:framing-extend}
If $X$ is framed, then the surgery of $X$ on $(V,\alpha\colon K\to \Diff(\partial V))$ of type I or II gives a framed bundle $\pi(\alpha)\colon (K\times X)^{V,\alpha}\to K$, $K=S^0$ or $S^{d-3}$, on which the framing agrees with the original framing outside $V$. In other words, the vertical framing on $K\times(X-\mathrm{Int}\,V)$ canonically induced from the original one on $X-\mathrm{Int}\,V$ extends to that on $(K\times X)^{V,\alpha}$. 
\end{Cor}

For $\ell\in\{1,2,3\}$, let $V_{[\ell]}$ denote the handlebody constructed in the same way as $V$ except we forget the $\ell$-th component in $H_1^{e_1}\cup H_2^{e_2}\cup H_3^{e_3}$. 

\begin{Lem}[{\cite[Lemma~A and Remark~7]{Wa3}}]\label{lem:A}
Let $\pi(\alpha)\colon \widetilde{V}\to K$ be the bundle obtained by twists $\alpha\colon K\to \Diff(\partial V)$ of type I or II. Let $\pi(\alpha)_{[\ell]}\colon \widetilde{V}_{[\ell]}\to K$ be the bundle obtained from $\pi(\alpha)$ by extension by filling a trivial framed family into the $\ell$-th complementary handle. Then $\pi(\alpha)$
\begin{enumerate}
\item admits a vertical framing that extends the standard one on the boundary induced from the given one on $V$, and
\item becomes trivial as a framed relative bundle if $\widetilde{V}$ is extended to $\widetilde{V}_{[\ell]}$.
\end{enumerate}
\end{Lem}
\begin{Rem}
Although Lemma~\ref{lem:A} is the statement for the standard model, it is also true for any other handlebody $V$ in a framed $d$-manifold $X$ that is obtained from the standard model in a small ball by an isotopy of the embedding $V\to X$ from the inclusion.
\end{Rem}

\begin{Prop}[Theorem~\ref{thm:Z(G)} (1),(2)]\label{prop:framed-surgery}
\begin{enumerate}
\item $\pi^\Gamma\colon E^\Gamma\to B_\Gamma$ is a $(D^d,\partial)$-bundle and admits a vertical framing.
\item There is a vertical framing $\tau^\Gamma$ on $\pi^\Gamma$ such that the framed $(D^d,\partial)$-bundle $(\pi^\Gamma,\tau^\Gamma)$ is oriented bundle bordant to a framed $(D^d,\partial)$-bundle $\varpi^\Gamma\colon F^\Gamma\to S^{(d-3)k}$ over $S^{(d-3)k}$ with some vertical framing $\sigma^\Gamma$. Namely, there exist a compact oriented $(d-3)k+1$-dimensional cobordism $\widetilde{B}$ with $\partial \widetilde{B}=B_\Gamma\tcoprod (-S^{(d-3)k})$ and a framed $(D^d,\partial)$-bundle $\widetilde{\pi}\colon \widetilde{E}\to \widetilde{B}$ such that the restriction of $\widetilde{\pi}$ on $\partial \widetilde{B}$ agrees with $(\pi^\Gamma,\tau^\Gamma)$ and $(\varpi^\Gamma,\sigma^\Gamma)$ (with the opposite orientation). 
\end{enumerate}
\end{Prop}
\begin{proof}
(1) We see that if $\alpha=\alpha_{\mathrm{I}}$ or $\alpha_{\mathrm{II}}$, then the bundle $\pi^{V,\alpha}\colon (S^a\times D^d)^{V,\alpha}\to S^a$, $a=0$ or $d-3$, obtained from the trivial $(D^d,\partial)$-bundle $S^a\times D^d$ by surgery along $V$ is a trivial $(D^d,\partial)$-bundle. Indeed, $V$ can be extended to $V_{[\ell]}$ in $D^d$ and the surgery along $V$ and $V_{[\ell]}$ produce equivalent results, where the surgery along $V_{[\ell]}$ is defined by replacing $S^a\times V_{[\ell]}$ with $\widetilde{V}_{[\ell]}$. By Lemma~\ref{lem:A} (2), the result is a trivial $D^d$-bundle. By the definition of the surgery along $V_{[\ell]}$, the trivialization on the $(V_{[\ell]},\partial)$-bundle $\widetilde{V}_{[\ell]}$ obtained by Lemma~\ref{lem:A} (2) can be extended to a trivialization of a $(D^d,\partial)$-bundle. Also, by Lemma~\ref{lem:A} (1), the restriction of the standard framing on $S^a\times D^d$ to $S^a\times (D^d-\mathrm{Int}\,V)$ extends over $(S^a\times D^d)^{V,\alpha}$. 

By applying the above for type I surgeries, it follows that the restriction of $\pi^\Gamma$ over $(S^0)^{k}\subset B_\Gamma$ has a trivialization as a $(D^d,\partial)$-bundle. Now we have a trivialization of $(D^d,\partial)$-bundle at the basepoint of each path-component of $B_\Gamma$, the whole bundle $\pi^\Gamma$ must be a $(D^d,\partial)$-bundle, by the definition of the type II surgery. The vertical framing on $E^\Gamma$ can be obtained by doing the parametrized gluing in \S\ref{ss:surgery-Y} with framing.

(2) The proof is parallel to that of \cite[Claim~3]{Wa3} (see also \cite[Remark~7]{Wa3}) for $d$ even and with $(S^{k-1})^{\times 2n}$ replaced by a product $(S^0)^{\times k}\times (S^{d-3})^{\times k}$, and we do not repeat that here. We should remark that we used in \cite[Lemma~B]{Wa3} the claim that $\Sigma A\to \Sigma X$ splits with cofiber $\Sigma(X/A)$, where $X$ is a product of spheres and $A$ is the maximal skeleton of $X$ of positive codimension for a certain cell decomposition. The splitting holds even for the products like $(S^0)^{\times \ell}\times (S^{d-3})^{\times m}$ (including 0-spheres), by the wedge decomposition of $\Sigma X$ given in \cite[Satz~20]{Pu}.
\end{proof}

%%%%%%%%%%%%%%%%%%%%%%%%%%%%%%%
%%%%%%%%%%%%%%%%%%%%%%%%%%%%%%%
\mysection{Computation of the invariant}{s:computation}

The strategy for computing the configuration space integrals taken in \cite{KT,Les2}, which we follow for higher-dimensional manifolds, is to reduce the computation of $Z_k$ to homological (or combinatorial) one, like the linking number. 

%%%%%%%%%%%%%%%%%%%%%%%%%%%%%%%
\subsection{Normal Thom class}\label{ss:thom}

For a topologically closed oriented smooth submanifold $A$ of an oriented manifold $N$, we denote by $\eta_A$ a closed form representative of the Thom class of the normal bundle $\nu_A$ of $A$. We identify the total space of $\nu_A$ with a small tubular neighborhood $N_A$ of $A\subset N$ and assume that $\eta_A$ has support in $N_A$. It has the useful property that $[\eta_A]$ is the Poincar\'{e}--Lefschetz dual of $[A]\in H_*(N,\partial N)$, when both $N$ and $A$ are compact. A basic textbook reference is \cite[Ch.~I, Section 6]{BTu}.

%%%%%%%%%%%%%%%%%%%%%%%%%%%%%%%
\subsection{Standard cycles on $\partial V$}\label{ss:std-cycles}

Recall that $V_i\subset X$ is defined in \S\ref{ss:Y-link} as a handlebody obtained by thickening a Y-graph $G_i$. 
In \S\ref{ss:coord-V}, we fixed a standard model $V$ of $V_i$ and we have taken cycles $b_1,b_2,b_3$ of $\partial V$. Now we take more standard cycles $a_1,a_2,a_3$ of $\partial V$, which are null-homologous in $V$, as follows. Here we again use the standard coordinates of $V$ fixed in \S\ref{ss:coord-V}.

We define disks $a_1^T,a_2^T,a_3^T\subset T$ by 
$a_1^T=\{p_1\}\times[-1,-\ve]\times[-1,1]^{d-3}$, 
$a_2^T=\{p_2\}\times[-1,-\ve]\times[-1,1]^{d-3}$, 
$a_3^T=\{p_3\}\times[-1,-\ve]\times\{\textstyle(\underbrace{0,\ldots,0}_{d-3})\}$, 
and put
\[ a_\ell=(a_\ell^T\times\{1\})\cup (\partial a_\ell^T\times I)\cup (-a_\ell^T\times\{0\})\subset \partial V.\]
(See Figure~\ref{fig:ab-type-I-II} (1).)
We orient $a_\ell$ so that 
\[ \mathrm{Lk}(b_\ell^-,a_\ell)=+1, \]
where $b_\ell^-$ is a copy of $b_\ell\subset T\times\{1\}$ in $T\times\{1-\ve\}$ obtained by shifting, and $\Lk$ is defined by using the Euclidean coordinates of $T_0\times I$ of \S\ref{ss:coord-V} and the formula (\ref{eq:Lk}). The collection $(a_1,b_1,a_2,b_2,a_3,b_3)$ of cycles gives a $\Z$-basis of $H_1(\partial V;\Z)\oplus H_{d-2}(\partial V;\Z)$ such that 
\[ \begin{split}
&\mathrm{Lk}(b_\ell^-,a_j)=\delta_{\ell j}\quad (\mbox{when $\dim{b_\ell}+\dim{a_j}=d-1$}), \mbox{ and}\\
&[a_j]\cdot [a_\ell]=[b_j]\cdot [b_\ell]=0 \quad\\
& (\mbox{when $\dim{a_j}+\dim{a_\ell}=d-1$ and $\dim{b_j}+\dim{b_\ell}=d-1$}).
\end{split}\]

When $V$ is of type II, we replace $b_2$ and $a_2^T$ for type I with
\[ b_2=S^{d-2}_{2\ve}(p_2,\underbrace{0,\ldots,0}_{d-2}), \quad 
a_2^T=\{p_2\}\times[-1,-\ve]\times\{(\underbrace{0,\ldots,0}_{d-3})\}. \]
We define the cycles $\widetilde{a}_\ell, \widetilde{b}_\ell$ of $S^{d-3}\times\partial V$ by 
\[ \widetilde{a}_\ell=S^{d-3}\times a_\ell,\quad  \widetilde{b}_\ell=S^{d-3}\times b_\ell,\]
and orient them by
\begin{equation}\label{eq:ori-a-tilde}
 o(\widetilde{a}_\ell)=(-1)^{d-3}o(S^{d-3})\wedge o(a_\ell),\quad
  o(\widetilde{b}_\ell)=(-1)^{d-3}o(S^{d-3})\wedge o(b_\ell). 
\end{equation}
This strange-looking orientation convention, which is not as in \S\ref{ss:ori-prod}, is for the coorientations of $S(\widetilde{a}_\ell)$, $S(\widetilde{b}_\ell)$ (defined in \S\ref{ss:eta-family}) to be compatible with $S(a_\ell)$, $S(b_\ell)$ (defined in \S\ref{ss:linking}), respectively (see Lemma~\ref{lem:ori-Sa-compatible}).

\begin{figure}
\begin{center}
\begin{tabular}{ccc}
\includegraphics[height=30mm]{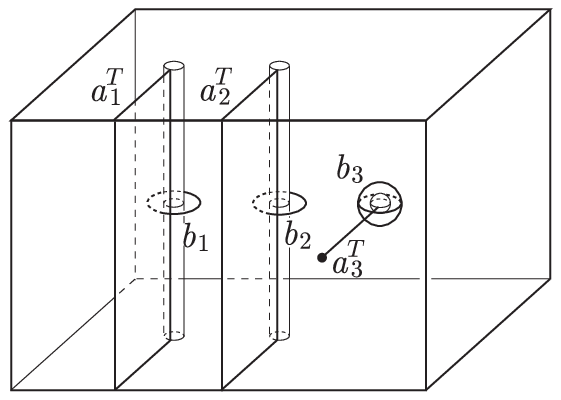} & \phantom{aaaaa} & \includegraphics[height=30mm]{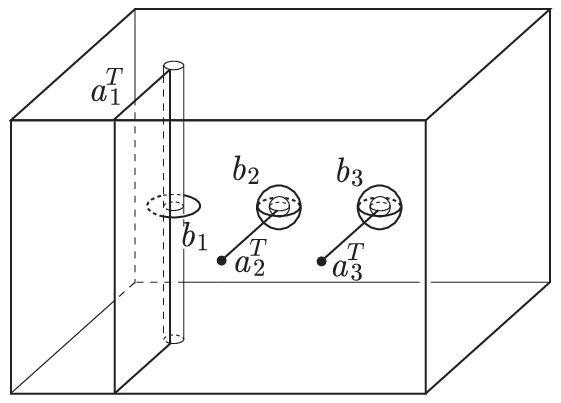}\\
(1) & & (2)
\end{tabular}
\end{center}
\caption{$a_1^T,a_2^T,a_3^T,b_1,b_2,b_3\subset T$, (1) in the top face of $V$ of type I, (2) in the top face of $V$ of type II. (Not the pictures of the whole of $V$.)}\label{fig:ab-type-I-II}
\end{figure}

%%%%%%%%%%%%%%%%%%%%%%%%%%%%%%%
\subsection{Normalization of linking pairing of Y-link}\label{ss:linking}

First, we consider the subspace $V_i\times V_j\subset \bConf_2(S^d;\infty)$, $i\neq j$, $i,j\neq \infty$, and see that a propagator can be described explicitly by means of the $\eta$ forms. Let $a_\ell^\lambda,b_\ell^\lambda$, $\ell=1,2,3$, be the generating cycles of $H_p(\partial V_\lambda;\Z)$ for $p=1,d-2$, corresponding to the standard cycles $a_\ell, b_\ell$ in the standard model given in \S\ref{ss:coord-V} and \S\ref{ss:std-cycles}. The spherical cycle $a_\ell^\lambda$ bounds a disk $S(a_\ell^\lambda)$ in $V_\lambda$, and moreover, by the construction of $\vec{V}_G=(V_1,\ldots,V_{2k})$, the spherical cycle $b_\ell^\lambda$ bounds a disk $S(b_\ell^\lambda)$ in $X-\mathrm{Int}\,V_\lambda$, which intersects some other $V_{\lambda'}$, $\lambda'\neq \lambda$. Then $H^*(V_\lambda)$ is spanned by the classes of 
\[ 1,\,\eta_{S(a_1^\lambda)},\,\eta_{S(a_2^\lambda)},\,\eta_{S(a_3^\lambda)}.\]
By the K\"{u}nneth formula, it follows that $H^{d-1}(V_i\times V_j)$ is spanned by $[\eta_{S(a_\ell^i)}]\otimes [\eta_{S(a_m^j)}]$, where $\ell, m$ are such that $\dim{a_\ell^i}+\dim{a_m^j}=d-1$. Thus a propagator $\omega\in\Omega_\dR^{d-1}(\bConf_2(S^d;\infty))$ satisfies
\begin{equation}\label{eq:omega-explicit-H}
 [\omega|_{V_i\times V_j}]=\sum_{\ell, m} L_{\ell m}^{ij}\,[\eta_{S(a_\ell^i)}]\otimes [\eta_{S(a_m^j)}]
\end{equation}
in $H^{d-1}(V_i\times V_j)$ for some $L_{\ell m}^{ij}\in\R$. Let $\Lk(b,b')=\int_{b\times b'}\omega$ for a link $b\coprod b'$.
\begin{Lem}[Proof in \S\ref{ss:proof-int-eta}]\label{lem:int-eta}
We have the following identities.
\begin{enumerate}
\item $\displaystyle\int_{b_\ell^-}\eta_{S(a_\ell)}=(-1)^{kd+k+d-1}$, 
where $k=\dim{a_\ell}$. 

\item $\displaystyle \int_{a_\ell^+}\eta_{S(b_\ell)}=(-1)^{d+k}$, where $k=\dim{a_\ell}$. 

\item $L_{\ell m}^{ij}=(-1)^{d-1}\Lk(b_\ell^i, b_m^j)$ 
for $i,j,\ell,m$ such that $\dim{b_\ell^i}+\dim{b_m^j}=d-1$.
\end{enumerate}
\end{Lem}
The identities (1) and (2) will be used later in \S\ref{ss:normalization_one}.
The integral of $\omega$ gives the linking pairing
\[ \Lk\colon \bigoplus_{p+q=d-1}H_p(V_i)\otimes H_q(V_j)\to \R. \]
The right hand side of (\ref{eq:omega-explicit-H}) has the following explicit closed form representative as a form on $V_i\times V_j$.
\begin{equation}\label{eq:omega-explicit}
\sum_{\ell, m} L_{\ell m}^{ij}\,\pr_1^*\,\eta_{S(a_\ell^i)}\wedge \pr_2^*\,\eta_{S(a_m^j)},
\end{equation}
where $\pr_n\colon \bConf_2(S^d;\infty)\to \bConf_1(S^d;\infty)$ is the map induced by the $n$-th projection.

%%%%%%%%%%%%%%%%%%%%%%%%%%%%%%%
\subsection{Spanning submanifolds and their $\eta$-forms in $\widetilde{V}_\lambda$}\label{ss:eta-family}

The formula (\ref{eq:omega-explicit}) can be naturally extended to families of $V_\lambda\times V_\mu$. Let $\pi(\alpha_\lambda)\colon \widetilde{V}_\lambda\to K_\lambda$ be the relative bundle obtained by the twists $\alpha_\lambda\colon K_\lambda\to \Diff(\partial V_\lambda)$ of type I or II in Definition~\ref{def:alpha-I} or \ref{def:alpha-II}. Let
\[ \widetilde{a}_\ell^\lambda:=K_\lambda\times a_\ell^\lambda \subset \partial \widetilde{V}_\lambda=K_\lambda\times \partial V_\lambda.   \]
The following lemma, which will be used to make the integrals in the main computation of the invariant in \S\ref{ss:normalize} explicit, follows from Lemmas~\ref{lem:spanning-disk-long} and \ref{lem:B(3,2,2)}.

\begin{Lem}\label{lem:S(a)}
For each $\ell$ there exists a compact oriented submanifold $S(\widetilde{a}_\ell^\lambda)$ of $\widetilde{V}_\lambda$ with boundary such that
\begin{enumerate}
\item $\partial S(\widetilde{a}_\ell^\lambda)=\widetilde{a}_\ell^\lambda=S(\widetilde{a}_\ell^\lambda)\cap \partial \widetilde{V}_\lambda$, and the intersection is transversal.
\item $S(\widetilde{a}_\ell^\lambda)\cap \pi(\alpha_\lambda)^{-1}(t^0)=S(a_\ell^\lambda)$ over the basepoint $t^0\in K_\lambda$.
\item $S(\widetilde{a}_\ell^\lambda)$ is diffeomorphic to the connected sum of $K_\lambda\times S(a_\ell^\lambda)$ with $S^u\times S^v$ for some $u,v$ such that $u+v=\dim{S(\widetilde{a}_\ell^\lambda)}$. 
\item The normal bundle of $S(\widetilde{a}_\ell^\lambda)$ is trivial.
\item $S(\widetilde{a}_1^\lambda)\cap S(\widetilde{a}_2^\lambda)\cap S(\widetilde{a}_3^\lambda)$ is one point, and the intersection is transversal. 
\end{enumerate}
\end{Lem}
\begin{proof}
By Lemma~\ref{lem:spanning-disk-long}, the three components in a Borromean string link have spanning submanifolds $\underline{D_1'},\underline{D_2'},\underline{D_3'}$. The restrictions of these submanifolds to the family of $I^d-(H_1^{e_1}\cup H_2^{e_2}\cup H_3^{e_3})$ give submanifolds satisfying the conditions (2), (3), (4), (5). To see that we can moreover assume (1), we need to show that a standard collar neighborhood of $\widetilde{a}_\ell^\lambda$ agrees with that induced by the spanning disk $\underline{D_\ell}$ of the corresponding component.

By a standard argument relating a normal framing of an embedding and a trivialization of its tubular neighborhood, it suffices to check the compatibility of the normal framings of the two models: one given in Definition~\ref{def:long-borromean-link} and one given by the parametrization of the family of handles $H_1^{e_1}\cup H_2^{e_2}\cup H_3^{e_3}$ in $I^d$. But this is proved in Lemma~\ref{lem:B(3,2,2)}.
\end{proof}

Note that $S(\widetilde{a}_\ell^\lambda)$ need not be a subbundle of $\pi(\alpha_\lambda)$. The product $\pi(\alpha_i)\times \pi(\alpha_j)\colon \widetilde{V}_i\times\widetilde{V}_j\to K_i\times K_j$ is a bundle whose fiber over the basepoint is $V_i\times V_j$. The formula (\ref{eq:omega-explicit}) is naturally extended over $\widetilde{V}_i\times\widetilde{V}_j$ by 
\begin{equation}\label{eq:omega-explicit-tilde}
 \sum_{\ell, m} L_{\ell m}^{ij}\,\pr_1^*\,\eta_{S(\widetilde{a}_\ell^i)}\wedge \pr_2^*\,\eta_{S(\widetilde{a}_m^j)}, 
\end{equation}
where $\eta_{S(\widetilde{a}_\ell^i)}$ etc. is a closed form on $\widetilde{V}_i$ etc. Note that the form (\ref{eq:omega-explicit-tilde}) is currently defined only on the space $\widetilde{V}_i\times\widetilde{V}_j$ and we still have not seen that this is a restriction of a propagator on the corresponding $(D^d,\partial)$-bundle over $K_i\times K_j$, although we will do so in Proposition~\ref{prop:localization} below.

%%%%%%%%%%%%%%%%%%%%%%%%%%%%%%%
\subsection{Normalization of propagator in family}\label{ss:localize}

To state Proposition~\ref{prop:localization}, we decompose bundles into pieces. Let $U_\infty$ is a small closed $d$-ball about $\infty$ and let $\pi^{\Gamma\infty}\colon E^{\Gamma\infty}\to B_\Gamma$ be the $(S^d,U_\infty)$-bundle obtained by extending the $(D^d,\partial)$-bundle $\pi^\Gamma\colon E^\Gamma\to B_\Gamma$ by the product bundle $B_\Gamma\times U_\infty$. 

We decompose $E^{\Gamma\infty}$ into subbundles compatible with surgery, as follows. We extend the vertical framing $\tau^\Gamma$ on $E^\Gamma$ over the complement of the $\infty$-section $B_\Gamma\times \{\infty\}$ in $E^{\Gamma\infty}$ by the standard framing $\tau_0$ on $\R^d=S^d-\{\infty\}$. This extension is possible since $\tau^\Gamma$ is standard near the boundary. Let 
\[ V_\infty=S^d-\mathrm{Int}(V_1\cup\cdots\cup V_{2k}).\]
For $\lambda\in\{1,2,\ldots,2k\}$, let 
\[ \widetilde{V}'_\lambda=K_1\times\cdots\times K_{\lambda-1}\times\widetilde{V}_\lambda\times K_{\lambda+1}\times\cdots\times K_{2k}. \]
This is a bundle over $B_\Gamma$, which is canonically isomorphic to the pullback of the bundle $\pi(\alpha_\lambda)\colon \widetilde{V}_\lambda\to K_\lambda$ by the projection $B_\Gamma\to K_\lambda$. Let 
\[ \widetilde{V}'_\infty=B_\Gamma\times V_\infty \]
and we consider the projection $\widetilde{V}'_\infty\to B_\Gamma$ as a trivial $V_\infty$-bundle over $B_\Gamma$. Then we have the decomposition
\[ E^{\Gamma\infty}=\widetilde{V}'_1\cup\cdots\cup \widetilde{V}'_{2k}\cup \widetilde{V}'_\infty, \]
where the gluing at the boundary is given by the natural trivializations $\partial\widetilde{V}'_\lambda=B_\Gamma\times \partial V_\lambda$ for $\lambda\in\{1,\ldots,2k\}$ (given in \S\ref{ss:mc_I} and \S\ref{ss:mc_II}) and $\partial\widetilde{V}'_\infty=B_\Gamma\times(\partial V_1\cup\cdots\cup\partial V_{2k})$. 

We also consider a natural decomposition of $E\bConf_2(\pi^\Gamma)$ accordingly, as follows.
\begin{Not}
For $i,j\in\{1,\ldots,2k\}$ such that $i\neq j$, let
\[ \Omega_{ij}^\Gamma=\widetilde{V}'_i\times_{B_\Gamma}\widetilde{V}'_j, \]
namely, the pullback of the diagram $\widetilde{V}'_i\to B_\Gamma \leftarrow \widetilde{V}'_j$, where the map $\widetilde{V}'_i\to B_\Gamma$ etc. is the projection of the $V_i$-bundle. For $i\in\{1,\ldots,2k,\infty\}$, let
\[ \Omega_{ii}^\Gamma=p_{B\ell}^{-1}\bigl(\widetilde{V}'_i\times_{B_\Gamma}\widetilde{V}'_i\bigr),\quad 
\Omega_{i\infty}^\Gamma=p_{B\ell}^{-1}\bigl(\widetilde{V}'_i\times_{B_\Gamma}\widetilde{V}'_\infty\bigr),\quad
\Omega_{\infty i}^\Gamma=p_{B\ell}^{-1}\bigl(\widetilde{V}'_\infty\times_{B_\Gamma}\widetilde{V}'_i\bigr),\]
where $p_{B\ell}\colon E\bConf_2(\pi^\Gamma)\to E^{\Gamma\infty}\times_{B_\Gamma}E^{\Gamma\infty}$ is the fiberwise blow-down map.
\end{Not}
The projection $\Omega_{ij}^\Gamma\to B_\Gamma$ is a subbundle of $\bConf_2(\pi^\Gamma)\colon E\bConf_2(\pi^\Gamma)\to B_\Gamma$, whose fiber over the basepoint $(t_1^0,\ldots,t_{2k}^0)\in B_\Gamma$ is $V_i\times V_j$ or $p_{B\ell}^{-1}(V_i\times V_i)$, which is either a manifold with corners or the image of a manifold with corners under a smooth map (Lemma~\ref{lem:Bl(Y2)}). Then we have
\[ E\bConf_2(\pi^\Gamma)=\bigcup_{i,j}\,\Omega_{ij}^\Gamma, \]
where the sum is over all $i,j\in\{1,\ldots,2k,\infty\}$. This decomposition is such that the interiors of the pieces do not overlap. The closed form (\ref{eq:omega-explicit-tilde}) can be defined on most terms in this decomposition, except those of the forms $\Omega_{ii}^\Gamma$ or those involving $\infty$. Over the latter exceptions we will extend by ``degenerate'' forms. 
\begin{Not}\label{not:Omega_J}
For $J\subset \{1,2,\ldots,2k\}$, let
\[  B_\Gamma(J)=\prod_{\lambda=1}^{2k}K_\lambda(J),\quad\mbox{where}\quad 
 K_\lambda(J)=\left\{\begin{array}{ll}
  K_\lambda & (\lambda\in J),\\
  \{t_\lambda^0\} & (\lambda\notin J),
\end{array}\right. 
 \]
and let $\Omega_{ij}^\Gamma(J)\to B_\Gamma(J)$ be the restriction of the bundle $\Omega_{ij}^\Gamma\to B_\Gamma$ on $B_\Gamma(J)$. More generally, for a bundle $\calE\to B_\Gamma$, we denote by $\calE(J)\to B_\Gamma(J)$ its restriction on $B_\Gamma(J)$.
\end{Not}
If we let $J_{ij}=(\{i\}\cup \{j\})\cap\{1,\ldots,2k\}$, we have $B_\Gamma(J_{ij})\cong \prod_{\lambda\in J_{ij}}K_\lambda$, and there is a natural bundle map
\[ \xymatrix{
  \Omega_{ij}^\Gamma \ar[r]^-{\widetilde{p}_{ij}} \ar[d] & \Omega_{ij}^\Gamma(J_{ij}) \ar[d]\\
  B_\Gamma \ar[r]^-{p_{ij}} & B_\Gamma(J_{ij})
} \]
over the projection $p_{ij}$. 
For example, if $i,j\in\{1,\ldots,2k\}$ and $i\neq j$, then $J_{ij}=\{i,j\}$, $B_\Gamma(J_{ij})\cong K_i\times K_j$, and $\Omega_{ij}^\Gamma=\widetilde{V}_i\times \widetilde{V}_j$. If $i\in\{1,\ldots,2k\}$, then $J_{ii}=\{i\}$, $J_{i\infty}=\{i\}$, and $B_\Gamma(J_{ii})\cong K_i\cong B_\Gamma(J_{i\infty})$. Also, $J_{\infty\infty}=\emptyset$ and $B_\Gamma(J_{\infty\infty})\cong *$.

\begin{Prop}[Normalization of propagator]\label{prop:localization}
There exists a propagator $\omega\in\Omega_\dR^{d-1}(E\bConf_2(\pi^\Gamma))$ satisfying the following conditions.
\begin{enumerate}
\item For $i,j\in\{1,\ldots,2k,\infty\}$, 
\[ \omega|_{\Omega_{ij}^\Gamma}=\widetilde{p}_{ij}^*\,\omega|_{\Omega_{ij}^\Gamma(J_{ij})}. \]
\item For $i,j\in\{1,\ldots,2k\}$, $i\neq j$, 
\[ \omega|_{\Omega_{ij}^\Gamma(J_{ij})}= \sum_{\ell, m} L_{\ell m}^{ij}\,\pr_1^*\,\eta_{S(\widetilde{a}_\ell^i)}\wedge \pr_2^*\,\eta_{S(\widetilde{a}_m^j)}, \]
where $L_{\ell m}^{ij}=(-1)^{d-1}\Lk(b_\ell^i, b_m^j)$ and the sum is over $\ell,m$ such that $\dim{a_\ell^i}+\dim{a_m^j}=d-1$.
\end{enumerate}
\end{Prop}
This is the heart of the computation of the invariant. The statement of Proposition~\ref{prop:localization} looks natural, although its proof given in \S\ref{s:proof-localization} and \S\ref{s:localization-ii}, mostly following Lescop's interpretation \cite{Les2} of Kuperberg--Thurston's theorem (\cite[Theorem~2]{KT}), is not short. In fact, as in \cite{Les2} we will prove a statement stronger than (2), which includes $\infty$. Nevertheless, Proposition~\ref{prop:localization} is sufficient for the main computation in \S\ref{ss:normalize} due to Lemma~\ref{lem:eval-A}. 

The following lemma is a restatement of Lemma~\ref{lem:S(a)}(5), which will also be used in the computation of the invariant.
\begin{Lem}[Integral at a trivalent vertex]\label{lem:int-triple} Let $S(\widetilde{a}_1^\lambda),S(\widetilde{a}_2^\lambda),S(\widetilde{a}_3^\lambda)$ be the submanifolds of $\widetilde{V}_\lambda$ of Lemma~\ref{lem:S(a)}. Then we have
\[\int_{\widetilde{V}_\lambda}\eta_{S(\widetilde{a}_1^\lambda)}\wedge \eta_{S(\widetilde{a}_2^\lambda)}\wedge \eta_{S(\widetilde{a}_3^\lambda)}=\pm 1. \]
\end{Lem}

%%%%%%%%%%%%%%%%%%%%%%%%%%%%%%%
\subsection{Evaluation of the configuration space integrals}\label{ss:normalize}

From now on we complete the proof of Theorem~\ref{thm:Z(G)}, assuming Proposition~\ref{prop:localization}, by proving the following theorem, whose idea of the proof is analogous to that of \cite[Theorem~2]{KT}, \cite[Theorem~2.4]{Les2}, and \cite[Theorem~6.1]{Wa2}.
\begin{Thm}[Theorem~\ref{thm:Z(G)}(3)]\label{thm:Z_k}
Let $d$ be an even integer such that $d\geq 4$ and let $\Gamma$ be a vertex oriented, vertex labelled arrow graph with $2k$ vertices without self-loop (as in Definition~\ref{def:pi-Gamma}). If moreover $\Gamma$ has no multiple edges, we have
\[ Z_k^\Omega(\pi^\Gamma;\tau^\Gamma)=\pm [\Gamma], \]
where the sign depends only on $k$ (not on $\Gamma$ in $P_k\calG_0^\even$).
\end{Thm}

\setcounter{footnote}{0}
For $i_1,i_2,\ldots,i_{2k}\in \{1,\ldots,2k,\infty\}$, let 
\[ \Omega_{i_1i_2\cdots i_{2k}}^\Gamma=p_{B\ell}^{-1}\bigl(\widetilde{V}'_{i_1}\times_{B_\Gamma}\widetilde{V}'_{i_2}\times_{B_\Gamma}\cdots\times_{B_\Gamma}\widetilde{V}'_{i_{2k}}\bigr), \]
where $p_{B\ell}\colon E\bConf_{2k}(\pi^\Gamma)\to E^{\Gamma\infty}\times_{B_\Gamma}\cdots\times_{B_\Gamma}E^{\Gamma\infty}$ is the canonical projection, which is induced by the $\Diff(S^d,U_\infty)$-equivariant projection $\bConf_{2k}(S^d;\infty)\to (S^d)^{\times 2k}$. 
This is the subspace of $E\bConf_{2k}(\pi^\Gamma)$ consisting of configurations $(x_1,x_2,\ldots,x_{2k})$ such that $\pi^{\Gamma\infty}(x_1)=\cdots=\pi^{\Gamma\infty}(x_{2k})$ and $x_r\in\widetilde{V}'_{i_r}$ for each $r$, and is either a manifold with corners or the image of a manifold with corners under a smooth map (Lemma~\ref{lem:Bl(Y2)} and Remark~\ref{rem:compact-Cn(Y)}). More precisely, $\Omega_{i_1i_2\cdots i_{2k}}^\Gamma$ is the image of a bundle over $B_\Gamma$ with fiber a product of the compactifications $\bConf_{p}(V_\ell;\partial V_\ell)$ of configuration spaces of $V_\ell$ (Definition~\ref{def:compact-C2(Y)} and Remark~\ref{rem:compact-Cn(Y)}) under a smooth projection\footnote{The stratification of the compactification of configuration spaces of manifolds with boundary is described in \cite[\S{3.6}]{CILW}.}. Then we have
\[ E\bConf_{2k}(\pi^\Gamma)=\bigcup_{i_1,i_2,\ldots,i_{2k}}\Omega_{i_1i_2\cdots i_{2k}}^\Gamma, \]
where the sum is taken for all possible choices $i_1,i_2,\ldots,i_{2k}\in \{1,\ldots,2k,\infty\}$. 
It follows from the formulas (\ref{eq:I-universal}) and (\ref{eq:push-pull-3}) that
\[ \begin{split}
  (2k)!(3k)!\,Z_k^\Omega(\pi^\Gamma;\tau^\Gamma)&=
\sum_{\Gamma'\in\calL_k^\even}\int_{B_\Gamma}I(\Gamma')[\Gamma']=\sum_{\Gamma'\in\calL_k^\even}\int_{E\bConf_{2k}(\pi^\Gamma)}\omega(\Gamma')[\Gamma']\\
&=\sum_{\Gamma'\in\calL_k^\even}\sum_{i_1,i_2,\ldots,i_{2k}}\int_{\Omega_{i_1i_2\cdots i_{2k}}^\Gamma}\omega(\Gamma')[\Gamma'].
\end{split}\]
Thus, to prove Theorem~\ref{thm:Z_k}, it suffices to compute the integrals
\begin{equation}\label{eq:int_omega}
 \int_{\Omega_{i_1i_2\cdots i_{2k}}^\Gamma}\omega(\Gamma') 
\end{equation}
for all $\Gamma'\in\calL_k^\even$. For a labelled graph $\Gamma'$, we denote its edges by $e_1,\ldots,e_{3k}$ according to the edge labels. Then the integral (\ref{eq:int_omega}) is the one over the configurations such that the vertices of $\Gamma'$ labelled by $1,2,\ldots,2k$ are mapped to a fiber of $\Omega_{i_1i_2\cdots i_{2k}}^\Gamma$. If the image of the ordered pair $(j_a',\ell_a')$ of the (labelled) endpoints of the edge $e_a$ under the map $\{1,2,\ldots,2k\}\to \{1,2,\ldots,2k\};q\mapsto i_q$ are $(j_a,\ell_a)$, namely $j_a=i_{j_a'}$ and $\ell_a=i_{\ell_a'}$, and if the propagator $\omega$ is normalized as in Proposition~\ref{prop:localization}, then by Proposition~\ref{prop:localization} (1), 
\begin{equation}\label{eq:omega|Omega}
 \omega(\Gamma')|_{\Omega_{i_1i_2\cdots i_{2k}}^\Gamma}=\bigwedge_{a=1}^{3k}\phi_{e_a}^*\widetilde{p}_{j_a\ell_a}^*\omega|_{\Omega_{j_a\ell_a}^\Gamma(J_{j_a\ell_a})}. 
\end{equation}
\begin{Lem}\label{lem:eval-A}
Suppose that the propagator $\omega\in\Omega_\dR^{d-1}(E\bConf_2(\pi^\Gamma))$ is normalized as in Proposition~\ref{prop:localization}.
Let $\lambda\in\{1,\ldots,2k\}$. If $i_1,\ldots,i_{2k}\in \{1,\ldots,2k,\infty\}-\{\lambda\}$, then 
\[ \int_{\Omega_{i_1i_2\cdots i_{2k}}^\Gamma}\omega(\Gamma')=0. \]
Hence the integral (\ref{eq:int_omega}) can be nonzero only if $\{i_1,\ldots,i_{2k}\}=\{1,\ldots,2k\}$.
\end{Lem}
\begin{proof}
We think $B_\Gamma(\{1,\ldots,2k\}-\{\lambda\})$ as a subspace of $B_\Gamma$ by taking the $\lambda$-th term to be the basepoint, and denote it by $B_\Gamma/K_\lambda$. Let 
\[ \calE/K_\lambda\to B_\Gamma/K_\lambda\] 
denote the restriction of a bundle $\calE\to B_\Gamma$ over the subspace $B_\Gamma/K_\lambda$. If $i_q\neq\lambda$ for all $q\in\{1,\ldots,2k\}$, the bundle map $\widetilde{p}_{j_a\ell_a}$ factors through the bundle map
\[ \xymatrix{
  \Omega_{j_a\ell_a}^\Gamma \ar[d] \ar[r] & \Omega_{j_a\ell_a}^\Gamma/K_\lambda \ar[d]\\
  B_\Gamma \ar[r] & B_\Gamma/K_\lambda
} \]
for each $a\in\{1,\ldots,3k\}$, since $B_\Gamma(J_{j_a\ell_a})$ does not have the factor $K_\lambda$ for all $a$. 
Hence by (\ref{eq:omega|Omega}), $\omega(\Gamma')$ is the pullback of $\omega(\Gamma')|_{E\bConf_{2k}(\pi^\Gamma)/K_\lambda}$ by the projection $E\bConf_{2k}(\pi^\Gamma)\to E\bConf_{2k}(\pi^\Gamma)/K_\lambda$. If $V_\lambda$ is of type II, $\omega(\Gamma')$ is the pullback of a $3k(d-1)$-form on a $3k(d-1)-(d-3)$-dimensional manifold $E\bConf_{2k}(\pi^\Gamma)/K_\lambda$, which is zero. If $V_\lambda$ is of type I, we can integrate $\omega(\Gamma')$ over $K_\lambda=S^0$ first:
\[ \begin{split}
  \int_{\Omega_{i_1i_2\cdots i_{2k}}^\Gamma}\omega(\Gamma')&=\pm\int_{\Omega_{i_1i_2\cdots i_{2k}}^\Gamma/K_\lambda}\int_{K_\lambda}\omega(\Gamma')\\
  &=\pm\left\{\int_{\Omega_{i_1i_2\cdots i_{2k}}^\Gamma/K_\lambda}\omega(\Gamma')-\int_{\Omega_{i_1i_2\cdots i_{2k}}^\Gamma/K_\lambda}\omega(\Gamma')\right\}=0.\\
\end{split} \]
This completes the proof.
\end{proof}

\begin{Lem}\label{lem:eval-B}
Suppose that the propagator $\omega\in\Omega_\dR^{d-1}(E\bConf_2(\pi^\Gamma))$ is normalized as in Proposition~\ref{prop:localization}. If $\Gamma$ has no multiple edges, we have
\[ \int_{\Omega_{12\cdots (2k)}^\Gamma}\omega(\Gamma')=\left\{\begin{array}{ll}
\pm 1 & \mbox{if $\Gamma'\cong \pm \Gamma$,}\\
0 & \mbox{otherwise}
\end{array}\right. \]
for each $\Gamma'\in\calL_k^\even$. 
Here, we write $\Gamma'\cong \pm \Gamma$ if there exists an isomorphism $\Gamma'\to \Gamma$ of graphs that sends the $i$-th vertex of $\Gamma'$ to the $i$-th vertex of $\Gamma$.
\end{Lem}
\begin{proof}
By (\ref{eq:omega|Omega}) and Proposition~\ref{prop:localization}(2), the restriction of $\omega(\Gamma')$ to $\Omega_{12\cdots (2k)}^\Gamma$ can be described explicitly as follows.
\begin{equation}\label{eq:123}
 \omega(\Gamma')|_{\Omega_{12\cdots (2k)}^\Gamma}
=\bigwedge_{{{(i,j)}\atop{\mathrm{edge\,of\,\Gamma'}}}} 
\left(
\sum_{\ell,m} L_{\ell m}^{ij}\,\pr_i^*\eta_\ell^i\wedge\pr_j^*\eta_m^j
 \right),
 \end{equation}
where $L_{\ell m}^{ij}=(-1)^{d-1}\Lk(b_\ell^i,b_m^j)$, $\eta_\ell^i=\eta_{S(\widetilde{a}_\ell^i)}$, $\eta_m^j=\eta_{S(\widetilde{a}_m^j)}$ and the sum is over $\ell,m$ such that $\dim{a_\ell^i}+\dim{a_m^j}=d-1$. Note that there is a symmetry of the linking number $L_{\ell m}^{ij}=L_{m\ell}^{ji}$ when $d$ is even and that one of $\eta_\ell^i$ and $\eta_m^j$ is of even degree, the result does not depend on the choice the order of $(i,j)$.
The form (\ref{eq:123}) is a linear combination of products of $6k$ $\eta$ forms. 

Furthermore, if $\Gamma$ does not have multiple edges, we may assume that each term in the linear combination is the product of $6k$ \emph{different} $\eta$ forms since there is at most one edge of $\Gamma$ between each pair $(i,j)$ of vertices with $i\neq j$, and for a given pair $(i,\ell)$ the coefficient $L_{\ell m}^{ij}$ is nonzero for at most unique pair $(j,m)$. Thus we have
\[ \omega(\Gamma')|_{\Omega_{12\cdots (2k)}^\Gamma}=\pm\prod_{(i,j)}\left(\sum_{(\ell,m)\in P_{ij}} L_{\ell m}^{ij}\right)\bigwedge_{q=1}^{2k}\bigl(\pr_q^*\eta_1^q\wedge\pr_q^*\eta_2^q\wedge\pr_q^*\eta_3^q\bigr), \]
where $P_{ij}=\{1\leq \ell,m\leq 3\mid \dim{a_{\ell}^i}+\dim{a_{m}^j}=d-1,\,L_{\ell m}^{ij}\neq 0\}$. The cardinality of $P_{ij}$ is the number of edges between $i$ and $j$ in $\Gamma$, which is 1 or 0 by assumption. Hence the right hand side is nonzero only if $|P_{ij}|=1$ for all edges $(i,j)$ of $\Gamma'$. This condition is equivalent to $\Gamma'\cong\pm \Gamma$. More precisely, if $\Gamma$ does not have multiple edges, we have
\[ \int_{\Omega_{12\cdots (2k)}^\Gamma}\omega(\Gamma')=\left\{\begin{array}{ll}
\displaystyle\pm (-1)^{3k(d-1)}\int_{\Omega_{12\cdots (2k)}^\Gamma}\bigwedge_{q=1}^{2k}\bigl(\pr_q^*\eta_1^q\wedge\pr_q^*\eta_2^q\wedge\pr_q^*\eta_3^q\bigr) & \mbox{if $\Gamma'\cong \pm \Gamma$},\\
0 & \mbox{otherwise}.
\end{array}\right. \]
Here the sign $\pm$ is determined by the graph orientations of $\Gamma$ and $\Gamma'$ (the interpretation of the graph orientation in terms of orderings of half-edges was given in \S\ref{ss:he-ori} and \S\ref{ss:vertex-ori-compatible}). Note that there is a canonical diffeomorphism
\[ \widehat{\pr}_1\times\cdots\times\widehat{\pr}_{2k}\colon \Omega_{12\cdots (2k)}^\Gamma\to \widetilde{V}_1\times\cdots\times\widetilde{V}_{2k}, \]
where $\widehat{\pr}_q\colon \Omega_{12\cdots (2k)}^\Gamma\to \widetilde{V}_q$ is the natural projection, which gives the $q$-th point. This diffeomorphism is orientation-preserving. Namely, $\Omega_{12\cdots(2k)}^\Gamma$ is oriented by
\[ \begin{split}
  &(o(K_1)\wedge o(K_2)\wedge\cdots\wedge o(K_{2k}))\wedge (o(V_1)\wedge o(V_2)\wedge\cdots\wedge o(V_{2k}))\\
  &=(o(K_1)\wedge o(V_1))\wedge (o(K_2)\wedge o(V_2))\wedge\cdots\wedge (o(K_{2k})\wedge o(V_{2k}))\\
  &=o(\widetilde{V}_1)\wedge o(\widetilde{V}_2)\wedge\cdots\wedge o(\widetilde{V}_{2k}),
\end{split}\]
where $o(W)$ denotes the orientation of $W$. Note that $o(V_j)$ is of even degree for each $j$. 
Hence in the case $\Gamma'\cong \pm\Gamma$ and $\Gamma$ does not have multiple edges, we have
\[ \begin{split}
\int_{\Omega_{12\cdots (2k)}^\Gamma}\omega(\Gamma')&=\pm\int_{\Omega_{12\cdots (2k)}^\Gamma}\bigwedge_{q=1}^{2k}\widehat{\pr}_q^*(\eta_{S(\widetilde{a}_1^q)}\wedge\eta_{S(\widetilde{a}_2^q)}\wedge\eta_{S(\widetilde{a}_3^q)})\\
&=\pm\prod_{q=1}^{2k}\int_{\widetilde{V}_q}\eta_{S(\widetilde{a}_1^q)}\wedge\eta_{S(\widetilde{a}_2^q)}\wedge\eta_{S(\widetilde{a}_3^q)}=\pm 1
\end{split} \]
by Lemma~\ref{lem:int-triple}. 
\end{proof}

\begin{Lem}\label{lem:eval-C}
Suppose that the propagator $\omega\in\Omega_\dR^{d-1}(E\bConf_2(\pi^\Gamma))$ is normalized as in Proposition~\ref{prop:localization}. If $\Gamma$ has no multiple edges and has no orientation-reversing automorphism, then we have
\[ \int_{\Omega_{\sigma(1)\sigma(2)\cdots \sigma(2k)}^\Gamma}\omega(\Gamma')=\int_{\Omega_{12\cdots (2k)}^\Gamma}\omega(\Gamma') \]
for each $\Gamma'\in\calL_k^\even$ and $\sigma\in\mathfrak{S}_{2k}$. 
\end{Lem}
\begin{proof}
If $\Gamma'\not\cong \pm \Gamma$, the vanishing of the integral on the LHS is the same as Lemma~\ref{lem:eval-B}. If $\Gamma'\cong \pm\Gamma$ and if $\Gamma$ (and $\Gamma'$) does not have an orientation-reversing automorphism, then for a permutation $\sigma\in\mathfrak{S}_{2k}$, we have
\begin{equation}\label{eq:123-sigma}
\begin{split}
 \omega(\Gamma')|_{\Omega_{\sigma(1)\sigma(2)\cdots \sigma(2k)}^\Gamma}
&=\bigwedge_{{{(i,j)}\atop{\mathrm{edge\,of\,\Gamma'}}}} 
\left(
\sum_{\ell,m} L_{\ell m}^{\sigma(i)\sigma(j)}\,\pr_i^*\eta_\ell^{\sigma(i)}\wedge\pr_j^*\eta_m^{\sigma(j)}
 \right)\\
&=\pm\prod_{(i,j)}\left(\sum_{(\ell,m)\in P_{\sigma(i)\sigma(j)}} L_{\ell m}^{\sigma(i)\sigma(j)}\right)\bigwedge_{q=1}^{2k}\pr_{\sigma^{-1}(q)}^*\bigl(\eta_1^q\,\eta_2^q\,\eta_3^q\bigr),
\end{split}
\end{equation}
where the sign is the same as for $\Omega_{12\cdots(2k)}^\Gamma$, and $\Omega_{\sigma(1)\sigma(2)\cdots\sigma(2k)}^\Gamma$ is oriented by
\[ \begin{split}
  &(o(K_1)\wedge o(K_2)\wedge\cdots\wedge o(K_{2k}))\wedge (o(V_{\sigma(1)})\wedge o(V_{\sigma(2)})\wedge\cdots\wedge o(V_{\sigma(2k)}))\\
  &=(o(K_1)\wedge o(V_1))\wedge (o(K_2)\wedge o(V_2))\wedge\cdots\wedge (o(K_{2k})\wedge o(V_{2k}))\\
  &=o(\widetilde{V}_1)\wedge o(\widetilde{V}_2)\wedge\cdots\wedge o(\widetilde{V}_{2k}).
\end{split}\]
We abbreviated $\eta_1^q\wedge\eta_2^q\wedge\eta_3^q$ as $\eta_1^q\,\eta_2^q\,\eta_3^q$ for a typesetting purpose. (Similar abbreviation is used in Example~\ref{ex:computation} below.)
Now (\ref{eq:123-sigma}) gives 
\[ \begin{split}
\int_{\Omega_{\sigma(1)\sigma(2)\cdots\sigma(2k)}^\Gamma}\omega(\Gamma')&=\pm\int_{\Omega_{\sigma(1)\sigma(2)\cdots\sigma(2k)}^\Gamma}\bigwedge_{q=1}^{2k}\widehat{\pr}_{\sigma^{-1}(q)}^*(\eta_{S(\widetilde{a}_1^q)}\wedge\eta_{S(\widetilde{a}_2^q)}\wedge\eta_{S(\widetilde{a}_3^q)})\\
&=\pm\prod_{q=1}^{2k}\int_{\widetilde{V}_q}\eta_{S(\widetilde{a}_1^q)}\wedge\eta_{S(\widetilde{a}_2^q)}\wedge\eta_{S(\widetilde{a}_3^q)}=\int_{\Omega_{12\cdots (2k)}^\Gamma}\omega(\Gamma'),
\end{split} \]
where $\widehat{\pr}_{\sigma^{-1}(q)}\colon \Omega_{\sigma(1)\sigma(2)\cdots\sigma(2k)}^\Gamma\to \widetilde{V}_q$ is the projection onto the $\sigma^{-1}(q)$-th factor, and the sign is the same as for $\Omega_{12\cdots(2k)}^\Gamma$. (An example of this computation is given below in Example~\ref{ex:computation}.) 
\end{proof}

\begin{proof}[Proof of Theorem~\ref{thm:Z_k}]
Let $\omega$ be a propagator normalized as in Proposition~\ref{prop:localization}. Suppose that $\Gamma$ does not have multiple edges. 
If $\Gamma'\cong \pm\Gamma$ and if $\Gamma$ (and $\Gamma'$) does not have an orientation-reversing automorphism, then the same value $\pm [\Gamma]$ (with the same sign) is counted $|\Aut\,\Gamma|$ times, according to Lemma~\ref{lem:eval-C}. Hence by Lemmas~\ref{lem:eval-A} and \ref{lem:eval-B}, 
\[ I(\Gamma')[\Gamma']=\pm |\Aut\,\Gamma|[\Gamma]. \]
If $\Gamma$ has an orientation-reversing automorphism, then the sum of the integrals for $\Gamma'\cong \pm \Gamma$ over the $(2k)!$ components $\Omega_{\sigma(1)\sigma(2)\cdots\sigma(2k)}^\Gamma$ cancels in pairs. Nevertheless, in this case we also have $[\Gamma']=0$. 

Hence, the term $I(\Gamma')[\Gamma']$ is nonzero only if $\Gamma'\cong\pm \Gamma$ and if $\Gamma'$ does not have an orientation-reversing automorphism, in which case $I(\Gamma')[\Gamma']=\pm|\Aut\,\Gamma|[\Gamma]$ by Lemma~\ref{lem:eval-B}. Moreover, \emph{the sign in $\pm|\Aut\,\Gamma|[\Gamma]$ is the same for different choices of $\Gamma'$ such that $\Gamma'\cong\pm \Gamma$}, since $I(-\Gamma')=-I(\Gamma')$ and the value $I(\Gamma')[\Gamma']$ does not depend on the labelling to orient $\Gamma'$. Now there are $\frac{(2k)!(3k)!}{|\Aut\,\Gamma|}$ labellings on each graph $\Gamma$ up to graph isomorphism, and hence we have
\[ Z_k^\Omega(\pi^\Gamma;\tau^\Gamma)=\pm \frac{1}{(2k)!(3k)!}\frac{(2k)!(3k)!}{|\Aut\,\Gamma|}|\Aut\,\Gamma|[\Gamma]=\pm[\Gamma]. \]
The sign $\pm$ in the last term is of the form $\alpha^k\beta^k$ for the signs $\alpha,\beta\in\{-1,1\}$ of Lemma~\ref{lem:int-triple} for types I and II families of handlebodies, respectively.
This completes the proof.
\end{proof}

\begin{Exa}\label{ex:computation}
Let us give an example which confirms the proofs of Lemma~\ref{lem:eval-B} and Theorem~\ref{thm:Z_k} for $k=2$. 
Let $\Gamma$ and $\Gamma'$ be the oriented trivalent graphs for $k=2$ given in the left and middle of Figure~\ref{fig:Gamma-ori}, respectively. We use $\Gamma$ to define surgery. (Recall the convention of \S\ref{ss:arrow-graph} for the orientation of $\Gamma$ for the surgery.) According to Lemmas~\ref{lem:eval-A} and \ref{lem:eval-B}, the integral $I(\Gamma')$ for $(\pi^\Gamma,\tau^\Gamma)$ may be nonzero only if $\Gamma'\cong \pm \Gamma$ and over $\Omega_{i_1i_2i_3i_4}^\Gamma$ with $\{i_1,i_2,i_3,i_4\}=\{1,2,3,4\}$. 
\begin{figure}
\[ \includegraphics[height=35mm]{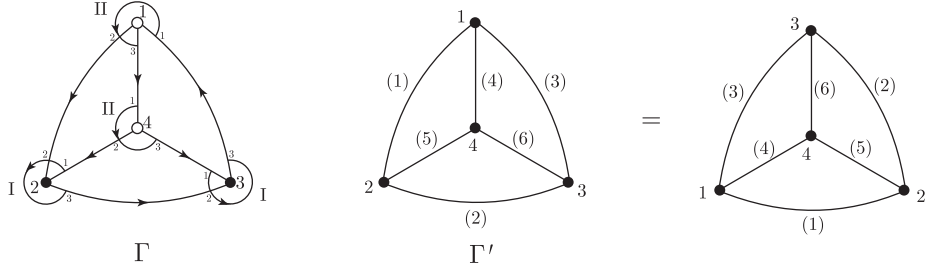} \]
\caption{Oriented graphs $\Gamma$ and $\Gamma'$ for $k=2$. The left $\Gamma$ is oriented in terms of the convention of \S\ref{ss:vertex-ori-compatible}. The middle and right $\Gamma'$ are oriented in terms of the convention (a) of \S\ref{ss:he-ori}.}\label{fig:Gamma-ori}
\end{figure}
By (\ref{eq:123}), 
\[ \begin{split}
  \omega(\Gamma')|_{\Omega_{1234}^\Gamma}&=\omega_{12}\,\omega_{23}\,\omega_{31}\,\omega_{14}\,\omega_{42}\,\omega_{43}\\
  &=\pr_1^*\eta_2^1\wedge\pr_2^*\underline{\eta_2^2}
  \wedge\pr_2^*\eta_3^2\wedge\pr_3^*\underline{\eta_2^3}
  \wedge\pr_3^*\eta_3^3\wedge\pr_1^*\underline{\eta_1^1}\\
  &\phantom{==}\wedge\pr_1^*\eta_3^1\wedge\pr_4^*\underline{\eta_1^4}
  \wedge\pr_4^*\eta_2^4\wedge\pr_2^*\underline{\eta_1^2}
  \wedge\pr_4^*\eta_3^4\wedge\pr_3^*\underline{\eta_1^3}\\
  &=\pr_1^*(\underline{\eta_1^1}\,\eta_2^1\,\eta_3^1)\wedge\pr_2^*(\underline{\eta_1^2}\,\underline{\eta_2^2}\,\eta_3^2)\wedge\pr_3^*(\underline{\eta_1^3}\,\underline{\eta_2^3}\,\eta_3^3)\wedge\pr_4^*(\underline{\eta_1^4}\,\eta_2^4\,\eta_3^4)
\end{split} \]
where $\omega_{12}=\pr_1^*\eta_2^1\wedge\pr_2^*\underline{\eta_2^2}$, $\omega_{23}=\pr_2^*\eta_3^2\wedge\pr_3^*\underline{\eta_2^3}$, $\omega_{31}=\pr_3^*\eta_3^3\wedge\pr_1^*\underline{\eta_1^1}$, $\omega_{14}=\pr_1^*\eta_3^1\wedge\pr_4^*\underline{\eta_1^4}$, $\omega_{42}=\pr_4^*\eta_2^4\wedge\pr_2^*\underline{\eta_1^2}$, $\omega_{43}=\pr_4^*\eta_3^4\wedge\pr_3^*\underline{\eta_1^3}$ (odd degree forms are underlined). Hence
\[ \int_{\Omega_{1234}^\Gamma}\omega(\Gamma')=\int_{\widetilde{V}_1}\eta_1^1\,\eta_2^1\,\eta_3^1\,\int_{\widetilde{V}_2}\eta_1^2\,\eta_2^2\,\eta_3^2\,\int_{\widetilde{V}_3}\eta_1^3\,\eta_2^3\,\eta_3^3\,\int_{\widetilde{V}_4}\eta_1^4\,\eta_2^4\,\eta_3^4=(\pm 1)^2(\pm 1)^2=1.\]
Here, the orientation of $\Omega_{1234}^\Gamma$ is given by $\epsilon_2\,\epsilon_3\,\partial t^{(1)}\wedge\partial t^{(4)}\wedge \partial v^{(1)}\wedge\cdots\wedge \partial v^{(4)}=\epsilon_2\,\epsilon_3\,(\partial t^{(1)}\wedge \partial v^{(1)})\wedge \partial v^{(2)}\wedge \partial v^{(3)}\wedge (\partial t^{(4)}\wedge \partial v^{(4)})$, where $\epsilon_j=\pm 1\in  K_j=\{-1,1\}$ ($j=2,3$), $\partial t^{(i)}$ is the orientation of $K_i=S^{d-3}$ ($i=1,4$), $\partial v^{(i)}$ is the orientation of the fiber $V_i$. 

We consider the permutation $\sigma\colon 1\mapsto 2$, $2\mapsto 3$, $3\mapsto 1$, $4\mapsto 4$, which gives rise to the graph automorphism from the right to the middle one in Figure~\ref{fig:Gamma-ori}. We have
\[ \begin{split}
  \omega(\Gamma')|_{\Omega_{2314}^\Gamma}&=\omega_{23}'\,\omega_{31}'\,\omega_{12}'\,\omega_{42}'\,\omega_{43}'\,\omega_{14}'\stackrel{(*)}{=}\omega_{12}'\,\omega_{23}'\,\omega_{31}'\,\omega_{14}'\,\omega_{42}'\,\omega_{43}'\\
  &=\pr_3^*\eta_2^1\wedge\pr_1^*\underline{\eta_2^2}
  \wedge\pr_1^*\eta_3^2\wedge\pr_2^*\underline{\eta_2^3}
  \wedge\pr_2^*\eta_3^3\wedge\pr_3^*\underline{\eta_1^1}\\
  &\phantom{==}\wedge\pr_3^*\eta_3^1\wedge\pr_4^*\underline{\eta_1^4}
  \wedge\pr_4^*\eta_2^4\wedge\pr_1^*\underline{\eta_1^2}
  \wedge\pr_4^*\eta_3^4\wedge\pr_2^*\underline{\eta_1^3}\\
  &=\pr_3^*(\underline{\eta_1^1}\,\eta_2^1\,\eta_3^1)\wedge\pr_1^*(\underline{\eta_1^2}\,\underline{\eta_2^2}\,\eta_3^2)\wedge\pr_2^*(\underline{\eta_1^3}\,\underline{\eta_2^3}\,\eta_3^3)\wedge\pr_4^*(\underline{\eta_1^4}\,\eta_2^4\,\eta_3^4),
\end{split} \]
where $\omega_{23}'=\pr_1^*\eta_3^2\wedge\pr_2^*\underline{\eta_2^3}$, 
$\omega_{31}'=\pr_2^*\eta_3^3\wedge\pr_3^*\underline{\eta_1^1}$, 
$\omega_{12}'=\pr_3^*\eta_2^1\wedge\pr_1^*\underline{\eta_2^2}$, 
$\omega_{42}'=\pr_4^*\eta_2^4\wedge\pr_1^*\underline{\eta_1^2}$, 
$\omega_{43}'=\pr_4^*\eta_3^4\wedge\pr_2^*\underline{\eta_1^3}$, 
$\omega_{14}'=\pr_3^*\eta_3^1\wedge\pr_4^*\underline{\eta_1^4}$ (odd degree forms are underlined). Hence
\[ \int_{\Omega_{2314}^\Gamma}\omega(\Gamma')=\int_{\widetilde{V}_1}\eta_1^1\,\eta_2^1\,\eta_3^1\,\int_{\widetilde{V}_2}\eta_1^2\,\eta_2^2\,\eta_3^2\,\int_{\widetilde{V}_3}\eta_1^3\,\eta_2^3\,\eta_3^3\,\int_{\widetilde{V}_4}\eta_1^4\,\eta_2^4\,\eta_3^4=1.\]
Here, the orientation of $\Omega_{2314}^\Gamma$ is given by 
\[ \begin{split}
  &\epsilon_2\,\epsilon_3\,\partial t^{(1)}\wedge\partial t^{(4)}\wedge \partial v^{(2)}\wedge \partial v^{(3)}\wedge \partial v^{(1)}\wedge \partial v^{(4)}\\
  &=\epsilon_2\,\epsilon_3\,(\partial t^{(1)}\wedge \partial v^{(1)})\wedge \partial v^{(2)}\wedge \partial v^{(3)}\wedge (\partial t^{(4)}\wedge \partial v^{(4)}).
\end{split} \]
The equality of the integrals of $\omega(\Gamma')$ over $\Omega_{1234}^\Gamma$ and $\Omega_{2314}^\Gamma$ can also be explained by means of the bundle isomorphism $g_\sigma\colon \Omega_{1234}^\Gamma\to \Omega_{2314}^\Gamma$ induced by the permutation $\sigma\colon  V_1\times V_2\times V_3\times V_4\to V_2\times V_3\times V_1\times V_4$; $(x_1,x_2,x_3,x_4)\mapsto (x_2,x_3,x_1,x_4)$. The map $g_\sigma$ preserves the orientation of the fiber in the sense that $g_{\sigma*}o(\Omega_{1234}^\Gamma)=o(\Omega_{2314}^\Gamma)$. Also, according to the computations above, we have
\[ g_\sigma^*\omega(\Gamma')|_{\Omega_{2314}^\Gamma}=\omega(\Gamma')|_{\Omega_{1234}^\Gamma}. \]
Hence 
\[ \int_{\Omega_{1234}^\Gamma}\omega(\Gamma')|_{\Omega_{1234}^\Gamma}=\int_{\Omega_{1234}^\Gamma}g_\sigma^*\omega(\Gamma')|_{\Omega_{2314}^\Gamma}
=\int_{\Omega_{2314}^\Gamma}\omega(\Gamma')|_{\Omega_{2314}^\Gamma}. \]

Similarly, the same value is obtained for other permutations of $\mathfrak{S}_4$ since a graph automorphism of $\Gamma'$ in Figure~\ref{fig:Gamma-ori} always preserves graph orientation and the equality as in $(*)$ above holds. Therefore, we have
\[ I(\Gamma')=\sum_{\sigma\in\mathfrak{S}_4}\int_{\Omega_{\sigma(1)\sigma(2)\sigma(3)\sigma(4)}^\Gamma}\omega(\Gamma')=4!=|\Aut\,\Gamma|. \]
The plus sign is because the graph orientations of $\Gamma'$ and $\Gamma$ are the same.
Hence $[\Gamma']=[\Gamma]$ and 
\[ I(\Gamma')[\Gamma']=|\Aut\,\Gamma|[\Gamma]. \]\qed
\end{Exa}

%%%%%%%%%%%%%%%%%%%%%%%%%%%%%%%
%%%%%%%%%%%%%%%%%%%%%%%%%%%%%%%
\mysection{Proofs of the properties of the Y-graph surgeries}{s:proof-borr}

We shall prove Propositions~\ref{prop:T-bundle-I}, \ref{prop:T-bundle-II}, and Lemma~\ref{lem:B(3,2,2)}, whose proofs are technical and were postponed. In \S\ref{ss:explicit-model}, we will give an explicit model for our parametrized surgery which will be used later in Lemma~\ref{lem:F(a)-2}.

\subsection{The idea}
The proofs of Propositions~\ref{prop:T-bundle-I} and \ref{prop:T-bundle-II} are instances of the same principle.
\begin{Lem}\label{lem:graphing}
If an element $x$ of $\pi_i(\fEmb(\underline{I}^p\cup\underline{I}^p\cup\underline{I}^p,I^d))$ lies in the image of the graphing map
\[ \Psi\colon \pi_{i+1}(\fEmb_0(\underline{I}^{p-1}\cup \underline{I}^{q-1}\cup \underline{I}^{r-1},I^{d-1}))
\to  \pi_i(\fEmb(\underline{I}^{p}\cup \underline{I}^{q}\cup \underline{I}^r,I^{d})), \]
which is defined by considering an $I^{i+1}$-family of embeddings $I^{p-1}\cup I^{q-1}\cup I^{r-1}\to I^{d-1}$ as an $I^i$-family of isotopies $(I^{p-1}\cup I^{q-1}\cup I^{r-1})\times I\to I^{d-1}\times I$,
then $c_*(x)$ as a bundle over $I^i$ can be realized as the mapping cylinder $C(\widetilde{\varphi})$ of a bundle isomorphism $\varphi$ of a trivial $(d-1)$-dimensional handlebody bundle over $I^i$.
\end{Lem}
\begin{proof}
We prove this only for $(i,p,q,r)=(0,d-2,d-2,1)$ and $(d-3,d-2,1,1)$, which correspond to type I and II handlebodies, respectively, for simplicity. Since the complement of a thickened tangle of $\Emb(\underline{H}_1^{e_1}\cup\underline{H}_2^{e_2}\cup\underline{H}_3^{e_3},I^d)$ ($\stackrel{\mathrm{res}}{\simeq}\fEmb_0(\underline{I}^{p-1}\cup \underline{I}^{q-1}\cup \underline{I}^{r-1},I^{d-1})$) is a handlebody relatively diffeomorphic to $T$, we have the following commutative diagram:
\[ \xymatrix{
\pi_{i+1}(\fEmb_0(\underline{I}^{p-1}\cup \underline{I}^{q-1}\cup \underline{I}^{r-1},I^{d-1}))
\ar[r]^-{\Psi} \ar[d]_-{c_*} &  \pi_i(\fEmb(\underline{I}^{p}\cup \underline{I}^{q}\cup \underline{I}^r,I^{d})) \ar[d]^-{c_*} \\
\pi_{i+1}(B\Diff(T,\partial)) \ar[r]^-{\overline{\Psi}} & \pi_i(\tcoprod_{[W,\partial W]}B\Diff(W,\partial))
}\]
where the disjoint union is taken for the class in $\calS^H(V,\partial V)$, and the bottom horizontal map $\overline{\Psi}$ is given by considering a $(T,\partial)$-bundle over $I^{i+1}$ as a mapping cylinder of a bundle isomorphism $\widetilde{\varphi}$ between two $(T,\partial)$-bundles over $I^i$.
If $x=\Psi(\widetilde{x})$, we have
$c_*(x)=c_*\circ \Psi(\widetilde{x})=\overline{\Psi}\circ c_*(\widetilde{x})$.
This completes the proof.
\end{proof}

\subsection{Proof of Proposition~\ref{prop:T-bundle-I}: mapping cylinder structure on $V'$}\label{ss:proof-T-bundle-I}
\begin{proof}[Proof of Proposition~\ref{prop:T-bundle-I}] The following argument is essentially based on the fact that $B(d-2,d-2,1)_d$ is the suspension of $B(d-3,d-3,1)_{d-1}$ (Definition~\ref{def:suspension}). By considering the third component of the framed tangle $B(\underline{d-3},\,\underline{d-3},\,1)_{d-1}$ (Figure~\ref{fig:beta-0} (1)) as a 1-parameter family of points, we obtain an element $\gamma$ of $\pi_1(\fEmb_0(\underline{I}^{d-3}\cup \underline{I}^{d-3}\cup I^0,I^{d-1}))$. Then the class of $B(\underline{d-2},\,\underline{d-2},\,\underline{1})_d$ lies in the image of $\gamma$ under the graphing map
\[ \Psi\colon \pi_1(\fEmb_0(\underline{I}^{d-3}\cup \underline{I}^{d-3}\cup I^0,I^{d-1}))
\to  \pi_0(\fEmb(\underline{I}^{d-2}\cup \underline{I}^{d-2}\cup \underline{I}^1,I^{d})). \]
Then the result follows by Lemma~\ref{lem:graphing}.
\end{proof}

\subsection{Proof of Proposition~\ref{prop:T-bundle-II}: mapping cylinder structure on $\widetilde{V}$}\label{ss:proof-T-bundle-II}

We now construct the family $\beta\in \Omega^{d-3}\fEmb_0(\underline{I}^{d-2}\cup \underline{I}^1\cup \underline{I}^1,I^d)$ of framed string links explicitly to find a parametrized twist map in 4 steps. The basic idea is to construct $\beta$ so that the projection of the second component onto its last coordinate of $I^d$ is a submersion. We will also give another explicit model for $\beta$ later in \S\ref{ss:explicit-model} which is more simple at least for the purpose of only defining the cycle. 
\par\medskip

\subsubsection{Step 1: From a Borromean string link $B(\underline{d-2},\,\underline{d-2},\,\underline{1})_d$ to an $I^{d-3}$-family $\beta''$ of string links in $\fEmb_0(\underline{I}^{d-2}\cup I^1\cup \underline{I}^1,I^d)$.}\label{ss:step1-beta}
Let $T_0=[-1,1]^{d-1}$. We assume that the first and second components of $B(\underline{d-2},\,\underline{d-2},\,\underline{1})_d$ are the standard inclusions
\[ L_i\colon [-1,1]^{d-3}\times I \to T_0\times I \quad (i=1,2)\]
given by $L_i(s,w)=(p_i,0,s,w)$ ($p_i$ is fixed in \S\ref{ss:coord-V}), which is possible by Lemma~\ref{lem:borromean}. A normal framing of $L_1$ is given explicitly by $(\partial x_1,\partial x_2)$. We consider $L_2$ as a $(d-3)$-parameter family of string knots $I\to T_0\times I$ given by the maps
\[L_{2,s}\colon  I\to \{(p_2,0)\}\times [-1,1]^{d-3}\times I\subset T_0\times I\quad(s\in I^{d-3});\]
$L_{2,s}(w)=(p_2,0,s,w)$. 
For each $s$, the endpoints of $L_{2,s}$ are mapped to $T_0\times\{0,1\}$ and depend on $s$. The tuple $(\partial x_1,\partial x_2,\partial x_3,\ldots,\partial x_{d-1})$ gives a normal framing of $L_{2,s}$. Moreover, we assume that the third component $L_3$ of $B(\underline{d-2},\,\underline{d-2},\,\underline{1})_d$ is equipped with a normal framing as in Definition~\ref{def:long-borromean-link}. Thus we obtain a map 
\[ \beta''\colon I^{d-3}\to \fEmb_0(\underline{I}^{d-2}\cup I^1\cup \underline{I}^1,I^d) \]
defined by mapping each $s$ to the family $L_1\cup L_{2,s}\cup L_3$ with the normal framings, where we consider $L_1$ and $L_3$ are independent of $s$, and by identifying $T_0\times I$ with $I^d$. 
\begin{figure}%
\begin{center}%
\begin{tabular}{ccccc}%
\includegraphics[height=35mm]{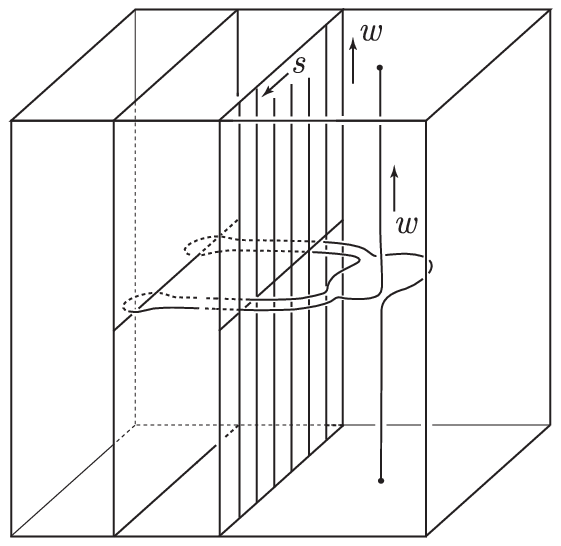} & & \includegraphics[height=37mm]{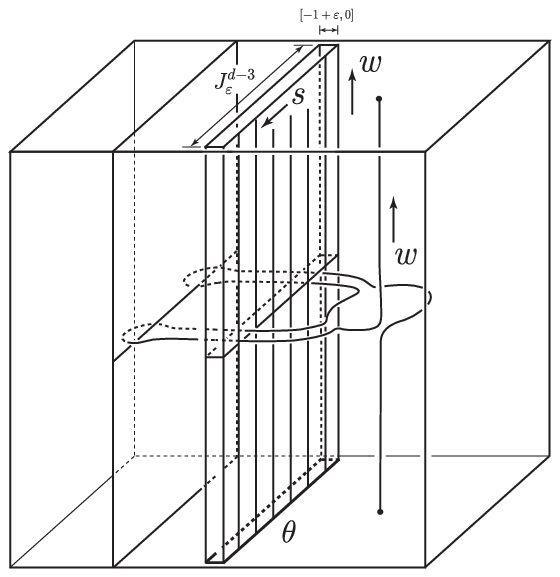} & & \includegraphics[height=35mm]{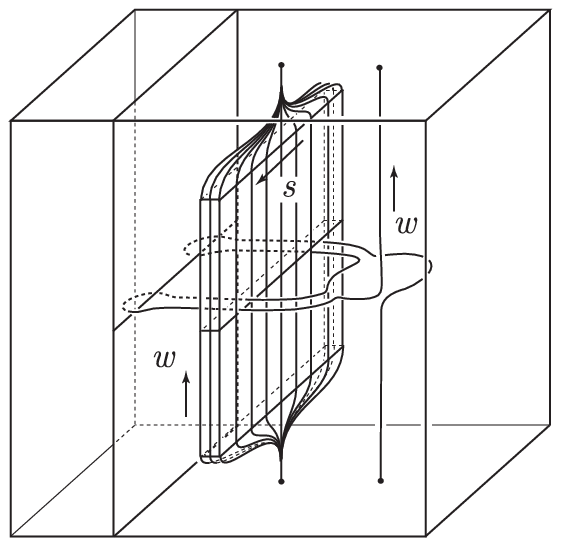}\\%
 (1) Step 1 & & (2) Step 2 & & (3) Step 3%
\end{tabular}%
\end{center}%
\caption{(1) Family of 1-disks in $\beta''$, parametrized by $I^{d-3}$, (2) in $\beta'$, parametrized by $s\in S^{d-3}$, endpoints on the top and bottom not fixed. 1-disks are drawn as vertical lines in the middle component. (3) $S^{d-3}$-family of (vertical) 1-disks in $\beta_a$, endpoints fixed. }\label{fig:closing-disk}%
\end{figure}%
\par\medskip

\subsubsection{Step 2: Closing the $I^{d-3}$-family $\beta''$ into a loop $\beta'$.} 
We alter the $I^{d-3}$-family $\beta''$ to a loop
\[ \beta'\colon (I^{d-3},\partial I^{d-3})\to (\fEmb_0(\underline{I}^{d-2}\cup I^1\cup \underline{I}^1,I^d),a) \]
for some point $a$ as follows. We consider the $(d-3)$-cycle $\theta$ in $T_0$ given by 
\[\begin{split}
 &\partial(\{p_2\}\times [-1+\ve,0]\times J_\ve^{d-3})\\
&=\Bigl(\{(p_2,0)\}\times J_\ve^{d-3}\Bigr)
\cup \Bigl(\{(p_2,-1+\ve)\}\times J_\ve^{d-3}\Bigr)
\cup \Bigl(\{p_2\}\times [-1+\ve,0]\times \partial J_\ve^{d-3}\Bigr),
\end{split} \]
where $0<\ve<1/100$, $J^{d-3}_\ve=[-1+\ve,1-\ve]^{d-3}$. Roughly, $\theta$ is a cycle obtained by closing the $(d-3)$-disk $\{(p_2,0)\}\times J_\ve^{d-3}$ in $T_0$ within the disk $\{p_2\}\times [-1,1]^{d-2}$ along its boundary. The part $\{(p_2,0)\}\times J_\ve^{d-3}$ of $\theta$ is a part of $\{(p_2,0)\}\times [-1,1]^{d-3}=\mathrm{Im}\,L_2\cap (T_0\times\{0\})$. We emphasize that the $(d-3)$-cycle $\theta$ is considered in a $(d-1)$-dimensional slice $T_0\times\{0\}$ in $T_0\times I$, which corresponds to the bottom horizontal disk in Figure~\ref{fig:closing-disk} (2). 
We fix a loop $\lambda\colon (I^{d-3},\partial I^{d-3})\to (\theta, (p_2,-1+\ve,0,\ldots,0))$ of degree one, and define the map
\[ L_{2,s}'\colon I\to \theta\times I\subset T_0\times I\quad(s\in I^{d-3}) \]
by $L_{2,s}'(w)=(\lambda(s),w)$. The tuple $(\partial x_1,\partial x_2,\partial x_3,\ldots,\partial x_{d-1})$ gives a normal framing of this family of 1-disks. Now we obtain the map $\beta'$ by mapping each $s$ to the family $L_1\cup L_{2,s}'\cup L_3$ (Figure~\ref{fig:closing-disk} (2)) with the normal framings, where we again consider $L_1$ and $L_3$ are independent of $s$. Note that $L_1\cup L_{2,s}'\cup L_3$ is a link since the closing disk $(\theta-\{(p_2,0)\}\times J_\ve^{d-3})\times I$ lies in a small neighborhood of $(\partial T_0)\times I$ and does not intersect the components $L_1$ and $L_3$. 
\par\medskip

\subsubsection{Step 3: Making $\beta'$ into a loop $\beta_a$ in $\fEmb_0(\underline{I}^{d-2}\cup \underline{I}^1\cup \underline{I}^1,I^d)$.}\label{ss:beta-a}
We make the family $\beta'$ into that of 1-disks whose boundaries are fixed with respect to $s$, as follows. Let $\rho\colon [0,1]\to [0,1]$ be a smooth function such that
\begin{enumerate}
\item[(i)] $\rho(x)=0$ on a neighborhood of $\{0,1\}$, and $\rho(x)=1$ on $[\ve',1-\ve']$ for some $0<\ve'<1/10$,
\item[(ii)] $\displaystyle\frac{d}{dx}\rho(x)\geq 0$ on $[0,\ve']$, $\displaystyle\frac{d}{dx}\rho(x)\leq 0$ on $[1-\ve',1]$.
\end{enumerate}
We define the `pressing-to-standard' map
$\rho'\colon T_0\times I\to T_0\times I$
by $\rho'(x,w)=\Bigl(\rho(w)x+(1-\rho(w))(p_2,0,\ldots,0),\,w\Bigr)$. By replacing $L_{2,s}'$ with $\rho'\circ L_{2,s}'$ and by a similar replacement for the closing disks $\theta-\{(p_2,0)\}\times J_\ve^{d-3}$, we obtain an $S^{d-3}$-family of 1-disks $I\to T_0\times I$ that are standard near $\partial I$ (Figure~\ref{fig:closing-disk} (3)). This replacement can be obtained by a family of isotopies of the second component which does not intersect the other components, so that the $S^{d-3}$-family of 1-disks obtained after composing $\rho'$ gives a family of embeddings $\underline{I}^{d-2}\cup \underline{I}^1\cup \underline{I}^1\to I^d$. This is because the locus of $\{0\}$ or $\{1\}$ in the family of $I\to T_0\times I$ for $\beta'$ forms a $(d-3)$-sphere in $T_0\times\{0\}$ or $T_0\times\{1\}$ which bounds a disk $\{p_2\}\times [-1+\ve,0]\times J_\ve^{d-3}\times\{i\}$ ($i=0$ or $1$) in $T_0\times\{0,1\}$ that is disjoint from other components, and the pressing map $\rho'$ retracts the spanning disk into a point on that disk. 

This family of embeddings of the second component admits a family of normal framings as follows. The orthogonal projection of the tuple $(\partial x_1,\partial x_2,\partial x_3,\ldots,\partial x_{d-1})$ of sections of $T(T_0)|_{\mathrm{Im}\,\rho'\circ L_{2,s}'}\subset T(T_0\times I)|_{\mathrm{Im}\,\rho'\circ L_{2,s}'}$ to the normal bundle $N(\mathrm{Im}\,\rho'\circ L_{2,s}')$ gives a normal framing of $\rho'\circ L_{2,s}'$. With this family of normal framings, we obtain a family
\[ \beta_a\colon  (I^{d-3},\partial I^{d-3})\to (\fEmb_0(\underline{I}^{d-2}\cup \underline{I}^1\cup \underline{I}^1,I^d),a). \]
Note that this map does not take $\partial I^{d-3}$ to the basepoint of $\fEmb_0(\underline{I}^{d-2}\cup \underline{I}^1\cup \underline{I}^1,I^d)$ since the third component $L_3$ is not standard.  
\par\medskip

\subsubsection{Step 4: Making $\beta_a$ into a loop $\beta$ based at the basepoint. }
We choose any path $\gamma$ in $\fEmb_0(\underline{I}^{d-2}\cup \underline{I}^1\cup \underline{I}^1,I^d)$ from $a$ to the basepoint which isotopes $L_3$ with framing into the standard one and fixes other components, and use it to extend $\beta_a$ to a slightly bigger cube $I'^{d-3}$ by taking the collar $I'^{d-3}-\mathrm{Int}\,I^{d-3}\cong \partial I^{d-3}\times I$ through the composition of the maps $\partial I^{d-3}\times I\to I$ and $\gamma\colon I\to \fEmb_0(\underline{I}^{d-2}\cup \underline{I}^1\cup \underline{I}^1,I^d)$. We assume $\gamma(t)$ is the basepoint for $1-\ve''\leq t\leq 1$ for some small $\ve''>0$. The extended map takes a neighborhood of $\partial I'^{d-3}$ to the basepoint and we obtain an $I'^{d-3}$-family of framed embeddings in $\fEmb_0(\underline{I}^{d-2}\cup \underline{I}^1\cup \underline{I}^1,I^d)$, which after a rescaling $I'^{d-3}\to I^{d-3}$ gives a loop 
\[ \beta\in \Omega^{d-3}\fEmb_0(\underline{I}^{d-2}\cup \underline{I}^1\cup \underline{I}^1,I^d). \]
Then this gives rise to a $(V,\partial)$-bundle $\widetilde{V}\to S^{d-3}$. 
\[ \includegraphics[height=50mm]{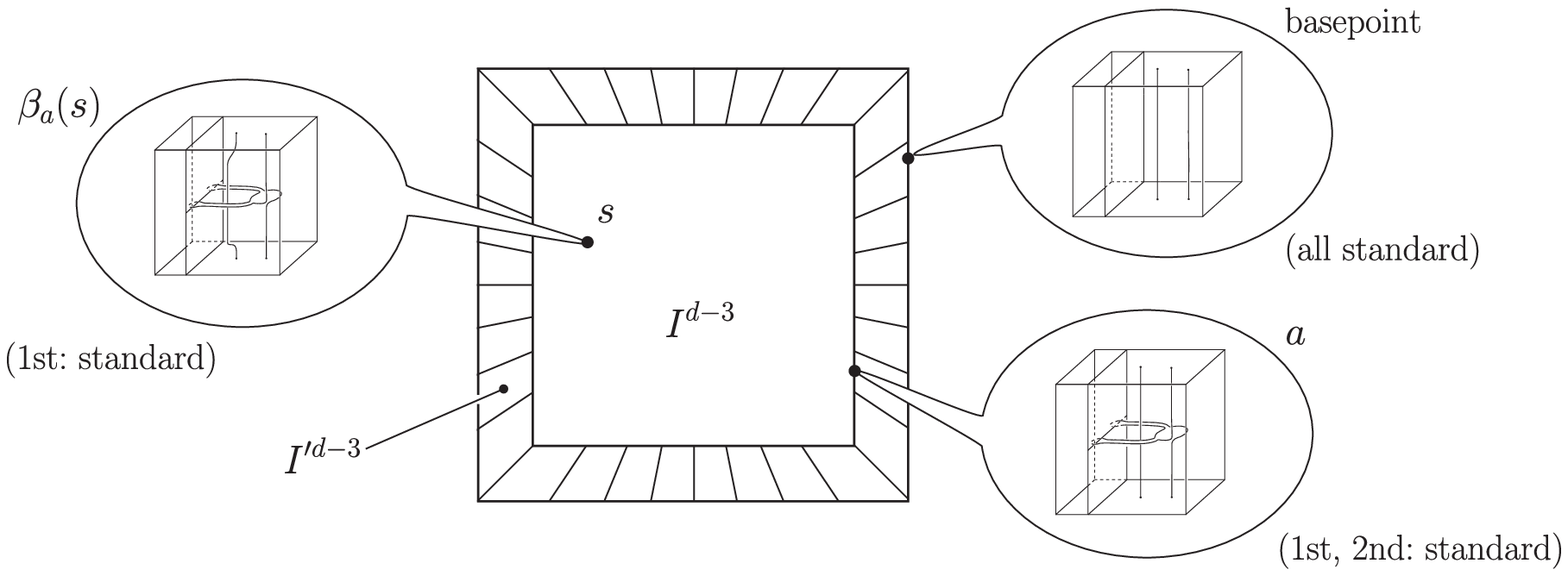} \]

\begin{proof}[Proof of Proposition~\ref{prop:T-bundle-II}]
We see that the loop $\beta\in \Omega^{d-3}\fEmb_0(\underline{I}^{d-2}\cup \underline{I}^1\cup \underline{I}^1,I^d)$ can also be obtained by considering certain element 
\[ \beta_0\in \Omega^{d-2}\fEmb_0(\underline{I}^{d-3}\cup I^0\cup I^0,I^{d-1}) \]
as an $I^{d-3}$-family of isotopies $(I^{d-3}\cup I^0\cup I^0)\times I\to I^{d-1}\times I$ where each isotopy gives rise to an embedding $\underline{I}^{d-2}\cup \underline{I}^1\cup \underline{I}^1\to I^d$. Then we have $[\beta]=\Psi([\beta_0])$ and we can apply Lemma~\ref{lem:graphing}.

\begin{figure}
\begin{center}
\begin{tabular}{ccccc}
\raisebox{5mm}{\includegraphics[height=25mm]{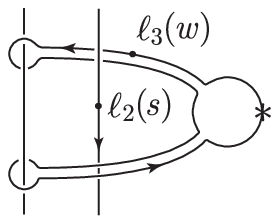}} &  & \raisebox{5mm}{\includegraphics[height=25mm]{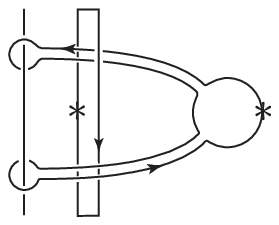}} & & \includegraphics[height=40mm]{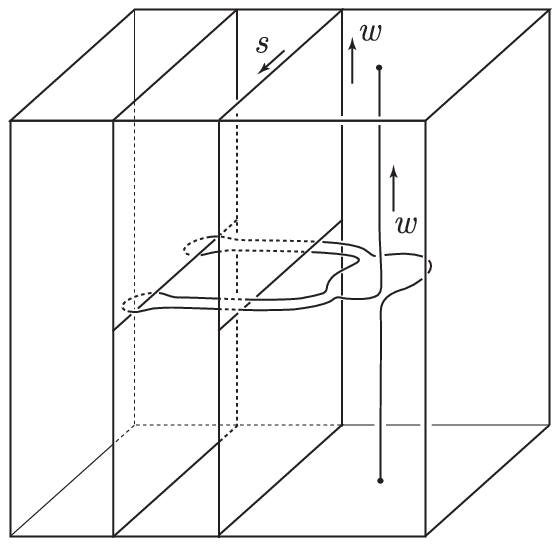}\\
(1) & & (2) & & (3)
\end{tabular}
\end{center}
\caption{(1) $B(\underline{d-3},\,\underline{d-3},\,1)_{d-1}$ parametrized by $(s,w)\in I^{d-3}\times I$.
(2) $B(\underline{d-3},\,d-3,\,1)_{d-1}$ parametrized by $S^{d-3}\times I$. 
(3) $\beta''\colon I^{d-3}\to \fEmb_0(\underline{I}^{d-2}\cup I^1\cup \underline{I}^1,I^d)$. Horizontal section is parallel to the $(d-1)$-disk $T_0$ on the top.}\label{fig:beta-0}
\end{figure}

We construct $\beta_0$ explicitly. The idea is to modify embeddings $I^{d-2}\cup I^1\cup I^1\to I^d$ into isotopies $(I^{d-3}\cup I^0\cup I^0)\times I\to I^{d-1}\times I$ (that are height-preserving). 
Recall that the open $(d-3)$-handles and 0-handles in $T_0$ given in \S\ref{ss:p-borr-I} become $(d-2)$-handles and 1-handles in $T_0\times I$, whose complement is $V$. We saw that $\beta$ is obtained by replacing the trivial $S^{d-3}$-family of the $(d-2)$- and 1-handles in $S^{d-3}\times(T_0\times I)$ by a family corresponding to the Borromean string link $B(\underline{d-2},\,\underline{d-2},\,\underline{1})_d$. We would like to find parametrizations of the family of string links that behave nicely with respect to the ``height'' parameter $I$ in $T_0\times I$, by modifying the family $L_1\cup (\rho'\circ L_{2,s}')\cup L_3$ of framed string links in $\fEmb_0(\underline{I}^{d-2}\cup \underline{I}^1\cup \underline{I}^1,I^d)$ in the definition of $\beta_a$. 

We observe that the first two components $L_1$, $\rho'\circ L_{2,s}'$ are already nice in the sense that the natural maps $\mathrm{pr}_I\circ L_1\colon [-1,1]^{d-3}\times I\to I$ and $\mathrm{pr}_I\circ L_{2,s}'\colon I\to I$ are submersions, where $\mathrm{pr}_I\colon T_0\times I\to I$ is the second projection. Also, we may assume that the third (1-dimensional) component $L_3$ is a section of the projection $\mathrm{pr}_I\colon T_0\times I\to I$, as $B(\underline{d-2},\,\underline{d-2},\,\underline{1})_d$ is the suspension of $B(\underline{d-3},\,\underline{d-3},\,\underline{1})_{d-1}$ for $d\geq 4$ (see \S\ref{ss:long-borromean} and \S\ref{ss:normal-framing-beta} (Definition~\ref{def:suspension-string} below) for the suspensions of the Borromean links). Furthermore, $L_3(w)$ ($w\in I$) can be taken as the lift of a simple closed curve $\ell_3(w)$ in $T_0$ as in Figure~\ref{fig:beta-0} (1). Then we obtain a $I^{d-3}\times I$-family $\beta_{0a}$ of {\it framed} embeddings in $\fEmb_0(\underline{I}^{d-3}\cup I^0\cup I^0,I^{d-1})$:
\[ x\mapsto L_1(x,w)\cup (\rho'\circ L_{2,s}')(w)\cup L_3(w)\quad (x\in [-1,1]^{d-3},s\in I^{d-3},w\in I). \]
Indeed, this family possesses a natural framing. Namely, since the first $I^{d-3}$-component agrees with a standard inclusion and does not depend on the parameter, the basis $(\partial x_1,\partial x_2)$ gives a normal framing of $L_1(\cdot,w)$ in $T_0$. Since the second and third components are family of points, the basis $(\partial x_1,\ldots,\partial x_{d-1})$ gives normal framings of $\rho'\circ L_{2,s}'(w)$ and $L_3(w)$ in $T_0$. One may see that $\beta_{0a}$ gives the $I^{d-3}$-family $\beta_a$ by considering the $I^{d-3}\times I$-family of framed embeddings $I^{d-3}\cup I^0\cup I^0\to T_0$ as an $I^{d-3}$-family of framed embeddings $(I^{d-3}\cup I^0\cup I^0)\times I\to T_0\times I$.

Extending the $I^{d-3}\times I$-family $\beta_{0a}$ to a slightly bigger cube by a null-isotopy of $L_3$ as in the step 4 above, we obtain a map 
\[ \beta_0\colon (I^{d-3}\times I,\partial)\to (\fEmb_0(\underline{I}^{d-3}\cup I^0\cup I^0,I^{d-1}),L_\mathrm{st}).\]
This is possible since the null-isotopy of $L_3$ can be chosen to be height-preserving.

Finally, we see that $[\beta]=\Psi([\beta_0])$ by construction, and the result follows by Lemma~\ref{lem:graphing}.
\end{proof}

%%%%%%%%%%%%%%%%%%%%%%%%%%%%%%%
\subsection{Equivalence of the two models: graph of spinning and iterated suspension}\label{ss:normal-framing-beta}

We prove Lemma~\ref{lem:B(3,2,2)}, which relates the graph of the spinning family construction $\beta$ with a Borromean string link obtained by iterated suspension. 

\begin{Def}[Suspension of string link]\label{def:suspension-string}
Let $L=L_1\cup L_2\cup L_3\colon I^p\cup I^q\cup I^r\to I^d$ ($0<p,q,r<d$) be a string link in $\fEmb(\underline{I}^p\cup\underline{I}^q\cup\underline{I}^r,I^d)$ equipped with a framed isotopy $L_{1,t}\cup L_{2,t}\colon I^p\cup I^q\to I^d$ ($t\in [0,1]$) of the first two components fixing a neighborhood of the boundary $\partial I^p\cup \partial I^q$, such that $L_{1,0}\cup L_{2,0}$ is the standard inclusions of the first two components and $L_{1,1}\cup L_{2,1}=L_1\cup L_2$. Suppose that $L_3$ agrees with the standard inclusion $I^r\to I^d$ outside a ball about $a=(\frac{1}{2},\ldots,\frac{1}{2})\in I^r$ with small radius $R\ll\frac{1}{2}$. Then the {\it suspension} $L'=L_1'\cup L_2'\cup L_3'\colon I^{p+1}\cup I^{q+1}\cup I^r\to I^{d+1}$ of $L$ is defined by
\[ \begin{split}
  &L_1'(u_1,w)=(L_{1,\chi(w)}(u_1),w),\quad L_2'(u_2,w)=(L_{2,\chi(w)}(u_2),w),\\
  &L_3'(u_3)=\left\{\begin{array}{ll}
  (L_3(u_3),\frac{1}{2}) & (|u_3-a|\leq R),\\
  (p_3,\mu_d^{-1}\circ\rho_r\circ\mu_r(u_3)) & (|u_3-a|\geq R),
  \end{array}\right.
\end{split} \]
where $\chi\colon I\to [0,1]$ is a smooth function supported on a small neighborhood of $\frac{1}{2}$ such that $\chi(\frac{1}{2})=1$, $\mu_n\colon [0,1]^n\to [-1,1]^d$ is the embedding defined by $\mu_n(t_1,\ldots,t_n)=(2t_1-1,\ldots,2t_n-1,0,\ldots,0)$, and $\rho_r\colon [-1,1]^d\to [-1,1]^d$ is the diffeomorphism defined by 
\begin{equation}\label{eq:rho_3}
\begin{split}
&\rho_r(x_1,\ldots,x_d)
=(x_1,\ldots,x_{r-1},x_r',x_{r+1},\ldots,x_{d-1},x_d'),\mbox{ where}\\
&x_r'=x_r\cos\psi(|\bvec{x}|)-x_d\sin\psi(|\bvec{x}|),\quad
x_d'=x_r\sin\psi(|\bvec{x}|)+x_d\cos\psi(|\bvec{x}|),\\
&|\bvec{x}|=\sqrt{x_1^2+\cdots+x_d^2}
\end{split}
\end{equation}
for a smooth function $\psi\colon [0,\sqrt{2}]\to [0,\frac{\pi}{2}]$ with $\frac{d}{dt}\psi(t)\geq 0$, which takes the value 0 on $[0,2R]$ and the value $\frac{\pi}{2}$ on $[R',\sqrt{2}]$ for some $R'$ with $2R<R'<\frac{\sqrt{2}}{2}$. (The diffeomorphism $\rho_r$ rotates the sphere of radius $|\bvec{x}|$ by angle $\psi(|\bvec{x}|)$ along the $x_rx_d$-plane.) The resulting embedding $L'$ has a canonical normal framing induced from the original one since the embedding $\rho_r\circ \mu_r$ can be extended to the diffeomorphism $\rho_r$.
By permuting the coordinates, $L'$ with the induced framing can be considered giving an element of $\fEmb(\underline{I}^{p+1}\cup\underline{I}^{q+1}\cup\underline{I}^r,I^{d+1})$. (Figure~\ref{fig:beta-graph} (b).) Suspensions for other choices of components are defined similarly by symmetry.
\end{Def}

Here, we interpret normal framings of embeddings by the model of the ``embedding modulo immersion'', as in \cite[(0.3)]{Wa3}. Let $\bEmb_0(\underline{I}^p\cup \underline{I}^q\cup \underline{I}^r, I^d)$ be the path-component of the point $(L_\st,\const)$ in the homotopy fiber of the derivative map
\[ \Emb_0(\underline{I}^p\cup \underline{I}^q\cup \underline{I}^r, I^d)\to \mathrm{Bun}(T(\underline{I}^p\cup \underline{I}^p\cup \underline{I}^p),TI^d), \]
where
\begin{itemize}
\item $\mathrm{Bun}(T(\underline{I}^p\cup \underline{I}^q\cup \underline{I}^r),TI^d)\simeq \Omega^p(\frac{SO_d}{SO_{d-p}})\times \Omega^q(\frac{SO_d}{SO_{d-q}})\times \Omega^r(\frac{SO_d}{SO_{d-r}})$ is the space of bundle monomorphisms $T(I^p\cup I^q\cup I^r)\to TI^d$ with fixed behavior on the boundary, and the identification in terms of the orthogonal groups is induced by the standard framings of the disks,
\item $\const$ is the constant path at the basepoint of $\mathrm{Bun}(T(\underline{I}^p\cup \underline{I}^q\cup \underline{I}^r),TI^d)$ given by the standard inclusion.
\end{itemize}
A point of $\bEmb_0(\underline{I}^p\cup \underline{I}^q\cup \underline{I}^r, I^d)$ can be represented by an element $f$ of $\Emb_0(\underline{I}^p\cup \underline{I}^q\cup \underline{I}^r, I^d)$ with a regular homotopy, which is a path of immersions, from $f$ to the standard inclusion.

The component $\fEmb_0(\underline{I}^p\cup \underline{I}^q\cup \underline{I}^r, I^d)$ of the standard inclusion $L_\st$ with the standard normal framing can be interpreted as the path-component of the point $(L_\st,\const^3)$ in the homotopy fiber of the map
\[ \Emb_0(\underline{I}^p\cup \underline{I}^q\cup \underline{I}^r, I^d)\to \Omega^p(BSO_{d-p})\times\Omega^q(BSO_{d-q})\times \Omega^r(BSO_{d-r}) \]
given by taking normal bundles. Then there is a natural map
\[ \mathrm{ind}\colon \bEmb_0(\underline{I}^p\cup \underline{I}^q\cup \underline{I}^r, I^d)
\to \fEmb_0(\underline{I}^p\cup \underline{I}^q\cup \underline{I}^r, I^d) \]
induced by the map $\mathrm{Bun}(T(\underline{I}^p\cup \underline{I}^q\cup \underline{I}^r),TI^d)\to 
\Omega^p(BSO_{d-p})\times\Omega^q(BSO_{d-q})\times \Omega^r(BSO_{d-r})$ given by taking normal bundles. Let 
\[ \mathrm{fg}\colon \fEmb(\underline{I}^p\cup \underline{I}^q\cup \underline{I}^r, I^d)\to \Emb(\underline{I}^p\cup \underline{I}^q\cup \underline{I}^r, I^d)\]
be the map given by forgetting framing. The following diagram is commutative:
\begin{equation}\label{eq:ind-graphing}
\xymatrix{
\pi_{d-3}(\bEmb_0(\underline{I}^{d-2}\cup \underline{I}^1\cup \underline{I}^1, I^d)) \ar[r]^-{\widetilde{\Psi}} \ar[d]_-{\mathrm{ind}_*} 
& \pi_0(\bEmb(\underline{I}^{2d-5}\cup \underline{I}^{d-2}\cup \underline{I}^{d-2}, I^{2d-3})) \ar[d]^-{\mathrm{ind}_*} \\
\pi_{d-3}(\fEmb_0(\underline{I}^{d-2}\cup \underline{I}^1\cup \underline{I}^1, I^d)) \ar[r]^-{\Psi} & \pi_0(\fEmb(\underline{I}^{2d-5}\cup \underline{I}^{d-2}\cup \underline{I}^{d-2}, I^{2d-3}))
} 
\end{equation}
where the horizontal maps are the ones induced by graphing.

\begin{Lem}\label{lem:beta-graph}
\begin{enumerate}
\item The class $\mathrm{fg}_*([\beta])\in\pi_{d-3}(\Emb(\underline{I}^{d-2}\cup \underline{I}^1\cup \underline{I}^1, I^d))$ has a canonical lift $[\widetilde{\beta}]\in\pi_{d-3}(\bEmb_0(\underline{I}^{d-2}\cup \underline{I}^1\cup \underline{I}^1, I^d))$ such that $\mathrm{ind}_*([\widetilde{\beta}])=[\beta]$.
\item The class $[B(\underline{2d-5},\underline{d-2},\underline{d-2})_{2d-3}]\in \pi_0(\Emb(\underline{I}^{2d-5}\cup \underline{I}^{d-2}\cup \underline{I}^{d-2}, I^{2d-3}))$ has a canonical lift $[\widetilde{B}(\underline{2d-5},\underline{d-2},\underline{d-2})_{2d-3}]\in \pi_0(\bEmb(\underline{I}^{2d-5}\cup \underline{I}^{d-2}\cup \underline{I}^{d-2}, I^{2d-3}))$ such that $\widetilde{\Psi}([\widetilde{\beta}])=[\widetilde{B}(\underline{2d-5},\underline{d-2},\underline{d-2})_{2d-3}]$.
\end{enumerate}
\end{Lem}
\begin{proof}
(1) This is a straightforward analogue of the proof of (1') in the proof of \cite[Lemma~A]{Wa3} (obtained just by replacing $(D^k\cup D^k\cup D^k,Q^{2k+1})$ with $(I^{d-2}\cup I^1\cup I^1,I^d)$, and by exchanging the role of the first and second component). 

(2) A lift $\widetilde{B}(\underline{2d-5},\underline{d-2},\underline{d-2})_{2d-3}$ is constructed as a result of iterated suspension of the first and third components in $B(\underline{d-2},\underline{d-2},\underline{1})_d$ with the spanning disks $\underline{D_i}$ ($i=1,2,3$) by extending the suspension of string links to those with spanning disks in a straightforward manner. 

To prove $\widetilde{\Psi}([\widetilde{\beta}])=[\widetilde{B}(\underline{2d-5},\underline{d-2},\underline{d-2})_{2d-3}]$, we compare the two elements of $\bEmb(\underline{I}^{2d-5}\cup \underline{I}^{d-2}\cup \underline{I}^{d-2}, I^{2d-3})$ represented by the following objects:
\begin{enumerate}
\item[(a)] The string link $(I^{d-2}\cup I^1\cup I^1)\times I^{d-3}\to I^d\times I^{d-3}$ with spanning disks obtained from $\widetilde{\beta}$ by graphing.
\item[(b)] The string link obtained from $B(\underline{d-2},\underline{d-2},\underline{1})_d$ with the spanning disks $\underline{D_i}$ by the $(d-3)$-fold suspension for the first and third components.
\end{enumerate}
The family of spanning disks of (a) is given by a straightforward analogue of those in the proof of (1') of \cite[Lemma~A]{Wa3}.

We assume without loss of generality the following. For (a), we assume that the first and third components agree with the ones obtained from the constant $I^{d-3}$-families of the standard inclusions $I^{d-2}\cup \emptyset\cup I^1\to I^d$. This is possible by Lemma~\ref{lem:borromean}. Moreover, we also assume similar condition for the second component outside a ball $D_R$ about $(\frac{1}{2},\ldots,\frac{1}{2})\in I^{1}\times I^{d-3}$ with small radius $R\ll \frac{1}{2}$. Then the associated graph is of the form that is obtained from the graph of the standard spinning model $\rho'\circ L_{2,s}'$ of \S\ref{ss:beta-a} (assumed to lie in a small $(2d-3)$-ball) by connect summing a $(d-2)$-sphere $\widetilde{L}_2$ in $I^{2d-3}-(I^{d-2}\cup\emptyset\cup I^1)\times I^{d-3}$, which is disjoint from the ball $U_R$ about $(p_2,\frac{1}{2},\ldots,\frac{1}{2})\in I^{2d-3}$ with radius $R$ and lies in a small tubular neighborhood of $I^d\times \{0\}$ in $I^{2d-3}$, by a thin band connecting the point $(p_2,\frac{1}{2},\ldots,\frac{1}{2})$ with a basepoint of $\widetilde{L}_2$. We may perturb the object (a) within the class $\widetilde{\Psi}([\widetilde{\beta}])$ into one such that $\widetilde{L}_2$ lies in $I^d\times\{0\}$ and the restriction of the embedding of the second component to $D_R$ collapses into $I^d\times\{0\}$ outside $U_R$ (Figure~\ref{fig:beta-graph} (a)). 

For (b), we assume that the first and third components are standard as for (a). Moreover, we may assume that the second component satisfies a similar condition as above for (a), namely, it is standard outside $D_R$ and is a connected sum of the standard model for the suspension with $\widetilde{L}_2\subset I^d\times\{0\}$ (Figure~\ref{fig:beta-graph} (b)).

\begin{figure}
\[\begin{array}{cc}
\includegraphics[height=20mm]{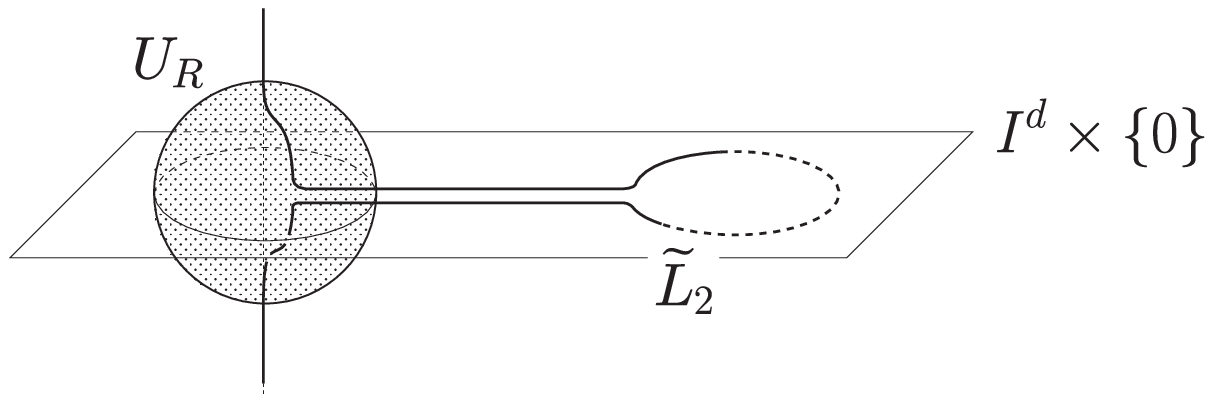}&
\includegraphics[height=20mm]{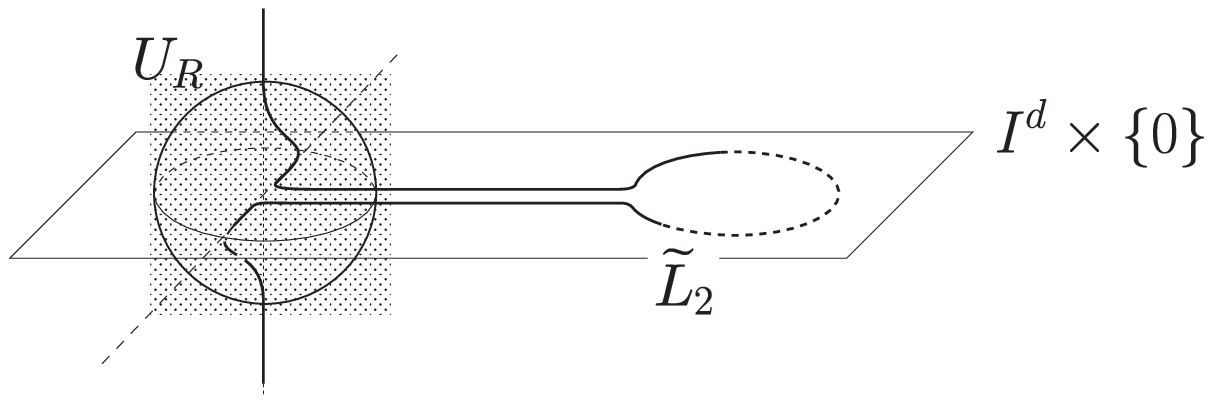}\\
\mbox{(a) graph of spinning} & \mbox{(b) suspension}
\end{array} \]
\caption{The two models for the second component.}\label{fig:beta-graph}
\end{figure}

Now we prove that the two models in $U_R$ are related by an isotopy in $U_R$ that fix a neighborhood of $\partial U_R$. Note that the first and third components do not intersect $U_R$, and hence the intersection of the images of the embeddings of $D_R$ with $U_R$ consist of a single component. By assuming that the bands for the connected sums with $\widetilde{L}_2$ is sufficiently thin, it suffices to prove that the two models without connected sums with $\widetilde{L}_2$ are related by an isotopy. Let $f_1,f_2\colon D_R\to U_R$ be the embeddings of the two models, respectively. As $f_1$ can be isotoped to the restriction of the standard inclusion, by collapsing the spinning model of \S\ref{ss:beta-a} onto a base-line, we need only to prove that $f_2$ can be so too. That $f_2$ can be isotoped to the restriction of the standard inclusion can be seen inductively by using the explicit model given in Definition~\ref{def:suspension-string}. More precisely, we replace the smooth function $\psi\colon [0,\sqrt{2}]\to [0,\frac{\pi}{2}]$ with $\psi_\ve=(1-\frac{2\ve}{\pi})\psi+\ve\colon [0,\sqrt{2}]\to [\ve,\frac{\pi}{2}]$ for small $\ve>0$. We only consider the case of the suspension from 
$B(\underline{d-2},\underline{d-2},\underline{1})_d$ to $B(\underline{d-1},\underline{d-2},\underline{2})_{d+1}$ as the subsequent steps are parallel to this case. When $d=4$, more steps are not necessary. Let $\rho_{r,\ve}\colon [-1,1]^d\to [-1,1]^d$ be the diffeomorphism defined similarly as $\rho_r$ in Definition~\ref{def:suspension-string} by replacing $\psi$ with $\psi_\ve$. 
Then it follows that $\rho_{r,\ve}\circ \rho_r^{-1}\circ f_2$ is the restriction of the graph of a 1-parameter family of $r$-cubes for $0<\ve\leq \frac{\pi}{2}$ since the derivative of $\rho_{r,\ve}\circ\rho_r^{-1}\circ f_2$ along the $x_d$-axis is positive. Then there is an ambient isotopy $\{\rho_{r,(1-s)\ve+s\frac{\pi}{2}}\circ \rho_r^{-1}\}_{s\in[0,1]}$ of $U_R$ perturbing $\rho_{r,\ve}\circ \rho_r^{-1}\circ f_2$ to the standard inclusion, since $\rho_{r,\frac{\pi}{2}}=\mathrm{id}_{U_R}$ and $\rho_r^{-1}\circ f_2$ is the standard inclusion. 
This completes the proof.
\end{proof}

\begin{proof}[Proof of Lemma~\ref{lem:B(3,2,2)}]
By the commutativity of (\ref{eq:ind-graphing}) and Lemma~\ref{lem:beta-graph}, we have
\[\begin{split}
 \Psi([\beta])&=\Psi(\mathrm{ind}_*([\widetilde{\beta}]))=\mathrm{ind}_*(\widetilde{\Psi}([\widetilde{\beta}]))\\
&=\mathrm{ind}_*([\widetilde{B}(\underline{2d-5},\underline{d-2},\underline{d-2})_{2d-3}])\\
&=[(B(\underline{2d-5},\underline{d-2},\underline{d-2})_{2d-3},F_D)].
\end{split}\]
This completes the proof.
\end{proof}

%%%%%%%%%%%%%%%%%%%%%%%%%%%%%%%
\subsection{Explicit model for type II surgery}\label{ss:explicit-model}

Recall that a surgery on a type II handlebody was defined by using a ``family of embeddings $I^{d-2}\cup I^1\cup I^1\to I^d$ obtained by parametrizing the second component in the Borromean string link''. Now we give an explicit model for the family of embeddings $I^1\to I^d$ of the second components, which will be used in Lemma~\ref{lem:F(a)-2}. 
Note that the description below gives the same element of $\pi_{d-3}(\fEmb_0(\underline{I}^{d-2}\cup\underline{I}^1\cup\underline{I}^1,I^d))$ as the one given in \S\ref{ss:proof-T-bundle-II} (Lemma~\ref{lem:eq-g-beta} below). 

\subsubsection{Family of arcs in the upper hemisphere $S_+^{d-2}$}

We consider the upper hemisphere $S_+^{d-2}=\{(x_1,\ldots,x_{d-1})\in\R^{d-1}\mid x_1^2+\cdots+x_{d-1}^2=1,\,x_{d-1}\geq 0\}$ and the smooth arcs
\[ f_\nu\colon (-1,1)\to S_+^{d-2};\quad t\mapsto te_1+\sqrt{1-t^2}\,\nu, \]
where $e_1=(1,0,\ldots,0)$, $\nu=(0,a_2,\ldots,a_{d-1})\in S_+^{d-2}$. We consider $\nu\in S_+^{d-3}=D^{d-3}$. 
\[ \includegraphics[height=35mm]{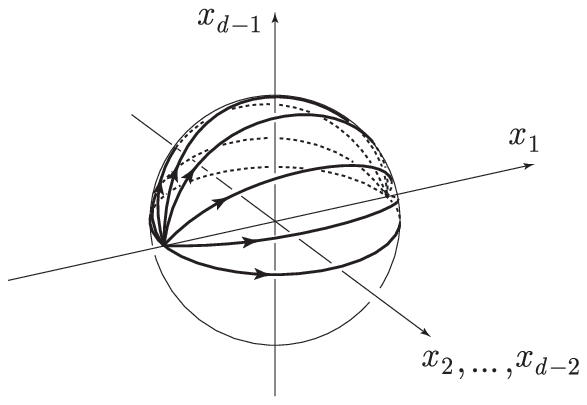} \]

\subsubsection{Extension to a family of smooth embeddings of lines}

We take a smooth function $\rho\colon \R^{d-2}\to [0,1]$ satisfying the following conditions:
\begin{enumerate}
\item $\rho$ is a radial function, i.e., $\rho(x)=\rho(x')$ whenever $|x|=|x'|$.
\item $\rho(x)=0$ for $|x|\geq 1-\delta$ for some $\delta$ such that $0<\delta<1/10$.
\item $\rho(x)=1$ for $|x|\leq a$ for some $a$ such that $0<a<1/10$.
\item $\displaystyle\frac{\partial}{\partial r}\rho(x)\leq 0$ for $r=|x|\leq 1$.
\end{enumerate}
Let $\rho'\colon S_+^{d-2}\to D_+^{d-1}$ be defined by 
\[ \rho'(x_1,\ldots,x_{d-1})=(x_1,\ldots,x_{d-2},\rho(x_1,\ldots,x_{d-2})x_{d-1}). \]
\[ \includegraphics[height=30mm]{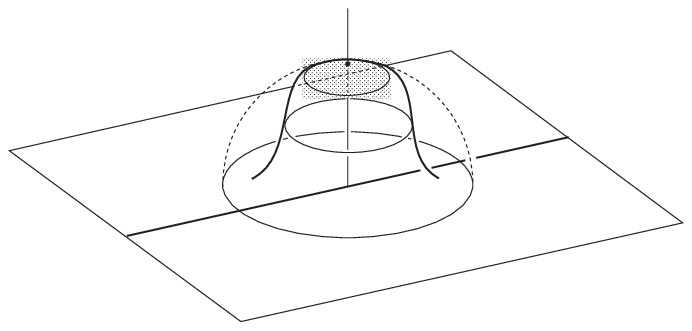} \]
Then we define $f_\nu'\colon \R\to \R^{d-1}$ by 
\[ f_\nu'(t)=\left\{\begin{array}{ll}
(t,0,\ldots,0) & (|t|\geq 1),\\
\rho'(f_\nu(t)) & (|t|<1).
\end{array}\right.\]
This gives a family of piecewise smooth embeddings of lines. 

\subsubsection{Pressing to a thick band}

Let $\kappa\colon [0,1]\to [0,1]$ be a smooth function such that $\kappa(h)=0$ for $0\leq h\leq m$ and $\kappa(h)=1$ for $m'\leq h\leq 1$ for some $0<m<m'<1$. Let $\kappa'\colon D_+^{d-1}\to D_+^{d-1}$ be defined by 
\[ \kappa'(x_1,\ldots,x_{d-1})=(x_1,\kappa(x_{d-1})x_2,\ldots,\kappa(x_{d-1})x_{d-2},x_{d-1}). \]
\[ \includegraphics[height=25mm]{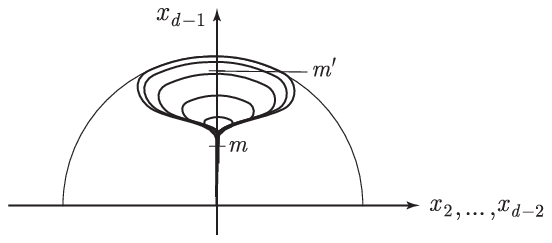} \]
Then we define $f_\nu''\colon \R\to \R^{d-1}$ by
\[ f_\nu''(t)=\kappa'(f_\nu'(t)). \]
Since $f_\nu''(t)$ agrees with $(t,0,\ldots,0)$ for $|t|\geq 1-\delta$ by the explicit formula, $f_\nu''$ is a smooth embedding. As $\nu=(0,a_2,\ldots,a_{d-1})\in S_+^{d-2}$ varies on a $(d-3)$-disk, $\{f_\nu''\}_\nu$ is a $D^{d-3}$-family of smooth embeddings of lines in $\R_+^{d-1}$ such that for $\nu\in \partial D^{d-3}$, $f_\nu''$ agrees with the standard inclusion $t\mapsto (t,0,\ldots,0)$, and the images of $f_\nu''$ for $\nu\in D^{d-3}$ covers the image of $\kappa'\circ \rho'\colon S_+^{d-2}\to D_+^{d-1}$. 

Now we define a normal framing of the embedding $f_\nu''$, which gives rise to a smooth $D^{d-3}$-family of normally framed embeddings of lines. 
Observe that the first coordinate of the tangent vector $\displaystyle\frac{d f_\nu''(t)}{dt}$ is $1$, it is transversal to the codimension 1 subspace of $T_{f_\nu''(t)}\R^{d-1}$ spanned by $\partial x_2,\ldots,\partial x_{d-1}$. We put $L_\nu''=\mathrm{Im}\,f_\nu''$ and let $NL_\nu''$ be the orthogonal complement of $TL_\nu''\subset T\R^{d-1}$. By the transversality of $T_{f_\nu''(t)}L_\nu''$ and $T_{f_\nu''(t)}\R^{d-1}$, the orthogonal projection $T\R^{d-1}|_{L_\nu''}\to NL_\nu''$ takes $\partial x_2,\ldots,\partial x_{d-1}\in T_{f_\nu''(t)}\R^{d-1}$ to a basis of $N_{f_\nu''(t)}L_\nu''$. This construction of a fiberwise basis of $NL_\nu''$ gives a normal framing $\tau_\nu$ of $L_\nu''$. Thus $\{(f_\nu'',\tau_\nu)\}$ gives a $D^{d-3}$-family of normally framed embeddings of lines, whose restriction to $\partial D^{d-3}$ consists of the standard inclusion with the standard framing. 

\subsubsection{Embedding into a small neighborhood of a sphere with arc}

For a small positive real number $\ve$, let 
\[ \begin{split}
Q&=\{(x_1,\ldots,x_{d-1})\in\R^{d-1}\mid\\
&\hspace{20mm} (1-\ve)^2\leq x_1^2+x_2^2+\cdots+x_{d-2}^2+(x_{d-1}-10)^2\leq (1+\ve)^2\},\\
R&=\{(x_1,\ldots,x_{d-1})\in\R^{d-1}\mid x_1^2+\cdots+x_{d-2}^2\leq\ve^2,\,0\leq x_{d-1}\leq 9\}.
\end{split}\]
Then $Q\cup R$ is a small closed neighborhood of an $(d-1)$-sphere connected to the origin by an arc. 
\[ \includegraphics[height=40mm]{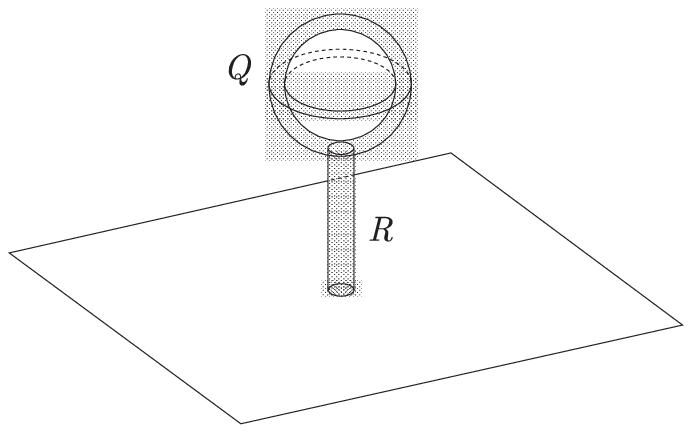} \]

We embed the family of embeddings $f_\nu''$ into $\R^{d-2}\cup Q\cup R$, as follows. We embed the part $\R- [-1,1]$ onto $\R- [-\ve,\ve]$ by scaling $t\mapsto \ve t$, where $\R$ is the $x_1$ axis in $\R^{d-1}$, and embed the locus $L=\mathrm{Im}(\kappa'\circ\rho')$ of $f_\nu''([-1,1])$ into $Q\cup R$. This is possible since there is a small ball $B$ in $\mathrm{Int}\,\kappa'\circ\rho'(D_+^{d-1})$, and $L$ is included in the punctured disk $D_+^{d-1}-\mathrm{Int}(B)$. It is easy to construct a diffeomorphism $\iota\colon D_+^{d-1}-\mathrm{Int}(B)\to Q\cup R$ which extends the scaling $x\mapsto \ve x$ of $\R^{d-2}$. Here, we consider that the corner of $Q\cup R$ is smoothed. Now we define an embedding $g_\nu\colon \R\to \R^{d-1}$ by
\[ g_\nu(t)=\left\{\begin{array}{ll}
\ve t & (t\in\R-[-1,1]),\\
\iota\circ f_\nu''(t) & (t\in [-1,1]).
\end{array}\right. \]
The family $\{g_\nu\}_\nu$ is again a smooth $I^{d-3}$-family of smooth embeddings $\R\to \R^{d-1}$. This is a standard model of a ``family of embeddings $I^1\to I^{d-1}\subset I^d$ that goes around a small neighborhood of a sphere with arc''. The differential of the diffeomorphism $\iota$ takes the normal framing $\tau_\nu$ to a normal framing $\sigma_\nu$ of $g_\nu$. Hence we obtain a $D^{d-3}$-family $\{(g_\nu,\sigma_\nu)\}$ of normally framed embeddings of lines. We may assume that this model has the following property: 
\begin{Propt}
For $\nu\in I^{d-3}$ such that the image of $g_\nu$ is included in $\R\cup R$, the image of $g_\nu$ is included in the 2-plane spanned by the vectors $(1,0,\ldots,0)$ and $(0,\ldots,0,1)$. For such a $\nu$, the normal framing $\sigma_\nu$ is of the form $(\sigma_\nu',\partial x_2,\ldots,\partial x_{d-2})$, where $\sigma_\nu'$ is a vector that lies in the plane $\langle \partial x_1,\partial x_{d-1}\rangle$.
\end{Propt}

The $D^{d-3}$-family $\{(g_\nu,\sigma_\nu)\}$ can be considered as obtained from the standard spinning model of the second component in the definition of $\beta_a:I^{d-3}\to \fEmb_0(\underline{I}^{d-2}\cup{I}^1\cup\underline{I}^1,I^d))$ of \S\ref{ss:beta-a} (Figure~\ref{fig:closing-disk} (3)) by a family of isotopies that deforms the bubble into $Q$, and the rest into a band in $R$, as in the following picture:
\[ \includegraphics[height=30mm]{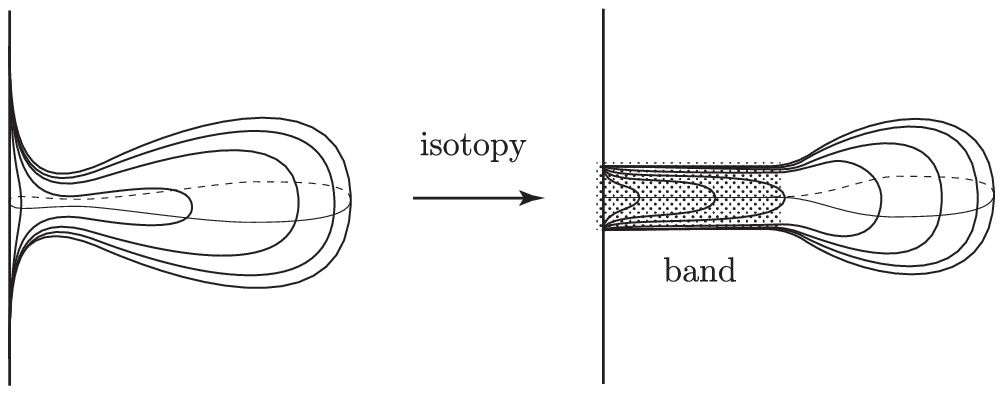} \]

By embedding $I^{d-2}\cup Q\cup R$ into the complement of the standard inclusions of the first and third components in $I^d$, so that $I^{d-2}$ is a standard inclusion and $Q$ is embedded along the second $(d-2)$-dimensional component of $B(\underline{d-2},d-2,\underline{1})_d$, we obtain a standard model of $\beta$. Now the following lemma is evident.
\begin{Lem}\label{lem:eq-g-beta}
The classes in $\pi_{d-3}(\fEmb_0(\underline{I}^{d-2}\cup\underline{I}^1\cup\underline{I}^1,I^d)))$ of the two constructions: through $\beta_a$ of \S\ref{ss:beta-a} (Figure~\ref{fig:closing-disk} (3)), and through the $D^{d-3}$-family $\{(g_\nu,\sigma_\nu)\}$ of embeddings of $I^1\to I^d$, agree.
\end{Lem}

%%%%%%%%%%%%%%%%%%%%%%%%%%%%%%%
%%%%%%%%%%%%%%%%%%%%%%%%%%%%%%%
\mysection{Normalization of propagator: Proof of Proposition~\ref{prop:localization}}{s:proof-localization}

In this section, we shall prove that the normalization of propagator as in Proposition~\ref{prop:localization} is possible on all the pieces $\Omega_{ij}^\Gamma$ except the diagonal ones $\Omega_{ii}^\Gamma$ ($i\neq\infty$), mostly following Lescop's interpretation given in \cite{Les2} of Kuperberg--Thurston's sketch proof for 3-manifolds (\cite[\S{6}]{KT}).

%%%%%%%%%%%%%%%%%%%%%%%%%%%%%%%
\subsection{Preliminaries}\label{ss:preliminaries}

In the rest of this section, we put $X=\bConf_1(S^d;\infty)$. 

\begin{enumerate}
\item[(i)] Let $a_1^i,a_2^i,a_3^i, b_1^i, b_2^i, b_3^i$ be the cycles in $\partial V_i$ defined in \S\ref{ss:coord-V} and \S\ref{ss:std-cycles}. We take a basepoint $p^i$ of $\partial V_i$ that is disjoint from the cycles $b_j^i,a_j^i$. If $V_i$ is of Type I, two of the cycles $b_j^i$ are circles and one of the cycles $b_j^i$ is $(d-2)$-dimensional sphere. If $V_i$ is of Type II, one of the cycles $b_j^i$ is a circle and two of the cycles $b_j^i$ are $(d-2)$-dimensional spheres. 

\item[(ii)] Let $S(a_\ell^i)$ be a disk in $V_i$ that is bounded by $a_\ell^i$. Let $S(b_\ell^i)$ be a disk in $X-\mathrm{Int}\,V_i$ that is bounded by $b_\ell^i$. Let $\gamma^i$ be a smoothly embedded path in $V_\infty$ from $p^i$ to $\infty\in S^d$, which is disjoint from $S(b_m^j)$ for all $(m,j)$. The exsistence of such a $\gamma^i$ follows from the particular construction of $V_i$ from Y-links as in \S\ref{ss:Y-link}. Further, we assume that $\gamma^i\cap \gamma^j=\emptyset$ for $i\neq j$.

\item[(iii)] $S(b_\ell^i)$ may intersect a handle of $V_j$ ($j\neq i$) transversally. We assume that the intersection agrees with $S(a_m^j)$ for some unique $(m,j)$ up to orientation. This is possible according to the special linking property of the handlebodies in graph surgery. 

\item[(iv)] For $i\neq\infty$, we identify a small tubular neighborhood of $\partial V_i$ in $X$ with $[-4,4]\times\partial V_i$ so that $\{0\}\times\partial V_i=\partial V_i$ and $\{-4\}\times\partial V_i\subset \mathrm{Int}\,V_i$. For a cycle $x$ of $\partial V_i$ represented by a manifold, let 
\[ x[h]=\{h\}\times x\subset [-4,4]\times\partial V_i \]
 and let $x^+$ denote a parallel copy of $x$ obtained by slightly shifting $x$ along positive direction in the coordinate $[-4,4]$. Here, $[-4,4]\times \partial V_i$ is a subset of a single fiber $X$. 
Also, let
\[ \begin{split}
  V_i[h]&=\left\{
  \begin{array}{ll}
    V_i\cup ([0,h]\times\partial V_i) & (h\geq 0),\\
    V_i-((h,0]\times\partial V_i) & (h<0),
  \end{array}\right.\\
  S_h(b_\ell^i)&=\left\{
  \begin{array}{ll}
    S(b_\ell^i)\cap (X-\mathrm{Int}({V}_i[h])) & (h\geq 0),\\
    S(b_\ell^i)\cup ([h,0]\times b_\ell^i) & (h<0),
  \end{array}
  \right.\\
  S_h(a_\ell^i)&=\left\{\begin{array}{ll}
  S(a_\ell^i)\cup ([0,h]\times a_\ell^i) & (h> 0),\\
  S(a_\ell^i)\cap V_i[h] & (h\leq  0),\\
  \end{array}\right.\\
  V_\infty[h]&=
    X-\mathrm{Int}\,(V_1[-h]\cup \cdots\cup V_{2k}[-h]),
\end{split}\]
where, $V_\infty$ was defined in \S\ref{ss:localize}. 
\[ \includegraphics[height=40mm]{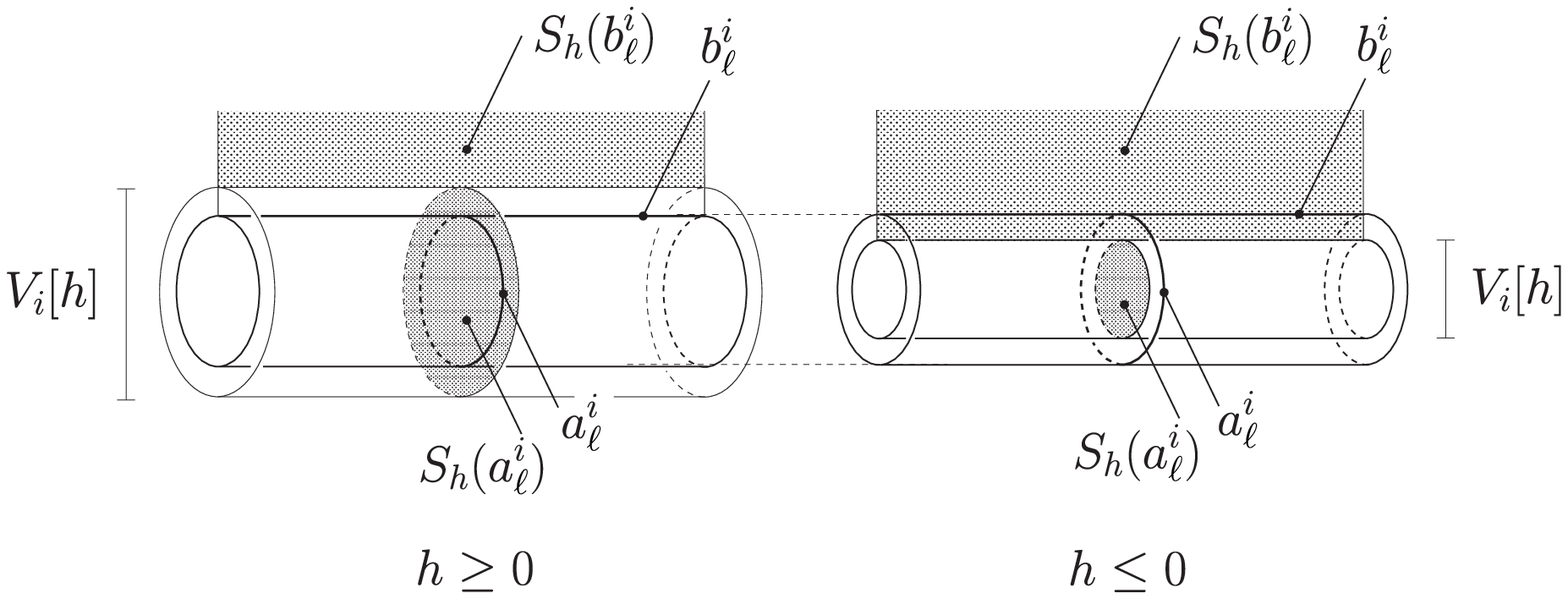} \]

\item[(v)] The boundary of $\widetilde{V}_i$ ($i\neq\infty$) is $K_i\times \partial V_i$. The factor $K_i$ has nothing to do with the $[-4,4]$ in the previous item.
Let 
\[ \widetilde{b}_\ell^i=K_i\times b_\ell^i\quad\mbox{ and }\quad\widetilde{a}_\ell^i=K_i\times a_\ell^i.  \]
Let $S(\widetilde{a}_\ell^i)$ be the compact submanifold of $\widetilde{V}_i$ with $\partial S(\widetilde{a}_\ell^i)=\widetilde{a}_\ell^i$ given by Lemma~\ref{lem:S(a)}. We assume without loss of generality that the intersection of $S(\widetilde{a}_\ell^i)$ with $[-4,4]\times \partial\widetilde{V}_i=K_i\times ([-4,4]\times \partial V_i)$ (in $\widetilde{V}_i$) agrees with $[-4,4]\times \widetilde{a}_\ell^i$. 

\item[(vi)] $\widetilde{V}_i[h]$, $\widetilde{V}_\infty[h]$, $\widetilde{V}_i'[h]$, $\widetilde{V}_\infty'[h]$, $S_h(\widetilde{a}_\ell^i)\subset E\bConf_2(\pi^\Gamma)(\{i\})$ etc. can be defined in a similar way. $\Omega_{ij}^\Gamma[h,h']$ is defined by replacing $\widetilde{V}_i'$, $\widetilde{V}_\infty'$ in the definition of $\Omega_{ij}^\Gamma$ with $\widetilde{V}_i'[h]$, $\widetilde{V}_\infty'[h]$, respectively.
\end{enumerate}

%%%%%%%%%%%%%%%%%%%%%%%%%%%%%%%
\subsection{Normalization of propagator with respect to one handlebody $V_j$, $j\neq\infty$, unparametrized case}\label{ss:normalization_one}

We put $V=V_j$ and abbreviate $a_i^j,b_\ell^j,\gamma^j$ etc. as $a_i,b_\ell,\gamma$ etc. for simplicity. We identify $\partial X$ with $S^{d-1}$, and its collar neighborhood with $[0,1]\times S^{d-1}$, where $\{0\}\times S^{d-1}=\partial X$. Let $\overline{\gamma}$ be the closure of the lift of $\gamma-\{\infty\}$ in $X=B\ell_{\{\infty\}}(S^d)$. 
Let $\overline{\eta}_\gamma$ be a closed $(d-1)$-form on $X$ supported on the union of a tubular neighborhood of $\gamma$ and $[0,1]\times \partial X\subset \bConf_1(S^d;\infty)$ whose restriction to a tubular neighborhood of $\gamma$ in $X-[0,1)\times \partial X$ agrees with $\eta_{\gamma}$ (defined on $\mathrm{Int}\,X$) and whose restriction to $\{0\}\times \partial X$ is the $SO_d$-invariant unit volume form on $\partial X=S^{d-1}$ which is consistent with the orientation of $\partial X$.
  
\begin{Prop}[Normalization for one handlebody]\label{prop:normalization1}
There exists a propagator $\omega$ on $\bConf_2(S^d;\infty)$ that satisfies the following ($x^+=x[h]$ for some small $h>0$).
\begin{enumerate}
\item $\displaystyle \omega|_{V\times (X-\mathring{V}[3])}=
\sum_{i,\ell}(-1)^{(\dim{a_i})d-1}\Lk(b_i,a_\ell^+)\,\pr_1^*\,\eta_{S(a_i)}\wedge \pr_2^*\,\eta_{S_3(b_\ell)}
+\pr_2^*\,\overline{\eta}_{\gamma[3]}$, where the sum is over $i,\ell$ such that $\dim{b_i}+\dim{a_\ell}=d-1$. 
\par\medskip
\item $\displaystyle \omega|_{(X-\mathring{V}[3])\times V}=
\sum_{i,\ell}(-1)^{(\dim{a_i})d-1}\Lk(a_i^+,b_\ell)\,
\pr_1^*\,\eta_{S_3(b_i)}\wedge \pr_2^*\,\eta_{S(a_\ell)}
+(-1)^d\pr_1^*\,\overline{\eta}_{\gamma[3]}$, where the sum is over $i,\ell$ such that $\dim{a_i}+\dim{b_\ell}=d-1$. 
\par\medskip
\item $\displaystyle \int_{p\times S_3(a_i)}\omega=0$,
$\displaystyle \int_{S_3(a_i)\times p}\omega=0$ when $\dim{a_i}=d-2$.

\item $\displaystyle \int_{b_j\times S_3(a_i)}\omega=0$,
$\displaystyle\int_{S_3(a_i)\times b_j}\omega=0$ when $d=4$ and $\dim{a_i}=\dim{b_j}=1$.
\end{enumerate}
\end{Prop}

See Figure~\ref{fig:V-normalized} for the domain where $\omega$ is normalized. The conditions (1), (2) imply that $\omega$ is an extension of (\ref{eq:omega-explicit}) on $V_i\times V_j$. The condition (3) and (4) are technical conditions which will only be needed so that the induction in the proof of Proposition~\ref{prop:normalize2} works. More precisely, in the proofs of Lemmas~\ref{lem:Vm-1} and \ref{lem:Vm-2}, respectively.
\begin{figure}
\[ \includegraphics[height=50mm]{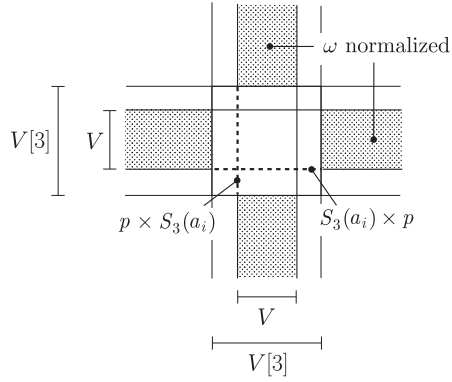} \]
\caption{Where $\omega$ is normalized for one $V$ (projection in $X\times X$).}\label{fig:V-normalized}
\end{figure}

Let $A=V\times (X-\mathring{V}[3])$, where $\mathring{V}$ denotes $\mathrm{Int}\,V$. Each term in the formula of Proposition~\ref{prop:normalization1} (1) represents the Poincar\'{e}--Lefschetz dual of an element of $H_{d+1}(A,\partial A)$, as shown in Lemma~\ref{lem:X-V} (3) below. We start with any propagator $\omega_0$ in $\bConf_2(S^d;\infty)$ and check that its restriction to $A$ gives the same class in $H^{d-1}(A)$ as the formula of Proposition~\ref{prop:normalization1} (1). Then it follows that by adding some exact form supported on a neighborhood of $A$ to $\omega_0$ we obtain a propagator satisfying Proposition~\ref{prop:normalization1} (1). To do so, we compare the values of the integrals along cycles that represent a basis of the dual $H_{d-1}(A)$. Verification of the condition (2) is similar. 

\begin{Lem}\label{lem:X-V}
\begin{enumerate}
\item $H_i(X-V)=H_{i+1}(V,\partial V)$ for $i>0$ and $H_0(X-V)=\R$. Namely, 
$H_*(X- V)=\langle [*],[a_1],[a_2],[a_3],[\partial V]\rangle$.
\item $H_*(A)=H_*(V)\otimes \langle [*],[a_1],[a_2],[a_3],[\partial V]\rangle$.
\item $H_{d-1}(A)$ is generated by $[p\times \partial V[3]]$, $[b_i\times a_\ell[3]]$ for $\dim{b_i}+\dim{a_\ell}=d-1$.
\item $H_{d+1}(A,\partial A)$ is generated by the following elements.
\[ [S(a_i)\times S_3(b_\ell)],\quad [V\times \gamma[3]],\]
where $\dim{a_i}+\dim{b_\ell}=d-1$.
\end{enumerate}
\end{Lem}
\begin{proof}
In the homology long exact sequence for the pair $(X,X- V)$, we have $H_*(X)=0$ for $*>0$. Also, by excision, we have $H_{i+1}(X,X- V)=H_{i+1}(V,\partial V)$. This gives (1). The rest is obtained by the K\"unneth formula and Poincar\'{e}--Lefschetz duality.
\end{proof}

\begin{proof}[Proof of Proposition~\ref{prop:normalization1}]
This proof is similar to \cite[Proposition~11.2, 11.6, 11.7]{Les3}. Let $\omega_0$ be any propagator and $\omega_A$ be the closed $(d-1)$-form on 
\[ A':=V[1]\times (X-\mathring{V}[2]) \]
defined by the natural extension of the one given by the condition (1). This domain $A'$ is the sum of $A$ with a collar neighborhood, on which we connect $\omega_0$ and $\omega_A$ by an exact form.
The integrals of $\omega_0$ over the generators $b_i\times a_\ell[3]$, $p\times \partial V[3]$ of $H_{d-1}(A)$ (Lemma~\ref{lem:X-V} (3)) are as follows.
\[ 
\int_{b_i\times a_\ell[3]}\omega_0=\Lk(b_i,a_\ell^+), \quad
\int_{p\times\partial V[3]}\omega_0=1.
\]
Also, by Lemma~\ref{lem:int-eta} (1) and (2), we compute
\[\begin{split}
 &\int_{b_i^-\times a_\ell[3]}\pr_1^*\,\eta_{S(a_{i})}\wedge \pr_2^*\,\eta_{S_3(b_{\ell})}\\
&=\int_{b_i^-}\eta_{S(a_{i})}\int_{a_\ell[3]}\eta_{S_3(b_{\ell})}
 =(-1)^{kd+k+d-1}(-1)^{d+k}=(-1)^{kd-1},
\end{split} \]
where $k=\dim{a_i}=\dim{a_\ell}$. 
From the identities
\[ \begin{split}
  &\int_{b_i\times a_\ell[3]}\pr_1^*\,\eta_{S(a_{i'})}\wedge \pr_2^*\,\eta_{S_3(b_{\ell'})}=(-1)^{(\dim{a_i})d-1}\delta_{ii'}\delta_{\ell\ell'},\quad \int_{p\times\partial V[3]}\pr_2^*\,\overline{\eta}_\gamma=1, \\
  &\int_{p\times\partial V[3]}\pr_1^*\,\eta_{S(a_{i'})}\wedge \pr_2^*\,\eta_{S_3(b_{\ell'})}=0,\quad 
\int_{b_i\times a_\ell[3]}\pr_2^*\,\overline{\eta}_\gamma=0,
\end{split}\]
it follows that the closed form $\omega_A$ and the restriction of $\omega_0$ to $A'$ gives the same element of $H^d(A')$. 
Hence there exists a $(d-2)$-form $\mu$ on $A'$ such that $\omega_A=\omega_0+d\mu$ and $d\mu=0$ on $V[1]\times \partial X$, since $\omega_A$ and $\omega_0$ agree with $\pr^*_2\mathrm{Vol}_{S^{d-1}}$ on $V[1]\times \partial X$ by assumption. Moreover, we may assume that $\mu=0$ on $V[1]\times \partial X$ by adding to $\mu$ a closed form on $A'$. Namely, since $\partial X$ is $(d-2)$-connected, the natural map $H^{d-2}(V[1]\times(X-\mathring{V}[2]))\to H^{d-2}(V[1]\times\partial X)$ is surjective, and there is a closed extension $\mu'$ of $\mu|_{V[1]\times \partial X}$ on $A'$. Then we replace $\mu$ with $\mu-\mu'$, which vanishes on $V[1]\times \partial X$. 

Let $\chi\colon \bConf_2(S^d;\infty)\to [0,1]$ be a smooth function such that $\mathrm{Supp}\,\chi=A'$ and $\chi=1$ on $A=V\times (X-\mathring{V}[3])$. Then let
\[ \omega_a:=\omega_0+d(\chi\mu). \]
This is a closed form on $\bConf_2(S^d;\infty)$ that is as required on $V\times(X-\mathring{V}[3])$ (as the condition (1)) and agrees with $\omega_0$ on $\partial\bConf_2(S^d;\infty)$ because $\chi=0$ on the diagonal stratum of $\partial \bConf_2(S^d;\infty)$ and $\mu=0$ on the infinity stratum. 

For the condition (3), let $r_j=\int_{p\times S_3(a_j)}\omega_a$ for $\dim{a_j}=d-2$. We would like to cancel this value by adding to $\omega_a$ a form $d(\chi\mu_c)$ for some closed form $\mu_c$ on $A'$, which vanishes on $V[1]\times \partial X$. This is possible because the addition of $d(\chi\mu_c)$ changes the integral $r_j$ by
\[\int_{p\times S_3(a_j)}d(\chi\mu_c)=\int_{p\times ([2,3]\times a_j)}d(\chi\mu_c)=\int_{p\times a_j[3]}\mu_c, \]
where the left equality is because $\mathrm{Supp}\,\chi\cap (p\times S_3(a_j))=p\times ([2,3]\times a_j)$, and the right equality is because $\chi=0$ on $p\times a_j[2]$. 
By $\int_{p\times a_j[3]}\pr_2^*\, \eta_{S_2(b_\ell)}=(-1)^{2d-2}\delta_{j\ell}=\delta_{j\ell}$ for $\dim{a_j}=d-2$, $\dim{b_\ell}=1$ from Lemma~\ref{lem:int-eta} (2), the first half of the condition (3) will be satisfied if we replace $\omega_a$ with
\[ \omega_a'=\omega_a + d(\chi \mu_c),\mbox{ where }
\mu_c=- \sum_{{j:\,}\atop{\dim{b_j}=1}} r_j (\pr_2^*\,\eta_{S_2(b_j)}). \]

For the condition (4) (only for $d=4$), let $\lambda_{ij}=\int_{b_i\times S_3(a_j)}\omega_a$ for $\dim{b_i}=\dim{a_j}=1$. For a closed form $\mu_c'$ on $A'$, which vanishes on $V[1]\times \partial X$, we have
\[ \int_{b_i\times S_3(a_j)}d(\chi\mu_c')=\int_{b_i\times ([2,3]\times a_j)}d(\chi\mu_c')=\int_{b_i\times a_j[3]}\mu_c'. \]
By $\int_{b_i\times a_j[3]}p_1^*\eta_{S_3(a_k)}\wedge p_2^*\eta_{S(b_\ell)}=(-1)^{d-1}(-1)^{d+1}\delta_{ik}\delta_{j\ell}=\delta_{ik}\delta_{j\ell}$ for $\dim{a_k}=\dim{b_\ell}=2$ from Lemma~\ref{lem:int-eta} (1) and (2), the first half of the condition (4) will be satisfied if we replace $\omega_a'$ with
\[ \omega_a''=\omega_a'+ d(\chi \mu_c'),\mbox{ where }
\mu_c'=-\sum_{{i,j:\,}\atop{\dim{a_i}=\dim{b_j}=1}}\lambda_{ij}( p_1^*\eta_{S_3(a_i)}\wedge p_2^*\eta_{S(b_j)}). \]
This change does not affect the previous modification since $\int_{b_i\times a_j[3]}p_2^*\eta_{S_2(b_\ell)}=0$
for $\dim{b_i}=\dim{a_j}=\dim{b_\ell}=1$ and $\int_{p\times a_j[3]}p_1^*\eta_{S_3(a_k)}\wedge p_2^*\eta_{S(b_\ell)}=0$ for $\dim{a_j}=\dim{a_k}=\dim{b_\ell}=2$.

A similar modification of $\omega_a''$ on $(X-\mathring{V}[3])\times V$ is possible without touching the previous modifications and yields another closed $(d-1)$-form $\omega$ that satisfies the conditions (1)--(4). In this case the coefficients are determined by the following identities:
\[ \begin{split}
  &\int_{a_i[3]\times b_\ell}\omega_0=\Lk(a_i^+,b_\ell), \quad
  \int_{\partial V[3]\times p}\omega_0=(-1)^d,\\
  &\int_{a_i[3]\times b_\ell}\pr_1^*\,\eta_{S_3(b_{i'})}\wedge \pr_2^*\,\eta_{S(a_{\ell'})}=(-1)^{(\dim{a_i})d-1}\delta_{ii'}\delta_{\ell\ell'},\quad \int_{\partial V[3]\times p}\pr_1^*\,\overline{\eta}_\gamma=1, \\
  &\int_{\partial V[3]\times p}\pr_1^*\,\eta_{S_3(b_{i'})}\wedge \pr_2^*\,\eta_{S(a_{\ell'})}=0,\quad 
\int_{a_i[3]\times b_\ell}\pr_1^*\,\overline{\eta}_\gamma=0,
\end{split}\]

\end{proof}

%%%%%%%%%%%%%%%%%%%%%%%%%%%%%%%
\subsection{Normalization of propagator with respect to a finite set of handlebodies, unparametrized case}\label{ss:normalize2}

Let $V_1,\ldots,V_{2k}$ be the disjoint handlebodies in $X$ that define $\pi^\Gamma$. We normalize propagator with respect to this set of handlebodies. 

\begin{Prop}[Normalization for several handlebodies]\label{prop:normalize2}
There exists a propagator $\omega$ on $\bConf_2(S^d;\infty)$ that satisfies the following conditions.
\begin{enumerate}
\item For each $j=1,2,\ldots,m$, 
\[ \omega|_{V_j\times (X-\mathring{V}_j[3])}=
\sum_{i,\ell} (-1)^{(\dim{a_i^j})d-1}\Lk(b_i^j,a_\ell^{j+})\,\pr_1^*\,\eta_{S(a_i^j)}\wedge \pr_2^*\,\eta_{S_3(b_\ell^j)} + \pr_2^*\, \overline{\eta}_{\gamma^j[3]}, \]
where the sum is over $i,\ell$ such that $\dim{b_i^j}+\dim{a_\ell^j}=d-1$.
\par\medskip

\item For each $j=1,2,\ldots,m$, 
\[ \omega|_{(X-\mathring{V}_j[3])\times V_j} 
=\sum_{i,\ell}(-1)^{(\dim{a_i^j})d-1}\Lk(a_i^{j+},b_\ell^{j})\,\pr_1^*\,\eta_{S_3(b_i^j)}\wedge \pr_2^*\,\eta_{S(a_\ell^j)} +(-1)^d\pr_1^*\,\overline{\eta}_{\gamma^j[3]}, \]
where the sum is over $i,\ell$ such that $\dim{a_i^j}+\dim{b_\ell^j}=d-1$. 

\item $\displaystyle\int_{p^j\times S_3(a_i^j)}\omega=0$,
$\displaystyle\int_{S_3(a_i^j)\times p^j}\omega=0$ ($j=1,2,\ldots,m$, $\dim{a_i^j}=d-2$).

\item $\displaystyle \int_{b_i^j\times S_3(a_k^j)}\omega=0$,
$\displaystyle\int_{S_3(a_k^j)\times b_i^j}\omega=0$ ($j=1,2,\ldots,m$) when $d=4$ and $\dim{b_i^j}=\dim{a_k^j}=1$.
\end{enumerate}
\end{Prop}
\begin{proof}
The following proof is an analogue of \cite[Proposition~5.1]{Les2}. We prove Proposition~\ref{prop:normalize2} by induction on $m$. The case $m=1$ is Proposition~\ref{prop:normalization1}. For $m>1$, we take a propagator $\omega_0$ that satisfies the conditions of Proposition~\ref{prop:normalize2} for all $j<m$, and $\omega_m$ that satisfies the conditions of Proposition~\ref{prop:normalize2} for a single $m$, with $V_m$ and $X-\mathring{V}_m[3]$ replaced by larger subspaces $V_m[1]$ and $X-\mathring{V}_m[2]$, respectively, so that $\omega_0$ and $\omega_m$ agree on $V_m[1]\times V_j$. By Lemma~\ref{lem:propagator1}, there exists a $(d-2)$-form $\mu$ on $\bConf_2(S^d;\infty)$ such that $\omega_m=\omega_0+d\mu$. We may assume that $\omega_m$ agrees with $\omega_0$ on $\partial \bConf_2(S^d;\infty)$ and moreover that $\mu=0$ there since $H^{d-2}(\partial\bConf_2(S^d;\infty))=0$ by the exact sequence:
\[ 0=H^{d-2}(\bConf_2(S^d;\infty))\to H^{d-2}(\partial \bConf_2(S^d;\infty))\to H^{d-1}(\bConf_2(S^d;\infty),\partial \bConf_2(S^d;\infty)), \]
and $H^{d-1}(\bConf_2(S^d;\infty),\partial \bConf_2(S^d;\infty))\cong H_{d+1}(\bConf_2(S^d;\infty))=0$ by Poincar\'{e}--Lefschetz duality. Then we set
\[ \omega_a=\omega_0+d(\chi\mu), \]
where $\chi\colon \bConf_2(S^d;\infty)\to [0,1]$ is a smooth function with $\mathrm{Supp}\,\chi=V_m[1]\times(X-\mathring{V}_m[2])$ (Figure~\ref{fig:supp-chi}) that takes the value 1 on $V_m\times(X-\mathring{V}_m[3])$. Then $\omega_a$ is a closed $(d-1)$-form on $\bConf_2(S^d;\infty)$, which is as desired on
\[ \partial\bConf_2(S^d;\infty)\cup \bigcup_{j=1}^m(V_j\times (X-\mathring{V}_j[3]))\cup \bigcup_{j=1}^{m-1}((X-(\mathring{V}_j[3]\cup \mathring{V}_m[1]))\times V_j). \]
(Figure~\ref{fig:Vm-Vj-normalized}.) We need to check that it can be normalized further on $V_m[1]\times \bigcup_{j=1}^{m-1}V_j$, since the addition of $d(\chi\mu)$ may change the previous normalization where the function $\chi$ is non-constant. 
\begin{figure}
\[ \hspace{25mm}\includegraphics[height=70mm]{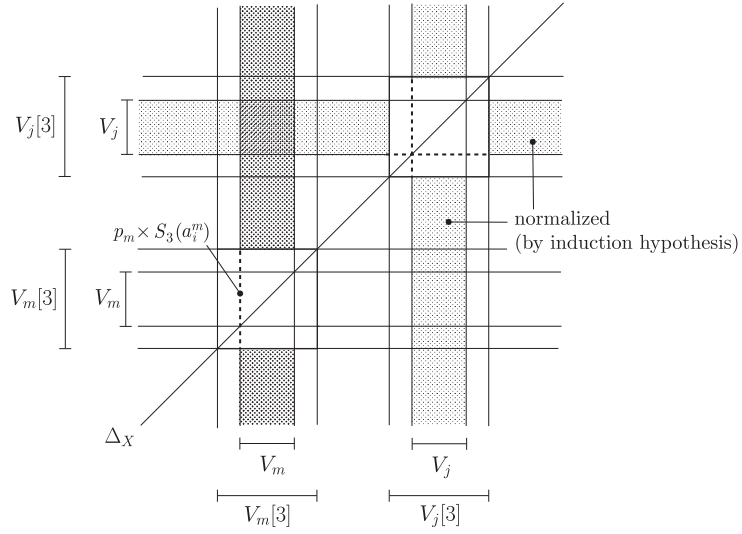} \]
\caption{Where $\omega_a'$ is normalized (projected on $X\times X$).}\label{fig:Vm-Vj-normalized}
\end{figure}

The assumptions on $\omega_0$ and $\omega_m$ imply that $\mu$ is closed on $V_m[1]\times V_j$ ($j<m$) and vanishes on $V_m[1]\times\partial X$. Moreover, by Lemmas~\ref{lem:Vm-1} and \ref{lem:Vm-2} below, we see that $\mu$ is exact on $V_m[1]\times V_j$ ($j<m$). Hence we may assume that $\mu=0$ on that part. Thus it remains to prove that we may assume moreover the conditions (3) and (4).%
\begin{figure}%
\[ \includegraphics[width=45mm]{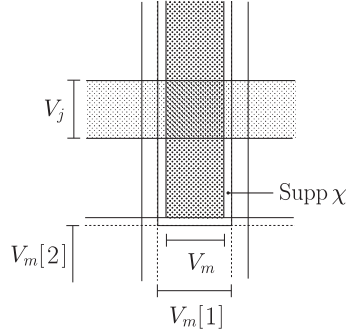} \]%
\caption{$\mathrm{Supp}\,\chi=V_m[1]\times (X-\mathring{V}_m[2])$.}\label{fig:supp-chi}%
\end{figure}%

Now we shall prove that there is a linear combination $\mu_c$ of $\pr_2^*\,\eta_{S_2(b_\ell^m)}$ and a linear combination $\mu_c'$ of $p_1^*\eta_{S_3(a_k^m)}\wedge p_2^*\eta_{S(b_\ell^m)}$ that vanish on $V_m[1]\times V_j$ for $j<m$ such that the new form $\omega_a'=\omega_a+d(\chi\mu_c)+d(\chi\mu_c')$ satisfies the following identities, which correspond to the former parts of the conditions (3) and (4), respectively.
\begin{align}
 &\int_{p^m\times S_3(a_\ell^m)}\omega_a'=0\quad(\mbox{for }\dim{a_\ell^m}=d-2),\label{eq:int-pxS}\\
 &\int_{b_k^m\times S_3(a_\ell^m)}\omega_a'=0\quad(\mbox{for }d=4,\,\dim{b_k^m}=\dim{a_\ell^m}=1).\label{eq:int-bxS}
\end{align}
We prove the existence of such $\mu_c$ and $\mu_c'$ by modifying the proof of the conditions (3) and (4) of Proposition~\ref{prop:normalization1} in a way that the induction works. Namely, let $r_\ell:=\int_{p^m\times S_3(a_\ell^m)}\omega_a$ and $\lambda_{k\ell}:=\int_{b_k^m\times S_3(a_\ell^m)}\omega_a$.
As in the proof of Proposition~\ref{prop:normalization1}, there exist unique linear combinations $\mu_c$ of $\pr_2^*\,\eta_{S_2(b_\ell^m)}$ and $\mu_c'$ of $p_1^*\eta_{S_3(a_k^m)}\wedge p_2^*\eta_{S(b_\ell^m)}$ (when $d=4$) such that $r_\ell=\int_{p^m\times a_\ell^m[3]}\mu_c$ for all $\ell$ with $\dim{a_\ell^m}=d-2$ ($\Leftrightarrow\,\deg\,\eta_{S_2(b_\ell^m)}=d-2$), and $\lambda_{k\ell}=\int_{b_k^m\times a_\ell^m[3]}\mu_c'$ for all $k,\ell$ with $\dim{b_k^m}=\dim{a_\ell^m}=1$ (when $d=4$). Then the form
\[\begin{split}
&\omega_a'=\omega_a+d(\chi\mu_c)+d(\chi\mu_c'), \mbox{ where }\\
&\mu_c=-\sum_k r_k(p_2^*\eta_{S_2(b_k^m)}),\quad
\mu_c'=-\sum_{k,\ell}\lambda_{k\ell}( p_1^*\eta_{S_3(a_k^m)}\wedge p_2^*\eta_{S(b_\ell^m)})
\end{split} \]
satisfies (\ref{eq:int-pxS}) and (\ref{eq:int-bxS}). In order that this modification does not affect the previous normalization, it suffices to prove that $\mu_c$ (resp. $\mu_c'$) does not have the term of $\pr_2^*\,\eta_{S_2(b_\ell^m)}$ (resp. $p_1^*\eta_{S_3(a_k^m)}\wedge p_2^*\eta_{S(b_\ell^m)}$) such that its support intersects $V_m[1]\times V_j$ for $j<m$. Under our assumption on the linking property of the handlebodies, this condition for the support of $\pr_2^*\,\eta_{S_2(b_\ell^m)}$ or $p_1^*\eta_{S_3(a_k^m)}\wedge p_2^*\eta_{S(b_\ell^m)}$ is equivalent to $S_2(b_\ell^m)\cap V_j\neq\emptyset$. 
By Lemma~\ref{lem:r=0} below, the condition $S_2(b_\ell^m)\cap V_j\neq\emptyset$ implies $r_\ell=0$ or $\lambda_{k\ell}=0$ (depending on $\dim{a_\ell^m}$), and it follows that $\mu_c$ (resp. $\mu_c'$) does not have the term of $\pr_2^*\,\eta_{S_2(b_\ell^m)}$ (resp. $p_1^*\eta_{S_3(a_k^m)}\wedge p_2^*\eta_{S(b_\ell^m)}$) for $b_\ell^m$ with $S_2(b_\ell^m)\cap V_j\neq\emptyset$.

The normalization on the symmetric part $(X-\mathring{V}_m[3])\times V_m$ can be done similarly and disjointly from the previous normalization, again by using the straightforward analogues of Lemmas~\ref{lem:Vm-1} and \ref{lem:Vm-2} for $V_j\times V_m[1]$ ($j<m$). 
\end{proof}

\begin{Lem}\label{lem:Vm-1}
Let $\mu$ be the $(d-2)$-form on $\bConf_2(S^d;\infty)$ in the proof of Proposition~\ref{prop:normalize2} such that $\mu=0$ on $\partial\bConf_2(S^d;\infty)$. For $j<m$ and for $\ell, \ell'$ such that $\dim{b_\ell^m}=\dim{b_{\ell'}^j}=d-2$, we have
\[ \int_{b_\ell^m\times p^j}\mu =0,\quad \int_{p^m\times b_{\ell'}^j}\mu=0. \] 
\end{Lem}
\begin{proof}
For the first identity, let $v^j_\infty\in \partial X$ be the endpoint of $\bgamma^j$ other than $p^j$. Since $\mu=0$ on $\partial \bConf_2(S^d;\infty)$, we have $\int_{b_\ell^m\times v^j_\infty}\mu=0$, and by the Stokes theorem,
\[ \int_{b_\ell^m\times p^j}\mu=(-1)^{d-1}\int_{\partial(b_\ell^m\times \bgamma^j)}\mu=(-1)^{d-1}\int_{b_\ell^m\times \bgamma^j}(\omega_m-\omega_0). \]
\[ \includegraphics[height=45mm]{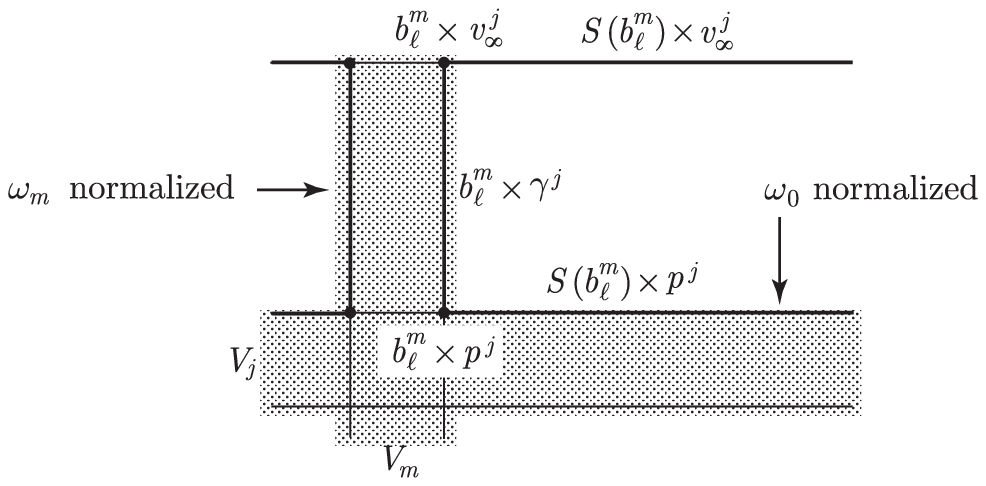} \]
Here, it follows from $b_\ell^m\times \bgamma^j\subset V_m\times (X-\mathring{V}_m[3])$ and the explicit formula for $\omega_m$ there (condition (1) of Proposition~\ref{prop:normalize2}) that $\int_{b_\ell^m\times \bgamma^j}\omega_m=0$, since $\bgamma^j$ is disjoint from $S(b_{\ell'}^m)$ for all $\ell'$, as assumed in \S\ref{ss:preliminaries}-(ii). 
Also,
\[ \int_{b_\ell^m\times\bgamma^j}\omega_0
=\pm\int_{S(b_\ell^m)\times\partial \bgamma^j}\omega_0
=\pm\int_{S(b_\ell^m)\times v^j_\infty}\omega_0\mp\int_{S(b_\ell^m)\times p^j}\omega_0
=\mp\int_{S(b_\ell^m)\times p^j}\omega_0,\]
where $\pm=(-1)^d$ and the first equality holds by $\partial(S(b_\ell^m)\times \bgamma^j)=b_\ell^m\times \bgamma^j+(-1)^{d-1}S(b_\ell^m)\times \partial\bgamma^j$ and $d\omega_0=0$, and the third equality holds by the explicit form of $\omega_0$ on $S(b_\ell^m)\times v^j_\infty\subset \partial \bConf_2(S^d;\infty)$. Then it suffices to prove that the last integral vanishes. 

If $S(b_\ell^m)\cap V_j=\emptyset$, the last integral vanishes by the explicit formula of $\omega_0$ on $(X-\mathring{V}_j[3])\times V_j$. If $S(b_\ell^m)\cap V_j\neq\emptyset$, the intersection of $S(b_\ell^m)$ with $V_j[3]$ is $\pm S_3(a_{\ell'}^j)$ for some $\ell'$ by the assumption \S\ref{ss:preliminaries}-(iii), as in the following picture.
\[ \includegraphics[height=30mm]{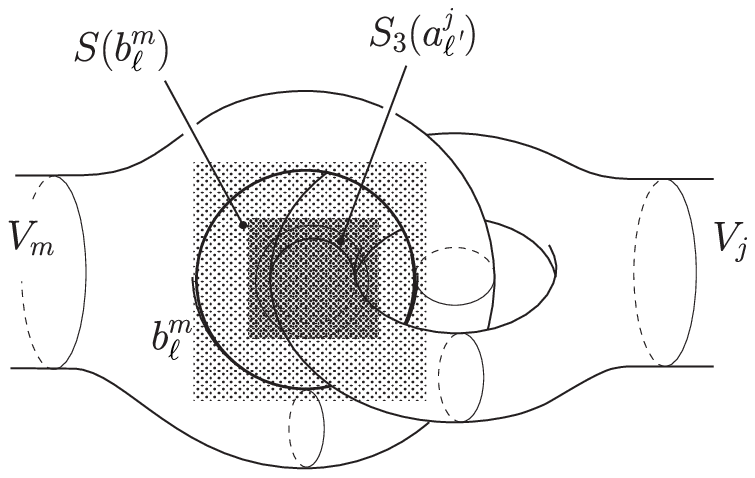} \]
Then we have
\[ \int_{S(b_\ell^m)\times p^j}\omega_0=\pm \int_{S_3(a_{\ell'}^j)\times p^j}\omega_0+\int_{(S(b_\ell^m)-\mathring{S}_3(a_{\ell'}^j))\times p^j}\omega_0, \]
where $S(b_\ell^m)-\mathring{S}_3(a_{\ell'}^j)$ is considered as a chain given by the submanifold $S(b_\ell^m)-\mathring{S}_3(a_{\ell'}^j)$ with orientation induced from $S(b_\ell^m)$, and the integral over $S_3(a_{\ell'}^j)\times p^j$ vanishes by the condition (3) of Proposition~\ref{prop:normalize2}. The integral over the remaining piece $(S(b_\ell^m)-\mathring{S}_3(a_{\ell'}^j))\times p^j$ vanishes by the explicit formula of $\omega_0$ on $(X-\mathring{V}_j[3])\times V_j$ and the assumption $S(b_\ell^m)\cap \overline{\gamma}^j=\emptyset$. This completes the proof of the first identity.

The second identity can be verified similarly, except the roles of $\omega_0$ and $\omega_m$ are exchanged. Since $\mu=0$ on $\partial \bConf_2(S^d;\infty)$, we have $\int_{v^m_\infty\times b_{\ell'}^j}\mu=0$ and 
\[ \int_{p^m\times b_{\ell'}^j}\mu=-\int_{\partial(\bgamma^m\times b_{\ell'}^j)}\mu=-\int_{\bgamma^m\times b_{\ell'}^j}(\omega_m-\omega_0). \]
Here, $\bgamma^m\times b_{\ell'}^j\subset (X-\mathring{V}_j[3])\times V_j$ and the explicit formula for $\omega_0$ there imply $\int_{\bgamma^m\times b_{\ell'}^j}\omega_0=0$, since $\bgamma^m$ is disjoint from $S(b_{\ell''}^j)$ for all $\ell''$, as assumed in \S\ref{ss:preliminaries}-(ii). Also, 
\[ \int_{\bgamma^m\times b_{\ell'}^j}\omega_m=\int_{\partial\bgamma^m\times S(b_{\ell'}^j)}\omega_m=\int_{v^m_\infty\times S(b_{\ell'}^j)}\omega_m- \int_{p^m\times S(b_{\ell'}^j)}\omega_m=- \int_{p^m\times S(b_{\ell'}^j)}\omega_m.\]
Again, we need only to consider the case $S(b_{\ell'}^j)\cap V_m\neq \emptyset$, in which case the integral on the right hand side vanishes by the condition (3) of Proposition~\ref{prop:normalize2} and by the explicit formula of $\omega_m$ on $V_m\times(X-\mathring{V}_m[3])$. 
\end{proof}

\begin{Lem}\label{lem:Vm-2}
Let $d=4$ and $\mu$ be the $2$-form on $\bConf_2(S^d;\infty)$ in the proof of Proposition~\ref{prop:normalize2} such that $\mu=0$ on $\partial\bConf_2(S^d;\infty)$. For $j<m$ and for $\ell, \ell'$ such that $\dim{b_\ell^m}=\dim{b_{\ell'}^j}=1$, we have
\[ \int_{b_\ell^m\times b_{\ell'}^j}\mu=0. \] 
\end{Lem}
\begin{proof}
The idea of the proof is similar to Lemma~\ref{lem:Vm-1}. We use the identity
\[
 \int_{b_\ell^m\times b_{\ell'}^j}\mu=-\int_{\partial(b_\ell^m\times S(b_{\ell'}^j))}\mu=-\int_{b_\ell^m\times S(b_{\ell'}^j)}(\omega_m-\omega_0) 
\]
given by the Stokes theorem. We have
\[ \int_{b_\ell^m\times S(b_{\ell'}^j)}\omega_0=\pm \int_{S(b_\ell^m)\times b_{\ell'}^j}\omega_0,\]
by $\partial(S(b_\ell^m)\times S(b_{\ell'}^j))=b_\ell^m\times S(b_{\ell'}^j)\pm S(b_\ell^m)\times b_{\ell'}^j$ and $d\omega_0=0$. If $S(b_\ell^m)\cap V_j=\emptyset$, the last integral vanishes by the explicit formula of $\omega_0$ on $(X-\mathring{V}_j[3])\times V_j$. If $S(b_\ell^m)\cap V_j\neq\emptyset$, the intersection of $S(b_\ell^m)$ with $V_j[3]$ is $\pm S_3(a_{\ell''}^j)$ for some $\ell''$ by the assumption \S\ref{ss:preliminaries}-(iii). Then we have
\[ \int_{S(b_\ell^m)\times b_{\ell'}^j}\omega_0
=\pm \int_{S_3(a_{\ell''}^j)\times b_{\ell'}^j}\omega_0+\int_{(S(b_\ell^m)-\mathring{S}_3(a_{\ell''}^j))\times b_{\ell'}^j}\omega_0, \]
where $\int_{S_3(a_{\ell''}^j)\times b_{\ell'}^j}\omega_0=0$ by the condition (4) of Proposition~\ref{prop:normalize2}. The integral over the remaining piece $(S(b_\ell^m)-\mathring{S}_3(a_{\ell''}^j))\times b_{\ell'}^j$ vanishes by the explicit formula of $\omega_0$ on $(X-\mathring{V}_j[3])\times V_j$ and the assumption $S(b_\ell^m)\cap \overline{\gamma}^j=\emptyset$. Thus we have $\int_{b_\ell^m\times S(b_{\ell'}^j)}\omega_0=0$. 

If $S(b_{\ell'}^j)\cap V_m=\emptyset$, then we have $b_\ell^m\times  S(b_{\ell'}^j)\subset V_m\times (X-\mathring{V}_m[3])$ and $\int_{b_\ell^m\times  S(b_{\ell'}^j)}\omega_m=0$ by the explicit formula of $\omega_m$ in Proposition~\ref{prop:normalization1} (1). If $S(b_{\ell'}^j)\cap V_m\neq\emptyset$, then the intersection of $S(b_{\ell'}^j)$ with $V_m[3]$ is $\pm S_3(a_k^m)$ for some $k$ by the assumption \S\ref{ss:preliminaries}-(iii). Thus we have 
\[ \int_{b_\ell^m\times S(b_{\ell'}^j)}\omega_m
=\pm\int_{b_\ell^m\times S_3(a_k^m)}\omega_m 
+\int_{b_\ell^m\times (S(b_{\ell'}^j)-\mathring{S}_3(a_k^m))}\omega_m
=\pm \int_{b_\ell^m\times S_3(a_k^m)}\omega_m, \]
where the second equality holds by $b_\ell^m\times (S(b_{\ell'}^j)-\mathring{S}_3(a_k^m))\subset V_m\times (X-\mathring{V}_m[3])$ and by the explicit formula of $\omega_m$ there. Moreover, the last integral vanishes by the condition (4) of Proposition~\ref{prop:normalization1}, and we have $\int_{b_\ell^m\times S(b_{\ell'}^j)}\omega_m=0$. This completes the proof.
\end{proof}

\begin{Lem}\label{lem:r=0}
Let $r_\ell$ and $\lambda_{k\ell}$ be as in the proof of Proposition~\ref{prop:normalize2}. If $S_2(b_\ell^m)\cap V_j\neq\emptyset$, then 
$r_\ell=0$ (when $\dim{a_\ell^m}=d-2$) and $\lambda_{k\ell}=0$ (when $d=4$ and $\dim{a_\ell^m}=1$).
\end{Lem}
\begin{proof}
Suppose $a_\ell^m$ is such that $S_2(b_\ell^m)\cap V_j\neq\emptyset$. By the assumption \S\ref{ss:preliminaries}-(iii), $S_3(a_\ell^m)\subset S(b_i^j)$ for some $i$. When $\dim{a_\ell^m}=d-2$, we have
\[ r_\ell=\int_{p^m\times S_3(a_\ell^m)}\omega_a
=\pm\int_{p^m\times S(b_i^j)}\omega_a-\int_{p^m\times(S(b_i^j)\cap(X-\mathring{V}_m[3]))}\omega_a. \]
We prove that the two terms on the right hand side both vanish.
\[ \includegraphics[height=47mm]{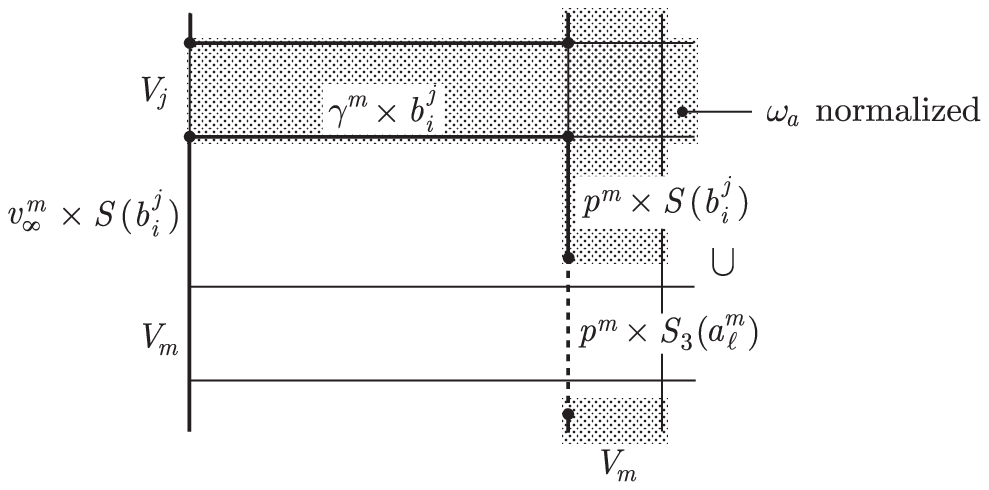} \]
For the first term, let $v_\infty^m\in \partial X$ be the other endpoint of $\bgamma^m$ than $p^m$. By $\partial(\bgamma^m\times S(b_i^j))=v^m_\infty\times S(b_i^j)-p^m\times S(b_i^j)-\bgamma^m\times b_i^j$, we have
\[ \int_{p^m\times S(b_i^j)}\omega_a=\int_{v^m_\infty\times S(b_i^j)}\omega_a-\int_{\bgamma^m\times b_i^j}\omega_a. \]
Since $v^m_\infty\times S(b_i^j)\subset \partial\bConf_2(S^d;\infty)$ and $\bgamma^m\times b_i^j\subset (X-\mathring{V}_j[3])\times V_j$, the integrals on the right hand side are both zero by the explicit formula of $\omega_a$ on $\partial\bConf_2(S^d;\infty)$ and $(X-\mathring{V}_j[3])\times V_j$. 
For the second term, since $p^m\times(S(b_i^j)\cap(X-\mathring{V}_m[3]))\subset V_m\times(X-\mathring{V}_m[3])$ and $S(b_i^j)$ is disjoint from $\bgamma^m$, we have
$\int_{p^m\times(S(b_i^j)\cap(X-\mathring{V}_m[3]))}\omega_a=0$
by the explicit formula of $\omega_a$ on $V_m\times(X-\mathring{V}_m[3])$. Hence we have $r_\ell=0$.

When $d=4$ and $\dim{a_\ell^m}=1$, we have
\[ \lambda_{k\ell}
=\int_{b_k^m\times S_3(a_\ell^m)}\omega_a
=\pm\int_{b_k^m\times S(b_i^j)}\omega_a
-\int_{b_k^m\times (S(b_i^j)\cap(X-\mathring{V}_m[3]))}\omega_a.
\]
The second term in the right hand side vanishes by
$b_k^m\times (S(b_i^j)\cap(X-\mathring{V}_m[3]))\subset V_m\times (X-\mathring{V}_m[3])$ and by the explicit formula of $\omega_a$ there. For the first term, we use the identity
\[ \int_{b_k^m\times S(b_i^j)}\omega_a
=\pm \int_{S(b_k^m)\times b_i^j}\omega_a \]
given by the Stokes theorem and $d\omega_a=0$. If $S(b_k^m)\cap V_j=\emptyset$, then $S(b_k^m)\times b_i^j\subset (X-\mathring{V}_j[3])\times V_j$ and the integral vanishes by the explicit formula of $\omega_a$ there. If $S(b_k^m)\cap V_j\neq \emptyset$, then the intersedction of $S(b_k^m)$ with $V_j[3]$ is $\pm S_3(a_{i'}^j)$ for some $i'$ by the assumption \S\ref{ss:preliminaries}-(iii). Then we have
\[ \int_{S(b_k^m)\times b_i^j}\omega_a
=\pm \int_{S_3(a_{i'}^j)\times b_i^j}\omega_a+\int_{(S(b_k^m)-\mathring{S}_3(a_{i'}^j))\times b_i^j}\omega_a. \]
The second term in the right hand side vanishes by $(S(b_k^m)-\mathring{S}_3(a_{i'}^j))\times b_i^j\subset (X-\mathring{V}_j[3])\times V_j$ and by the explicit formula of $\omega_a$ there. The first term vanishes too by the condition (4) of Proposition~\ref{prop:normalize2}. Hence we have $\int_{b_k^m\times S(b_i^j)}\omega_a=0$. This completes the proof.
\end{proof}

%%%%%%%%%%%%%%%%%%%%%%%%%%%%%%%
\subsection{Normalization of propagator in parametrized pieces}\label{ss:normalize-piece}

The normalization conditions of Proposition~\ref{prop:normalize2} for a single fiber allows us to extend the normalized propagator to most pieces $\Omega_{ij}^\Gamma$ in $E\bConf_2(\pi^\Gamma)$. We shall do this and complete the proof of Proposition~\ref{prop:localization} in five steps.

%%%%%%%%%%%%%%%
\subsubsection{Step 1: Normalization in a single fiber}

In the following, let $\omega_1$ be the normalized propagator on $\bConf_2(S^d;\infty)$ with respect to $V_1\cup\cdots\cup V_{2k}\subset \mathrm{Int}\,X$, as in Proposition~\ref{prop:normalize2}. We consider $\omega_1$ as a normalized propagator on the fiber over the basepoint of $B_\Gamma$.

%%%%%%%%%%%%%%%
\subsubsection{Step 2: The most ``degenerate'' entry $\Omega_{\infty\infty}^\Gamma$}

There is a bundle map 
\[ \xymatrix{
  \Omega_{\infty\infty}^\Gamma \ar[r]^-{\widetilde{p}_{\infty\infty}} \ar[d] & \bConf_2(V_\infty;\infty) \ar[d]\\
  B_\Gamma \ar[r]^-{p_{\infty\infty}} & \ast
} \]
which can be slightly enlarged to a map $\widetilde{p}_{\infty\infty}^{[2]}\colon \Omega_{\infty\infty}^\Gamma[2,2]\to \bConf_2(V_\infty[2];\infty)$, where $\Omega_{\infty\infty}^\Gamma[h,h']=p_{B\ell}^{-1}(\widetilde{V}_\infty'[h]\times_{B_\Gamma}\widetilde{V}_\infty'[h'])$. (See \S\ref{ss:preliminaries}(vi) for the definition of $\Omega_{ij}^\Gamma[h,h']$.)
We set
\begin{equation}\label{eq:step2}
 \omega_2=(\widetilde{p}_{\infty\infty}^{[2]})^*\omega_1\in\Omega_\dR^{d-1}(\Omega_{\infty\infty}^\Gamma[2,2]). 
\end{equation}

\begin{figure}
\[ \begin{array}{cc}
\includegraphics[height=33mm]{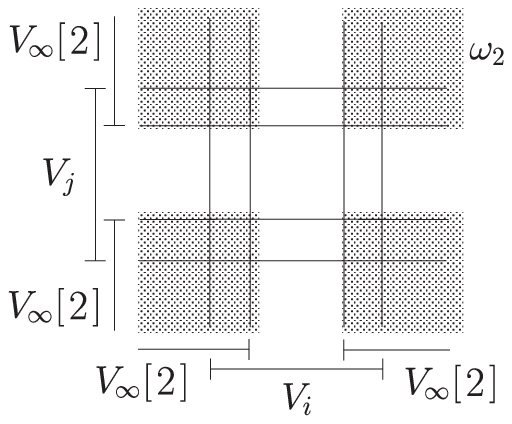}
&\includegraphics[height=33mm]{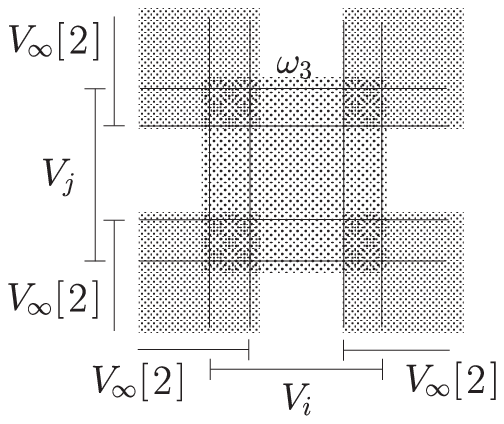}\\
\mbox{(a)} & \mbox{(b)}\\
&\\
\includegraphics[height=35mm]{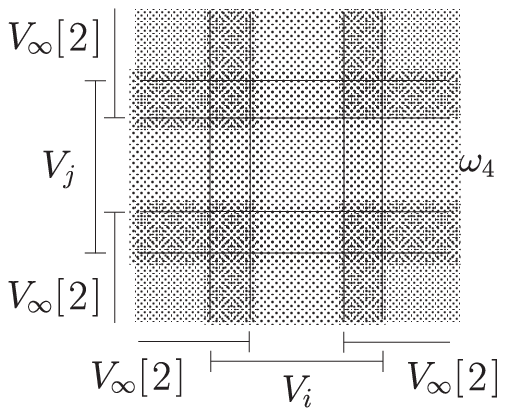}
&\includegraphics[height=35mm]{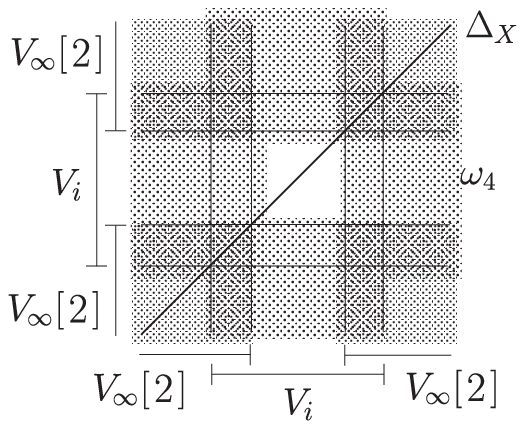}\\
\mbox{(c)} & \mbox{(c')}
\end{array} \]
\caption{The domains of (a) $\omega_2$, (b) $\omega_3$, (c),(c') $\omega_4$.}\label{eq:D-omega}
\end{figure}

%%%%%%%%%%%%%%%%%%%%%%%%%%%%%%%
\subsubsection{Step 3: Explicit form in ``generic'' entry $\Omega_{ij}^\Gamma$, $i\neq j$, $\{i,j\}\cap\{\infty\}=\emptyset$}\label{ss:bubble}

There is a bundle map
\[ \xymatrix{
  \Omega_{ij}^\Gamma \ar[r]^-{\widetilde{p}_{ij}} \ar[d] & \widetilde{V}_i\times \widetilde{V}_j \ar[d]\\
  B_\Gamma \ar[r]^-{p_{ij}} & K_i\times K_j
} \]
We define
\begin{equation}\label{eq:ij-explicit}
 \widetilde{\omega}_{ij}=\sum_{\ell, m} L_{\ell m}^{ij}\,\pr_1^*\,\eta_{S(\widetilde{a}_\ell^i)}\wedge \pr_2^*\,\eta_{S(\widetilde{a}_m^j)}, 
\end{equation}
which is a form on $\Omega_{ij}^\Gamma(\{i,j\})=\widetilde{V}_i\times \widetilde{V}_j$. It is immediate from the explicit formula that $\widetilde{\omega}_{ij}$ agrees with $\omega_2$ on 
\[ \begin{split}
  &\bigl\{(K_i\times K_j)\times (V_\infty[2]\times V_\infty[2])\bigr\}\cap (\widetilde{V}_i\times \widetilde{V}_j)\\
  &=(K_i\times K_j)\times \bigl\{([-2,0]\times\partial V_i)\times ([-2,0]\times \partial V_j)\bigr\},
\end{split}\]
where the identification is given by the partial trivialization of $\widetilde{V}_\lambda$ over the subbundle with fiber $[-2,0]\times\partial{V}_\lambda$. 
Hence $\widetilde{\omega}_{ij}$ can be glued to $\omega_2$. Namely, the two forms $(\widetilde{p}_{ij}^{[2]})^*\widetilde{\omega}_{ij}$ and $\omega_2$ agrees on $\Omega_{ij}^\Gamma\cap \Omega_{\infty\infty}^\Gamma[2,2]$, where $\widetilde{p}_{ij}^{[2]}\colon \Omega_{ij}^\Gamma[2,2]\to \widetilde{V}_i[2]\times\widetilde{V}_j[2]$ is the fiberwise extension of $\widetilde{p}_{ij}$, and they are glued together to give a new form on $\Omega_{ij}^\Gamma\cup \Omega_{\infty\infty}^\Gamma[2,2]$, by just extending the domain. Doing similar gluings for all $(i,j)$ such that $i\neq j$, $\{i,j\}\cap\{\infty\}=\emptyset$, we obtain a form $\omega_3$ defined on
\[ D_3=\Omega_{\infty\infty}^\Gamma[2,2]\cup\bigcup_{(i,j)}\Omega_{ij}^\Gamma. \]
Then the following identity holds.
\begin{equation}\label{eq:step3}
\begin{split}
\omega_3|_{\Omega_{ij}^\Gamma}=\widetilde{p}_{ij}^*\widetilde{\omega}_{ij}=\widetilde{p}_{ij}^*\omega_3|_{\Omega_{ij}^\Gamma(\{i,j\})}.
\end{split}
\end{equation}

%%%%%%%%%%%%%%%%%%%%%%%%%%%%%%%
\subsubsection{Step 4: Extension over $\Omega_{i\infty}^\Gamma\cup \Omega_{\infty i}^\Gamma$, $i\neq \infty$}\label{ss:base2}

There are bundle maps
\begin{equation}\label{eq:bun-i-infty}
 \xymatrix{
  \Omega_{i\infty}^\Gamma \ar[r]^-{\widetilde{p}_{i\infty}} \ar[d] &
  \widetilde{V}_i\times V_\infty \ar[d] &
  \Omega_{\infty i}^\Gamma \ar[r]^-{\widetilde{p}_{\infty i}} \ar[d] &
  V_\infty\times \widetilde{V}_i \ar[d]\\
  B_\Gamma \ar[r]^-{p_{i\infty}}  &
  K_i &
  B_\Gamma \ar[r]^-{p_{\infty i}}  &
  K_i 
} 
\end{equation}
Let $\Omega_{(i)\infty}^\Gamma$ and $\Omega_{\infty(i)}^\Gamma$ be the subspaces $\widetilde{p}_{i\infty}^{\,-1}(\widetilde{V}_i\times (V_\infty[2]\cap (X-\mathring{V}_i[3])))$ and $\widetilde{p}_{\infty i}^{\,-1}((V_\infty[2]\cap (X-\mathring{V}_i[3]))\times \widetilde{V}_i)$ of $\Omega_{i\infty}^\Gamma$ and $\Omega_{\infty i}^\Gamma$, respectively. 
We define the closed forms
\[ \begin{split}
  \widetilde{\omega}_{i\infty}=&\sum_{j,\ell}(-1)^{(\dim{a_j^i})d-1}\Lk(b_j^i,a_\ell^{i+})\,\pr_1^*\,\eta_{S(\widetilde{a}_j^i)}\wedge \pr_2^*\,\eta_{S_3(b_\ell^i)}+\pr_2^*\,\overline{\eta}_{\gamma^i[3]}\\
  &(\mbox{for $j,\ell$ such that }\dim{b_j^i}+\dim{a_\ell^i}=d-1),\\
  &\\
  \widetilde{\omega}_{\infty i}=&\sum_{j,\ell}(-1)^{(\dim{a_j^i})d-1}\Lk(a_j^{i+},b_\ell^{i})\,\pr_1^*\,\eta_{S_3(b_j^i)}\wedge \pr_2^*\,\eta_{S(\widetilde{a}_\ell^i)}+(-1)^d\pr_1^*\,\overline{\eta}_{\gamma^i[3]}\\
  &(\mbox{for $j,\ell$ such that }\dim{a_j^i}+\dim{b_\ell^i}=d-1)
\end{split} \]
on $\widetilde{V}_i\times (V_\infty[2]\cap (X-\mathring{V}_i[3]))$ and $(V_\infty[2]\cap (X-\mathring{V}_i[3]))\times \widetilde{V}_i$, respectively. These formulas are consistent with the formulas of Proposition~\ref{prop:normalize2} on the fiber over the basepoint of $K_i$.
It follows from the explicit formulas that on the overlap of these domains with $D_3(\{i\})$, which is the restriction of the bundle $D_3\to B_\Gamma$ on $B_\Gamma(\{i\})$ as in Notation~\ref{not:Omega_J}, the values of the overlapping forms agree. Hence $\widetilde{p}_{i\infty}^*\omega_3|_{D_3(\{i\})}$ and $\widetilde{p}_{\infty i}^*\omega_3|_{D_3(\{i\})}$ can be extended by $\widetilde{p}_{i\infty}^*\widetilde{\omega}_{i\infty}$ and $\widetilde{p}_{\infty i}^*\widetilde{\omega}_{\infty i}$ to a closed form $\omega_4$ on 
\[ D_4:=D_3\cup \bigcup_{i\neq\infty}\bigl(\Omega_{(i)\infty}^\Gamma\cup\Omega_{\infty(i)}^\Gamma\bigr). \]
Then we have the following identities.
\begin{equation}\label{eq:step4}
\begin{split}
&\omega_4|_{\Omega_{(i)\infty}^\Gamma}=\widetilde{p}_{i\infty}^*\widetilde{\omega}_{i\infty}=\widetilde{p}_{i\infty}^*\omega_4|_{\Omega_{(i)\infty}^\Gamma(\{i\})},\\
&\omega_4|_{\Omega_{\infty (i)}^\Gamma}=\widetilde{p}_{\infty i}^*\widetilde{\omega}_{\infty i}=\widetilde{p}_{\infty i}^*\omega_4|_{\Omega_{\infty (i)}^\Gamma(\{i\})},
\end{split}
\end{equation}
where $\Omega_{(i)\infty}^\Gamma(\{i\})$ and $\Omega_{\infty(i)}^\Gamma(\{i\})$ are the restrictions of the bundles $\Omega_{(i)\infty}^\Gamma\to B_\Gamma$ and $\Omega_{\infty(i)}^\Gamma(\{i\})\to B_\Gamma$ on $B_\Gamma(\{i\})$, respectively, as in Notation~\ref{not:Omega_J}.

%%%%%%%%%%%%%%%%%%%%%%%%%%%%%%%
\subsubsection{Step 5: Extension over $\Omega_{ii}^\Gamma[4,4]$, $i\neq \infty$}\label{ss:base3}

There is a bundle map
\[ \xymatrix{
  \Omega_{ii}^\Gamma[4,4] \ar[r]^-{\widetilde{p}_{ii}} \ar[d] &
  E\bConf_2(\pi(\alpha_i))[4,4] \ar[d] \\  
  B_\Gamma \ar[r]^-{p_{ii}} &
  K_i 
} \]
where $E\bConf_2(\pi(\alpha_i))[4,4]=B\ell_{\Delta_{\widetilde{V}_i[4]}}(\widetilde{V}_i[4]\times_{K_i}\widetilde{V}_i[4])=\Omega_{ii}^\Gamma[4,4](\{i\})$. Let $ST^v\Delta_{\widetilde{V}_i[4]}=p_{B\ell}^{-1}(\Delta_{\widetilde{V}_i[4]})$ denote the diagonal stratum in $E\bConf_2(\pi(\alpha_i))[4,4]$. By Lemma~\ref{lem:A}, the standard vertical framing on $K_i\times V_\infty$ extends over $\widetilde{V}_i$. Hence by pulling back the symmetric unit volume form on $S^{d-1}$ by the framing as in Lemma~\ref{lem:propagator2}, we obtain a closed $(d-1)$-form extension $\omega_{4,i}'$ of $\omega_4$ over $ST^v\Delta_{\widetilde{V}_i[4]}$. We will see in the next section (in Lemma~\ref{lem:extension-ii}) that $\omega_{4,i}'$ on 
\[ \bigl(D_4(\{i\})\cap E\bConf_2(\pi(\alpha_i))[4,4]\bigr)\cup ST^v\Delta_{\widetilde{V}_i[4]} \]
can be extended to a closed $(d-1)$-form on $E\bConf_2(\pi(\alpha_i))[4,4]$. We postpone the proof of this fact and assume this now. By pulling back this extension to $\Omega_{ii}^\Gamma[4,4]$ by $\widetilde{p}_{ii}$, we obtain a closed form $\omega_{5,i}$ on $\Omega_{ii}^\Gamma[4,4]$. By doing similar extensions on $\Omega_{ii}^\Gamma[4,4]$ for all $i\neq\infty$, we obtain a closed form $\omega_5$ defined on $E\bConf_2(\pi^\Gamma)$ that extends $\omega_4$, which satisfies the boundary condition for a propagator. By definition, we have the following identity.
\begin{equation}\label{eq:step5}
\omega_5|_{\Omega_{ii}^\Gamma[4,4]}=\widetilde{p}_{ii}^*\omega_5|_{\Omega_{ii}^\Gamma[4,4](\{i\})}.
\end{equation}

\begin{proof}[Proof of Proposition~\ref{prop:localization}]
Now the closed form $\omega_5$ on $E\bConf_2(\pi^\Gamma)$ is as desired in Proposition~\ref{prop:localization}. Namely, the condition (1) of Proposition~\ref{prop:localization} follows by (\ref{eq:step2}), (\ref{eq:step3}), (\ref{eq:step4}), (\ref{eq:step5}). Note that (\ref{eq:step4}) can be extended to the identity for $\Omega_{i\infty}^\Gamma\cup\Omega_{\infty i}^\Gamma$ by using (\ref{eq:step5}), both hold in subspaces of the same bundle $E\bConf_2(\pi^\Gamma)(\{i\})$ over $K_i$. The condition (2) of Proposition~\ref{prop:localization} follows from (\ref{eq:ij-explicit}). 
\end{proof}

%%%%%%%%%%%%%%%%%%%%%%%%%%%%%%%
%%%%%%%%%%%%%%%%%%%%%%%%%%%%%%%
\mysection{Extension over the final piece $\Omega_{ii}^\Gamma$, $i\neq\infty$}{s:localization-ii}

To simplify notation, we set $V=V_i[4]$, $\widetilde{V}=\widetilde{V}_i[4]$, and $E\bConf_2(\widetilde{V})=\Omega_{ii}^\Gamma[4,4](\{i\})$.
We shall prove the following lemma, whose proof was postponed. 
\begin{Lem}\label{lem:extension-ii}
The closed form $\omega_{4,i}'$ on $P=\bigl(D_4(\{i\})\cap E\bConf_2(\widetilde{V})\bigr)\cup ST^v\Delta_{\widetilde{V}}$ can be extended to a closed $(d-1)$-form on $E\bConf_2(\widetilde{V})$.
\end{Lem}
The problem is to show that the class of $\omega_{4,i}'$ in the cohomology $H^{d-1}(P)$ is mapped to zero by the connecting homomorphism
\[ H^{d-1}(P)\to H^d(E\bConf_2(\widetilde{V}),P). \]
It is easy to see that $P$ deformation retracts onto $\partial E\bConf_2(\widetilde{V})$ by shrinking the collar neighborhoods. Thus the problem is equivalent to the analogous one for the pair 
\[ (E\bConf_2(\widetilde{V}),\partial E\bConf_2(\widetilde{V})), \]
and we consider the latter. In this section, we will prove the above cohomological property of $\omega_{4,i}'$ by evaluating on some explicit $(d-1)$-cycle in $\partial E\bConf_2(\widetilde{V})$ by a higher dimensional analogue of Lescop's proof of \cite[Lemma~11.11]{Les3}.

%%%%%%%%%%%%%%%%%%%%%%%%%%%%%%%
\subsection{On the homology of $\bConf_2(V)$}\label{ss:prelim-ext-ii}

In this section, a {\it chain} is a piecewise smooth singular chain, namely, a linear combination of smooth maps from simplices. Since a manifold with corners admits a smooth triangulation, a linear combination of smooth maps from compact oriented manifolds with corners can be considered as a chain.

\begin{Lem}\label{lem:H(V)}
Let $d$ be an integer such that $d\geq 4$. 
Let $\Lambda_n=\langle[b_j\times b_\ell]\mid \dim{b_j}+\dim{b_\ell}=n\rangle$.
\begin{enumerate}
\item[\rm (i)]
$H_{d-2}(V^2)=\left\{\begin{array}{ll}
  \langle [b_j\times *], [*\times b_j]\mid \dim{b_j}=2\rangle\oplus \Lambda_2\quad(d=4),\\
  \langle [b_j\times *], [*\times b_j]\mid \dim{b_j}=d-2\rangle\quad(d>4),\\
\end{array}\right.$\\
$H_{d-1}(V^2)=\Lambda_{d-1}$, \\
$H_d(V^2)=\left\{\begin{array}{ll}
  \Lambda_4 & (d=4),\\
  0 & (\mbox{otherwise}),
\end{array}\right.$\\
$H_{d+1}(V^2)=\left\{\begin{array}{ll}
  \Lambda_6 & (d=5),\\
  0 & (\mbox{otherwise}).
\end{array}\right.$

\item[\rm (ii)] 
$H_{d-1}(\bConf_2(V))=H_{d-1}(V^2)\oplus \langle [ST(*)]\rangle$,\\ 
$H_d(\bConf_2(V))=H_d(V^2)\oplus \langle [ST(b_i)]\mid \dim{b_i}=1\rangle$,\\
$H_{2d-3}(\bConf_2(V))=\langle [ST(b_i)]\mid \dim{b_i}=d-2\rangle$,\\
$H_{i}(\bConf_2(V))=H_{i}(V^2)$ if $i\neq d-1,d,2d-3$, where $ST(\sigma)$ for a submanifold cycle $\sigma\subset V$ denotes $ST(V)|_{\sigma}=SN(\Delta_V)|_{\Delta_\sigma}$ (see \S\ref{ss:notations} (a)). 
\end{enumerate}
\end{Lem}
\begin{proof}
We replace for simplicity $V$ and $\bConf_2(V)$ with $\mathring{V}$ and $\Conf_2(\mathring{V})$, respectively, without changing their homotopy types (especially for the excision argument below).
The assertion (i) follows immediately from the K\"{u}nneth formula.
In the homology exact sequence for the pair
\[ \to H_{p+1}(\mathring{V}^2)
	\to H_{p+1}(\mathring{V}^2,\mathring{V}^2-\Delta_{\mathring{V}})
	\to H_{p}(\Conf_2(\mathring{V}))\to \]
we see that the map $H_{p+1}(\mathring{V}^2)\to H_{p+1}(\mathring{V}^2,\mathring{V}^2-\Delta_{\mathring{V}})$ is zero since the explicit basis $\{[*],[b_1],[b_2],[b_3]\}^{\otimes 2}$ of $H_*(V^2)$ can be given by cycles in $\mathring{V}^2-\Delta_{\mathring{V}}$. Hence we have the isomorphism 
\[ H_{p}(\Conf_2(\mathring{V}))\cong H_{p+1}(\mathring{V}^2,\mathring{V}^2-\Delta_{\mathring{V}})\oplus H_p(\mathring{V}^2). \]
We have $H_i(\mathring{V}^2,\mathring{V}^2-\Delta_{\mathring{V}})=H_d(D^d,\partial D^d)\otimes H_{i-d}(\Delta_{\mathring{V}})\,(\cong H_{i-d}(\mathring{V}))$ by excision, and 
\[H_{d+r}(\mathring{V}^2,\mathring{V}^2-\Delta_{\mathring{V}})
=\left\{\begin{array}{ll}
H_r(V)& (r\geq 0),\\ 
0 & (r<0).
\end{array}\right.
\]
The assertion (ii) follows from this.
\end{proof}

Let $a$ be $a_j[4]\subset \partial V$ that is $(d-2)$-dimensional. Let $\Sigma=S_4(a_j)$. Suppose that $V$ is of type I. We assume the following for $\Sigma$.
\begin{Assum} 
\begin{enumerate}
\item If $V$ is the fiber over the non-basepoint $1\in K_i$, we assume $\Sigma$ is given by a normally framed embedding from $S^1\times S^{d-2}-\mbox{(open disk)}$. This is possible since $\Sigma$ is a Seifert surface of one component in the Borromean rings that is disjoint from other components, as in Lemma~\ref{lem:S(a)}. 

\item If $V$ is the fiber over the basepoint $-1\in K_i$, we assume that $\Sigma$ is either $D^{d-1}$ or $S^1\times S^{d-2}-\mbox{(open disk)}$, the connect sum of a small $S^1\times S^{d-2}$ to a $(d-1)$-disk.
\end{enumerate}
\end{Assum}
In any case, $\Sigma=D^{d-1}\#(S^1\times S^{d-2})^{\# g}$ for some $g\geq 0$. 
Let $c_1,c_2,\ldots, c_{2g}$ be the cycles of $\Sigma$ that form a basis of the reduced homology of $\Sigma$ over $\Z$. Let $c_1^*, c_2^*,\ldots, c_{2g}^*$ be the cycles of $\Sigma$ that represent the basis of $\tilde{H}_*(\Sigma;\Z)$ dual to $c_1,c_2,\ldots,c_{2g}$ with respect to the intersection form on $\Sigma$, so that $c_i\cdot c_j^*=\delta_{ij}$. Let $c_i^+,c_j^{*+}$ be the cycles in $V$ obtained by slightly shifting $c_i,c_j^*$ along positive normal vectors on $\Sigma$. The following lemma will be used in Lemma~\ref{lem:F(a)} to study a part of the homology class of the diagonal in $\Sigma\times\Sigma^+$.

\begin{Lem}\label{lem:cxc}
\begin{enumerate}
\item[\rm (a)] The $(d-1)$-cycle $\sum_k c_k\times c_k^*$ is homologous to 
\[ \sum_{j,\ell}\lambda_{j\ell}^V\,b_j\times b_\ell \quad \mbox{in $V^2$ for some $\lambda_{j\ell}^V\in\R$}, \]
where the sum is over $j,\ell$ such that $\dim{b_j}+\dim{b_\ell}=d-1$.
\item[\rm (b)] The $(d-1)$-cycle $\sum_k c_k\times c_k^{*+}$ is homologous to \\
\[ \sum_{j,\ell}\lambda_{j\ell}^V\,b_j\times b_\ell+\delta(\Sigma)ST(*) \quad\mbox{in $\bConf_2(V)$} \]
for some constant $\delta(\Sigma)$ depending on the submanifold $\Sigma\subset V$, where the sum is over $j,\ell$ such that $\dim{b_j}+\dim{b_\ell}=d-1$.
\end{enumerate}
\end{Lem}
\begin{proof}
The assertion (a) follows from Lemma~\ref{lem:H(V)}(i). For (b), one can show by using the computation of $H_{d-1}(\bConf_2(V))$ in Lemma~\ref{lem:H(V)}(ii) that the component of $b_j\times b_\ell$ in the homology class of $\sum_k c_k\times c_k^{*+}$ agrees with that of (a). The coefficient $\delta(\Sigma)$ of $ST(*)$ in the homology class is
$\sum_k\Lk(c_k,c_k^{*+})$.
\end{proof}

\begin{Rem}
If we chose $\Sigma$ to be a $(d-1)$-disk, then the coefficient $\delta(\Sigma)$ of $ST(*)$ of Lemma~\ref{lem:cxc}(b) was zero.
In \cite[Lemma~11.12]{Les3}, an explicit formula for the coefficient $\lambda_{j\ell}^V$ is given, which is not necessary for our purpose. 
\end{Rem}

%%%%%%%%%%%%%%%%%%%%%%%%%%%%%%%%%%%%%%%%%%%
\subsection{Extension over type I handlebody}\label{ss:ext-type-I}

We consider an analogue of Lescop's chain $F^2(a)$ of \cite[Lemma~11.13]{Les3}. We fix some notations to define the analogous chain. 
Recall that we have put $V=V_i[4]$, $V[h]=V_j[h]$ and chosen $a\subset \partial V$ that is $(d-2)$-dimensional in \S\ref{ss:prelim-ext-ii}.
\begin{enumerate}
\item We identify a small tubular neighborhood of $a$ in $\partial V$ with $a\times [-1,1]$ so that $a\times\{0\}=a$.

\item Let $\Sigma^+=(\Sigma\cap V[-1])\cup \{(5t-1,a(v),t)\mid v\in S^{d-2},\,t\in[0,1]\}$, where $(5t-1,a(v),t)\in [-4,4]\times (a\times [-1,1])$. We will also write $\Sigma^+_V=\Sigma^+$ or $\Sigma_V=\Sigma$ to emphasize that $\Sigma^+$ or $\Sigma$ is considered in a particular $V$ when $V$ is a single fiber in a family of handlebodies. Recall that we assumed that $\Sigma\cap ([-4,4]\times \partial V[0])=[-4,4]\times a[0]$ (\S\ref{ss:preliminaries}(iv)). 
\[ \includegraphics[height=25mm]{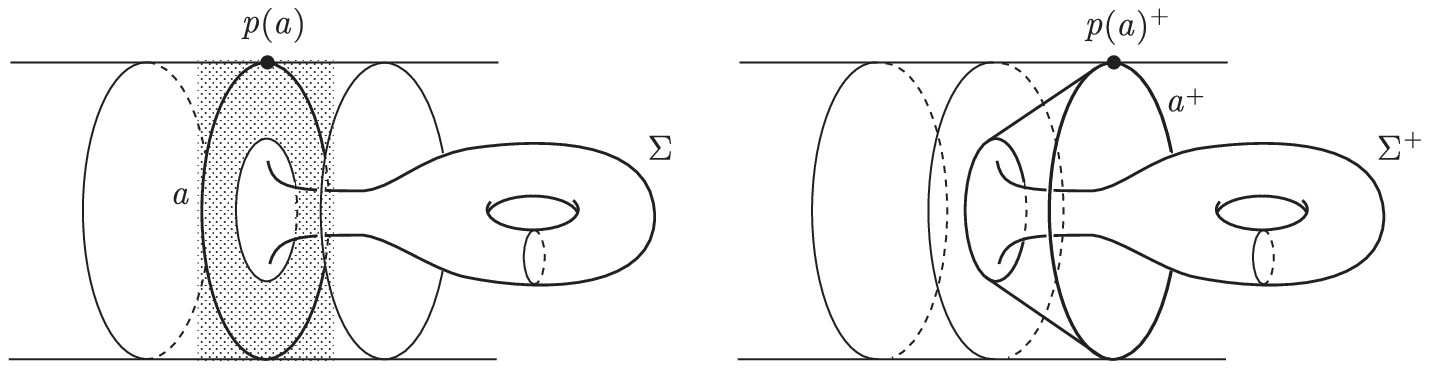} \]

\item By $S^{d-2}=([0,1]\times S^{d-3})/(\{0,1\}\times S^{d-3}\cup [0,1]\times \{\infty\})$ (reduced suspension of $S^{d-3}$), we equip $a$ with coordinates from $[0,1]\times S^{d-3}$. Let $p(a)$ be the basepoint of $a$ that corresponds to $\infty\in S^{d-2}$, the basepoint for the reduced suspension. Let $p(a)^+=(p(a),1)\in a\times [-1,1]\subset \partial V$.

\item Let $\mathrm{diag}(\nu)\Sigma$ be the chain given by the section of $ST(V)|_{\Sigma}$ by the unit normal vector field $\nu$ on $\Sigma$ compatible with the coorientation of the codimension 1 submanifold $\Sigma$ of $V$. The restriction $\nu_\Sigma:=\nu|_{\Sigma}\colon \Sigma\to ST V$ gives a submanifold chain $\mathrm{diag}(\nu)\Sigma$ of $ST\Delta_{V}\subset \partial\bConf_2(V)$. We will also write $\mathrm{diag}(\nu_\Sigma)\Sigma$ to emphasize the choice of $\Sigma$.

\item Let $T(a)\colon S^{d-3}\times T\to (a\times \{0\})\times (a\times\{1\})$ be the $(d-1)$-chain defined for $(v';y,z)\in S^{d-3}\times T$, where $T=\{(y,z)\in[0,1]^2\mid y\geq z\}$, by
\[ T(a)(v';y,z)=((a(y,v'),0),(a(z,v'),1)). \]
To make this into a chain, we orient $T(a)$ by the one induced from $\partial y\wedge \partial z\wedge o(S^{d-3})$, where $\partial y\wedge o(S^{d-3})=o(a)$. 

\item Let $A(a)$ be the closure of $\{((a(v),0),(a(v),t))\mid t\in(0,1],\,v\in [0,1]\times S^{d-3}\}$ in $\bConf_2(X)$, which is a compact $(d-1)$-submanifold with boundary and is diffeomorphic to $S^{d-2}\times [0,1]$. We orient $A(a)$ by the one induced from $o((0,1])\wedge o(a)$.
\end{enumerate}

We assume the following without loss of generality.
\begin{Assum}\label{assum:nu}
\begin{enumerate}
\item The unit normal vector field $\nu$ on $\Sigma$ is such that its restriction to $[-1,4]\times a$ is included in $T(\partial V)$. 
\item Let $\tau_V$ be the framing on $V$ as in Corollary~\ref{cor:framing-extend} and let $p(\tau_V)\colon ST(V)|_{\Sigma}\to S^{d-1}$ be the composition
$ST(V)|_{\Sigma}\stackrel{\tau_V}{\longrightarrow} \Sigma\times S^{d-1} \stackrel{\mathrm{pr}}{\longrightarrow} S^{d-1}$. We assume that the restriction of $p(\tau_V)\circ \nu$ to $[-1,4]\times a$ is a constant map.
\end{enumerate}
\end{Assum}
Thanks to Assumption~\ref{assum:nu} (2), the mapping degree $\deg\,(p(\tau_V)\circ \nu)$ of $p(\tau_V)\circ \nu$ makes sense.

\begin{Lem}[Type I]\label{lem:F(a)} The $(d-1)$-chain
\[ \begin{split}
	F_V^{d-1}(a)=& \,\mathrm{diag}(\nu)\Sigma_V-p(a)\times \Sigma^+_V -\Sigma_V\times p(a)^++T(a)+A(a) \\
	& -\Bigl\{\sum_{j,\ell}\lambda_{j\ell}^V\,b_j\times b_\ell+\delta(\Sigma_V)\,ST(*)\Bigr\}
\end{split} \]
in $\partial\bConf_2(V)$ is a cycle and is null-homologous in $\bConf_2(V)$.
\end{Lem}
\begin{proof}
Let $C_{*,\geq}'(\Sigma,\Sigma^+)$ denote the first line of the formula of $F_V^{d-1}(a)$, which is obtained from an analogue of $C_{*,\geq}(\Sigma,\Sigma^+)$ in \cite[Lemma~8.11]{Les3} by homotopy. Namely, if we let
\[ \begin{split}
	&a\times_{*,\geq}a^+= 
	\{(a(v',y),a(v',z)^+)\mid v'\in S^{d-3},\,y,z\in [-1,1],\,y\geq z\},\\
	&\mathrm{diag}(\Sigma\times \Sigma^+)=\{(x,x^+)\mid x\in\Sigma\},
\end{split}\]
where the superscript $+$ denotes the parallel copy in $\Sigma^+$, and orient $a\times_{*,\geq}a^+$ by $\partial y\wedge \partial z\wedge o(\Delta_{S^{d-3}})$, then the chain
\[ C_{*,\geq}(\Sigma,\Sigma^+)
=\mathrm{diag}(\Sigma\times\Sigma^+)-*\times \Sigma^+-\Sigma\times *^+ +a\times_{*,\geq} a^+ \]
of $\Sigma\times \Sigma^+$ is a $(d-1)$-cycle since 
\begin{align}
&\partial(a\times_{*,\geq} a^+)=-\mathrm{diag}(a\times a^+)+*\times a^+ + a\times *^+,\label{eq:diag(axa)}\\
&\partial\bigl(\mathrm{diag}(\Sigma\times\Sigma^+)-*\times \Sigma^+-\Sigma\times *^+\bigr)=\mathrm{diag}(a\times a^+)-*\times a^+ - a\times *^+,\label{eq:diag(SxS)}
\end{align}
where $\mathrm{diag}(a\times a^+)=\mathrm{diag}(\Sigma\times\Sigma^+)\cap (a\times a^+)$. The following holds in $H_{d-1}(\Sigma\times \Sigma^+;\Z)$.
\begin{equation}\label{eq:[C]}
 [C_{*,\geq}(\Sigma,\Sigma^+)]=\sum_k[c_k\times c_k^{*+}] 
\end{equation}
This identity can be proved by considering the closed manifold $S$ obtained from $\Sigma$ by gluing a $(d-1)$-disk $D$ along their boundary. It can be shown that
\[ [\mathrm{diag}(S\times S^+)]=[S\times *^+] + [*\times S^+] +\sum_k[c_k\times c_k^{*+}] \]
holds in $H_{d-1}(S\times S^+;\Z)$ (Proposition~\ref{prop:homol-diag}). We may define the cycle $C_{*,\geq}(-D,-D^+)$ analogously to $C_{*,\geq}(\Sigma,\Sigma^+)$ by replacing $\Sigma$ with $-D$ in the definition of $C_{*,\geq}(\Sigma,\Sigma^+)$. Then we have
\[  [C_{*,\geq}(\Sigma,\Sigma^+)]+[C_{*,\geq}(-D,-D^+)]=\sum_k[c_k\times c_k^{*+}] \]
in $H_{d-1}(S\times S^+;\Z)$, 
and that $[C_{*,\geq}(-D,-D^+)]=0$ in $H_{d-1}(D\times D^+;\Z)=0$. 

Now $\Sigma\times\Sigma^+$ can be considered as embedded in $\bConf_2(V)$ by considering the points on $\Sigma^+$ in $\mathrm{diag}(\Sigma[-1]\times \Sigma^+[-1])$ as lying on $\nu(\Sigma)$ in $SN(\Delta_V)$. In this way, we may identify $C_{*,\geq}'(\Sigma,\Sigma^+)$ with $C_{*,\geq}(\Sigma,\Sigma^+)$ up to boundaries, where $\mathrm{diag}(\Sigma\times\Sigma^+)+a\times_{*,\geq} a^+$ corresponds to $\mathrm{diag}(\nu)\Sigma+T(a)+A(a)$. Note that the boundaries of the three chains $\mathrm{diag}(\nu)\Sigma,T(a),A(a)$ cancel at their common boundaries since
\[ \begin{split}
  &\partial T(a)=-\mathrm{diag}(a\times a^+)+p(a)\times a^+ + a\times p(a)^+,\\
  &\partial A(a)=\mathrm{diag}(a\times a^+)-\mathrm{diag}(\nu)a,\\
  &\partial \mathrm{diag}(\nu)\Sigma=\mathrm{diag}(\nu)a,
\end{split} \]
where $\mathrm{diag}(\nu)a$ is defined by replacing $\Sigma$ by $a$ in the definition of $\mathrm{diag}(\nu)\Sigma$. 
The identity (\ref{eq:[C]}) also holds for $C_{*,\geq}'(\Sigma,\Sigma^+)$ in $H_*(\bConf_2(V);\Z)$. Then the result follows from Lemma~\ref{lem:cxc}.

We need to check that the signs of the right hand side of (\ref{eq:diag(axa)}) are correct. Suppose that $o(\Delta_{S^{d-3}})$ at a point $(v,v)$ is given by $\bigwedge_{i=1}^{d-3}(e_i+e_i')$, where $\{e_i\}$, $\{e_i'\}$ are copies of a basis of $T_v(S^{d-3})$. Then the orientation of $a\times_{*,\geq}a^+$ at $(y,v)\times(z,v)$ is $\partial y\wedge \partial z\wedge \bigwedge_{i=1}^{d-3}(e_i+e_i')$. The outward normal vectors at $\mathrm{diag}(a\times a^+)$, $a\times *^+$, $*\times a^+$ are $\partial z-\partial y$, $-\partial z$, $\partial y$, respectively. Hence the induced orientation on these parts are as follows:
\[ \begin{split}
  & i(\partial z-\partial y)\,\partial y\wedge \partial z\wedge \bigwedge_i(e_i+e_i')=-(\partial y+\partial z)\wedge \bigwedge_i(e_i+e_i'),\\
  & i(-\partial z)\,\partial y\wedge \partial z\wedge \bigwedge_i(e_i+e_i')=\partial y\wedge \bigwedge_i(e_i+e_i')\to \partial y\wedge \bigwedge_i e_i\quad(\mbox{$(z,v)\to *^+$}),\\
  & i(\partial y)\,\partial y\wedge \partial z\wedge \bigwedge_i(e_i+e_i')=\partial z\wedge \bigwedge_i(e_i+e_i')\to \partial z\wedge \bigwedge_i e_i'\quad(\mbox{$(y,v)\to *$}),
\end{split} \]
where $i(\cdot)$ is the interior multiplication defined by $i(w)u=\langle u,w\rangle$ for the inner product on $T_{(y,v)\times(z,v)}(a\times a^+)$ such that $\partial y, \partial z, e_i, e_i'$ forms an orthonormal basis. The results agree with $-o(\mathrm{diag}(a\times a^+))$, $o(a\times *^+)$, $o(*\times a^+)$, respectively. Hence the signs of the right hand side of (\ref{eq:diag(axa)}) are correct.
\end{proof}

When $\widetilde{V}$ is of type I, we write $\widetilde{V}={V}'\cup (-V)$. By Lemma~\ref{lem:F(a)}, there exist $d$-chains $G_{V'}^d(a_1),G_{V'}^d(a_2)$ of $\bConf_2(V')$ with coefficients in $\Z$ such that $\partial G_{V'}^d(a_i)=F_{V'}^{d-1}(a_i)$ ($i=1,2$). 
\begin{Lem}[Type I]\label{lem:basis_G_V}
$H_d(\bConf_2(V'),\partial \bConf_2(V'))$ has the following basis. 
\[\begin{array}{ll}
\{ [G_{V'}^{4}(a_1)],\, [G_{V'}^{4}(a_2)],\, [S_4(a_3)\times S_4(a_3)^+]\}
&(\mbox{if $d=4$}),\\
\{ [G_{V'}^d(a_1)],\, [G_{V'}^d(a_2)]\}
&(\mbox{if $d>4$}),
\end{array}
\]
where $S_4(a_3)^+$ is a parallel copy of $S_4(a_3)$.
\end{Lem}
\begin{proof}
By Lemma~\ref{lem:H(V)} (ii), $H_d(\bConf_2(V'))$ has the following basis:
\[\begin{array}{ll}
\{ [ST(b_1[4])],[ST(b_2[4])],\, [b_3\times b_3^+]\}
&(\mbox{if $d=4$}),\\
\{ [ST(b_1[4])],[ST(b_2[4])]\}
&(\mbox{if $d>4$}).
\end{array}
\]
Then the result follows by Poincar\'{e}--Lefschetz duality (see Lemma~\ref{lem:equiv-C2}) and the following intersections:
\[ \begin{split}
  [G_{V'}^d(a_i)]\cdot [ST(b_j[4])] &= [F_{V'}^{d-1}(a_i)]\cdot [ST(b_j[4])]\\
  &=[\mathrm{diag}(\nu)S_4(a_i)]\cdot [ST(b_j[4])]=\pm \delta_{ij}\quad (1\leq i,j\leq 2),\\
  [G_{V'}^d(a_i)]\cdot [b_3\times b_3^+]&=[F_{V'}^{d-1}(a_i)]\cdot [b_3\times b_3^+]=0\quad (\mbox{if }d=4),\\
  [S_4(a_3)\times S_4(a_3)^+]\cdot [ST(b_j[4])] &= 0\quad (\mbox{if }d=4,\, 1\leq j\leq 2),\\
  [S_4(a_3)\times S_4(a_3)^+]\cdot [b_3\times b_3^+] &= \pm 1\quad (\mbox{if }d=4),\\
\end{split} \] 
where $\cdot$ is the intersection pairing between homologies.
\end{proof}

\begin{Lem}[Type I]\label{lem:ext-propagator-I}
For the propagator $\omega_{4,i}'$ of Lemma~\ref{lem:extension-ii}, the closed form
\[ \omega_\partial=\omega_{4,i}'|_{\partial \bConf_2({V}')} \]
on $\partial \bConf_2({V}')$ extends to a closed form on $\bConf_2({V}')$.
\end{Lem}
\begin{proof}
We consider the following commutative diagram.
\[ \xymatrix{
  H^d(\bConf_2(V'))  \ar[r]^-{\cong} & H_d(\bConf_2(V')) \ar[d]^-{0} \ar[l]\\
  \delta([\omega_\partial])\in H^d(\bConf_2(V'),\partial\bConf_2(V')) \ar[u]^{0} \ar[r]^-{\cong} & H_d(\bConf_2(V'),\partial\bConf_2(V')) \ar[d] \ar[l]\\
  [\omega_\partial]\in H^{d-1}(\partial \bConf_2(V')) \ar[u]^-{\delta} \ar[r]^-{\cong} & H_{d-1}(\partial\bConf_2(V')) \ar[d] \ar[l]\\
  H^{d-1}(\bConf_2(V')) \ar[u]^-{r} \ar[r]^-{\cong} & H_{d-1}(\bConf_2(V')) \ar[l]
} \]
where the horizontal isomorphisms are given by the evaluation pairing. To prove that $[\omega_\partial]$ is in the image of the restriction induced map $r$, we prove $\delta([\omega_\partial])=0$. Here, the natural map
$H_d(\bConf_2(V'))\to H_d(\bConf_2(V'),\partial\bConf_2(V'))$
is zero since by Lemma~\ref{lem:H(V)}, we have $H_d(\bConf_2(V'))=H_d({V'}^2)\oplus \langle [ST(b_i)]\rangle$, where $H_d({V'}^2)$ is $\Lambda_4$ or 0 and $\dim{b_i}=1$, and all the generators are mapped to zero in $H_d(\bConf_2(V'),\partial\bConf_2(V'))$. To prove $\delta([\omega_\partial])=0$, it suffices to show the vanishing of the evaluation of $\delta([\omega_\partial])$ at the basis of $H_d(\bConf_2(V'),\partial \bConf_2(V'))$ in Lemma~\ref{lem:basis_G_V}. 

The class $\delta[\omega_\partial]$ can be represented by $d\,\widetilde{\omega}_\partial$, where $\widetilde{\omega}_\partial$ is an extension of $\omega_\partial$ over $\bConf_2(V)$ as a smooth $(d-1)$-form. 
Since 
\[ \begin{split}
  &\int_{G_{V'}^d(a_i)}d\,\widetilde{\omega}_\partial=\int_{F_{V'}^{d-1}(a_i)}\omega_\partial\quad (i=1,2),\\
  &\int_{S_4(a_3)\times S_4(a_3)^+}d\,\widetilde{\omega}_\partial=\int_{\partial(S_4(a_3)\times S_4(a_3)^+)}\omega_\partial\quad(\mbox{if $d=4$})
\end{split} \]
by the Stokes theorem, it suffices to check that the right hand sides vanish. By Lemma~\ref{lem:int_D} below, we have
\begin{equation}\label{eq:int-F}
 \begin{split}
   &\int_{F_{V'}^{d-1}(a_i)}\omega_\partial=\int_{F_{V}^{d-1}(a_i)}\omega_1\quad (i=1,2), \\
   &\int_{\partial(S_4(a_3)\times S_4(a_3)^+)}\omega_\partial=\int_{\partial(S_4(a_3)\times S_4(a_3)^+)}\omega_1,
\end{split}
\end{equation}
where $\omega_1$ is a form as in Proposition~\ref{prop:normalize2}. The right hand sides of (\ref{eq:int-F}) vanish since $F_V^{d-1}(a_i)$ and $\partial(S_4(a_3)\times S_4(a_3)^+)$ are null-homologous in $\bConf_2(V)$ by Lemma~\ref{lem:F(a)} and $\omega_1$ is defined there. Hence the left hand side of (\ref{eq:int-F}) vanishes, too. 
\end{proof}

We give some lemmas to prove Lemma~\ref{lem:int_D}.

\begin{Lem}\label{lem:int-pt-Sigma}
 Let $(V,\Sigma)$ be as above, let $\omega_1$ be a propagator normalized as in Proposition~\ref{prop:normalize2}, and let $\omega_\partial$ is the form of Lemma~\ref{lem:ext-propagator-I}. Then we have
\[ \int_{p(a)\times\Sigma_{V'}^+}\omega_\partial
=\int_{p(a)\times\Sigma_V^+}\omega_1\mbox{ and }
\int_{\Sigma_{V'}\times p(a)^+}\omega_\partial
=\int_{\Sigma_V\times p(a)^+}\omega_1. \]
\end{Lem}
\begin{proof}
We see that
\begin{equation}\label{eq:int-pt-Sigma}
 \int_{p(a)\times\Sigma_{V'}^+[-1]}\omega_\partial
=\int_{p(a)\times\Sigma_V^+[-1]}\omega_1=0
\end{equation}
since $p(a)\times \Sigma_{V'}^+[-1]\subset (X-\mathring{V}'[3])\times V'[0]$ and $\Sigma_{V'}[-1]\times p(a)^+\subset V'[0]\times (X-\mathring{V}'[3])$, and we have explicit formula for $\omega_\partial$ there. Note that we are assuming $V'=V_i'[4]$ and $a=\{4\}\times a_j^i$, but we consider $V'[3]$, $\Sigma_{V'}^+[-1]$ etc. denotes $V_i'[3]$, $S(a_j^i[-1])^+$ etc. By the same reason, the second integral of (\ref{eq:int-pt-Sigma}) vanishes. We have similar identities for the integrals over $\Sigma_{V'}[-1]\times p(a)^+$ and $\Sigma_V[-1]\times p(a)^+$.

Also,  we have
\[ \int_{p(a)\times(\Sigma_{V'}^+-\mathring{\Sigma}_{V'}^+[-1])}\omega_\partial
=\int_{p(a)\times(\Sigma_V^+-\mathring{\Sigma}_V^+[-1])}\omega_1 \]
since the domains are both included in the common subspace $p_{B\ell}^{-1}(([-1,4]\times \partial V')^2)=p_{B\ell}^{-1}(([-1,4]\times \partial V)^2)$, where the two forms $\omega_\partial$ and $\omega_1$ agree. We have similar identities for the integrals over $(\Sigma_{V'}-\mathring{\Sigma}_{V'}[-1])\times p(a)^+$ and $(\Sigma_{V}-\mathring{\Sigma}_{V}[-1])\times p(a)^+$. This completes the proof.
\end{proof}

\begin{Lem}\label{lem:int-T-A}
 Let $(V,\Sigma)$ be as above, let $\omega_1$ be a propagator normalized as in Proposition~\ref{prop:normalize2}, and let $\omega_\partial$ is the form of Lemma~\ref{lem:ext-propagator-I}. Then we have
\[ \int_{T(a)+A(a)}\omega_\partial
=\int_{T(a)+A(a)}\omega_1. \]
\end{Lem}
\begin{proof}
The identity holds
since the domains are both included in the common subspace $p_{B\ell}^{-1}(([-1,4]\times \partial V')^2)=p_{B\ell}^{-1}(([-1,4]\times \partial V)^2)$, where the two forms $\omega_\partial$ and $\omega_1$ agree. 
\end{proof}

\begin{Lem}\label{lem:deg-Sigma}
 Let $(V,\Sigma)$ be as above and let $\omega_1$ be a propagator normalized as in Proposition~\ref{prop:normalize2}. Then we have
\[ \int_{\mathrm{diag}(\nu)\Sigma}\omega_1=\delta(\Sigma). \]
\end{Lem}
\begin{proof}
First we prove that 
\[ \int_{\mathrm{diag}(\nu)\Sigma-\delta(\Sigma)ST(*)}\omega_1
= \int_{\mathrm{diag}(\nu)\Sigma}\omega_1-\delta(\Sigma)
\]
does not change if $\Sigma$ is replaced with the spanning disk $\Sigma_0=(a_j^T\times I)[4]$ bounded by $a=a_j[4]$. Namely, by the analogues of Lemmas~\ref{lem:int-pt-Sigma} and \ref{lem:int-T-A} obtained by replacing $(V',\Sigma_{V'})$ and $\omega_\partial$ with $(V,\Sigma_0)$ and $\omega_1$, respectively, we have
\[ \int_{C_{*,\geq}'(\Sigma,\Sigma^+)}\omega_1
-\int_{C_{*,\geq}'(\Sigma_0,\Sigma_0^+)}\omega_1
=\int_{\mathrm{diag}(\nu)\Sigma}\omega_1
-\int_{\mathrm{diag}(\nu)\Sigma_0}\omega_1.
\]
On the other hand, it follows from Lemma~\ref{lem:cxc} (b) that 
\[ \begin{split}
\int_{\sum_k c_k\times c_k^{*+}}\omega_1
&=\int_{\sum_{j,\ell}\lambda_{j\ell}^Vb_j\times b_\ell+\delta(\Sigma)ST(*)}\omega_1=\delta(\Sigma),
\end{split}
\]
where the right equality holds since $\Lk(b_p,b_q)=0$ for $p\neq q$.
Since
\[ [C_{*,\geq}'(\Sigma,\Sigma^+)]-[C_{*,\geq}'(\Sigma_0,\Sigma_0^+)]=\sum_k[c_k\times c_k^{*+}] \]
in $\bConf_2(V)$, it follows that
\[  \int_{\mathrm{diag}(\nu)\Sigma}\omega_1
-\delta(\Sigma)
=\int_{\mathrm{diag}(\nu)\Sigma_0}\omega_1-\delta(\Sigma_0).\]
It is easy to see that the right hand side of this identity is zero.
\end{proof}

\begin{Lem}\label{lem:int-omega-deg}
Let $\tau_V$ be the framing on $V$ as in Corollary~\ref{cor:framing-extend} and let $p(\tau_V)\colon ST(V)|_{\Sigma}\to S^{d-1}$ be the composition
$ST(V)|_{\Sigma}\stackrel{\tau_V}{\longrightarrow} \Sigma\times S^{d-1} \stackrel{\mathrm{pr}}{\longrightarrow} S^{d-1}$.
Let $\nu$ be the unit normal vector field on $\Sigma$ in $V$. Then we have
\[ \int_{\mathrm{diag}(\nu)\Sigma}\omega_1=\deg\,(p(\tau_V)\circ \nu). \]
Similarly, we have
\[ \int_{\mathrm{diag}(\nu_{\Sigma_{V'}})\Sigma_{V'}}\omega_\partial=\deg\,(p(\tau_{V'})\circ \nu_{\Sigma_{V'}}). \]
\end{Lem}
\begin{proof}
This follows since $\omega_1|_{SN(\Delta_V)}=p(\tau_V)^*\mathrm{Vol}_{S^{d-1}}$ and its integral is the mapping degree. The latter identity holds since $\omega_\partial$ is defined on $SN(\Delta_{V'})$.
\end{proof}

\begin{Lem}\label{lem:Sigma-Sigma-deg}
Let $\tau_V$ and $\tau_{V'}$ be the framings on $V$ and $V'$, respectively, as in Corollary~\ref{cor:framing-extend}. Let $\Sigma_{V'}$ be the restriction of $S(\widetilde{a}_j)$ in Lemma~\ref{lem:S(a)}. There is a submanifold $\Sigma_V$ bounded by $a=a_j[4]$ in $V$, respectively, such that 
\begin{enumerate}
\item $\Sigma_{V'}$ and $\Sigma_V$ agree on $[-4,4]\times \partial V_j=[-4,4]\times \partial V_j'$.
\item There is a diffeomorphism $\Sigma_{V'}\cong \Sigma_V$ relative to $[-4,4]\times \partial V_j$.
\item $\deg\,(p(\tau_{V'})\circ \nu_{\Sigma_{V'}})=\deg\,(p(\tau_V)\circ \nu_{\Sigma_V})$.
\item $\delta(\Sigma_{V'})=\delta(\Sigma_V)$.
\end{enumerate}
\end{Lem}
\begin{proof}
Recall that $\tau_{V'}$ was obtained from the standard framing $\mathrm{st}$ on the string link complement model in the Euclidean space by perturbing $\mathrm{st}$ in a neighborhood of the link components. We show that the pair $(\Sigma_V,\tau_V)$ has an interpretation similar to this. Namely, consider the string link $\underline{L}[j]$ in $\Emb(\underline{I}^{d-2}\cup \underline{I}^{d-2}\cup\underline{I}^1,I^d)$ whose $j$-th component is the $j$-th component $\underline{L_j}$ of $B(\underline{d-2},\underline{d-2},\underline{1})_d=\underline{L_1}\cup\underline{L_2}\cup \underline{L_3}$ and other components are the standard embedding $L_{\mathrm{st}}$. As $\underline{L_j}$ has the spanning disk $\underline{D_j}$ and the spanning submanifold $\underline{D'_j}$ as before, and the restrictions of the framings $\tau_{V'}$ and $\tau_V$ to $\underline{D'_j}$ agree, we obtain $\Sigma_{V'}$ for $\underline{L}[j]$ that satisfies (1) and (2), and 
we have $p(\tau_{V'})\circ \nu_{\Sigma_{V'}}=p(\tau_V)\circ \nu_{\Sigma_V}$ for this particular model, proving (3). For (4), it follows from the proof of Lemma~\ref{lem:deg-Sigma} that
\[ \delta(\Sigma_V)=\int_{(\sum_k c_k\times c_k^{*+})(\Sigma_V)}\omega_1\quad\mbox{ and }\quad
 \delta(\Sigma_{V'})=\int_{(\sum_k c_k\times c_k^{*+})(\Sigma_{V'})}\omega_1'\]
for any propagator $\omega_1'$ on $\bConf_2(V')$ that does not detect $H_{d-1}(V'^2)$ (see Lemma~\ref{lem:H(V)}). The right hand sides of these identities are the sum of the linking numbers that can be computed via the same submanifold $\underline{D'_j}$ with the same normal vector field. Thus the two integrals agree. 
\end{proof}

\begin{Lem}\label{lem:int_D}
Let $\omega_\partial$ and $\omega_1$ be as in the proof of Lemma~\ref{lem:ext-propagator-I}. We have
\begin{equation}\label{eq:int_D}
 \int_{D_i(V')}\omega_{\partial}=\int_{D_i(V)}\omega_1\quad(i=1,2,3), 
\end{equation}
where for $U=V'$ or $V$,
\begin{enumerate}
\item $D_1(U)=-p(a)\times \Sigma_{U}^+-\Sigma_{U}\times p(a)^++A(a)+T(a)$, 

\item $D_2(U)=\mathrm{diag}(\nu)\Sigma_U-\sum_{p,q}\lambda_{pq}^U\,b_p\times b_q-\delta(\Sigma_U)ST(*)$,

\item $D_3(U)=\partial(S_4(a_3)_U\times S_4(a_3)^+_U)$ (only for $d=4$).
\end{enumerate}
The superscript $+$ denotes the parallel copy in $\Sigma^+$.
\end{Lem}
\begin{proof}
(1) The identity (\ref{eq:int_D}) for $i=1$ holds by Lemmas~\ref{lem:int-pt-Sigma} and \ref{lem:int-T-A}.

(2) We prove the identity (\ref{eq:int_D}) for $i=2$, which is equivalent to the following:
\begin{equation}\label{eq:int_D2}
 \int_{\mathrm{diag}(\nu_{\Sigma_{V'}})\Sigma_{V'}}\omega_\partial-\delta(\Sigma_{V'})
= \int_{\mathrm{diag}(\nu_{\Sigma_{V}})\Sigma_{V}}\omega_1-\delta(\Sigma_{V}), 
\end{equation}
as in the proof of Lemma~\ref{lem:deg-Sigma}.
By Lemma~\ref{lem:deg-Sigma}, the right hand side of this identity does not depend on the choice of $\Sigma_V$. Thus we may choose $\Sigma_V$ as in Lemma~\ref{lem:Sigma-Sigma-deg}. For such a $\Sigma_V$, we have $\deg\,(p(\tau_{V'})\circ \nu_{\Sigma_{V'}})=\deg\,(p(\tau_{V})\circ \nu_{\Sigma_{V}})$ and $\delta(\Sigma_{V'})=\delta(\Sigma_V)$, which imply (\ref{eq:int_D2}) by Lemma~\ref{lem:int-omega-deg}.

(3) For $d=4$, we prove (\ref{eq:int_D}) for $i=3$ as follows. The proof is similar to that of $D_1(U)$. Namely, for $U=V'$, we have
\[ \begin{split}
  &\partial(S_4(a_3)\times S_4(a_3)^+)=\,a_3[4]\times S_4(a_3)^+ + S_4(a_3)\times a_3[4]^+\\
  &=\,a_3[4]\times S_{-1}(a_3)^++S_{-1}(a_3)\times a_3[4]^+\\
  &\quad+a_3[4]\times \bigl(S_4(a_3)^+\cap([-1,4]\times \partial U)\bigr)
  +\bigl(S_4(a_3)\cap([-1,4]\times \partial U)\bigr)\times a_3[4]^+.
\end{split} \]
Here, $a_3[4]\times S_{-1}(a_3)^+\subset (X-\mathring{V}'[3])\times V'[0]$ and $S_{-1}(a_3)\times a_3[4]^+\subset V'[0]\times(X-\mathring{V}'[3])$, and the integral vanishes by the explicit formula of $\omega_\partial$ there. The same is true for the integral of $\omega_1$. The part $a_3[4]\times \bigl(S_4(a_3)^+\cap([-1,4]\times \partial U)\bigr)
  +\bigl(S_4(a_3)\cap([-1,4]\times \partial U)\bigr)\times a_3[4]^+$ is included in $p_{B\ell}^{-1}(([-1,4]\times\partial V')^2)=p_{B\ell}^{-1}(([-1,4]\times\partial V)^2)$, where the two forms $\omega_\partial$ and $\omega_1$ agree, and the integrals are equal. 
\end{proof}

%%%%%%%%%%%%%%%%%%%%%%%%%5
\subsection{Extension over family of type II handlebodies}\label{ss:ext-typeII}

Now we consider $\widetilde{V}$ of type II. Recall that $\widetilde{V}=\widetilde{V}_j[4]$. 

\begin{enumerate}
\item Let $\widetilde{a}$ be $\widetilde{a}_i=S^{d-3}\times a_i[4]\subset \partial \widetilde{V}$ that is of dimension $(d-3)+1=d-2$. Namely, $i=2$ or $3$ in the model of \S\ref{ss:std-cycles}. 
Let $\widetilde{a}\times [-1,1]=S^{d-3}\times(a\times[-1,1])\subset S^{d-3}\times\partial V=\partial\widetilde{V}$ be a parametrization of a $S^{d-3}$-family of small embedded annuli in $\partial V$ such that $\widetilde{a}\times\{0\}=\widetilde{a}$. 

\item Let $p(\widetilde{a})=S^{d-3}\times p(a)$, $p(\widetilde{a})^+=S^{d-3}\times p(a)^+$.

\item Let $\widetilde{\Sigma}$ be the submanifold $S(\widetilde{a})$ of $\widetilde{V}$ of Lemma~\ref{lem:S(a)} (such that $\partial S(\widetilde{a})=\widetilde{a}$), and let $\widetilde{\Sigma}^+=(S(\widetilde{a})\cap \widetilde{V}[-1])\cup \{(5t-1,\widetilde{a}(s,v),t)\mid (s,v)\in S^{d-3}\times S^1,\,t\in[0,1]\}$, where $(5t-1,\widetilde{a}(s,v),t)\in [-4,4]\times (\widetilde{a}\times[-1,1])$. We will also denote $\widetilde{\Sigma}$ and $\widetilde{\Sigma}^+$ by $\widetilde{\Sigma}_{\widetilde{V}}$ and $\widetilde{\Sigma}^+_{\widetilde{V}}$, respectively, to emphasize that $\widetilde{\Sigma}$ and $\widetilde{\Sigma}^+$ is in $\widetilde{V}$. 

\item Let $\mathrm{diag}(\widetilde{\nu})\widetilde{\Sigma}$ be the chain given by the section $\widetilde{\nu}$ of $ST^v(\widetilde{V})|_{\widetilde{\Sigma}}\subset \partial E\bConf_2(\widetilde{V})$ obtained by the normalization of a vector field on $\widetilde{\Sigma}$.

\item Let $A(\widetilde{a})=S^{d-3}\times A(a)$, $T(\widetilde{a})=S^{d-3}\times T(a)$, where $T(a)$ and $A(a)$ are defined analogously for 1-cycle $a$ as in \S\ref{ss:ext-type-I} (5), (6). We orient $T(\widetilde{a})$ by the one induced from $\partial y\wedge \partial z\wedge o(S^{d-3})$, where $\partial y\wedge o(S^{d-3})=(-1)^{d-3}o(S^{d-3})\wedge \partial y=o(\widetilde{a})$ (\S\ref{ss:std-cycles}). Also, we orient $A(\widetilde{a})$ by the one induced from $o((0,1])\wedge o(\widetilde{a})$. 
We consider $A(\widetilde{a})$ and $T(\widetilde{a})$ as chains in $\partial E\bConf_2(\widetilde{V})=S^{d-3}\times\partial\bConf_2(V)$. 

\item Let $V'$ be a type I handlebody included in the type II handlebody, corresponding to the inclusion of an $S^1$ leaf into an $S^{d-2}$ leaf of Y-graphs. Such an embedding is possible since the $S^1$-leaf bounds a 2-disk in a type II handlebody.
\end{enumerate}

We assume the following without loss of generality.
\begin{Assum}\label{assum:nu-2}
\begin{enumerate}
\item The unit vertical vector field $\widetilde{\nu}$ on $\widetilde{\Sigma}$ is such that its restriction to $[-1,4]\times\widetilde{a}$ is included in the subspace $T^v(\widetilde{a}\times [-1,1])\subset T^v(\partial \widetilde{V})$ of $T^v\widetilde{V}$ and is orthogonal to $[-1,4]\times\widetilde{a}$. 
\item Let $\tau_{\widetilde{V}}$ be the vertical framing on $\widetilde{V}$ as in Corollary~\ref{cor:framing-extend} and let $p(\tau_{\widetilde{V}})\colon ST^v(\widetilde{V})|_{\widetilde{\Sigma}}\to S^{d-1}$ be the composition
$ST^v(\widetilde{V})|_{\widetilde{\Sigma}}\stackrel{\tau_{\widetilde{V}}}{\longrightarrow} \widetilde{\Sigma}\times S^{d-1} \stackrel{\mathrm{pr}}{\longrightarrow} S^{d-1}$. We assume that the restriction of $p(\tau_{\widetilde{V}})\circ \widetilde{\nu}$ to $[-1,4]\times \widetilde{a}$ is a constant map.
\end{enumerate}
\end{Assum}

\begin{Lem}[Type II]\label{lem:F(a)-2} The $(d-1)$-cycle
\[ \begin{split}
	F_{\widetilde{V}}^{d-1}(\widetilde{a})=& \mathrm{diag}(\widetilde{\nu})\widetilde{\Sigma}_{\widetilde{V}}-p(\widetilde{a})\times_{S^{d-3}} \widetilde{\Sigma}_{\widetilde{V}}^+ -\widetilde{\Sigma}_{\widetilde{V}}\times_{S^{d-3}} p(\widetilde{a})^+ +A(\widetilde{a})+T(\widetilde{a})\\
	& - \Bigl\{\sum_{j,\ell}\lambda_{j\ell}^{V'}\,b_j\times b_\ell+\delta(\Sigma_{V'})ST(*)\Bigr\}
\end{split} \]
in $\partial E\bConf_2(\widetilde{V})$ is null-homologous in $E\bConf_2(\widetilde{V})$, where $\lambda_{j\ell}^{V'}$ is the same as that of Lemma~\ref{lem:F(a)} for $V'$.
\end{Lem}

We first assume that $\widetilde{a}$ is the second component $S^{d-3}\times a_2[4]$, which corresponds to the second component in the spinning construction in \S\ref{ss:p-borr-II}. To prove Lemma~\ref{lem:F(a)-2}, we decompose $\widetilde{\Sigma}$ into two parts $\widetilde{\Sigma}_0$ and $\widetilde{\Sigma}_1$, and $F^{d-1}_{\widetilde{V}}(\widetilde{a})$ accordingly, and prove the nullity of the two parts separately. 

\subsubsection{Pushing most of $\widetilde{\Sigma}$ into a single fiber}
To simplify the proof of Lemma~\ref{lem:F(a)-2}, we make an assumption on the string link in the construction of $\widetilde{V}$. Recall that the $S^{d-3}$-family of embeddings $I^{d-2}\cup I^1\cup I^1\to I^{d}$ that defines $\widetilde{V}$ can be taken so that the first and third components are constant families, and the locus of the second component with the (unparametrized) first and third components forms a Borromean string link $B(\underline{d-2},\,\underline{d-2},\,\underline{1})_d$ (\S\ref{ss:p-borr-II}). {\it We assume that the family of the second component is constructed according to the model described in \S\ref{ss:explicit-model}}, which is possible by Lemma~\ref{lem:eq-g-beta}. By precomposing with an isotopy of the parameter space $S^{d-3}$ of the family of framed embeddings, we may assume that the second component agrees with the standard inclusion outside a small neighborhood $U_s$ of a single parameter $s\in S^{d-3}$. 

\subsubsection{Decomposition of $\widetilde{\Sigma}$}
After perturbing $\widetilde{\Sigma}$ suitably, it can be decomposed as the sum of the submanifolds with corners $\widetilde{\Sigma}_0$ and $\widetilde{\Sigma}_1$ (Figure~\ref{fig:F(a)-2}) satisfying the following conditions. 
\begin{enumerate}
\item $\widetilde{\Sigma}_0\cap \widetilde{\Sigma}_1=\partial\widetilde{\Sigma}_0\cap \partial\widetilde{\Sigma}_1$ and this is a $(d-2)$-disk $\delta$ such that $\partial\delta$ is included in $\partial\widetilde{V}$. 
\item $\widetilde{\Sigma}_1$ is diffeomorphic to $S^1\times S^{d-2}-\mbox{(open disk)}$ and included in $\pi_V^{-1}(U_s)$, where $\pi_V\colon \widetilde{V}\to S^{d-3}$ is the bundle projection.
\item $\widetilde{\Sigma}_0$ is diffeomorphic to $S^{d-3}\times I^2$. The bundle structure of $\widetilde{V}$ induces a product structure $S^{d-3}\times I^2$ of $\widetilde{\Sigma}_0$.
\item Let $\widetilde{a}_{(0)}=\partial \widetilde{\Sigma}_0$ and $\widetilde{a}_{(1)}=\partial\widetilde{\Sigma}_1$. Then we have $\widetilde{a}_{(0)}\cong S^{d-3}\times S^1$ and $\widetilde{a}_{(1)}\cong S^{d-2}$. As a chain, $\widetilde{a}_{(0)}+\widetilde{a}_{(1)}=\widetilde{a}$.
\end{enumerate}
\begin{figure}
\begin{center}
\includegraphics[height=35mm]{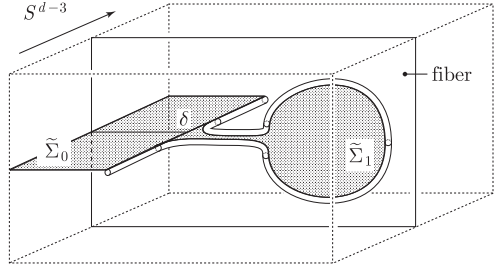}
\end{center}
\caption{$\widetilde{\Sigma}=\widetilde{\Sigma}_0\cup_{\delta}\widetilde{\Sigma}_1$, where $\widetilde{\Sigma}_1$ is included in a small neighborhood of a single fiber.}\label{fig:F(a)-2}
\end{figure}
Let us look more closely at $\widetilde{\Sigma}$ near the intersection disk $\delta$. According to the explicit model described in \S\ref{ss:explicit-model}, the intersection $\widetilde{a}_{(0)}\cap \widetilde{a}_{(1)}$ forms a $(d-3)$-disk family of singular intervals in $\widetilde{a}_{(1)}$ that restricts to a family of points over the boundary of the $(d-3)$-disk, and to a family of nondegenerate intervals over the interior, which is a ``lens'' (Figure~\ref{fig:delta}, right).

The fiberwise nonsingular vector field $\widetilde{\nu}\in\Gamma(ST^v\widetilde{V}|_{\widetilde{\Sigma}})$ on $\widetilde{\Sigma}$ can be chosen so that it is orthogonal to $\widetilde{\Sigma}$ near $\partial\widetilde{\Sigma}$ and orthogonal to both $\widetilde{\Sigma}_0$ and $\widetilde{\Sigma}_1$ on $\delta$, and we choose such. By pushing $\widetilde{a}$ slightly in a direction of $\widetilde{\nu}$, we obtain parallels $\widetilde{a}_{(0)}^+$ and $\widetilde{a}_{(1)}^+$ of $\widetilde{a}_{(0)}$ and $\widetilde{a}_{(1)}$, respectively. The chains $\widetilde{\Sigma}_0^+$ and $\widetilde{\Sigma}_1^+$ are defined by decomposing $\widetilde{\Sigma}[-1]$ into two pieces $\widetilde{\Sigma}_0[-1]=S^{d-3}\times I^2[-1]$ and $\widetilde{\Sigma}_1[-1]$ (Figure~\ref{fig:S_plus}) so that $\widetilde{\Sigma}^+=\widetilde{\Sigma}_0^++\widetilde{\Sigma}_1^+$ as chains. Note that $\widetilde{\Sigma}_1[-1]$ is not a subspace of $\widetilde{\Sigma}_1$.
Then the chains $F_{\widetilde{V}}^{d-1}(\widetilde{a}_{(0)})$, $F_{\widetilde{V}}^{d-1}(\widetilde{a}_{(1)})$ are defined similarly as above:
\[\begin{split}
F^{d-1}_{\widetilde{V}}(\widetilde{a}_{(0)})=&\,\mathrm{diag}(\widetilde{\nu})\widetilde{\Sigma}_{0}-p(\widetilde{a})\times_{S^{d-3}} \widetilde{\Sigma}_{0}^+ -\widetilde{\Sigma}_{0}\times_{S^{d-3}} p(\widetilde{a})^+ +A(\widetilde{a}_{(0)})+T(\widetilde{a}_{(0)}),\\
F^{d-1}_{\widetilde{V}}(\widetilde{a}_{(1)})=&\,\mathrm{diag}(\widetilde{\nu})\widetilde{\Sigma}_{1}-p(\widetilde{a})\times_{S^{d-3}} \widetilde{\Sigma}_{1}^+ -\widetilde{\Sigma}_{1}\times_{S^{d-3}} p(\widetilde{a})^+ +A(\widetilde{a}_{(1)})+T(\widetilde{a}_{(1)})\\
	& - \Bigl\{\sum_{j,\ell}\lambda_{j\ell}^{V'}\,b_j\times b_\ell+\delta(\Sigma_{V'})ST(*)\Bigr\}.\\
\end{split} \]
Note that these are chains of $E\bConf_2(\widetilde{V})$ but not of $\partial E\bConf_2(\widetilde{V})$. 
\begin{figure}
\begin{center}
\includegraphics[height=35mm]{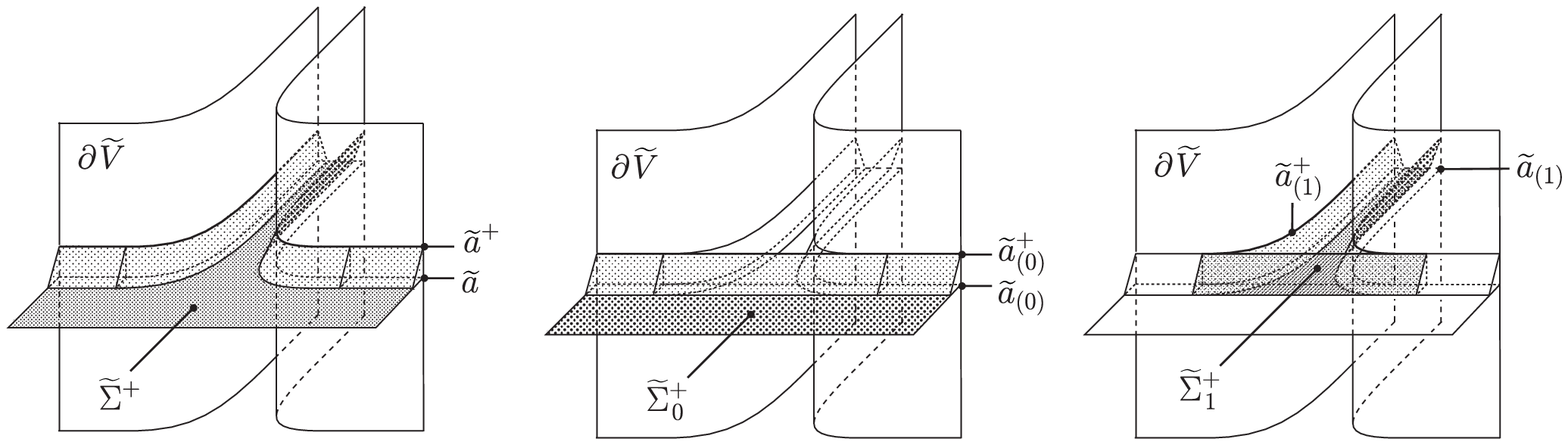}
\end{center}
\caption{$\widetilde{\Sigma}^+$, $\widetilde{\Sigma}^+_0$, and $\widetilde{\Sigma}^+_1$ near $\delta$. $\widetilde{\Sigma}^+=\widetilde{\Sigma}_0^++\widetilde{\Sigma}_1^+$.}\label{fig:S_plus}
\end{figure}

\begin{Lem}\label{lem:F-additive}
$[F_{\widetilde{V}}^{d-1}(\widetilde{a})]=[F_{\widetilde{V}}^{d-1}(\widetilde{a}_{(0)})]+[F_{\widetilde{V}}^{d-1}(\widetilde{a}_{(1)})]$.
\end{Lem}
\begin{proof}
To see this, we need only to prove the additivity of the term $T(\widetilde{a})=S^{d-3}\times T(a)$ when the loci of the basepoints of $\widetilde{a}_{(0)}$ and $\widetilde{a}_{(1)}$ are chosen compatibly, as this is the only term in $F^{d-1}_{\widetilde{V}}(\widetilde{a})$ for which the additivity is not obvious. Recall that $T(a)$ was defined by taking coordinates on the sphere $a$ by the reduced suspension of a lower dimensional sphere. Here we consider a pair $(\widetilde{a}_{(0)},\widetilde{a}_{(1)})$ of $(d-3)$-parameter families of singular 1-spheres over $U_s$ such that $\widetilde{a}_{(1)}\subset \pi_V^{-1}(U_s)$. We modify the definition of $T(a)$ at some fibers $a$ of $\widetilde{a}_{(0)}$ or $\widetilde{a}_{(1)}$ over $U_s$ slightly in a way that we consider a 1-sphere as {\it unreduced} suspension of $S^0$, which is suspended between the points $\pm \infty$, instead of the {\it reduced} suspension (Figure~\ref{fig:delta}, left). Thus we consider a 1-sphere as the quotient of $S^0\times [-1,1]$, where $S^0\times\{-1\}$ is identified with $-\infty$ and $S^0\times\{1\}$ is identified with $\infty$. Then $T(a)\colon S^0\times T\to (a\times \{0\})\times (a\times\{1\})$, where $T=\{(y,z)\in[-1,1]^2\mid y\geq z\}$, is redefined with these coordinates by the same formula:
\[ T(a)(v';y,z)=((a(v',y),0),(a(v',z),1))\quad((v';y,z)\in S^0\times T).\]
and the following holds, similarly as (\ref{eq:diag(axa)}).
\[ \partial T(a) = -\mathrm{diag}(a\times a^+)+ \infty\times a^+ + a\times (-\infty)^+. \]
Thus we need to modify accordingly the definitions of $p(\widetilde{a})$ and $p(\widetilde{a})^+$ over $U_s$ into the ones given by the loci of $+\infty$ and $-\infty$ in $\widetilde{a}$, respectively, so that $F(\widetilde{a})$ is still a cycle. 
 We take the locus of basepoints $+\infty$ to be the locus of the maximal points of the intervals in the ``lens'' $\delta$ (Figure~\ref{fig:delta}, right). Also, we take the locus of $-\infty$ to be the locus of the minimal points of the intervals. 
Then one can choose coordinates on $T(\widetilde{a}_{(0)})$ and $T(\widetilde{a}_{(1)})$ so that they are consistent on $\delta=\widetilde{a}_{(0)}\cap \widetilde{a}_{(1)}$. With this choice of coordinates, the additivity $T(\widetilde{a})=T(\widetilde{a}_{(0)})+T(\widetilde{a}_{(1)})$ is obvious. 

Note that the introduction of the two basepoints and the corresponding modification of $F_{\widetilde{V}}^{d-1}(\widetilde{a})$ does not change its homology class. More precisely, what may be changed under the modification of $F_{\widetilde{V}}^{d-1}(\widetilde{a})$ are the chains $p(\widetilde{a})\times_{S^{d-3}} \widetilde{\Sigma}_{\widetilde{V}}^+$, $\widetilde{\Sigma}_{\widetilde{V}}\times_{S^{d-3}} p(\widetilde{a})^+$, and $T(\widetilde{a})$. The changes of the first two chains are given by homotopies. If we consider that the single point $\infty$ (for reduced suspension) is the special case of the double basepoint $\pm \infty$ (for unreduced suspension) where the two basepoints agree, then the change of $T(\widetilde{a})$ is given by a homotopy. Note that considering a single basepoint as a special case of double basepoint does not change the chain $T(\widetilde{a})$. 
\begin{figure}
\[ \includegraphics[height=30mm]{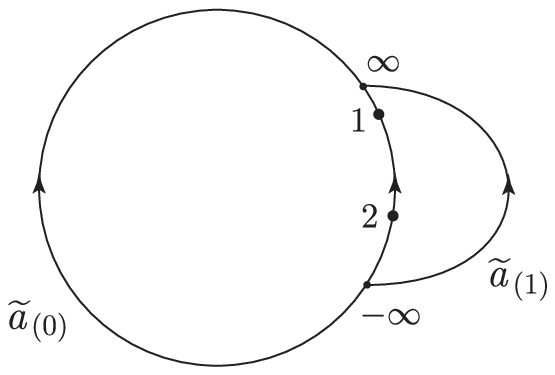}\qquad \includegraphics[height=35mm]{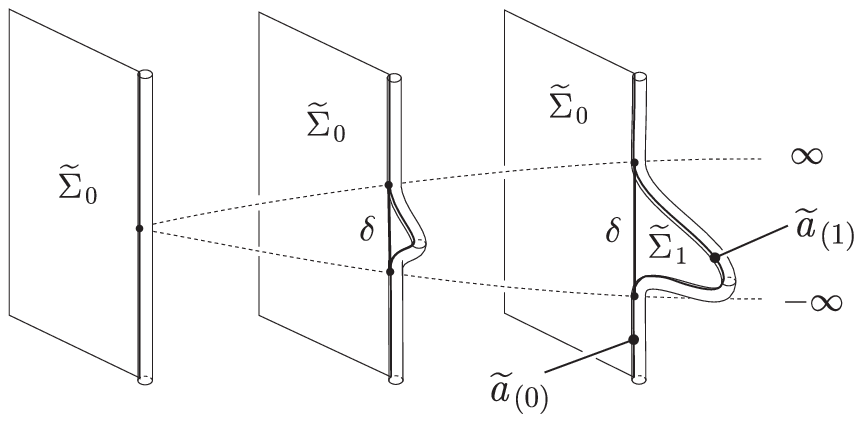} \]
\caption{Left: Introducing a pair of basepoints $\pm \infty$ to modify $T(\widetilde{a})$. Right: Appearance of $\delta$.}\label{fig:delta}
\end{figure}
\end{proof}

\subsubsection{Homological triviality of $F^{d-1}_{\widetilde{V}}(\widetilde{a})$: Proof of Lemma~\ref{lem:F(a)-2} for the second component}
Once the additivity Lemma~\ref{lem:F-additive} has been proved, the terms $[F_{\widetilde{V}}^{d-1}(\widetilde{a}_{(0)})]$ and $[F_{\widetilde{V}}^{d-1}(\widetilde{a}_{(1)})]$ can be separately altered by homotopies or addition of boundaries since the two terms are both represented by cycles. 
We have $[F^{d-1}_{\widetilde{V}}(\widetilde{a}_{(0)})]=0$ since $C_{*,\geq}(I^2,(I^2)^+)$ as in the proof of Lemma~\ref{lem:F(a)} is null-homologous. 

For $[F^{d-1}_{\widetilde{V}}(\widetilde{a}_{(1)})]$, if the radius of $U_s$ is sufficiently small, then $\widetilde{\Sigma}_1$ is close to a part of $S(a')$ for a $(d-2)$-cycle $a'$ of the boundary of a type I handlebody $V'$ included in a single fiber of $\widetilde{V}$, 
\[ \includegraphics[height=30mm]{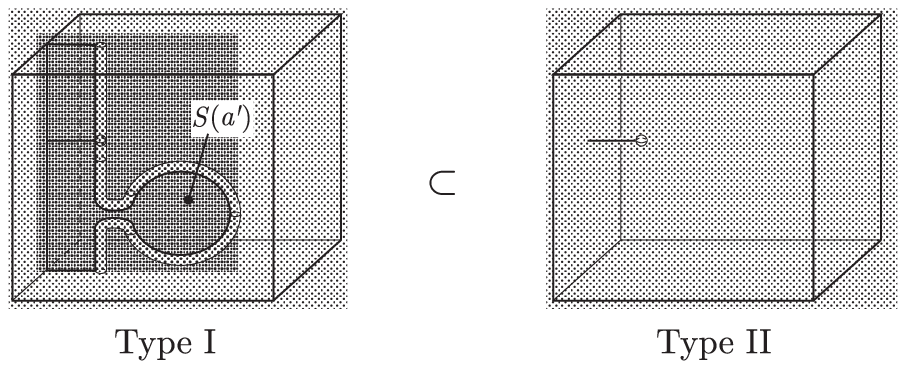} \]
and there is a homotopy of $\widetilde{\Sigma}_1$ in $\pi_V^{-1}(U_s)$ which shrinks the part near $\delta$ and then make the whole coincide with $S(a')$ that lies in a single fiber. 
\[ \includegraphics[height=30mm]{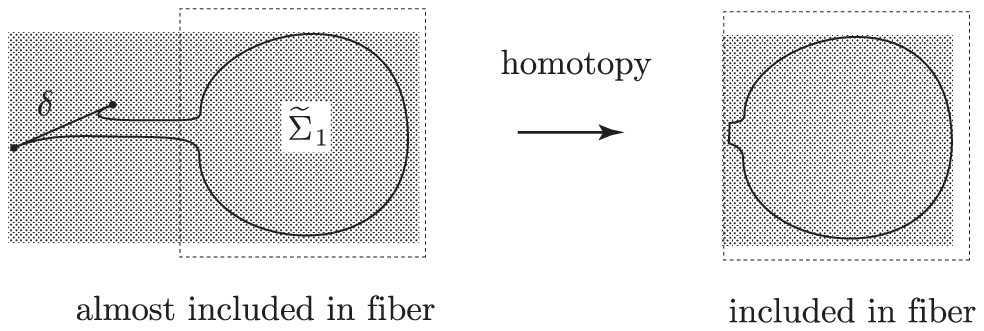} \]
This deformation is similar to the one considered in the proof of Lemma~\ref{lem:beta-graph} (2).
It does not matter if the boundary of $\widetilde{\Sigma}_1$ becomes disjoint from the boundary of $\widetilde{V}$ during the homotopy, as long as it does not go out of $\widetilde{V}$. 
Hence $F_{\widetilde{V}}^{d-1}(\widetilde{a}_{(1)})$ is homologous to $F_{V'}^{d-1}(a')$ in $E\bConf_2(\widetilde{V})$. 
By Lemma~\ref{lem:F(a)} for the single fiber, we have $[F_{\widetilde{V}}^{d-1}(\widetilde{a}_{(1)})]=[F_{V'}^{d-1}(a')]=0$. Hence we have $[F_{\widetilde{V}}^{d-1}(\widetilde{a})]=0$.
\qed

\subsubsection{Proof of Lemma~\ref{lem:F(a)-2} for the third component}
The case of the third component can be proved similarly. Namely, in the proof of Lemma~\ref{lem:beta-graph}, we have seen that $B(\underline{2d-5},\underline{d-2},\underline{d-2})_{2d-3}$ can be represented as the graph of the suspension model (b) in the proof of Lemma~\ref{lem:beta-graph}. The transformation from $B(\underline{2d-5},\underline{d-2},\underline{d-2})_{2d-3}$ to the graph model of the suspension can be applied to the third component in the same ambient space $I^{d-3}\times I^d$ parametrized over $I^{d-3}$ and can be done in a parameter preserving manner, as $B(\underline{2d-5},\underline{d-2},\underline{d-2})_{2d-3}$ has the symmetry of the last two componets, and there are two fiberwise isotopies in $I^{d-3}\times I^d$ over $I^{d-3}$ from the suspensions of $B(\underline{d-2},\underline{d-2},\underline{1})_d$ and $B(\underline{d-2},\underline{1},\underline{d-2})_d$, respectively, to the same string link $B(\underline{2d-5},\underline{d-2},\underline{d-2})_{2d-3}$. Thus there is a fiberwise isotopy between the two models. The third component of the family $[\beta]\in\pi_{d-3}(\fEmb_0(\underline{I}^{d-2}\cup \underline{I}^1\cup \underline{I}^1,I^d))$ can be treated similarly as the second component, and the proof of Lemma~\ref{lem:F(a)-2} above for the second component works also for the third component if we assume the first and second components are standard. 
\qed

\subsubsection{Homology of $E\bConf_2(\widetilde{V})$}
\begin{Lem}[Type II]\label{lem:H(EV)}
$H_{2d-3}(E\bConf_2(\widetilde{V}))=
\Lambda\oplus \Lambda'$, 
where
\[ \begin{split}
  &\Lambda=\langle[ST^v(b_2)],\,[ST^v(b_3)]\rangle,\quad
  \Lambda'=\langle[ST^v(\widetilde{b}_1)]\rangle \oplus H_{2d-3}(\widetilde{V}\times_{S^{d-3}}\widetilde{V}),
\end{split} \]
and $H_{2d-3}(\widetilde{V}\times_{S^{d-3}}\widetilde{V})$ is nonzero only if $d=4$, in which case $H_5(\widetilde{V}\times_{S^1}\widetilde{V})$ has the following basis.
\[ \bigl\{[S^1\times (b_j\times b_\ell')]\mid \dim{b_j}=\dim{b_\ell'}=2 \bigr\}, \]
where $b_\ell'$ is a parallel copy of $b_\ell$ in $\partial V$ obtained by slightly shifting in a direction of a normal vector field of $b_\ell\subset \partial V$.
\end{Lem}
\begin{proof}
The proof is an analogue of Lemma~\ref{lem:H(V)}(ii). Put $\widetilde{V}^\circ=\mathrm{Int}\,\widetilde{V}$ and $K=S^{d-3}$. We consider the homology exact sequence for the pair
\[ \to H_{p+1}(\widetilde{V}^\circ\times_K\widetilde{V}^\circ)\stackrel{i}{\to} H_{p+1}(\widetilde{V}^\circ\times_K\widetilde{V}^\circ,\widetilde{V}^\circ\times_K\widetilde{V}^\circ-\Delta_{\widetilde{V}^\circ})\to H_p(E\Conf_2(\widetilde{V}^\circ))\to \]
The bundle isomorphism $\widetilde{\varphi}$ of Proposition~\ref{prop:framed-handle-II} induces trivializations of the bundles $\widetilde{V}^\circ\times_K\widetilde{V}^\circ$ and $E\Conf_2(\widetilde{V}^\circ)$ over $K$, which are natural with respect to the exact sequence above. Hence the long exact sequence splits into tensor product of that of the fiber and the homology of $K$. It follows from triviality of $H_*(\mathring{V}^2)\to H_*(\mathring{V}^2,\mathring{V}^2-\Delta_{\mathring{V}})$ shown in the proof of Lemma~\ref{lem:H(V)} that $i$ is zero, and we have the isomorphism
\[ H_p(E\Conf_2(\widetilde{V}^\circ))\cong H_{p+1}(\widetilde{V}^\circ\times_K\widetilde{V}^\circ,\widetilde{V}^\circ\times_K\widetilde{V}^\circ-\Delta_{\widetilde{V}^\circ})\oplus H_p(\widetilde{V}^\circ\times_K\widetilde{V}^\circ). \]
By excision, we have
\[ H_{d+r}(\widetilde{V}^\circ\times_K\widetilde{V}^\circ,\widetilde{V}^\circ\times_K\widetilde{V}^\circ-\Delta_{\widetilde{V}^\circ})
=\left\{\begin{array}{ll}
\langle [D^d,\partial D^d]\rangle\otimes H_r(\widetilde{V}) & (r\geq 0),\\
0 & (r<0),
\end{array}\right. \]
where the image of $\langle [D^d,\partial D^d]\rangle\otimes H_r(\widetilde{V})$ in $H_{d+r-1}(E\Conf_2(\widetilde{V}^\circ))$ is spanned by $ST^v(\alpha)$ for $r$-cycles $\alpha$ of $\widetilde{V}$ generating $H_r(\widetilde{V})$. The generators $\alpha$ can be given explicitly. We have the following commutative diagram
\[ \xymatrix{
  \widetilde{V} \ar[r]^-{\widetilde{\varphi}_{\mathrm{II}}}_-{\cong} & K\times V \\
  \partial\widetilde{V} \ar[u]^-{\cup} \ar[r]_-{=} & K\times \partial V \ar[u]_-{\cup}
} \]
where $\widetilde{\varphi}_{\mathrm{II}}$ is a bundle isomorphism by Proposition~\ref{prop:framed-handle-II}. It follows from this that $H_{d-2}(\widetilde{V})$ is generated by the classes of the following cycles in $K\times\partial V$.
\[ *\times b_2,\quad *\times b_3,\quad \widetilde{b}_1=K\times b_1. \]
Namely, the image of $H_{d+(d-2)}(\widetilde{V}^\circ\times_K\widetilde{V}^\circ,\widetilde{V}^\circ\times_K\widetilde{V}^\circ-\Delta_{\widetilde{V}^\circ})$ ($*\geq 0$) in $H_{2d-3}(E\Conf_2(\widetilde{V}^\circ))$ 
is generated by $ST^v(b_2),ST^v(b_3)$ and $ST^v(\widetilde{b}_1)$. 

Since by Proposition~\ref{prop:framed-handle-II} the bundle $\widetilde{V}^\circ\times_K\widetilde{V}^\circ$ over $K$ is a trivial $\mathring{V}^2$-bundle, 
we have
\[ H_{2d-3}(\widetilde{V}^\circ\times_K\widetilde{V}^\circ)\cong H_{2d-3}(K\times V^2). \]
It follows from Lemma~\ref{lem:H(V)}(i) and the K\"{u}nneth formula that 
\[ \begin{split}
H_{2d-3}(K\times V^2)&=H_{d-3}(K)\otimes H_d(V^2)\\
&=\left\{\begin{array}{ll}
\langle[S^1\times (b_j\times b_\ell)]\mid \dim{b_j}=\dim{b_\ell}=2\rangle & (d=4),\\
0 & (\mbox{otherwise}).
\end{array}\right. 
\end{split} \]
The expression $S^1\times (b_j\times b_\ell)$ also makes sense in $\widetilde{V}^\circ\times_K\widetilde{V}^\circ$ since it is a cycle in $\partial \widetilde{V}\times_K\partial\widetilde{V}=K\times (\partial V\times \partial V)$, where the identification is given by the trivialization $\partial \widetilde{V}=K\times \partial V$. This completes the proof.
\end{proof}

By Lemma~\ref{lem:F(a)-2}, there exist $d$-chains $G_{\widetilde{V}}^{d}(\widetilde{a}_2),G_{\widetilde{V}}^{d}(\widetilde{a}_3)$ of $E\bConf_2(\widetilde{V})$ such that $\partial G_{\widetilde{V}}^{d}(\widetilde{a}_i)=F_{\widetilde{V}}^{d-1}(\widetilde{a}_i)$ ($i=2,3$). 
\begin{Lem}\label{lem:basis_G_V-2}
$H_{d}(E\bConf_2(\widetilde{V}),\partial E\bConf_2(\widetilde{V}))$ has the following basis.
\[ \begin{split}
  &\{ [G_{V}^{d}(a_1)],\,[G_{\widetilde{V}}^{d}(\widetilde{a}_2)],\,[G_{\widetilde{V}}^{d}(\widetilde{a}_3)]\}\\
  &\cup\left\{\begin{array}{ll}
  \{[S(a_j)\times S(a_\ell)^+]\mid \dim{a_j}=\dim{a_\ell}=1\} &(d=4),\\
  \emptyset & (d>4).
  \end{array}\right.
\end{split}\]
\end{Lem}
\begin{proof}
As in the proof of Lemma~\ref{lem:basis_G_V}, the dimension of $H_{d}(E\bConf_2(\widetilde{V}),\partial E\bConf_2(\widetilde{V}))$ is determined by Lemma~\ref{lem:H(V)} and by Poincar\'{e}--Lefschetz duality, the linear independence of the generating $d$-chains can be checked by computing the intersection numbers with the basis of Lemma~\ref{lem:H(EV)}. 
\end{proof}

\subsubsection{Extension of $\omega_{4,i}'$}
\begin{Lem}[Type II]\label{lem:ext-propagator-II}
For the propagator $\omega_{4,i}'$ of Lemma~\ref{lem:extension-ii}, the closed form
\[ \omega_\partial = \omega_{4,i}'|_{\partial E\bConf_2({\widetilde{V}})} \]
on $\partial E\bConf_2({\widetilde{V}})$ extends to a closed form on $E\bConf_2({\widetilde{V}})$.
\end{Lem}
\begin{proof}
We consider the following commutative diagram.
\[ \xymatrix{
	\delta([\omega_\partial])\in H^{d}(E\bConf_2(\widetilde{V}),\partial E\bConf_2(\widetilde{V})) \ar[r]^-{\cong} & H_{d}(E\bConf_2(\widetilde{V}),\partial E\bConf_2(\widetilde{V})) \ar[l] \ar[d]\\
	[\omega_\partial]\in H^{d-1}(\partial E\bConf_2(\widetilde{V})) \ar[r]^-{\cong} \ar[u]^-{\delta} & H_{d-1}(\partial E\bConf_2(\widetilde{V})) \ar[l]\\
} \]
We would like to prove that $\delta([\omega_\partial])=0$. As in the proof of Lemma~\ref{lem:ext-propagator-I}, it suffices to show that the evaluation of $\delta([\omega_\partial])$ with a basis of $H_{d}(E\bConf_2(\widetilde{V}),\partial E\bConf_2(\widetilde{V}))$ of Lemma~\ref{lem:basis_G_V-2} vanishes.

Moreover, by an argument similar to the type I case, we need only to check that the following integrals are zero.
\[\begin{split} 
&\int_{F_{V}^{d-1}(a_1)}\omega_\partial,\quad 
\int_{F_{\widetilde{V}}^{d-1}(\widetilde{a}_i)}\omega_\partial \quad\mbox{($i=2,3$), \quad and}\quad\\
&\int_{\partial (S(a_j)\times S(a_\ell)^+)}\omega_\partial\quad\mbox{(if $d=4$ and $\dim{a}_j=\dim{a}_\ell=1$).}
\end{split}\]
The computations of these integrals are similar to the proof of Lemma~\ref{lem:ext-propagator-I}. Namely, by Lemma~\ref{lem:int_D-2} below, we have
\[  \int_{F_{\widetilde{V}}^{d-1}(\widetilde{a}_i)}\omega_{\partial}=0
\quad\mbox{and}\quad
\int_{\partial (S(a_j)\times S(a_\ell)^+)}\omega_\partial=0. \]
This completes the proof.
\end{proof}

The idea to prove Lemma~\ref{lem:int_D-2} is similar to that of Lemma~\ref{lem:int_D}. We give some lemmas to prove Lemma~\ref{lem:int_D-2}.

\begin{Lem}\label{lem:int-pt-Sigma-2}
 Let $(\widetilde{V},\widetilde{\Sigma})$ be as above and let $\omega_\partial$ is the form of Lemma~\ref{lem:ext-propagator-II}. Then we have
\[ \int_{p(\widetilde{a})\times_{S^{d-3}}\widetilde{\Sigma}_{\widetilde{V}}^+}\omega_\partial
=\int_{\widetilde{\Sigma}_{\widetilde{V}}\times_{S^{d-3}} p(\widetilde{a})^+}\omega_\partial
=0. \]
\end{Lem}
\begin{proof}
We see that
\begin{equation}\label{eq:int-pt-Sigma-2}
 \int_{p(\widetilde{a})\times_{S^{d-3}}\widetilde{\Sigma}_{\widetilde{V}}[-1]^+}\omega_\partial
=0
\end{equation}
since $p(\widetilde{a})\times_{S^{d-3}}\widetilde{\Sigma}_{\widetilde{V}}[-1]^+\subset (E^\Gamma-\mathrm{Int}\,\widetilde{V}[3])\times_{S^{d-3}} \widetilde{V}[0]$ and $\widetilde{\Sigma}_{\widetilde{V}}[-1]\times_{S^{d-3}}p(\widetilde{a})^+\subset \widetilde{V}[0]\times_{S^{d-3}} (E^\Gamma-\mathrm{Int}\,\widetilde{V}[3])$, and we have explicit formula for $\omega_\partial$ there. We have similar identities for the integrals over $\widetilde{\Sigma}_{\widetilde{V}}[-1]\times_{S^{d-3}} p(\widetilde{a})^+$.

Also,  we have
\[ \int_{p(\widetilde{a})\times_{S^{d-3}}(\widetilde{\Sigma}_{\widetilde{V}}^+-\mathrm{Int}\,\widetilde{\Sigma}_{\widetilde{V}}^+[-1])}\omega_\partial
=0 \]
since the domain is included in the subbundle $S^{d-3}\times p_{B\ell}^{-1}(([-1,4]\times \partial V)^2)$, where $\omega_\partial$ is the pullback of $\omega_1$ in a single fiber $p_{B\ell}^{-1}(([-1,4]\times \partial V)^2)$ and the integral vanishes by a dimensional reason.
We have a similar vanishing of the integral over $(\widetilde{\Sigma}_{\widetilde{V}}^+-\mathrm{Int}\,\widetilde{\Sigma}_{\widetilde{V}}^+[-1])\times_{S^{d-3}} p(\widetilde{a})$. This completes the proof.
\end{proof}

\begin{Lem}\label{lem:int-T-A-2}
 Let $(\widetilde{V},\widetilde{\Sigma})$ be as above and let $\omega_\partial$ is the form of Lemma~\ref{lem:ext-propagator-II}. Then we have
\[ \int_{T(\widetilde{a})+A(\widetilde{a})}\omega_\partial=0. \]
\end{Lem}
\begin{proof}
The identity holds
since $T(\widetilde{a})=S^{d-3}\times T(a)$ and $A(\widetilde{a})=S^{d-3}\times A(a)$ are included in the subbundle $S^{d-3}\times p_{B\ell}^{-1}(([-1,4]\times \partial V)^2)$, 
where $\omega_\partial$ is the pullback of $\omega_1$ in a single fiber $p_{B\ell}^{-1}(([-1,4]\times \partial V)^2)$ and the integral vanishes by a dimensional reason.
\end{proof}

\begin{Lem}\label{lem:int_D-2}
Let $\omega_\partial$ be as in the proof of Lemma~\ref{lem:ext-propagator-II}. We have
\begin{equation}\label{eq:int_D-2}
 \int_{D_i(\widetilde{V})}\omega_{\partial}=0\quad(i=1,2,3), 
\end{equation}
where 
\begin{enumerate}
\item $D_1(\widetilde{V})=-p(\widetilde{a})\times_{S^{d-3}}\widetilde{\Sigma}^+ - \widetilde{\Sigma}\times_{S^{d-3}}p(\widetilde{a})^+
  +A(\widetilde{a})+T(\widetilde{a})$, 

\item $D_2(\widetilde{V})=\mathrm{diag}(\widetilde{\nu})\widetilde{\Sigma}-\sum_{p,q}\lambda_{pq}^{V'}\,b_p\times b_q-\delta(\Sigma_{V'})ST(*)$,

\item $D_3(\widetilde{V})=\partial(S_4(a_j)_{V'}\times S_4(a_\ell)^+_{V'})$ ($\dim{a}_j=\dim{a}_\ell=1$, only for $d=4$).
\end{enumerate}
The superscript $+$ denotes the parallel copy in $\Sigma^+$.
\end{Lem}
\begin{proof}
(1) The identity (\ref{eq:int_D-2}) for $i=1$ holds by Lemmas~\ref{lem:int-pt-Sigma-2} and \ref{lem:int-T-A-2}.

(2) To prove the identity (\ref{eq:int_D-2}) for $i=2$, we prove the identity
\[ \int_{\mathrm{diag}(\widetilde{\nu})\widetilde{\Sigma}}\omega_\partial=\delta(\Sigma_{V'}). \]
Let $\tau_{\widetilde{V}}$ be the vertical framing on $\widetilde{V}$ as in Corollary~\ref{cor:framing-extend} and let $p(\tau_{\widetilde{V}})\colon ST^v(\widetilde{V})|_{\widetilde{\Sigma}}\to S^{d-1}$ be the composition $ST^v(\widetilde{V})|_{\widetilde{\Sigma}}\stackrel{\tau_{\widetilde{V}}}{\longrightarrow} \widetilde{\Sigma}\times S^{d-1}\stackrel{\mathrm{pr}}{\longrightarrow} S^{d-1}$. 
We use the decomposition $\widetilde{\Sigma}=\widetilde{\Sigma}_0\cup \widetilde{\Sigma}_1$ given before Lemma~\ref{lem:F-additive}. By Assumption~\ref{assum:nu-2} for the vertical framing $\tau_{\widetilde{V}}$ and $\widetilde{\nu}$ near $\partial \widetilde{V}$, we see that
\[ \int_{\mathrm{diag}(\widetilde{\nu})\widetilde{\Sigma}_0}\omega_\partial=0. \]
Moreover, as we assume $p(\tau_{\widetilde{V}})$ is constant near $\delta=\widetilde{\Sigma}_0\cap\widetilde{\Sigma}_1$ and near $\partial\widetilde{V}$, we may assume by a small perturbation of $\widetilde{\Sigma}_1$ in $\widetilde{V}$ that the result $\widetilde{\Sigma}_1'$ of the perturbation is included in a single fiber $\pi_V^{-1}(s)$, without changing the relative homotopy class of $p(\tau_{\widetilde{V}})\circ \widetilde{\nu}_{\widetilde{\Sigma}_1}\colon (\widetilde{\Sigma}_1,\partial \widetilde{\Sigma}_1)\to (S^{d-1},*)$. Thus we have
\[ \begin{split}
&\int_{\mathrm{diag}(\widetilde{\nu})\widetilde{\Sigma}_1}\omega_\partial=\int_{\mathrm{diag}(\widetilde{\nu})\widetilde{\Sigma}_1'}\omega_\partial=\int_{\mathrm{diag}(\nu_{\Sigma_{V'}})\Sigma_{V'}}\omega_\partial|_{\pi_V^{-1}(s)}=\delta(\Sigma_{V'}),\\
&\int_{D_2(\widetilde{V})}\omega_\partial=\int_{\mathrm{diag}(\widetilde{\nu})\widetilde{\Sigma}_0}\omega_\partial+\int_{\mathrm{diag}(\widetilde{\nu})\widetilde{\Sigma}_1}\omega_\partial-\delta(\Sigma_{V'})=0.
\end{split} \]

(3) The identity (\ref{eq:int_D-2}) for $i=3$ is for the integral in a single fiber and the same as Lemma~\ref{lem:int_D} (3).
\end{proof}

%%%%%%%%%%%%%%%%%%%%%%%%%%%%%%%%%%%%%
\begin{appendix}

%%%%%%%%%%%%%%%%%%%%%%%%%%%%%%%%%%%%%
%%%%%%%%%%%%%%%%%%%%%%%%%%%%%%%%%%%%%
\mysection{Smooth manifolds with corners}{s:mfd-corners}

We follow the convention in \cite[Appendix]{BTa} for manifolds with corners, smooth maps between them and their (strata) transversality. We quote some necessary terminology from \cite{BTa}. We refer the reader to \cite{Jo} for more detail.

\begin{Def}
\begin{enumerate}
\item A {\it manifold with corners} of dimension $k>0$ is a topological manifold $X$ such that every point in $X$ has a neighborhood which is homeomorphic to $[0,\infty)^m\times \R^{k-m}$ for some integer $0\leq m\leq k$. The transition function between two such coordinate charts is required to be smooth. 
\item A map between manifolds with corners is {\it smooth} if it has a local extension, at any point of the domain, to a smooth map from a manifold without boundary, as usual.
\item Let $Y,Z$ be smooth manifolds with corners, and let $f\colon Y\to Z$ be a bijective smooth map. This map is a {\it diffeomorphism} if both $f$ and $f^{-1}$ are smooth.
\item Let $Y, Z$ be smooth manifolds with corners, and let $f\colon Y\to Z$ be a smooth
map. This map is {\it strata preserving} if the inverse image by $f$ of a connected
component $S$ of a stratum of $Z$ of codimension $i$ is a union of connected components of strata
of $Y$ of codimension $i$.
\item Let $X,Y$ be smooth manifolds with corners and $Z$ be a smooth manifold without boundary. Let $f\colon X\to Z$ and $g\colon Y\to Z$ be smooth maps. Say that $f$ and $g$ are {\it (strata) transversal} when the following is true: Let $U$ and $V$ be connected components in strata of $X$ and $Y$ respectively. Then $f\colon U\to S$ and $g\colon V\to S$ are transversal.
\end{enumerate}
\end{Def}

%%%%%%%%%%%%%%%%%%%%%%%%%%%%%%%%%%%%%
%%%%%%%%%%%%%%%%%%%%%%%%%%%%%%%%%%%%%
\mysection{Blow-up in differentiable manifold}{s:blow-up}

\subsection{Blow-up of $\R^i$ along the origin}

Let $\wgamma^1(\R^i)$ denote the total space of the tautological oriented half-line ($[0,\infty)$) bundle over the oriented Grassmannian $\widetilde{G}_1(\R^i)= S^{i-1}$. Namely, $\wgamma^1(\R^i)=\{(x,y)\in S^{i-1}\times \R^i; \exists t\in[0,\infty), y=tx\}$. Then the tautological bundle is trivial and $\wgamma^1(\R^i)$ is diffeomorphic to $S^{i-1}\times [0,\infty)$. 
\begin{Def} Let
\[ B\ell_{\{0\}}(\R^i)=\wgamma^1(\R^i) \]
and call $B\ell_{\{0\}}(\R^i)$ the {\it blow-up} of $\R^i$ along $0$.
\end{Def}
Let $\pi\colon B\ell_{\{0\}}(\R^i)=\wgamma^1(\R^i)\to \R^i$ be the map defined by $\pi=\pr_1\circ \varphi$ in the following commutative diagram:
\begin{equation}\label{eq:bl}
 \xymatrix{
	B\ell_{\{0\}}(\R^i)=\wgamma^1(\R^i) \ar[r]^-{\varphi} \ar[rd]_-{\pi}&
	 \R^i \times S^{i-1}  \ar[d]^-{\pr_1} \ar[r]^-{\pr_2} & S^{i-1}\\
	& \R^i &
	}
\end{equation}
where $\varphi\colon \wgamma^1(\R^i)\to S^{i-1}\times\R^i$ is the embedding which maps a pair $(x,y)\in S^{i-1}\times \R^i$ with $y=tx$ to $(x,y)$. If $y\neq 0$, then $\varphi(x,y)=(\frac{y}{|y|},y)$. We call $\pi$ the {\it blow-down map} of the blow-up. Here, $\pi^{-1}(0)=\partial\wgamma^1(\R^i)$ is the image of the zero section of the tautological bundle $\pr_2\circ \varphi\colon \wgamma^1(\R^i)\to S^{i-1}$ and is diffeomorphic to $S^{i-1}$. 
\begin{Lem}\label{lem:bl_extension}
\begin{enumerate}
\item The restriction of $\pi$ to the complement of $\pi^{-1}(0)=\partial\wgamma^1(\R^i)$ is a diffeomorphism onto $\R^i-\{0\}$. 
\item The restriction of $\varphi$ to the complement of $\pi^{-1}(0)$ has the image in $\R^i\times S^{i-1}$ whose closure agrees with the full image of $\varphi$ from $\wgamma^1(\R^i)$.
\item The map $\phi\colon \R^i-\{0\}\to S^{i-1}$ defined by $y\mapsto \frac{y}{|y|}$ extends to a smooth map $\phi'=p_2\circ \varphi\colon B\ell_{\{0\}}(\R^i)\to S^{i-1}$, in the sense that the composition
\[ 
\R^i-\{0\} \stackrel{\pi^{-1}}{\longrightarrow} 
\mathrm{Int}\,B\ell_{\{0\}}(\R^i) \stackrel{\varphi}{\longrightarrow}
 \R^i\times S^{i-1} \stackrel{\pr_2}{\longrightarrow}
S^{i-1} \]
agrees with $\phi$.
\item $B\ell_{\{0\}}(\R^i)$ admits a collar neighborhood $\partial B\ell_{\{0\}}(\R^i)\times [0,\ve)$ such that $\{(0,x)\}\times [0,\ve)$ is the preimage of the half-ray $\{x\}\times \{tx\mid t\geq 0\}$ under $\varphi$, which agrees with $\phi'^{-1}(x)$.
\end{enumerate}
\end{Lem}
\subsection{Blow-up along a submanifold}

\begin{Def}
When $d>i\geq 0$, we put $B\ell_{\R^i}(\R^d)=\R^i\times\wgamma^1(\R^{d-i})$ (the blow-up of $\R^d$ along $\R^i$) and define the projection $\pr_{B\ell}\colon B\ell_{\R^i}(\R^d)\to \R^d$ by $\mathrm{id}_{\R^i}\times \pi$. 
\end{Def}
This can be straightforwardly extended to the blow-up $B\ell_X(Y)$ of a manifold $Y$ along a submanifold $X$, by working on one chart at a time and the naturality properties of the blow-up with respect to linear isomorphisms (\cite[Corollary~2.6]{ArK}).

\begin{Lem}
Let $Y$ be a smooth manifold with corners and let $X$ be a submanifold of $Y$ that is strata transversal to $\partial Y$. Then $B\ell_X(Y)$ is a smooth manifold with corners.
\end{Lem}
\begin{proof}
By strata transversality, a standard local model of $X$ at a corner point $x\in X\cap \partial Y$ can be given by the following subspace
\[ [0,\infty)^k\times \R^\ell\subset [0,\infty)^k\times \R^{k-m}\quad(0\leq \ell\leq k-m). \]
Hence the blow-up along $X$ can be locally given by
\[ [0,\infty)^k\times B\ell_{\R^\ell}(\R^{k-m}), \]
which is a manifold with corners. 
\end{proof}

%%%%%%%%%%%%%%%%%%%%%%%%%%%%%%%%%%%%%
%%%%%%%%%%%%%%%%%%%%%%%%%%%%%%%%%%%%%
\mysection{Fulton--MacPherson compactification}{s:corners-compactification}

\subsection{Compactification by a sequence of blow-ups}

\begin{Lem}\label{lem:seq-blowup}
Let $r>2$ and $\bConf_n^{(r)}(X)$ be the closure of the image of the embedding
\[ \iota_r\colon \Conf_n(X)\to X^n(r)=X^n\times\prod_{|\Lambda|\geq r}B\ell_{\Delta(\Lambda)}(X^\Lambda) \]
of (\ref{eq:seq-emb}). Then $\bConf_n^{(r-1)}(X)$ can be obtained from $\bConf_n^{(r)}(X)$ by a sequence of blow-ups:
\[ \bConf_n^{(r)}(X)=M_0\stackrel{}{\leftarrow} 
M_1 \stackrel{}{\leftarrow} 
M_2 \stackrel{}{\leftarrow} \cdots
\stackrel{}{\leftarrow} 
M_N=\bConf_n^{(r-1)}(X), \]
where each $M_\ell$ is a manifold with corners and each step $M_\ell\stackrel{}{\leftarrow} M_{\ell+1}$ is the blow-up along a submanifold of $M_\ell$ of codimension $d(r-2)$ that is strata-transversal to the boundary. Thus $\bConf_n(X)=\bConf_n^{(2)}(X)$ can be obtained from $X^n$ by a sequence of blow-ups.
\end{Lem}
\begin{proof}
By definition, $X^n(r-1)=X^n(r)\times \prod_{i=1}^NE_i$ where $E_i=B\ell_{\Delta(\Lambda)}(X^\Lambda)$ for some $\Lambda$ with $|\Lambda|=r-1$. For $1\leq \ell\leq N$, let $X^n(r-1,\ell)=X^n(r)\times\prod_{i=1}^\ell E_i$ and let $\kappa_\ell\colon C_n(X)\to X^n(r-1,\ell)$ be the natural embedding defined similarly as $\iota_r$. Let $M_\ell$ be the closure of the image of $\kappa_\ell$ in $X^n(r-1,\ell)$. Let $\Sigma_\ell$ be the image of $\partial M_\ell$ under the natural projection $X^n(r-1,\ell)\to X^n$. Then by the local structure of the natural stratification of $X^n$ formed by the diagonals, it follows that $\Sigma_{\ell+1}-\Sigma_\ell$ is a submanifold of $X^n$ of codimension $d(r-2)$, which has a submanifold lift in $X^n(r-1,\ell)$. Moreover, by the standard local model of the successive blow-ups in Lemma~\ref{lem:succ-blowup} below, we have that the closure of the lift of $\Sigma_{\ell+1}-\Sigma_\ell$ in $X^n(r-1,\ell)$ is strata-transversal to $\partial M_\ell$, and that $M_{\ell+1}$ can be obtained by the blow-up along the closure of $\Sigma_{\ell+1}-\Sigma_\ell$ in $X^n(r-1,\ell)$. (We say that the preimage of the face of $\partial M_{\ell+1}$ under the projection $X^n(2)\to X^n(r-1,\ell+1)$ that is not in the preimage of $\partial M_\ell\subset X^n(r-1,\ell)$ is {\it caused by} the factor $E_{\ell+1}$.)
\end{proof}

\subsection{Standard local model of successive blow-ups}\label{ss:std-model-blowup}

For a collection $\calL=\{L_\lambda\}$ of linear subspaces $L_\lambda$ of a fixed real vector space $R$, we say that another collection $\calL'=\{L_\mu'\}$ of linear subspaces of $R$ is {\it transversal modulo $\calL$} if for any finitely many elements $L_{\mu_1}',\ldots,L_{\mu_r}'\in\calL'$, either of the following holds:
\begin{itemize}
\item the intersection $L_{\mu_1}'\cap\cdots\cap L_{\mu_r}'$ is transversal in $R$, or
\item the intersection $L_{\mu_1}'\cap\cdots\cap L_{\mu_r}'$ is $L_{\lambda}$ in $\calL$ for some $\lambda$, or
\item there is a partition $\{\mu_1,\ldots,\mu_r\}=M_1\coprod\cdots\coprod M_{r'}$ ($r'\geq 2$) such that $\bigcap_{\mu\in M_\ell}L_\mu'$ is $L_{\lambda_\ell}$ in $\calL$ for some $\lambda_\ell$.
\end{itemize}
Examples of collections $\calL$ transversal modulo $\calL$ are given in Example~\ref{ex:trans-L} and Lemma~\ref{lem:str-diagonal} below.

\begin{Lem}\label{lem:succ-blowup}
Let $L_0=\emptyset$ and let $L_1\subset L_2\subset\cdots\subset L_m\subset \R^e$ be a sequence of subspaces satisfying the following conditions.
\begin{enumerate}
\item $L_i$ is the union of $L_{i-1}$ and some finitely many linear subspaces $L_i^{(1)},\ldots,L_i^{(k_i)}$ of $\R^e$ of dimension $\ell_i$. 
\item $0\leq \ell_1<\ell_2<\cdots<\ell_m<e$.
\item The collection $\calL=\{L_i^{(j)}\mid 1\leq i\leq m,\,1\leq j\leq k_i\}$ of linear subspaces of $\R^e$ is transversal modulo $\calL$. 
\end{enumerate}
Let $Y_0=\R^e$. Then there is a sequence $Y_i$ ($1\leq i\leq m$) of smooth manifolds with corners such that for $i\geq 1$, $Y_i$ is obtained from $Y_{i-1}$ by blowing-ups along the lifts of $L_i^{(1)},\ldots,L_i^{(k_i)}$. More precisely, $Y_i$ is obtained from $Y_{i-1}$ by a sequence of blow-ups:
\[ Y_{i-1}=M_0\leftarrow M_1\leftarrow\cdots\leftarrow M_{k_i}=Y_i, \]
where $M_j$ is obtained from $M_{j-1}$ by blowing up along the closure of the lift of $L_i^{(j)}-(L_{i-1}\cup L_i^{(1)}\cup\cdots\cup L_i^{(j-1)})$ in $M_{j-1}$, which is a smooth submanifold of $M_{j-1}$ with corners.
\end{Lem}

\begin{Lem}\label{lem:R-L-P}
Let $R$ be an $e$-dimensional manifold with corners, let $L$ be an $\ell$-dimensional submanifold of dimension $\ell<e$, and let $P$ be a $p$-dimensional submanifold of $L$ of dimension $p<\ell$. Suppose that $\partial P\subset\partial L\subset \partial R$, $L$ is strata transversal to $\partial R$, and $P$ is strata transversal to $\partial L$. Then the closure of the image of the lift $L-P\to B\ell_P(R)$ of the inclusion $L-P\subset R-P$ is a submanifold of $B\ell_P(R)$ strata transversal to $\partial B\ell_P(R)$.
\end{Lem}
\begin{proof}
It suffices to prove the assertion for a standard local model: $L$, $P$ are the linear subspaces $\R^\ell\times 0$, $0\times \R^{e-\ell}$ of $R=\R^e$, respectively. In this case, the normal bundle $NP$ can be identified with $R$. We consider the following diagram:
\[ \xymatrix{
  B\ell_P(R) \ar[r]^-{\varphi_P} & \R^{e-p}\times S^{e-p-1} \ar[r]^-{\pr_2} & S^{e-p-1} \ar[d]^-{g} \ar[r]^-{f} & \R^{e-\ell}\\
  & R-P \ar[ul]^-{\pi^{-1}} \ar[u]_-{\iota} \ar[r]^-{\mathrm{incl}} & R \ar[ur]_-{\pr_{e-\ell}} & 
} \] 
where $\varphi_P$ is the map induced by $\varphi\colon B\ell_{\{0\}}(\R^{e-p})\to \R^{e-p}\times S^{e-p-1}$, $\pi^{-1}$ is the natural inclusion, $\iota=\varphi_P\circ \pi^{-1}$, $p_{e-\ell}\colon R=\R^\ell\times\R^{e-\ell}\to \R^{e-\ell}$ is the projection, $g$ is the composition of the inclusions $S^{e-p-1}\subset \R^{e-p}\subset R=\R^\ell\times\R^{e-\ell}$, and $f=\pr_{e-\ell}\circ g$. In this diagram, the left triangle is induced by (\ref{eq:bl}) and is commutative. The right triangle is commutative, too, and the middle square is not commutative.

It can be shown that
\begin{enumerate}
\item[(i)] $p_{e-\ell}^{-1}(0)=L$, $\mathrm{incl}^{-1}(p_{e-\ell}^{-1}(0))=L-P$, 
\item[(ii)] $(f\circ \pr_2\circ \iota)^{-1}(0)=L-P$,
\item[(iii)] $(f\circ \pr_2\circ \varphi_P)^{-1}(0)$ is the closure of $\pi^{-1}(L-P)$,
\item[(iv)] $0\in \R^{e-\ell}$ is a regular value of $f\circ \pr_2\circ \varphi_P$.
\end{enumerate}
The result for this standard local model follows immediately from (iii) and (iv). The claim (i) is obvious. The claim (ii) holds since $f^{-1}(0)$ is the $(\ell-p-1)$-sphere consisting of the directions perpendicular to $0\times\R^{e-\ell}$ in $\R^{e-p}$. As $0\times\R^{e-\ell}$ is the orthogonal complement of $L$ in $R$, onto which $f\circ \pr_2\circ\iota$ projects, we get (ii).
The claim (iii) follows from (ii) and by the commutativity of the left triangle. The claim (iv) holds since $\pr_2\circ \varphi_P\colon B\ell_P(R)\to S^{e-p-1}$ is the projection to a fiber of the normal sphere bundle of $P$, which is a submersion, and $0$ is a regular value of $f$, as $g$ is transversal to $L$. 
This completes the proof for a standard local model at a non-boundary point on $P$.

At a point of $\partial P$ ($\subset \partial L\subset \partial R$), there are non-zero linear functions $v_1,\ldots,v_m\colon R\to \R$ so that $L-P$ is determined by the conditions $f\circ \pr_2\circ \iota=0$, and $v_1\geq 0,\ldots,v_m\geq 0$. Moreover, $L-P$ is the preimage of $0\times [0,\infty)^m$ of the smooth map
\[ h\colon R-P\to \R^{e-\ell}\times \R^m;\quad h(x)=((f\circ \pr_2\circ \iota)(x), v_1(x),\ldots,v_m(x)), \]
for which $(0,0)$ is a regular value, i.e., $h$ is strata transversal to $(0,0)\in\R^{e-\ell}\times\R^m$, by the strata transversality assumptions at the boundaries of $P,L,R$. The map $h$ can be naturally extended to smooth maps on $R$ and on $B\ell_P(R)$, for each of which $(0,0)$ is a regular value. Thus the preimage of $(0,0)$ gives the closure of $\pi^{-1}(L-P)$, which is a submanifold of $B\ell_P(R)$ strata transversal to the boundary.
\end{proof}

\begin{proof}[Proof of Lemma~\ref{lem:succ-blowup}]
We prove this by induction on the pair $(m,e)$. When $e=1$, the only nontrivial case is $L_1=\{0\}\subset \R^1$, where $m=1$. In this case, the assertion is obvious by the property of the blow-up $B\ell_{\{0\}}(\R)$. When $e\geq 1$ and $m=1$, the conditions (1) and (3) imply that $L_1$ is the union $L_1^{(1)}\cup\cdots\cup L_1^{(k_1)}$ of $\ell_1$-dimensional subspaces of $\R^e$ whose intersection is transversal. In this case, $Y_1$ can be identified with the closure of the image of the embedding
\[ \iota=(\iota_0,\iota_1,\ldots,\iota_{k_1})\colon \R^e-L_1\to \R^e\times B\ell_{\{0\}}((L_1^{(1)})^\perp)\times\cdots\times B\ell_{\{0\}}((L_1^{(k_1)})^\perp), \]
where $\iota_0\colon \R^e-L_1\to \R^e$ is the inclusion, $\iota_j\colon \R^e-L_1\to B\ell_{\{0\}}((L_1^{(j)})^\perp)$ is induced by the orthogonal projection to $(L_1^{(j)})^\perp$. Then the closure $Y_1$ of the image of $\iota$ is diffeomorphic to $(\bigcap_{j=1}^{k_1} L_1^{(j)})\times B\ell_{\{0\}}((L_1^{(1)})^\perp)\times\cdots\times B\ell_{\{0\}}((L_1^{(k_1)})^\perp)$, which is a smooth manifold with corners.

For a general pair $(m,e)$ such that $m>1$, $e>1$, we suppose that the assertion holds for $(m',e')$ with $e'<e$ and for $(m',e)$ with $m'<m$. Thus we have a sequence $Y_1,Y_2,\ldots,Y_{m-1}$ of smooth manifolds with corners, which are obtained by blowing-ups along some submanifold lifts of $L_1,\ldots,L_{m-1}$ as in the statement. Now we would like to blow-up $Y_{m-1}$ along the lifts of $L_m^{(1)},\ldots,L_m^{(k_m)}$, which are $\ell_m$-dimensional by (1). We first observe that the closure $\overline{L}_m^{(j)}$ of the lift of $L_m^{(j)}-L_{m-1}$ in $Y_{m-1}$
\begin{enumerate}
\item[(a)] is a smooth submanifold with corners that is 
\item[(b)] strata-transversal to $\partial Y_{m-1}$.
\end{enumerate}
Indeed, the sequence of blow-ups of $\R^e$ along $L_1,\ldots, L_{m-1}$ induces a sequence of blow-ups of $L_m^{(j)}$ along $\calL_1',\ldots, \calL_{m-1}'$, where $\calL_i'=\{L_i^{(j')}\cap L_m^{(j)}\mid 1\leq j'\leq k_i\}$ for $i<m$. By applying the induction hypothesis for $e=\ell_m$, which is less than the original $e$ by (2), we obtain that the induced blow-ups turn $L_m^{(j)}$ into a smooth manifold with corners in $Y_{m-1}$, and the result can be identified with $\overline{L}_m^{(j)}$. Hence (a) is proved.

To prove (b), it suffices to check that the property that the closure of the lift of $L_m^{(j)}-\bigcup_{i=1}^{m-1}(\bigcup\calL_i')$ is transversal to the boundary is preserved under each step in the sequence of the blow-ups of $\R^e$ along $\calL_1',\ldots, \calL_{m-1}'$. 
When a plane $P=L_i^{(\ell)}\in\calL_i'$ is included in $L_m^{(j)}$, it follows from Lemma~\ref{lem:R-L-P} that $L_m^{(j)}$ is transversal to the face caused by the blowing-up along such $P$. When a plane $P=L_i^{(\ell)}\in\calL_i'$ ($i<m$) is not included in $L_m^{(j)}$, Lemma~\ref{lem:R-L-P} can be generalized to the case where $P$ may not be included in $L$ by modifying its proof by replacing $S^{e-p-1}$ with $S^{e-p-1}\times \R^{p-q}$ for $q=\dim{P\cap L}$. Hence $L_m^{(j)}$ is also transversal to the face added by the blowing-up along such $P$. This completes the proof of (b).

Now we see that any subset of $\{\overline{L}_m^{(1)},\ldots,\overline{L}_m^{(k_m)}\}$ intersect transversally in $Y_{m-1}$. Indeed, since $\{L_m^{(1)},\ldots,L_m^{(k_m)}\}$ is transversal modulo $\calL$ by the assumption (3) and lower dimensional intersection planes in $\calL$ have been resolved by blow-ups to construct $Y_{m-1}$, there are only transversal intersections among the planes in $\{L_m^{(1)},\ldots,L_m^{(k_m)}\}$ that haven't been resolved previously. We show that each step in the sequence of blow-ups of $\R^e$ along $L_1,\ldots,L_{m-1}$ to construct $Y_{m-1}$ does not destroy transversality modulo $\calL$ property among $L_m^{(1)},\ldots,L_m^{(k_m)}$. 
As Lemma~\ref{lem:R-L-P}, we see that for a plane $P=L_i^{(\ell)}$ ($i<m$) the transversal intersection $L_m^{(j_1)}\cap \cdots\cap L_m^{(j_\lambda)}$ that may or may not be transversal to $P$ is locally the preimage of a regular value and is a framed submanifold whose framing extends to the boundary of the blow-up along $P$. Thus the transversality of the intersection of $L_m^{(j_1)},\ldots, L_m^{(j_\lambda)}$ extends to the boundary of the blow-up along $P$.
\[ \includegraphics[height=27mm]{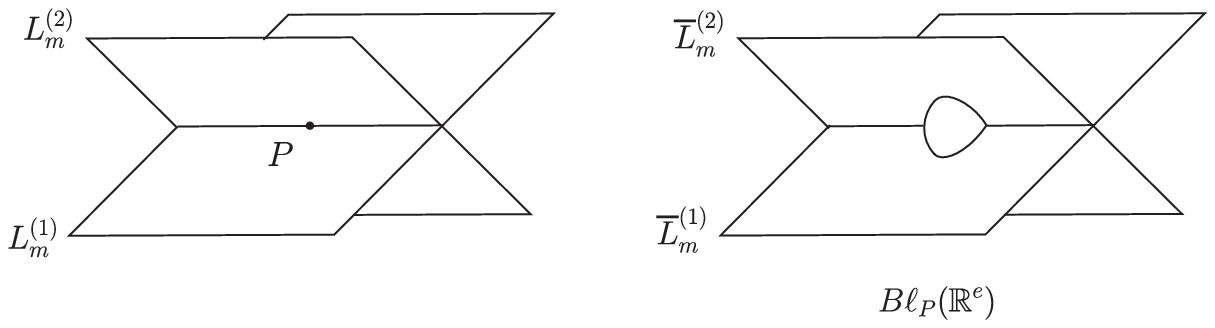} \]
Then the blow-ups along $\overline{L}_m^{(1)},\ldots,\overline{L}_m^{(k_m)}$ yield $Y_m$, which is a smooth manifold with corners. 
\end{proof}

\begin{Exa}\label{ex:trans-L}
Let $L_1\subset L_2\subset L_3\subset \R^3$ be a sequence of (topological) subspaces given by
$L_1=L_1^{(1)}=\{0\}$, $L_2=L_2^{(1)}\cup L_2^{(2)}\cup L_2^{(3)}\cup L_2^{(4)}$, $L_3=L_3^{(1)}\cup L_3^{(2)}\cup L_3^{(3)}\cup L_3^{(4)}\cup L_3^{(5)}\cup L_3^{(6)}$, where
\[ \begin{split}
 &L_2^{(1)}=\{(x_1,0,0)\mid x_1\in\R\}, L_2^{(2)}=\{(0,x_2,0)\mid x_2\in\R\}, L_2^{(3)}=\{(0,0,x_3)\mid x_3\in\R\},\\
 &L_2^{(4)}=\{(x_1,x_2,x_3)\in\R^3\mid x_1=x_2=x_3\},\\
 &L_3^{(1)}=\{(0,x_2,x_3)\mid x_2,x_3\in\R\}, 
 L_3^{(2)}=\{(x_1,0,x_3)\mid x_1,x_3\in\R\}, \\
 &L_3^{(3)}=\{(x_1,x_2,0)\mid x_1,x_2\in\R\}, 
 L_3^{(4)}=\{(x_1,x_2,x_3)\in\R^3\mid x_2=x_3\},\\
 &L_3^{(5)}=\{(x_1,x_2,x_3)\in\R^3\mid x_1=x_3\}, 
 L_3^{(6)}=\{(x_1,x_2,x_3)\in\R^3\mid x_1=x_2\}.
\end{split} \]
This is a local model of $\bConf_3(S^1;\infty)$ near $(\infty,\infty,\infty)\in (S^1)^{\times 3}$. It is easy to see that $\calL=\{L_i^{(j)}\}$ is transversal modulo $\calL$ (Figure~\ref{fig:Bl-ex}).\qed
\begin{figure}
\[ \includegraphics[height=45mm]{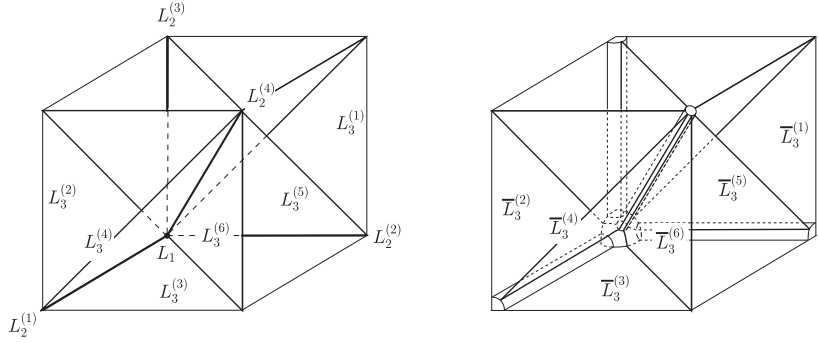} \]
\caption{Left: $\calL=\{L_i^{(j)}\}$. Right: The result $Y_2$ of blowing-ups along $L_1$ and $L_2$.}\label{fig:Bl-ex}
\end{figure}
\end{Exa}

\begin{Lem}\label{lem:str-diagonal}
For a subset $\Lambda$ of $\mathbf{n}=\{1,2,\ldots,n\}$ with $|\Lambda|\geq 2$, let
\[ \Delta_\Lambda=\{(x_1,\ldots,x_n)\in (\R^e)^{\times n}
\mid x_i=x_j\mbox{ for all $i,j\in\Lambda$}\}, \]
and let $\calL=\{\Delta_\Lambda\mid \Lambda\subset \mathbf{n},\,|\Lambda|\geq 2\}$. Then $\calL$ is transversal modulo $\calL$.
\end{Lem}
\begin{proof}
We take $\Delta_{\Lambda_1},\ldots,\Delta_{\Lambda_r}\in\calL$ for subsets $\Lambda_1,\ldots,\Lambda_r\subset \mathbf{n}$. We consider the graph $G(S)$ in the power set $2^{\mathbf{n}}$ on the vertex set $S=\{\Lambda_1,\ldots,\Lambda_r\}$ with edges $\{\Lambda_i,\Lambda_j\}$ for $i\neq j$ and $\Lambda_i\cap \Lambda_j\neq\emptyset$. We say that $S$ is connected if $G(S)$ is connected. 
When $S=\{\Lambda_1,\ldots,\Lambda_r\}$ is connected, then by the definition of $\Delta_\Lambda$, we have
\[ \Delta_{\Lambda_1}\cap\cdots\cap\Delta_{\Lambda_r}=\Delta_{\Lambda_1\cup\cdots\cup \Lambda_r}\in\calL. \]
When $S=\{\Lambda_1,\ldots,\Lambda_r\}$ is not connected, then there is a partition $\{\Lambda_1,\ldots,\Lambda_r\}=M_1\coprod\cdots\coprod M_{r'}$ into connected subcollections such that
\[ \begin{split}
  &\bigcap_{\Lambda\in M_i}\Delta_\Lambda=\Delta_{\bigcup M_i},\\
  &\Delta_{\Lambda_1}\cap\cdots\cap\Delta_{\Lambda_r}=(\Delta_{\bigcup M_1})\cap\cdots\cap(\Delta_{\bigcup M_{r'}}),
\end{split} \]
where $\Delta_{\bigcup M_j}\in\calL$ and the intersection on the right hand side is transversal. The transversality holds since the sum of the orthogonal complements $(\Delta_{\bigcup M_j})^\perp$ is the direct sum in $(\R^e)^{\times n}$.
\end{proof}

%%%%%%%%%%%%%%%%%%%%%%%%
\subsection{Configuration space of a manifold with boundary}

\begin{Lem}\label{lem:Bl(Y2)}
Let $Y$ be a smooth $m$-manifold with nonempty boundary that is a submanifold of a manifold $X$ without boundary. Let $B\ell_{\Delta_Y}(Y\times Y)$ denote the closure of $p_{B\ell}^{-1}(Y\times Y-\Delta_Y)$ in $B\ell_{\Delta_X}(X\times X)$. Then $B\ell_{\Delta_Y}(Y\times Y)$ is the image of a smooth manifold with corners under a smooth map.
\end{Lem}
\begin{proof}
A standard local model of $\Delta_Y$ at a corner point in $\partial Y\times \partial Y\subset Y\times Y$ can be given by the pair
$((\R^{m-1})^2\times [0,\infty)^2,\Delta_{\R^{m-1}}\times \Delta_{[0,\infty)})$, 
which is identified with
$\R^{m-1}\times (\R^{m-1}\times [0,\infty)^2,0\times \Delta_{[0,\infty)})$. In this model
\[ \R^{m-1}\times (0\times \Delta_{[0,\infty)}),\quad
\R^{m-1}\times(\R^{m-1}\times (0,0)),\quad
\R^{m-1}\times(0\times (0,0)) \]
give local models of $\Delta_Y,\partial Y\times \partial Y,\Delta_{\partial Y}$, respectively. 
We consider the sequence $L_1\subset L_2\subset L_3$ of subspaces of $\R^{m-1}\times \R^2$, where
\[\begin{split}
  &L_1=\{0\},\quad 
  L_2=\R^{m-1}\times (0,0),\quad\\
  &L_3=L_2\cup (0\times \Delta_\R)\cup (\R^{m-1}\times\R\times 0)\cup (\R^{m-1}\times 0\times \R), 
\end{split}\]
and consider the successive blow-ups $\R^{m-1}\times\R^2=Y_0\leftarrow Y_1\leftarrow Y_2\leftarrow Y_3$ along this sequence. This gives a local model of the blow-ups along the sequence $\Delta_{\partial Y}\subset \partial Y\times \partial Y\subset (\partial Y\times \partial Y)\cup \Delta_Y\cup (Y\times \partial Y)\cup (\partial Y\times Y)$. By Lemma~\ref{lem:succ-blowup}, $Y_3$ is a smooth manifold with corners. 

Let $Y_3^{++}$ be the component of $Y_3$ that is projected to $\R^{m-1}\times [0,\infty)^2$. Then there is a smooth projection
$Y_3^{++}\to B\ell_{0\times \Delta_{[0,\infty)}}(\R^{m-1}\times [0,\infty)^2)$,
which is induced by the smooth projection $Y_3\to B\ell_{0\times\Delta_\R}(\R^{m-1}\times \R^2)$. 
Since $\R^{m-1}\times Y_3^{++}$ is a smooth manifold with corners and $\R^{m-1}\times B\ell_{0\times \Delta_{[0,\infty)}}(\R^{m-1}\times [0,\infty)^2)$ is a local model of $B\ell_{\Delta_Y}(Y\times Y)$ at a corner point in $\partial Y\times \partial Y$, the result follows.
\end{proof}

\begin{Def}[Compactification of $\Conf_2(Y)$]\label{def:compact-C2(Y)}
Let $Y$ be as in Lemma~\ref{lem:Bl(Y2)}. 
Let $\bConf_2(Y;\partial Y)$ denote the manifold with corners obtained by the blow-ups of $Y\times Y$ along the sequence
\[ \Delta_{\partial Y}\subset \partial Y\times \partial Y\subset (\partial Y\times \partial Y)\cup \Delta_Y\cup (Y\times \partial Y)\cup (\partial Y\times Y) \]
of strata as in the proof of Lemma~\ref{lem:Bl(Y2)}. Let $\pr_{B\ell}'\colon \bConf_2(Y;\partial Y)\to B\ell_{\Delta_Y}(Y\times Y)$ denote the smooth projection of Lemma~\ref{lem:Bl(Y2)}. 
\end{Def}

\begin{Rem}\label{rem:compact-Cn(Y)}
\begin{enumerate}
\item $\bConf_2(Y)=B\ell_{\Delta_Y}(Y\times Y)$ is not a smooth manifold with corners. In particular, along the restriction of the normal sphere bundle over $\Delta_Y$ to $\partial \Delta_Y$ in $\partial B\ell_{\Delta_Y}(Y\times Y)$. 

\item In Definition~\ref{def:compact-C2(Y)}, the blow-ups along $(Y\times\partial Y)\cup (\partial Y\times Y)$ is in fact not necessary since without this yields a diffeomorphic result. This was necessary in the proof of Lemma~\ref{lem:Bl(Y2)} to cut out one piece from $\R^m\times\R^m$. 

\item The blow-up compactification to a manifold with corners in Definition~\ref{def:compact-C2(Y)} can be generalized to compactify $\Conf_n(Y)$ to a manifold with corners $\bConf_n(Y;\partial Y)$ by considering the stratification of $Y^n$ given by the conditions that some points agree and that some points are on $\partial Y$. Such a stratification of $Y^n$ gives a local model satisfying the assumption of Lemma~\ref{lem:succ-blowup}. A detail about a compactification of $\Conf_n(Y)$ is given in \cite{CILW}.
\end{enumerate}
\end{Rem}

\begin{Lem}\label{lem:equiv-C2}
Let $Y$ be as in Lemma~\ref{lem:Bl(Y2)}. Let $\bConf_2(Y)=B\ell_{\Delta_Y}(Y\times Y)$. Then the maps $\pr'_{B\ell}\colon \bConf_2(Y;\partial Y)\to \bConf_2(Y)$ and $\mathrm{incl}\colon \Conf_2(V)\to \bConf_2(Y)$ are homotopy equivalences. Moreover, the induced map $\pr'_{B\ell}\colon (\bConf_2(Y;\partial Y),\partial\bConf_2(Y;\partial Y))\to (\bConf_2(Y),\partial \bConf_2(Y))$ is a homotopy equivalence.
\end{Lem}
\begin{proof}
This is evident from the local model in the proof of Lemma~\ref{lem:Bl(Y2)}, as it is easy to give explicit deformation retractions. Namely, we observe that $\bConf_2(Y;\partial Y)$ is embedded as the complement of the lift of a small tubular neighborhood of $\Delta_{\partial Y}$ in $\bConf_2(Y)$ by pressing a small collar neighborhood the boundary of the blow-up along $\partial Y\times \partial Y$ into the interior of $\bConf_2(Y)$. Then there is a deformation retract of $\bConf_2(Y)$ onto $\bConf_2(Y;\partial Y)$, which gives a homotopy inverse. 
\end{proof}

%%%
\subsection{Proof of Lemma~\ref{lem:compactif-Sd-mfd}}\label{ss:proof-compactif-Sd-mfd}

\begin{Lem}[Lemma~\ref{lem:compactif-Sd-mfd}]\label{lem:compactif-Sd-mfd-2}
The map $\rho^{n+1}\colon \bConf_{n+1}(S^d)\to S^d$ is a fiber bundle such that the fiber $\bConf_n(S^d;\infty)$ is a manifold with corners.
\end{Lem}
\begin{proof}
The projection map $\rho^{n+1}\colon C_{n+1}(S^d)\to S^d$ is a fiber bundle whose fiber over $\infty$ is $C_n(\R^d)\times\{\infty\}$. Now we have the following commutative diagram:
\begin{equation}\label{eq:proj_infty}
\vcenter{
 \xymatrix{
  C_n(\R^d)\times\{\infty\} \ar[r]^-{\iota'} \ar[d]_-{\mathrm{incl}} & \displaystyle(X^n\times\{\infty\})\times\prod_{{\Lambda_0\subset\{1,2,\ldots,n,\infty\}}\atop{|\Lambda_0|\geq 2}} B\ell_{\Delta(\Lambda_0)}(X^{\Lambda_0}) \ar[d]^-{\mathrm{incl}}\\
  C_{n+1}(S^d) \ar[r]^-{\iota} \ar[d]_-{\rho^{n+1}} & \displaystyle X^{n+1}\times \prod_{{\Lambda\subset\{1,2,\ldots,n+1\}}\atop{|\Lambda|\geq 2}} B\ell_{\Delta(\Lambda)}(X^{\Lambda}) \ar[d]^-{\Pi}\\
  S^d \ar[r]^-{\delta} & X \times \displaystyle\prod_{{M\subset\{1,2,\ldots,n+1\}}\atop{|M|\geq 2,\,n+1\in M}}X
}}
\end{equation}
where $X=S^d$, the vertical lines are fiber bundles, and
\begin{itemize}
\item $\Pi$ is the projection defined by forgetting the factors $B\ell_{\Delta(\Lambda)}(X^\Lambda)$ for $\Lambda$ such that $n+1\notin \Lambda$, and by juxtaposing $X^{n+1}\to X; (x_1,\ldots,x_{n+1})\mapsto x_{n+1}$ and 
the composition of the projections $B\ell_{\Delta(M)}(X^M)\to X^M$ and $X^M\to X$;\\ $(x_{b_1},\ldots,x_{b_m},x_{n+1})\mapsto x_{n+1} \mbox{ for $M=\{b_1,\ldots,b_m,n+1\}$}$,
\item on the top row, $B\ell_{\Delta(\{a_1,\ldots,a_m,\infty\})}(X^{\{a_1,\ldots,a_m,\infty\}})$ denotes the blow-up of $X^{\{a_1,\ldots,a_m\}}\times\{\infty\}$ along $\{(\infty,\ldots,\infty,\infty)\}$,
\item $\iota$ is the embedding (\ref{eq:AS}), $\iota'$ is the embedding induced by $\iota$,
\item $\delta$ is defined by $\delta(x)=(x,\prod_M x)$.
\end{itemize}
The embedding $\iota$ is a map of fiber bundles. It follows from the diagram (\ref{eq:proj_infty}) that the closure $\bConf_{n+1}(S^d)$ of the image of $\iota$ induces the closure of the image of $\iota'$ in the fiber. Thus we obtain the fiber bundle
\begin{equation}\label{eq:fiberwise-closure}
 \bConf_n(S^d;\infty)=\mathrm{Closure}\,(\mathrm{Im}\,\iota')\to \bConf_{n+1}(S^d)\to S^d.
\end{equation}
Now the proof that $\bConf_n(S^d;\infty)$ is a smooth manifold with corners is analogous to that of $\bConf_{n+1}(S^d)$. 
\end{proof}

%%%
\subsection{Proof of Lemma~\ref{lem:phi-extend}}\label{s:phi-extend}

\begin{Lem}[Lemma~\ref{lem:phi-extend}]\label{lem:phi-extend-2}
The smooth map $\phi\colon \Conf_2(\R^d)\to S^{d-1}$ defined by 
\[ \phi(x_1,x_2)=\frac{x_2-x_1}{|x_2-x_1|} \]
extends to a smooth map $\overline{\phi}\colon \bConf_2(S^d;\infty)\to S^{d-1}$. The extension $\overline{\phi}$ on the boundary of $\bConf_2(S^d;\infty)$ is explicitly given as follows:
\begin{enumerate}
\item On the stratum $\overline{S}_{\{1,2,\infty\}}=B\ell_D(\{(y_1,y_2)\in(\R^d)^2\mid |y_1|^2+|y_2|^2=1\})$, $\overline{\phi}=\phi'\circ \overline{i}$, where $\overline{i}\colon\overline{S}_{\{1,2,\infty\}}\to \bConf_2(\R^d)$ is the map induced by the embedding $i:S_{\{1,2,\infty\}}=\Conf_3^*(T_\infty X)\to \Conf_2(\R^d-\{0\})$ given by (\ref{eq:Cr-infty}), and $\phi'\colon \bConf_2(\R^d)\to S^{d-1}$ is the smooth extension of $\phi$ defined by the coordinates of the blow-up (Lemma~\ref{lem:bl_extension}(3)). 

\item On the stratum $\overline{S}_{\{1,\infty\}}$, $\overline{\phi}$ is the composition
\begin{equation}\label{eq:phi-extend-2}
 \overline{S}_{\{1,\infty\}}=\bwConf_2(T_\infty X)\times \bConf_1(S^d;\infty)\stackrel{\pr_1}{\longrightarrow}\bwConf_2(T_\infty X)\xrightarrow[\cong]{-\phi'} S^{d-1}. 
\end{equation}
\item On the stratum $\overline{S}_{\{2,\infty\}}$, $\overline{\phi}$ is the composition
\begin{equation}\label{eq:phi-extend-3}
 \overline{S}_{\{2,\infty\}}=\bConf_1(S^d;\infty)\times \bwConf_2(T_\infty X)\stackrel{\pr_2}{\longrightarrow}\bwConf_2(T_\infty X)\xrightarrow[\cong]{\phi'} S^{d-1}.  
\end{equation}
\item On the stratum $\overline{S}_{\{1,2\}}$, $\overline{\phi}$ is the composition
\begin{equation}\label{eq:phi-extend-4}
 \overline{S}_{\{1,2\}}=\Delta_{\bConf_1(S^d;\infty)}\times \bwConf_2(\R^d)\stackrel{\pr_2}{\longrightarrow}\bwConf_2(\R^d)\xrightarrow[\cong]{\phi'} S^{d-1}. 
\end{equation}
\end{enumerate}
\end{Lem}
\begin{proof}
First, we prove that $\phi$ extends to a smooth map $\phi'\colon \bConf_2(\R^d)\to S^{d-1}$. Near $\Delta_{\R^d}$, $\phi$ factors into the orthogonal projection $\R^d\times \R^d\to \Delta_{\R^d}^{\perp}$; $(x_1,x_2)\mapsto (\frac{x_1-x_2}{2},\frac{x_2-x_1}{2})$, the identification $\Delta_{\R^d}^\perp\stackrel{\cong}{\to}\R^d$; $(-y,y)\mapsto y$, and the normalization $v\mapsto \frac{v}{|v|}$. It follows from the definition of the blow-up, the orthogonal projection is extended to a projection $\bConf_2(\R^d)=B\ell_{\Delta_{\R^d}}(\R^d\times \R^d)\to B\ell_{\{(0,0)\}}(\Delta_{\R^d}^\perp)=B\ell_{\{0\}}(\R^d)$, which is smooth, and the normalization is extended to a smooth map $B\ell_{\{0\}}(\R^d)\to S^{d-1}$ by Lemma~\ref{lem:bl_extension}(3). Hence the composition of the extended maps gives a smooth extension $\phi'$.

From now on, we prove that $\phi$ has a smooth extension on a collar neighborhood of each of $\overline{S}_\Lambda$ ($\Lambda=\{1,2,\infty\},\{1,\infty\},\{2,\infty\},\{1,2\}$, Figure~\ref{fig:collar-nh}) in $\bConf_2(S^d;\infty)$ in a way that the extensions are consistent on the intersection of two collar neighborhoods. 

For $\overline{S}_{\{1,2,\infty\}}$, we recall that $\overline{S}_{\{1,2,\infty\}}$ is obtained by the following sequence of blow-ups of $X\times X$:
\[ \partial B\ell_{\{(\infty,\infty)\}}(X^2)=S^{2d-1}\longleftarrow
\bConf_2^\infty(\R^d)\longleftarrow
\overline{S}_{\{1,2,\infty\}},\]
where $\bConf_2^\infty(\R^d)=B\ell_{\Delta}(S^{2d-1})$, $\Delta=\Delta_{\R^d}\cap S^{2d-1}\subset \R^d\times\R^d$, and the second blow-up is done along the preimage of the locus $L_0=\{(y_1,0)\in (\R^d)^2\mid |y_1|=1\}\cup \{(0,y_2)\in (\R^d)^2\mid |y_2|=1\}$ in $S^{2d-1}$. 
According to Lemma~\ref{lem:bl_extension} (3), (4), $B\ell_{\{(\infty,\infty)\}}(X^2)$ admits a collar neighborhood $\partial B\ell_{\{(\infty,\infty)\}}(X^2)\times [0,\ve)=S^{2d-1}\times [0,\ve)$ such that $\{v\}\times [0,\ve)$ is the preimage of $v\in S^{d-1}$ under the smooth extension $\overline{\psi}\colon \partial B\ell_{\{(\infty,\infty)\}}(X^2)\times [0,\ve)\to S^{2d-1}$ of the map
\[ \psi\colon \partial B\ell_{\{(\infty,\infty)\}}(X^2)\times (0,\ve)\to S^{2d-1};\quad \psi(x_1,x_2)=\frac{(x_1,x_2)}{|(x_1,x_2)|}. \]
Since $\Delta_{\R^d}\cap (\partial B\ell_{\{(\infty,\infty)\}}(X^2)\times (0,\ve))$ is the preimage of $\Delta\subset S^{2d-1}$ under $\psi$, the blow-up of $\partial B\ell_{\{(\infty,\infty)\}}(X^2)\times [0,\ve)$ along $\overline{\psi}{}^{-1}(\Delta)$ is a collar neighborhood $\bConf_2^\infty(\R^d)\times [0,\ve)$ of $\bConf_2^\infty(\R^d)$ in $B\ell_{\overline{\psi}{}^{-1}(\Delta)}\bigl(\partial B\ell_{\{(\infty,\infty)\}}(X^2)\times[0,\ve)\bigr)$. Moreover, since the preimage of the locus
$L=L_0\times [0,\ve)$
in $\bConf_2^\infty(\R^d)\times [0,\ve)$ is the preimage of $L_0$ under $\overline{\psi}$, the blow-up of $\bConf_2^\infty(\R^d)\times [0,\ve)$ along the preimage of $L$ gives a collar neighborhood $\overline{S}_{\{1,2,\infty\}}\times[0,\ve)$ of the stratum $\overline{S}_{\{1,2,\infty\}}$, as in the following picture.
\[ \includegraphics[height=25mm]{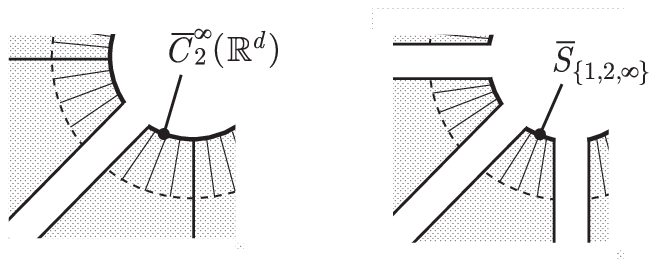} \]
Let $\overline{\psi}{}'\colon \overline{S}_{\{1,2,\infty\}}\times[0,\ve)\to  \bConf_2^\infty(\R^d)$ be the projection, which is smooth, and induces $\overline{\psi}$. Now $\bConf_2^\infty(\R^d)$ is a submanifold of $\bConf_2(\R^d)$ and the map $\phi$ admits a smooth extension $\phi'\colon \bConf_2(\R^d)\to S^{d-1}$. Hence the composition 
\[ \phi'\circ \overline{\psi}{}'\colon \overline{S}_{\{1,2,\infty\}}\times [0,\ve)\to S^{d-1}\]
is a smooth extension of $\phi|_{S_{\{1,2,\infty\}}\times (0,\ve)}$. By definition, the restriction of $\phi'\circ \overline{\psi}{}'$ to $\overline{S}_{\{1,2,\infty\}}$ agrees with $\phi'$, as in (1).
\begin{figure}
\[ \includegraphics[height=30mm]{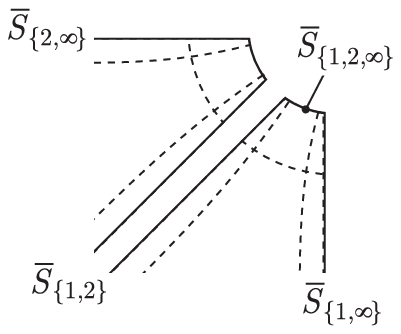}\]
\caption{Collar neighborhoods of $\overline{S}_{\{1,2,\infty\}}, \overline{S}_{\{1,\infty\}}, \overline{S}_{\{2,\infty\}}, \overline{S}_{\{1,2\}}$}\label{fig:collar-nh}
\end{figure}

Next we consider a collar neighborhood of $\overline{S}_{\{1,\infty\}}$, which is given by the blow-up of a tubular neighborhood of $\{\infty\}\times \bConf_1(S^d;\infty)$ in $\bConf_2^{(3)}(S^d;\infty)$ (\S\ref{ss:two-pt}) along the submanifold $\{\infty\}\times \bConf_1(S^d;\infty)$, so that the fiber of the normal sphere bundle in the blow-up is $\partial B\ell_{\{\infty\}}(X)=\bwConf_2(T_\infty X)$. 
It can be seen that the restriction of $\phi'\circ \overline{\psi}{}'$ to $\overline{S}_{\{1,\infty\}}\cap (\overline{S}_{\{1,2,\infty\}}\times [0,\ve))$ agrees with the composition (\ref{eq:phi-extend-2}),
where the minus sign of $-\phi'$ is because in $\bwConf_2(T_\infty X)$ we consider that the second point is restrained at the origin (see \S\ref{ss:unusual-coord}), the point where the non-infinite point gather, so that $(y_2-y_1)/|y_2-y_1|=-y_1/|y_1|$. 

We take smooth coordinates on a collar neighborhood of $S_{\{1,\infty\}}$ and a smooth extension of $\phi$ over it that agrees with $\phi'\circ\overline{\psi}{}'$ on $\overline{S}_{\{1,\infty\}}\cap (\overline{S}_{\{1,2,\infty\}}\times [0,\ve))$, as follows. We take a diffeomorphism
\[ \varphi\colon X^\circ\times X^\circ\stackrel{\cong}{\longrightarrow} \R^d\times X^\circ;\quad (z,x)\mapsto (z-x,x), \]
under which the diagonal $\Delta_{X^\circ}$ corresponds to the zero section $\{0\}\times X^\circ$. With this transformation, $\phi$ can be given by the composition
\[ X^\circ\times X^\circ-\Delta_{X^\circ} \stackrel{\varphi}{\longrightarrow} 
(\R^d-\{0\})\times X^\circ \stackrel{\mathrm{pr}}{\longrightarrow}
S^{d-1}, \]
where $\mathrm{pr}$ is the composition of the projection onto the first factor with the normalization $y\mapsto -\frac{y}{|y|}$.  
The diffeomorphism $\sigma$ of \S\ref{ss:unusual-coord} for the stereographic projection gives a coordinate transformation
\[ \sigma\colon  T_\infty X-\{0\}\stackrel{\cong}{\longrightarrow} T_0 X-\{0\};\quad y\mapsto \frac{y}{|y|^2}. \]
Thus $\varphi$ can be locally described near $\infty$ (the origin of $T_\infty X$) in terms of the (unusual) coordinates of $T_\infty X-\{0\}$ as follows: for $t>0$ small, $v\in S^{d-1}=ST_\infty X$, $x\in X^\circ$, 
\[
(\sigma^{-1}\times\mathrm{id})\circ \varphi\circ(\sigma\times\mathrm{id})(tv,x)
=\Bigl(\sigma^{-1}(\sigma(tv)-x),x\Bigr)=\Bigl(\frac{\frac{v}{t}-x}{|\frac{v}{t}-x|^2},x\Bigr). 
\]
Here $t$ should be small enough so that $\frac{v}{t}-x\neq 0$, which is satisfied if $t<\frac{1}{|x|}$. Applying $\mathrm{pr}$ to this, we obtain
\[ \phi(tv,x)=-\frac{\frac{v}{t}-x}{|\frac{v}{t}-x|}=-\frac{v-tx}{|v-tx|}. \]
This shows that the radial half-ray $t\mapsto tv$ in $X$ from $\infty$ is mapped by $\phi$ to the projection of the smooth curve $t\mapsto -(v-tx)$, which extends smoothly to $0\leq t<\frac{1}{|x|}$. Since $(v,t,x)\in (ST_\infty X\times [0,\ve))\times X^\circ$ $(t<\frac{1}{|x|})$ is a smooth coordinate system on a collar neighborhood of $S_{\{1,\infty\}}$ in $B\ell_{\{\infty\}\times X^\circ}(X\times X^\circ-\Delta_{X^\circ})$, it follows that $\phi$ extends smoothly to a map $\phi'_{1\infty}\colon S_{\{1,\infty\}}\times[0,\ve)\to S^{d-1}$. The restriction of $\phi'_{1\infty}$ to $S_{\{1,\infty\}}=\partial B\ell_{\{\infty\}\times X^\circ}(X\times X^\circ-\Delta_{X^\circ})$ is given by letting $t=0$ in the formula above, which agrees with the formula (\ref{eq:phi-extend-2}).

Extension on a collar neighborhood of $\overline{S}_{\{2,\infty\}}$ is similar. We consider a collar neighborhood of the stratum $\overline{S}_{\{2,\infty\}}$ given by the blow-up of a tubular neighborhood of $\bConf_1(S^d;\infty)\times \{\infty\}$ in $\bConf_2^{(3)}(S^d;\infty)$ along $\bConf_1(S^d;\infty)\times \{\infty\}$ so that the fiber of the normal sphere bundle in the blow-up is $\partial B\ell_{\{\infty\}}(X)=\bwConf_2(T_\infty X)$. 
It can be seen that the restriction of $\phi'\circ \overline{\psi}{}'$ to $\overline{S}_{\{2,\infty\}}\cap (\overline{S}_{\{1,2,\infty\}}\times [0,\ve))$ agrees with the composition (\ref{eq:phi-extend-3}). 
The rest of the proof is the same as for $\overline{S}_{\{1,\infty\}}$.

On a collar neighborhood of $S_{\{1,2\}}$, $\phi$ is extended by the smooth extension $\phi'\colon \bConf_2(\R^d)\to S^{d-1}$ of $\phi$. 
\end{proof}

%%%%%%%%%%%%%%%%%%%%%%%%%%%%%%%
\mysection{Orientations on manifolds and on their intersections}{s:ori}

%%%%%
\subsection{Coorientation}\label{ss:coori}
If $M$ is a submanifold of an oriented Riemannian $r$-dimensional manifold $R$, then we may alternatively define $o(M)$ from an orientation $o^*_R(M)$ of the normal bundle of $M$ by the rule
\begin{equation}\label{eq:coori}
 o(M)\wedge o^*_R(M)\sim o(R). 
\end{equation}
$o^*_R(M)$ is called a {\it coorientation} of $M$ in $R$. We assume that (\ref{eq:coori}) is always satisfied so that giving a coorientation is an alternative way to represent orientation. 

One could instead represent an orientation by a section of $\bigwedge^*T^*M$. The two interpretations are related by the duality $T_xM\cong T^*_xM$; $v\mapsto \langle v,\cdot\rangle$ given by a Riemannian metric.

%%%%%
\subsection{Orientation of intersection}
Suppose $M$ and $N$ are two cooriented submanifolds of $R$ of dimension $m$ and $n$ that intersect transversally. The transversality implies that at an intersection point $x$, the product $o^*_R(M)_x\wedge o^*_R(N)_x$ is a non-trivial $(2r-m-n)$-tensor. We define 
\begin{equation}\label{eq:coori_int}
 o^*_R(M\pitchfork N)_x=o^*_R(M)_x\wedge o^*_R(N)_x. 
\end{equation}
This depends on the order of the product. When $M$ and $N$ are compact and $m+n=r$, this convention is the same as the integral interpretation of the intersection number:
\[ \int_R \eta_M\wedge \eta_{N} \]
under the identification $\Gamma(\bigwedge^*T^*M)=\Gamma(\bigwedge^*TM)$ by the metric duality. 
See \S\ref{ss:thom} for the $\eta$-forms representing the Thom classes of the normal bundles. There are other interpretations of the intersection of submanifolds, such as $\int_M\eta_{N}$ or $\int_{N}\eta_M$. The relationship between these interpretations is as follows:
\[ (-1)^{m(r-m)}\int_M\eta_{N}=\int_R \eta_M\wedge \eta_{N}=\int_{N}\eta_M. \]
Indeed, the integral $\int_M\eta_{N}$ counts an intersection point by $+1$ if $o(M)\sim o^*_R(N)$, which is equivalent to $o^*_R(M)\wedge o^*_R(N)\sim (-1)^{i(r-i)}o(R)$ by (\ref{eq:coori}). The integral $\int_{N}\eta_M$ counts an intersection point by $+1$ if $o(N)\sim o^*_R(M)$, which is equivalent to $o^*_R(M)\wedge o^*_R(N)\sim o(R)$ by (\ref{eq:coori}).

%%%%%
\subsection{Orientation of direct product}\label{ss:ori-prod}
\setcounter{footnote}{0}
For oriented manifolds $M_1$ and $M_2$ of dimensions $m_1$ and $m_2$, respectively, we orient the direct product $M_1\times M_2$ as follows. Considering the natural identification $T(M_1\times M_2)=TM_1\times TM_2$, let $q_i\colon \Gamma(\bigwedge^{m_i}TM_i)\to \Gamma(\bigwedge^{m_i}T(M_1\times M_2))$ be the map defined by
\[ \gamma\mapsto \left\{\begin{array}{ll}
(x_1,x_2)\mapsto (\gamma(x_1),0)\in \bigwedge^*T_{x_1}M_1\otimes \bigwedge^* T_{x_2}M_2 & (i=1),\\
(x_1,x_2)\mapsto (0,\gamma(x_2))\in \bigwedge^*T_{x_1}M_1\otimes \bigwedge^*T_{x_2}M_2 & (i=2).
\end{array}\right. \]
Then for the orientations $o(M_i)\in \Gamma(\bigwedge^{m_i}TM_i)$, the product\footnote{Under the identification by the duality $T_xM\cong T^*_xM$, the map $q_i$ is just the pullback $\pr_i^*$ by the projection. }
\[ q_1\,o(M_1)\wedge q_2\,o(M_2) \]
gives an orientation of $T(M_1\times M_2)=TM_1\times TM_2$. To simplify notation, we will denote this orientation simply by 
\[ o(M_1)\wedge o(M_2). \]
We do not always assume this rule to orient products, as this orienatation for the product $M_1\times M_2$ is not always the natural one (e.g., (\ref{eq:ori-a-tilde}) and Lemma~\ref{lem:ori-Sa-compatible}), although it depends on the purpose. 

Suppose moreover that $M_1$ is a submanifold of $R_1$ and $M_2$ is a submanifold of $R_2$, both oriented. Then $M_1\times M_2$ is a submanifold of $R_1\times R_2$, which we orient by $o(M_1)\wedge o(M_2)$. Suppose that $M_i$ has a geometric dual $T_i$ of $R_i$, namely, $M_i$ intersects $T_i$ transversally in one point (we do not assume the sign of the intersection is $+1$). Suppose that $T_i$ is coorientable in $R_i$, and let $\eta_{T_i}$ be an $\eta$-form for $T_i$ in $R_i$ (\S\ref{ss:thom}).
Then $T_1\times T_2$ is a geometric dual of $M_1\times M_2$ in $R_1\times R_2$, and moreover the following identity holds.
\begin{equation}\label{eq:int_eta}
 \int_{M_1\times M_2}\pr_1^*\eta_{T_1}\wedge \pr_2^*\eta_{T_2}=\int_{M_1}\eta_{T_1}\int_{M_2}\eta_{T_2}. 
\end{equation}
Indeed, the sign of this integral is determined by the sign of the evaluation
\[ (\pr_1^*\eta_{T_1}\wedge \pr_2^*\eta_{T_2})(o(M_1)\wedge o(M_2))=\pr_1^*\eta_{T_1}(o(M_1))\,\pr_2^*\eta_{T_2}(o(M_2)). \]

%%%
\subsection{Proof of Lemma~\ref{lem:int-eta}}\label{ss:proof-int-eta}

\begin{Lem}[Lemma~\ref{lem:int-eta}]\label{lem:int-eta-2}
We have the following identities.
\begin{enumerate}
\item $\displaystyle\int_{b_\ell^-}\eta_{S(a_\ell)}=(-1)^{kd+k+d-1}$, 
where $k=\dim{a_\ell}$. 

\item $\displaystyle \int_{a_\ell^+}\eta_{S(b_\ell)}=(-1)^{d+k}$, where $k=\dim{a_\ell}$. 

\item $L_{\ell m}^{ij}=(-1)^{d-1}\Lk(b_\ell^i, b_m^j)$ 
for $i,j,\ell,m$ such that $\dim{b_\ell^i}+\dim{b_m^j}=d-1$.
\end{enumerate}
\end{Lem}
\begin{proof}
We assume without loss of generality that $a_\ell$ and $b_\ell$ intersect orthogonally at one point, say $x$, in $\partial V$ if they intersect. Moreover, we assume that $S(a_\ell)$ is orthogonal to $\partial V$ at $x$. To prove (1), we take a Euclidean local coordinate system $(x_1,x_2,\ldots,x_d)$ around $x$, in which $a_\ell$ agrees with the $x_1\cdots x_k$-plane, $b_\ell$ agrees with the $x_{k+1}\cdots x_{d-1}$-plane, the outward normal vector at $x$ corresponds to the positive direction in the $x_d$ coordinate. We let
\[ o(a_\ell)_x=\alpha\,\partial x_1\wedge\cdots\wedge \partial x_k,\quad
o(b_\ell)_x=\beta\,\partial x_{k+1}\wedge\cdots\wedge \partial x_{d-1} \]
for $\alpha=\pm 1$, $\beta=\pm 1$. Then we see that
\[ \begin{split}
  &o(S(a_\ell))_x=(-1)^k\alpha\,\partial x_1\wedge\cdots\wedge\partial x_k\wedge \partial x_d,\quad \\
  &o(S(b_\ell))_x=(-1)^{d-k}\beta\,\partial x_{k+1}\wedge\cdots\wedge\partial x_{d-1}\wedge \partial x_d
\end{split}  \]
by the outward-normal-first convention for the boundary orientations. This implies
\[\begin{split}
 &o_V^*(S(a_\ell))_x=(-1)^{d-1}\alpha\,\partial x_{k+1}\wedge\cdots\wedge \partial x_{d-1},\quad\\
&o_V^*(S(b_\ell))_x=(-1)^{k(d-k)+d-k}\beta\,\partial x_1\wedge\cdots\wedge \partial x_k.
\end{split}\]
(See \S\ref{ss:coori} for the convention of coorientation.) 
By comparing $o(b_\ell)_x$ and $o_V^*(S(a_\ell))_x$, we get
\begin{equation}\label{eq:int_b-eta-alpha}
 \int_{b_\ell^-}\eta_{S(a_\ell)}=(-1)^{d-1}\alpha\beta. 
\end{equation}
Now we recall that $\alpha$ and $\beta$ are related by the condition $\Lk(b_\ell^-,a_\ell)=+1$. More precisely, suppose that the embeddings $b_\ell^-$ and $a_\ell$ are localy given near $x$ by
\[ \begin{split}
  &b_\ell^-(x_{k+1}',\ldots,x_{d-1}')=(0,\ldots,0,x_{k+1}',\ldots,x_{d-1}',-\ve)\quad(\ve>0),\\
  &a_\ell(x_1'',\ldots,x_k'')=(x_1'',\ldots,x_k'',0,\ldots,0,0). 
\end{split} \]
Applying the rule of \S\ref{ss:ori-prod}, we have
\[ o(b_\ell^-\times a_\ell)_{(x',x'')}=o(b_\ell^-)_{x'}\wedge o(a_\ell)_{x''}
=\alpha\beta\,\partial x_{k+1}'\wedge\cdots\wedge \partial x_{d-1}'
\wedge\partial x_1''\wedge\cdots\wedge \partial x_k'', \]
where $x'=(x_{k+1}',\ldots,x_{d-1}')$, $x''=(x_1'',\ldots,x_k'')$.
To obtain $\Lk(b_\ell^-,a_\ell)$, we compute
\[ \phi(b_\ell^-(x'),a_\ell(x''))
=\frac{a_\ell(x'')-b_\ell^-(x')}{|a_\ell(x'')-b_\ell^-(x')|}
=\frac{(x_1'',\ldots,x_k'',-x_{k+1}',\ldots,-x_{d-1}',\ve)}{|(x_1'',\ldots,x_k'',-x_{k+1}',\ldots,-x_{d-1}',\ve)|}, \]
and we have that $\phi^*\mathrm{Vol}_{S^{d-1}}$ at $(x',x'')=(0,0)$ is a positive multiple of 
\[\begin{split}
& (-1)^{d-1}\ve\,dx_1''\wedge\cdots\wedge dx_k''\wedge d(-x_{k+1}')\wedge\cdots\wedge d(-x_{d-1}')\\
&=(-1)^k(-1)^{k(d-1-k)}\ve\,dx_{k+1}'\wedge\cdots\wedge dx_{d-1}'\wedge dx_1''\wedge\cdots\wedge dx_k''.
\end{split} \]
Thus we have
\[ 1=\Lk(b_\ell^-,a_\ell)=\int_{b_\ell^-\times a_\ell}\phi^*\mathrm{Vol}_{S^{d-1}}
=(-1)^{kd+k}\alpha\beta.
 \]
By (\ref{eq:int_b-eta-alpha}), we obtain (1).

The assertion (2) follows by using the coorientation of $S(b_\ell)$ and the value of $\alpha\beta$ obtained above, as 
\[ \int_{a_\ell^+}\eta_{S(b_\ell)}=(-1)^{k(d-k)}(-1)^{d-k}\alpha\beta
=(-1)^{kd+d}(-1)^{kd+k}=(-1)^{d+k}.
\]

The assertion (3) follows from $\int_{b_\ell^i\times b_m^j}\omega=\Lk(b_\ell^i, b_m^j)$, and 
\[ \begin{split}
&\int_{b_\ell^{i-}\times b_m^{j-}}\eta_{S(a_\ell^i)}\wedge \eta_{S(a_m^j)}=\int_{b_\ell^{i-}}\eta_{S(a_\ell^i)}\int_{b_m^{j-}} \eta_{S(a_m^j)} \\
&=(-1)^{kd+k+d-1}\times(-1)^{k'd+k'+d-1}=(-1)^{(k+k')(d+1)}=(-1)^{d-1}
\end{split}\]
by (\ref{eq:int_eta}) and (1), where $k=\dim{a_\ell^i}$ and $k'=\dim{a_m^j}=d-1-k$. 
\end{proof}

Now we check the compatibility of the coorientations of $S(\widetilde{a}_\ell),S(\widetilde{b}_\ell)$ in $\widetilde{V}$ of type II, induced by the orientations $o(\widetilde{a}_\ell), o(\widetilde{b}_\ell)$ fixed in (\ref{eq:ori-a-tilde}). Let $(t,x)\in S^{d-3}\times \partial V=\partial \widetilde{V}$ be a point on $\partial S(\widetilde{a}_\ell)$ or $\partial S(\widetilde{b}_\ell)$, where they intersect as in the local model in the proof of Lemma~\ref{lem:int-eta-2}. By assumption, the restrictions of $S(\widetilde{a}_\ell)$ and $S(\widetilde{b}_\ell)$ near $\partial\widetilde{V}$ is canonically identified with those of $S^{d-3}\times S(a_\ell)$ and $S^{d-3}\times S(b_\ell)$, respectively. According to the outward-normal-first convention for the boundary orientations, the orientations (\ref{eq:ori-a-tilde}) induce
\begin{equation}\label{eq:ori-Sa-tilde}
\begin{split}
  &o(S(\widetilde{a}_\ell))_{(t,x)}=o(S^{d-3})_t\wedge o(S(a_\ell))_x,\quad \\
  &o(S(\widetilde{b}_\ell))_{(t,x)}=o(S^{d-3})_t\wedge o(S(b_\ell))_x,
\end{split}
\end{equation}
which make sense at $(t,x)$, even if $S(\widetilde{a}_\ell)$ etc. may not agree with $S^{d-3}\times S(a_\ell)$ etc.
\begin{Lem}\label{lem:ori-Sa-compatible}
Suppose that the type II family $\widetilde{V}$ is oriented so that $o(\widetilde{V})_{(t,x)}=o(S^{d-3})_t\wedge o(V)_x$ at $(t,x)\in \partial\widetilde{V}$. Then the coorientations $o_{\widetilde{V}}^*(S(\widetilde{a}_\ell))_{(t,x)}$ (resp. $o_{\widetilde{V}}^*(S(\widetilde{b}_\ell))_{(t,x)}$) and $o_V^*(S(a_\ell))_x$ (resp. $o_V^*(S(b_\ell))_x$) are compatible. Namely, the restriction of the tensor $o_{\widetilde{V}}^*(S(\widetilde{a}_\ell))_{(t,x)}$ (resp. $o_{\widetilde{V}}^*(S(\widetilde{b}_\ell))_{(t,x)}$) to a fiber $\{t\}\times \partial V$ agrees with $o_V^*(S(a_\ell))_x$ (resp. $o_V^*(S(b_\ell))_x$).
\end{Lem}
\begin{proof}
By (\ref{eq:ori-Sa-tilde}), $o(S(a_\ell))_x\wedge o_V^*(S(a_\ell))_x=o(V)_x$, $o(S(b_\ell))_x\wedge o_V^*(S(b_\ell))_x=o(V)_x$, and $o(\widetilde{V})_{(t,x)}=o(S^{d-3})_t\wedge o(V)_x$, we have
\[
  o_{\widetilde{V}}^*(S(\widetilde{a}_\ell))_{(t,x)}=o_V^*(S(a_\ell))_x,\quad
  o_{\widetilde{V}}^*(S(\widetilde{b}_\ell))_{(t,x)}=o_V^*(S(b_\ell))_x.
\]
This completes the proof.
\end{proof}

%%%%%%%%%%%%%%%%%%%%%%%%%%%%%%%
\mysection{Well-definedness of Kontsevich's characteristic class}{s:wd-kon}

%%%%%%%%%%%%%%%%%
\subsection{Integral along the fiber (e.g., \cite[\S{6}]{BTu}, \cite[Ch.VII]{GHV})}\label{ss:pushforward}

\begin{Prop}[{Generalized Stokes theorem, e.g., \cite[Ch.VII]{GHV}}]\label{prop:stokes}
For a $p$-form $\alpha$ on the total space of a fiber bundle $\pi\colon E\to B$ with compact oriented $n$-dimensional fiber with $n\leq p$, the following identity holds.
\[ d\pi_*\alpha=\pi_*d\alpha+(-1)^{p-n}\pi^\partial_*\alpha, \]
where $\pi^\partial\colon \partial^vE\to B$ is the restriction of $\pi$ to the fiberwise boundary\footnote{The sign convention is different from that of \cite{Wa2}, where the boundary was oriented by the inward-normal-first convention.}.
\end{Prop}
The following identities for the pushforward, which are direct consequences of the definition of $\pi_*$, will be frequently used.
\begin{equation}\label{eq:push-pull}
\pi_*(\pi^*\beta\wedge \alpha)=\beta\wedge \pi_*\alpha
\end{equation}
for forms $\alpha$ on $E$ and $\beta$ on $B$.
If $\pi\colon E\to B$ is an orientation preserving diffeomorphism between oriented manifolds, then (\ref{eq:push-pull}) gives 
\begin{equation}\label{eq:push-pull-2}
 \pi_*(\pi^*\beta)=\beta.
\end{equation}
When $\deg\,\alpha=\dim{E}$, we have
\begin{equation}\label{eq:push-pull-3}
 \int_B\pi_*\alpha=\int_E\alpha,
\end{equation}
by the definition of $\pi_*$.

We need to consider pushforward in a fiber bundle with fiber a manifold with corners. In general, the map $\bConf_r(X)\to \bConf_s(X)$ induced by the forgetful map $\Conf_r(X)\to \Conf_s(X)$ may not be a submersion and pushforwards may produce non-smooth forms. We need only to consider pushforwards of submersions for our purpose, in which case we have smooth forms as in the following lemma, whose proof is standard.
\begin{Lem}
Suppose that $\pi\colon E\to B$ is a fiber bundle with fiber a compact oriented $n$-manifold with corners. 
Then pushforward of a smooth form on $E$ gives a smooth form on $B$.
\end{Lem}

%%%%%%%%%%%%%%%%%
\subsection{Family of codimension 1 strata}\label{ss:family-codim-1}

According to the description of the codimension 1 strata of $\partial\bConf_v(S^d;\infty)$, the codimension 1 strata of $E\bConf_v(\pi)$ in $\partial^vE\bConf_v(\pi)$ are parametrized by subsets $\Lambda\subset\{1,2,\ldots,v,\infty\}$ such that $|\Lambda|\geq 2$. Let 
\begin{equation}\label{eq:pi_Lambda}
\pi_\Lambda\colon E\overline{S}_\Lambda(\pi)\to B
\end{equation}
denote the $\overline{S}_\Lambda$-bundle associated to the given bundle $\pi\colon E\to B$. 

If $\infty\notin \Lambda$, the stratum $E\overline{S}_\Lambda(\pi)$ can be written as
\begin{equation}\label{eq:S_A-1}
 E\overline{S}_\Lambda(\pi)\cong E\bConf_{v,\Lambda}(\pi)\times \bConf_r^*(\R^d). 
\end{equation}
Here, $r=|\Lambda|$, the identification is induced by the vertical framing $\tau_E$ at the multiple point, and $E\bConf_{v,\Lambda}(\pi)$ is the total space of the $\bConf_{v,\Lambda}(S^d;\infty)$-bundle associated to $\pi$. Recall that $\bConf_{v,\Lambda}(S^d;\infty)\cong \bConf_{v-r+1}(S^d;\infty)$. Under the identification (\ref{eq:S_A-1}), the restriction of $\omega(\Gamma)$ can be written as
\begin{equation}\label{eq:omega-S_A-1}
 \omega(\Gamma)|_{E\overline{S}_\Lambda(\pi)}=\pm \pr_1^*\,\omega(\Gamma/\Lambda)\wedge\pr_2^*\,\omega(\Gamma_\Lambda), 
\end{equation}
where $\Gamma_\Lambda$ is the subgraph of $\Gamma$ spanned by the vertices labelled by $\Lambda$, $\Gamma/\Lambda$ is the graph obtained from $\Gamma$ by contracting $\Gamma_\Lambda$, $\omega(\Gamma/\Lambda)$ and $\omega(\Gamma_\Lambda)$ are defined similarly as (\ref{eq:omega(Gamma)}), where $\phi_i$ may be replaced with $\phi_i'\colon \bwConf_r(\R^d)\to \bwConf_2(\R^d)=S^{d-1}$ to pullback $\mathrm{Vol}_{S^{d-1}}$ if $i$ is an edge of $\Gamma_\Lambda$. The sign is determined by the permutation $\{1,2,\ldots,e\}\to \{\mbox{edges of $\Gamma/\Lambda$}\}\cup\{\mbox{edges of $\Gamma_\Lambda$}\}$. 

If $\infty\in \Lambda$, then we have
\begin{equation}\label{eq:S_A-2}
 E\overline{S}_\Lambda(\pi)=E\bConf_{N-\Lambda}(\pi)\times \bConf_r^*(\R^d), 
\end{equation}
where $r=|\Lambda|$, $E\bConf_{N-\Lambda}(\pi)$ is the $\bConf_{N-\Lambda}(S^d;\infty)$-bundle associated to $\pi$. Recall that $\bConf_{N-\Lambda}(S^d;\infty)\cong \bConf_{v-r+1}(S^d;\infty)$ and we identify $\bConf_r^*(T_\infty X)$ with $\bConf_r^*(\R^d)$ as in \S\ref{ss:unusual-coord}. Under the identification (\ref{eq:S_A-2}), the restriction of $\omega(\Gamma)$ can be written as
\begin{equation}\label{eq:omega-S_A-2}
 \omega(\Gamma)|_{E\overline{S}_\Lambda(\pi)}=\pm \pr_1^*\,\omega(\Gamma_{\Lambda^c})\wedge\pr_2^*\,\omega(\Gamma/\Lambda^c), 
\end{equation}
where $\Lambda^c=N-\Lambda$, and $\omega(\Gamma_{\Lambda^c})$, $\omega(\Gamma/\Lambda^c)$ are defined similarly as the previous case. The sign is also similar to the previous case.

%%%%%%%%%%%%%%%%%
\subsection{Proof of Theorem~\ref{thm:kon}}

By the generalized Stokes theorem (Proposition~\ref{prop:stokes}), we have
\[ \begin{split}
dI(\Gamma)&=(-1)^{(d-3)k+\ell}\,\bConf_v(\pi)^\partial_*\, \omega(\Gamma)=(-1)^{(d-3)k+\ell}\,\sum_{{\Lambda\subset\{1,\ldots,v,\infty\}}\atop{|\Lambda|\geq 2}}\pi_{\Lambda*}\,\omega(\Gamma).
\end{split} \]
Moreover, by Lemmas~\ref{lem:A1}, \ref{lem:A2} and \ref{lem:A3} below, we have
\[ dI(\Gamma)
=(-1)^{(d-3)k+\ell}\,\sum_{{\Lambda\subset\{1,\ldots,v,\infty\}}\atop{|\Lambda|=2}}\pi_{\Lambda*}\,\omega(\Gamma)
=(-1)^{(d-3)k+\ell+1}\,I(\delta \Gamma),\]
where $\pi_\Lambda$ is the bundle projection (\ref{eq:pi_Lambda}). 
This completes the proof of (1) (that $I$ is a chain map). 

For (2) (independence of $\omega$), we consider the cylinder $\bConf_v(I\times \pi)\colon I\times E\bConf_v(\pi)\to I\times B$, which is a $\bConf_v(S^d;\infty)$-bundle obtained by direct product with $I$. We extend the vertical framing $\tau_E$ on $I\times E$ naturally by the product structure. Now we take two propagators $\omega_0$ and $\omega_1$ on the ends $\{0,1\}\times E\bConf_2(\pi)$. Then by Corollary~\ref{cor:omega-cobordism}, there exists a propagator $\omega$ on $I\times E\bConf_2(\pi)$ for the extended framing that extends both $\omega_0$ and $\omega_1$ on the ends. 
Then the form $\omega(\Gamma)$ on $I\times E\bConf_v(\pi)$ is defined by (\ref{eq:omega(Gamma)}) by using the extended propagator $\omega$. Let $\bConf_v(\pi)^I=\pr\circ \bConf_v(I\times\pi)\colon I\times E\bConf_v(\pi)\to B$, where $\pr\colon I\times B\to B$ is the projection. Then by the generalized Stokes theorem for this $I\times\bConf_v(S^d;\infty)$-bundle, we have
\[ \begin{split}
  d\bConf_v(\pi)^I_*\omega(\Gamma)&=\epsilon\,\bConf_v(\pi)^{I\partial}_*\omega(\Gamma)\\
  &=\epsilon\,\Bigl\{\bConf_v(\pi)_*\omega_1(\Gamma)-\bConf_v(\pi)_*\omega_0(\Gamma)-\int_I\bConf_v(\pi)^\partial_* \omega(\Gamma)\Bigr\},
\end{split} \]
where $\epsilon=(-1)^{(d-3)k+\ell-1}$. This is the identity between $(d-3)k+\ell$-forms on $B$ and $\int_I$ is the pushforward along $I$.
The linear combination of this identity for a $\delta$-cocycle $\gamma=\sum_\Gamma W(\Gamma)\Gamma$ of $P_k\calG^\even_\ell$ gives rise to
\[ \begin{split}
  dI(\gamma)(\omega)&=\epsilon\,\Bigl\{I(\gamma)(\omega_1)-I(\gamma)(\omega_0)+\int_I I(\delta \gamma)(\omega)\Bigr\}\\
  &=\epsilon\,\Bigl\{I(\gamma)(\omega_1)-I(\gamma)(\omega_0)\Bigr\}
\end{split} \]
by a similar argument as in the proof of (1) and by $\delta\gamma=0$. This implies (2). 

We remark that the same proof as in the previous paragraph shows a stronger statement that $I_*$ is invariant even if $\omega(\Gamma)$ were defined by $\bigwedge_{i:\,\mathrm{edge}} \phi_i^*\omega_i$ for propagators $(\omega_1,\ldots,\omega_e)$, which may consist of different forms for different edges, since the proof of (1) does not require that the propagators are the same on the complement of the boundary of $E\bConf_2(\pi)$. 

The assertion (3) (independence of edge orientation) follows since a propagator has a symmetry on the boundary by Lemma~\ref{lem:vol-symmetry} and the assertion (2). A change of the order of the two boundary vertices of an edge gives rise to a diffeomorphism $E\bConf_2(\pi)\to E\bConf_2(\pi)$, the pullback of a propagator under which gives another propagator. Then the result follows by applying the remark in the previous paragraph. 

The assertion (4) (invariance under homotopy of $\tau_E$) can be proved similarly by extending the vertical framing over $I\times E$ by the given homotopy, and by Corollary~\ref{cor:omega-cobordism} again. 

The assertion (5) (naturality under bundle map) follows since the bundle map over $f$ can be used to pullback propagator. Since the integral along the fiber commutes with the pullback by bundle map: $\bConf_v(\pi)_*\tilde{f}^*=f^*\bConf_v(\pi')_*$, the result follows. \qed

\begin{Lem}\label{lem:A1}
When $|\Lambda|\geq 3$, 
\[ \pi_{\Lambda*}\,\omega(\Gamma)=0. \]
\end{Lem}
\begin{proof}
When $\infty\notin \Lambda$, let $\Gamma_\Lambda$ be as defined in \S\ref{ss:family-codim-1}. When $\infty\in \Lambda$, let $\Gamma_\Lambda$ be the $\Gamma/\Lambda^c$ in \S\ref{ss:family-codim-1}. 
There are two cases to be considered.
\begin{enumerate}
\item Every vertex of $\Gamma_\Lambda$ is at least trivalent.
\item $\Gamma_\Lambda$ has a vertex with valence 2, 1 or 0.
\end{enumerate}

\underline{Case (1)}: Suppose that $\Gamma_\Lambda$ has $v'$ vertices and $e'$ edges. The condition (1) implies the inequality
\begin{equation}\label{eq:trivalent}
 2e'-3v'\geq 0. 
\end{equation}
The product structure (\ref{eq:S_A-1}) or (\ref{eq:S_A-2}) and the decomposition (\ref{eq:omega-S_A-1}) or (\ref{eq:omega-S_A-2}) allows us to integrate $\omega(\Gamma_\Lambda)$ first along the fiber $\bwConf_r(\R^d)$, where $r=|\Lambda|=v'$. The integral of $\omega(\Gamma_\Lambda)$ is non-trivial only if $\deg\,{\omega(\Gamma_\Lambda)}=\dim\bwConf_r(\R^d)$,
that is,
\begin{equation}\label{eq:deg_dim}
 (d-1)e'=dv'-d-1, 
\end{equation}
since if $\deg\,{\omega(\Gamma_\Lambda)}<\dim\bwConf_r(\R^d)$ the integral over $\bwConf_r(\R^d)$ vanishes, and if $\deg\,{\omega(\Gamma_\Lambda)}>\dim\bwConf_r(\R^d)$ the result of the integral of $\omega(\Gamma_\Lambda)$ along $\wConf_r(\R^d)$ is a form of positive degree that is the pullback of some form on one point, which vanishes. Now (\ref{eq:trivalent}) and (\ref{eq:deg_dim}) imply $(d-3)v'+2d+2\leq 0$, which is a contradiction when $d\geq 3$. 

\underline{Case (2)}: In this case, we follow \cite[Lemma~2.20]{Les1}, which also uses a symmetry due to Kontsevich (\cite[Lemma~2.1]{Kon}), and \cite[Lemma~2.18]{Les1}\footnote{There are other approaches to prove this lemma (\cite{LV,KT}), which work with compactifications.}. If $\Gamma_\Lambda$ has a bivalent vertex, say $a$, then there are two edges of $\Gamma_\Lambda$ incident to $a$, say with the boundary vertices $\{a,b\}$ and $\{a,c\}$, respectively. Here, we may assume that $b\neq c$, as we may assume $\Gamma$ does not have multiple edges, since otherwise $\omega(\Gamma)=0$ {\it if $d$ is even}. Let $C$ be the subset of $\Conf_r(\R^d)$ consisting of configurations $\bvec{x}=(\ldots,x_a,x_b,x_c,\ldots)$ such that $x_b+x_c-x_a=x_e$ for some $e\neq a$, where we assume the points are labelled by $\Lambda$. Then $C$ is a disjoint union of codimension $d$ submanifolds, which has measure 0. We consider $\Conf_r^*(\R^d)$ as the subspace of $\Conf_r(\R^d)$  by letting 
\[ \wConf_r(\R^d)=\bigl\{(y_1,\ldots,y_r)\in(\R^d)^r\mid |y_1|^2+\cdots+|y_r|^2=1,\,y_i\neq y_j\mbox{ if }i\neq j,\,y_r=0\bigr\}. \]

We consider the automorphism $\iota_\Lambda\colon \Conf_r^*(\R^d)-C\to \Conf_r^*(\R^d)-C$, which 
\[ \mbox{takes $x_a$ to $x_a':=x_b+x_c-x_a$ and fixes other points}. \]
\begin{figure}%
\[ \includegraphics[height=35mm]{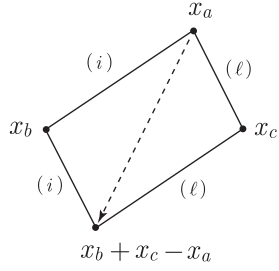}\]%
\caption{The automorphism $\iota_\Lambda$.}\label{fig:involx}%
\end{figure}%
See Figure~\ref{fig:involx}. Note that $C\cap \Conf_r^*(\R^d)$ is codimension $d$ in $\Conf_r^*(\R^d)$, too. 
Then $\iota_\Lambda^*\omega(\Gamma_\Lambda)=-\omega(\Gamma_\Lambda)$ because 
\[ \iota_\Lambda^*(\phi_i'^*\upsilon\wedge \phi_\ell'^*\upsilon)=\iota_\Lambda^*\phi_i'^*\upsilon\wedge \iota_\Lambda^*\phi_\ell'^*\upsilon=\phi_\ell'^*\upsilon\wedge \phi_i'^*\upsilon=-\phi_i'^*\upsilon\wedge \phi_\ell'^*\upsilon \]
($\upsilon=\mathrm{Vol}_{S^{d-1}}$) and $\iota_\Lambda^*$ acts trivially on other edge forms. Here the relations $\iota_\Lambda^*\phi_i'^*\upsilon=\phi_\ell'^*\upsilon$ etc. follow from the commutativity of the following diagram and Lemma~\ref{lem:vol-symmetry}.
\[ \xymatrix{
  \Conf_r^*(\R^d) \ar[r]^-{\phi_\ell'} \ar[d]_-{\iota_\Lambda} & S^{d-1} \ar[d]^-{\iota} & (\ldots,x_a,x_b,x_c,\ldots) \ar@{|->}[r] \ar@{|->}[d] & \frac{x_a-x_c}{|x_a-x_c|} \ar@{|->}[d] \\
  \Conf_r^*(\R^d) \ar[r]^-{\phi_i'} & S^{d-1} & (\ldots,x_a',x_b,x_c,\ldots) \ar@{|->}[r] & \frac{x_a'-x_b}{|x_a'-x_b|}
} \]
Moreover, the automorphism $\iota_\Lambda$ preserves the orientation of $\wConf_r(\R^d)-C$. Since the integral of $\omega(\Gamma_\Lambda)$ on the noncompact manifold $\wConf_r(\R^d)-C$ is absolutely convergent and $C$ has measure zero, we have that the integral over $\wConf_r(\R^d)$ can be replaced with that over $\wConf_r(\R^d)-C$, and
\[ \int_{\wConf_r(\R^d)-C}\omega(\Gamma_\Lambda)=\int_{\iota_\Lambda(\wConf_r(\R^d)-C)}\omega(\Gamma_\Lambda)=\int_{\wConf_r(\R^d)-C}\iota_\Lambda^*\omega(\Gamma_\Lambda)=-\int_{\wConf_r(\R^d)-C}\omega(\Gamma_\Lambda).\]
Note that the integral depends on the orientation of the domain of integral. 
Hence the integral $\pi_{\Lambda*}\omega(\Gamma)$ vanishes. 

If $\Gamma_\Lambda$ has a univalent vertex, say $a$, then there is an edge $i$ of $\Gamma_\Lambda$ incident to $a$, say with the boundary vertices $\{a,b\}$. 
Let $\wConf_{r-1,i}(\R^d)=\wConf_{r-1}(\R^{d-1})\times S^{d-1}$. 
We consider the map $q\colon \wConf_r(\R^d)\to \wConf_{r-1,i}(\R^d)$ given by 
\[ 
  q(x_1,\ldots,x_r)=\,(\mu x_1,\ldots,\widehat{\mu x}_a,\ldots,\mu x_r,(x_a-x_b)/|x_a-x_b|)
\]
(the factor $\mu x_a$ deleted), where $\mu=1/\sqrt{1-|x_a|^2}$. Then the form $\omega(\Gamma_\Lambda)$ restricted to $\wConf_r(\R^d)$ is basic with respect to $q$, namely, it is the pullback of some $(d-1)e'$-form on the manifold $\wConf_{r-1,i}(\R^d)$ of one less dimension since $r=|\Lambda|\geq 3$. It follows that the integral of $\omega(\Gamma_\Lambda)$ over $\wConf_r(\R^d)$ is zero. 
The case where $\Gamma_\Lambda$ has a zerovalent vertex is similar to this case. 
\end{proof}

\begin{Lem}\label{lem:A2}
When $|\Lambda|=2$ and $\infty\in \Lambda$, 
\[ \pi_{\Lambda*}\,\omega(\Gamma)=0. \]
\end{Lem}
\begin{proof}
If $\Lambda=\{j,\infty\}$ for some $j\neq \infty$, and if $j$ has valence $\ell$ in $\Gamma$, then the form $\omega(\Gamma/\Lambda^c)$ on $\bwConf_2(\R^d)$ in (\ref{eq:omega-S_A-2}) is $(\mathrm{Vol}_{S^{d-1}})^\ell$ for the volume form $\mathrm{Vol}_{S^{d-1}}$ on $\bwConf_2(\R^d)=S^{d-1}$, which vanishes. 
\end{proof}

\begin{Lem}\label{lem:A3}
When $|\Lambda|=2$ and $\infty\notin \Lambda$,
\[ \pi_{\Lambda*}\,\omega(\Gamma)=-I(\Gamma/\Lambda,\mbox{induced ori}). \]
\end{Lem}
\begin{proof}
Let $\Lambda=\{a,b\}\subset N$. We first describe the orientation on the stratum $\overline{S}_\Lambda$ induced from that of $\bConf_v(S^d;\infty)$. The stratum $\overline{S}_\Lambda$ is the face produced by the blow-up along the locus $\{x_a=x_b\}$. A neighborhood of a generic point of $\overline{S}_\Lambda$ can be canonically identified with that of a generic point of $\partial B\ell_{\Delta_{\R^d}}(\R^d\times \R^d)\times (\R^d)^{n-2}$ in $B\ell_{\Delta_{\R^d}}(\R^d\times \R^d)\times (\R^d)^{n-2}$. Here, the order of the factors $\R^d$ is not important since $d$ is even and their permutation does not affect the orientation. Coordinates on $\R^d\times\R^d$ with respect to the decomposition $\Delta_{\R^d}\times \Delta_{\R^d}^\perp$ are given by the map
\[ \R^d\times \R^d\to \Delta_{\R^d}\times \Delta_{\R^d}^\perp;\quad
(t,t')\mapsto \Bigl(\Bigl(\frac{t+t'}{2},\frac{t+t'}{2}\Bigr),\Bigl(\frac{t-t'}{2},\frac{t'-t}{2}\Bigr)\Bigr).\]
We fix the following identifications
\begin{equation}\label{eq:tangent-normal}
 \begin{split}
  &\varpi\colon \R^d\stackrel{\cong}{\to} \Delta_{\R^d};\quad\varpi(t)=(t,t),\\
  &\varpi^\perp\colon \R^d\stackrel{\cong}{\to} \Delta_{\R^d}^\perp;\quad\varpi^\perp(t)=(-t,t).
\end{split} 
\end{equation}
The pushforwards of the orientation $\partial t=\partial t_1\wedge\cdots\wedge \partial t_d$ of $\R^d$, where $\partial t_i=\frac{\partial}{\partial t_i}$, gives
\[ \begin{split}
\varpi_*(\partial t)\wedge \varpi^\perp_*(\partial t)&=(\partial t_1+\partial t'_1)\wedge\cdots\wedge(\partial t_d+\partial t_d')\wedge (\partial t'_1-\partial t_1)\wedge\cdots\wedge(\partial t_d'-\partial t_d)\\
&=2^d\,\partial t\wedge \partial t', 
\end{split}\]
which agrees with the orientation of $\R^d\times \R^d$. Thus $\varpi_*(\partial t)$ and $\varpi^\perp_*(\partial t)$ give natural orientations on the subspaces $\Delta_{\R^d}$ and $\Delta_{\R^d}^\perp$. 

Since $B\ell_{\Delta_{\R^d}}(\R^d\times \R^d)=\Delta_{\R^d}\times B\ell_{\{0\}}(\Delta_{\R^d}^\perp)$, it suffices to determine the orientation induced on $\Delta_{\R^d}\times \partial B\ell_{\{0\}}(\Delta_{\R^d}^\perp)$ from $\varpi_*(\partial t)\wedge \varpi^\perp_*(\partial t)$ by the outward-normal-first convention. Further, as $\varpi_*(\partial t)$ is of even degree, we need only to determine the induced orientation of $\partial B\ell_{\{0\}}(\R^d)$ from $\partial t$. Since the outward normal vector at a point $u$ of $\partial B\ell_{\{0\}}(\R^d)=S^{d-1}$ is the preimage of $-u$ under the blow-down map, the induced orientation on $\partial B\ell_{\{0\}}(\R^d)$ is $-\mathrm{Vol}_{S^{d-1}}$. Thus we have obtained the following formula of the orientation of $\partial B\ell_{\Delta_{\R^d}}(\R^d\times\R^d)\times(\R^d)^{n-2}$:
\begin{equation}\label{eq:ori-dS_A}
 -\varpi^\perp_*(\mathrm{Vol}_{S^{d-1}})\wedge \varpi_*(\partial t^{(a)})\wedge \bigwedge_{j\neq a,b}\partial t^{(j)}, 
\end{equation}
where we identified $\partial B\ell_{\{0\}}(\R^d)$ with the unit sphere $S^{d-1}\subset \R^d$ via the isotopy in $B\ell_{\{0\}}(\R^d)$ generated by the preimages of the radial rays from the origin. 

Next, we need to determine the sign caused by the permutation of propagators in $\omega(\Gamma)$. Namely, as in (\ref{eq:omega-S_A-2}), one may transform as
\begin{equation}\label{eq:omega-dS_A}
 \omega(\Gamma)|_{E\overline{S}_\Lambda(\pi)}=\pm p_2^*\omega(\Gamma_\Lambda)\wedge p_1^*\omega(\Gamma/\Lambda)=p_2^*\omega(\Gamma_\Lambda)\wedge (\pm p_1^*\omega(\Gamma/\Lambda)). 
\end{equation}
The term $\pm p_1^*\omega(\Gamma/\Lambda)$ corresponds to the induced orientation $o(\Gamma/i)$ in (\ref{eq:induced-ori}). Hence it turns out that the $\pm$ is in fact $+$. By (\ref{eq:ori-dS_A}) and (\ref{eq:omega-dS_A}), the integral along the fiber gives
\[ \pi_{\Lambda*}\omega(\Gamma)=-I(\Gamma/\Lambda,\mbox{induced ori}). \]
\end{proof}

%%%%%%%%%%%%%%%%%%%%%%%%%%%%%%%
\mysection{Homology class of the diagonal}{s:diag}

\begin{Prop}\label{prop:homol-diag}
Let $S$ be a closed oriented manifold. Suppose that $H_*(S;\Z)$ is free and has finite $\Z$-bases $\{e_i\}$ and $\{e_i^*\}$, which are represented by oriented submanifold cycles $\{\gamma_i\}$ and $\{\gamma_i^*\}$, respectively, and are dual to each other, namely, $\gamma_i\cdot \gamma_j^*=\delta_{ij}$ (the algebraic intersection number, $\alpha\cdot\beta=0$ if $\dim{\alpha}+\dim{\beta}\neq \dim{S}$). Then we have
\[ [\Delta_S]=\sum_{i} e_i\otimes e_i^* \]
in $H_*(S\times S;\Z)$. 
\end{Prop}
This can be deduced from the cohomology version in \cite[Theorem~11.11]{MS}, except for a sign.
\begin{proof}
By assumption, there are $\eta$-forms $\eta_i$, $\eta_i^*$ for some submanifolds $T_i,T_i^*$ of $S$ such that
\begin{equation}\label{eq:gamma-eta}
 \int_{\gamma_i}\eta_j=\delta_{ij},\quad \int_{\gamma_i^*}\eta_j^*=\delta_{ij}. 
\end{equation}
By the duality of the bases $\{e_i\}$ and $\{e_i^*\}$, we have
\[ T_i\cdot T_j^*=\pm 1. \]
We assume without loss of generality that the intersection of $T_i$ and $T_j^*$ is transversal for each $i,j$. 
By the K\"{u}nneth formula, $[\Delta_S]$ can be written as
\[ [\Delta_S]=\sum_{i,j} c_{ij}[\gamma_i\times \gamma_j^*] \]
for some integers $c_{ij}$. First, we have
\[ \int_{\gamma_i\times \gamma_j^*} \pr_1^*\eta_k\wedge \pr_2^*\eta_\ell^*=\int_{\gamma_i}\eta_k\int_{\gamma_j^*}\eta_\ell^*=\delta_{ik}\delta_{j\ell}. \]
Hence we have
\[ \int_{\Delta_S}\pr_1^*\eta_i\wedge \pr_2^*\eta_j^*=\sum_{k,\ell} c_{k \ell}\int_{\gamma_k\times\gamma_\ell^*}\pr_1^*\eta_i\wedge \pr_2^*\eta_j^*=c_{ij}
 \]
and that $c_{ij}=\delta_{ij}$.
Indeed, $\pr_1^*\eta_i\wedge \pr_2^*\eta_j^*$ has support on a small neighborhood of $\Delta_S\cap (T_i\times T_j^*)$. If $(x,x)\in\Delta_S\cap (T_i\times T_j^*)$, then $x$ is an intersection point of $T_i$ and $T_j^*$. Thus the integral counts the intersection $T_i\cap T_j^*$ with signs and the integral is nonzero only if $i=j$. When $i=j$ and $x\in T_i\cap T_j^*$, let $x_1,\ldots,x_k$ be coordinates of $T_x(T_i)^\perp$ and let $x_{k+1},\ldots,x_n$ be coordinates of $T_x(T_j^*)^\perp$. In these coordinates, the orientation of $\Delta_S$ induced from the given one of $S$ is given by
\begin{equation}\label{eq:ori-delta}
 \bigwedge_{a=1}^n (\partial x_a + \partial x_a'), 
\end{equation}
where $(x_1',\ldots,x_n')$ are the coordinates of a copy of $T_xM$ corresponding to $(x_1,\ldots,x_n)$. 
It follows from the assumption (\ref{eq:gamma-eta}) that 
\[ \eta_i(\partial x_1,\ldots,\partial x_k)>0,\quad \eta_j^*(\partial x_{k+1},\ldots,\partial x_n)>0. \]
Hence the form $\pr_1^*\eta_i\wedge \pr_2^*\eta_j^*$ is of the form
\[ f\,dx_1\wedge\cdots \wedge dx_k\wedge dx_{k+1}'\wedge\cdots\wedge dx_n' \]
for some positive function $f$ supported on a neighborhood of $(x,x)$. The evaluation of the volume $\pr_1^*\eta_i\wedge \pr_2^*\eta_j^*$ for (\ref{eq:ori-delta}) at $(x,x)$ is 
\[ f(x,x)\,\prod_{i=a}^k dx_a(\partial x_a + \partial x_a') \prod_{a=k+1}^n dx_a'(\partial x_a + \partial x_a')=f(x,x)>0. \]
The result follows.
\end{proof}
\end{appendix}

%%%%%%%%%%%%%%%%%%%%%%%%%%%%%%%
%%%%%%%%%%%%%%%%%%%%%%%%%%%%%%%

\end{document}